\pdfoutput=1
\documentclass[11pt,a4paper,reqno]{amsart}

\usepackage{times}
\usepackage{xspace}
\usepackage[inline]{enumitem}
\usepackage[numbers,sort&compress,longnamesfirst]{natbib}
\usepackage{geometry}
\usepackage[colorlinks=true, linkcolor=blue ,citecolor=blue]{hyperref}
\usepackage[capitalize]{cleveref}
\usepackage{dsfont}
\usepackage{enumitem}
\usepackage{soul}
\usepackage{amsfonts,comment}
\usepackage{amsmath,mathtools}
\usepackage{amssymb,amsmath,latexsym}
\numberwithin{equation}{section}
\usepackage{ color, graphicx}
\usepackage{todonotes}
\usepackage{stmaryrd}
\usepackage[mathscr]{euscript}

\def\<{\langle}
\def\>{\rangle}

\newcommand{\be}{\begin{equation}}
\newcommand{\ee}{\end{equation}}
\newcommand{\bea}{\begin{eqnarray}}
\newcommand{\eea}{\end{eqnarray}}
\newcommand{\beas}{\begin{eqnarray*}}
\newcommand{\eeas}{\end{eqnarray*}}

\theoremstyle{plain}
\newtheorem{theorem}{Theorem}[section]
\newtheorem{lemma}[theorem]{Lemma}
\newtheorem{proposition}[theorem]{Proposition}
\newtheorem{corollary}[theorem]{Corollary}

\newtheorem{definition}[theorem]{Definition}

\theoremstyle{definition}
\newtheorem{remark}[theorem]{Remark}

\newtheorem{example}[theorem]{Example}

\newcommand{\tnorm}[1]{\left\vert\kern-0.25ex\left\vert\kern-0.25ex\left\vert #1 
    \right\vert\kern-0.25ex\right\vert\kern-0.25ex\right\vert}
\newcommand{\numberthis}{\addtocounter{equation}{1}\tag{\theequation}}

\def\\ud sum{\displaystyle\sum}
\def\Prod{\displaystyle\prod}

\numberwithin{equation}{section}

\def\ue{\mathrm e}
\def\ud{\mathrm d}

\setlist[enumerate]{label=(\roman*),align=left, leftmargin=*}
\crefname{subsection}{subsection}{subsections}

\allowdisplaybreaks

\makeatletter
\newcommand \Dotfill {\leavevmode \leaders \hb@xt@ 6pt{\hss \hss }\hfill \kern \z@}
\makeatother

\makeatletter
\def\@tocline#1#2#3#4#5#6#7{\relax
  \ifnum #1>\c@tocdepth 
  \else
    \par \addpenalty\@secpenalty\addvspace{#2}%
    \begingroup \hyphenpenalty\@M
    \@ifempty{#4}{%
      \@tempdima\csname r@tocindent\number#1\endcsname\relax
    }{%
      \@tempdima#4\relax
    }%
    \parindent\z@ \leftskip#3\relax \advance\leftskip\@tempdima\relax
    \rightskip\@pnumwidth plus4em \parfillskip-\@pnumwidth
    #5\leavevmode\hskip-\@tempdima
      \ifcase #1
       \or\or \hskip 1.65em \or \hskip 3.3em \else \hskip 4.95em \fi%
      #6\nobreak\relax
    \Dotfill
    \hbox to\@pnumwidth{\@tocpagenum{#7}}\par
    \nobreak
    \endgroup
  \fi}
\makeatother

\makeatletter
\def\l@section{\@tocline{1}{0pt}{1pc}{}{\scshape}}
\renewcommand{\tocsection}[3]{%
\indentlabel{\@ifnotempty{#2}{\ignorespaces#1 #2.\hskip 0.7em}}#3}
\def\l@subsection{\@tocline{2}{0pt}{1pc}{5pc}{}}

\def\l@subsubsection{\@tocline{3}{0pt}{1pc}{7pc}{}}

\makeatother

\begin{document}

\frenchspacing

\title[Existence, uniqueness and propagation of chaos for general BSDE\MakeLowercase{s}]{Existence, uniqueness and propagation of chaos\\for general McKean--Vlasov and mean-field BSDE\MakeLowercase{s}}

 \author[A. Papapantoleon]{Antonis Papapantoleon}
 \author[A. Saplaouras]{Alexandros Saplaouras}
 \author[S. Theodorakopoulos]{Stefanos Theodorakopoulos}

 \address{TU Delft \& Forth}
 \email{a.papapantoleon@tudelft.nl}

 \address{University of the Aegean}
 \email{alexandros.saplaouras@aegean.gr}

 \address{TU Berlin}
 \email{stefanos.theodorakopoulos@tu-berlin.de}

\thanks{AS gratefully acknowledges the financial support from the Hellenic Foundation for Research and Innovation Grant No. 235 (2nd Call for H.F.R.I. Research Projects to support Post-Doctoral Researchers).
 ST gratefully acknowledges the financial support from the Hellenic Foundation for Research and Innovation Grant No. 05724 (3rd Call for H.F.R.I. Scholarships for PhD candidates).}



\date{}

\begin{abstract}
We consider backward stochastic differential equations (BSDEs) with mean-field and McKean--Vlasov interaction in their generators in a general setting, where the drivers are square--integrable martingales, typically with independent increments, and the filtrations are (possibly) stochastically discontinuous.
In other words, we consider discrete- and continuous-time systems of mean-field BSDEs and McKean--Vlasov BSDEs in a unified setting.
Moreover, the measure in the generator depends on all components of the solution process.
We provide existence and uniqueness results for these BSDEs using adapted \textit{a priori} estimates that utilize the stochastic exponential. 
Then, we derive propagation of chaos results for systems of particles that satisfy BSDEs, \textit{i.e.} we show that the asymptotic behavior of the solutions of mean-field systems of BSDEs, as the multitude of the systems grows to infinity, converges to I.I.D. solutions of McKean--Vlasov BSDEs.
We introduce a new  coupling technique for showing the backward propagation of chaos, that makes repeated use of the \textit{a priori} estimates, inequalities for the Wasserstein distance and the ''conservation of solutions'' under different filtrations.
This new approach does not require the solutions of the mean field systems to be exchangeable or symmetric, which also allows to consider cases with multiple generators.
Finally, we deduce convergence rates for the propagation of chaos, under advanced integrability conditions on the solutions of the BSDEs.
\end{abstract}

\maketitle

\tableofcontents

\section{Introduction}\label{sec:Introduction}

The notion of ``propagation of chaos'' received renewed attention, especially from the mathematical economics and mathematical finance communities, when it was used in a series of lectures by \citet{lasry2006jeux,lasry2006jeux1,lasry2007mean} in order to simplify the study of mean-field games.
They introduced ideas from statistical physics in the study of Nash equilibria for stochastic differential games with symmetric interactions, along with \citet{huang2007large,huang2006large}.
In general, problems with a large number of players are notoriously difficult to control. 
However, as statistical physics has shown us, under the appropriate assumptions, the most important being symmetry, one can study the asymptotic behavior of a system as the number of players grows to infinity much more easily.
The first instance of this phenomenon appears already in the (Strong) Law of Large Numbers, which states that, under certain conditions, the asymptotic behavior of the random empirical mean converges to the deterministic mean as the number of samples grows to infinity.
Of course, there is not a single way to mathematically express the notion that players, or agents in economics and finance, or particles in statistical physics, interact with one another and a choice must be made. 
Perhaps inspired by the Law of Large Numbers, the interaction that has been extensively studied in the statistical physics literature is the one that emerges from the empirical measure of the states of the particles. 
Hence, an interaction that involves the empirical measure of the states of the players in a game is called mean-field interaction.

The theory of propagation of chaos was initiated in the 1950's by \citet{kac1956foundations}.
In the process of investigating particle system approximations for some nonlocal partial differential equations arising in thermodynamics, he made an important observation about a characteristic of large systems.
Assume that the behavior of the particles is symmetric and they interact in a weak way, \textit{i.e.} such that its magnitude decreases inversely proportional to the size of the system, maybe due to cancelations of the contributions of different particles.
Then, if the initial positions of the particles are chaotic, here understood as independent and identically distributed, this initial state of the system could be seen asymptotically to propagate (or spread) to the other points in time, when its size grows to infinity.
This idea of propagation of chaos has been used ever since in various topics with several applications; see \citet{sznitman1991topics} for an important treatise and \citet{chaintron2022propagation,Chaintron_Diez_2022, jabin2016mean,jabin2018quantitative} and \citet{malrieu2003convergence} for more recent results and applications. 

In this work, combining the ideas above, we want to study the so-called backward propagation of chaos, \textit{i.e.} the asymptotic behavior of solutions of mean-field systems of backward stochastic differential equations (BSDEs) in a general setting. 
The backward propagation is understood as exhibiting chaotic behavior in the terminal conditions, instead of the initial conditions. 
Let $N\in\mathbb{N}$. 
We model the motion of $N-$particles or players in a closed system as a system of mean-field BSDEs, where the generator of the BSDE depends on the empirical measure of the solution of the whole system of BSDEs.
Let $\overline{\mathbb{S}}^{N}:=(\overline{\mathbf{S}}^{1,N}_\cdot, \dots, \overline{\mathbf{S}}^{N,N}_\cdot)$, where $\overline{\mathbf{S}}^{i,N}_\cdot := \big(Y^{i,N}_\cdot,Z^{i,N}_\cdot, U^{i,N}_\cdot, M^{i,N}_\cdot\big)$ denotes the solution of the $i$-th component of the system, for $i\in\{1,\dots,N\}$.
The $(Y^{i,N})_{i\in\{1,\dots,N\}}-$part of the solution $\overline{\mathbb{S}}^{N}$ satisfies the following BSDE system:
\begin{align}
\begin{multlined}[0.8\textwidth]
Y^{i,N}_t = \xi^{i,N} + \int^{T}_{t}f^{i,N} \big(s,\overline{\mathbf{S}}^{\text{mod},i,N}_s,L^N(\overline{\mathbb{S}}^{\text{mod},N}_s) \big) \, \ud C^{\mathbf{X}^i}_s\\
- \int^{T}_{t}Z^{i,N}_s \,  \ud X^{i,\circ}_s 
- \int^{T}_{t}\int_{\mathbb{R}^n}U^{i,N}_s(x) \, \widetilde{\mu}^{(\mathbb{F}^{1,\dots,N},X^{i,\natural})}(\ud s,\ud x) 
- \int^{T}_{t} \,\ud M^{i,N}_s, 
\label{mfBSDE_instantaneous_triplet}
\tag{$\textrm{MF}_{\textrm{Inst}}^{\overline{\mathbb{S}}^{\textrm{mod}}}$}
\end{multlined}
\end{align}
for  $i\in\{1,\dots,N\}$.
Here the modified solution equals
\begin{gather*}
    \overline{\mathbf{S}}^{\text{mod},i,N}_\cdot
    := \big(Y^{i,N}_\cdot,Z^{i,N}_\cdot  c^i_\cdot,
    \Gamma^{(\mathbb{F}^{1,\dots,N},\mathbf{X}^i,\Theta)}(U^{i,N})_\cdot\big),\\
    \shortintertext{and collectively we denote}
    \overline{\mathbb{S}}^{\text{mod},N}_\cdot := (\overline{\mathbf{S}}^{\text{mod},1,N}_\cdot, \dots, \overline{\mathbf{S}}^{\text{mod},N,N}_\cdot).
\end{gather*}
The mean-field interaction appears in the last argument of the generator $f^{i,N}$ as the respective (random) empirical measure $L^N$, and depends on all components of the solution, \textit{i.e.} on $\overline{\mathbb{S}}^{\text{mod},N}$. 
The notation will be clarified at the end of this and in the next section.
However, let us quickly state the following, for $i\in\{1,\dots,N\}$:
$\xi^{i,N}$ is the terminal condition of the BSDE for the $i$-th player, 
$\mathbf{X}^i:=(X^{i,\circ},X^{i,\natural})$ is a tuple of square-integrable martingales (where one could think of them as the continuous and purely discontinuous part of a martingale) with $C^{\mathbf{X}^i}$ associated to their predictable quadratic covariation,
$M^{i,N}$ is the orthogonal martingale with respect to $\mathbf{X}^{i}$,
$Z^{i,N}$ is the integrand of the It\=o integral with respect to $X^{i,\circ}$,
$U^{i,N}$ is the integrand of the stochastic integral with respect to the random measure associated with $X^{i,\natural}$,
and $\Gamma$ is a suitable function of $U^{i,N}$ chosen from a large class.
In particular, the $\Gamma$ function can be chosen as the Radon--Nikodym derivative of the stochastic integral of $U^{i,N}$ with respect to the jump measure relative to the measure induced by $C^{\mathbf{X}^i}$, for a range of free parameters $\Theta$; see \Cref{Gamma} for more details.
 
Let us return to the interaction of the BSDE system via the empirical measure $L^N$.
Motivated by the earlier discussion, one expects that, by an appropriate application of the law of large numbers, the empirical measures converge to a deterministic law as $N\to\infty$, thus rendering the interactions weaker and weaker.  
Obviously, the absence of the empirical measure translates the above system to $N$ non-interacting equations of the same type.
In order to identify the asymptotic behavior of the $(Y^{i,N})_{i\in\mathbb\{1,\dots,N\}}-$part of the solution $\mathbb{S}^{N}$ of the aforementioned mean-field system as $N\to\infty$, we will need a McKean--Vlasov BSDE of the form
\begin{align}
\begin{multlined}[0.85\textwidth]
Y_t = \xi + \int^{T}_{t}f\big(s,\widetilde{\mathbf{S}}^{\text{mod}}_s,\mathcal{L}(\widetilde{\mathbf{S}}^{\text{mod}}_s)\big) \, \ud C^{(\mathbb{G},\mathbf{X})}_s\\
- \int^{T}_{t}Z_s \,  \ud X^\circ_s - \int^{T}_{t}\int_{\mathbb{R}^n}U_s (x)\, \widetilde{\mu}^{(\mathbb{G},X^{\natural})}(\ud s,\ud x)
 - \int^{T}_{t} \,\ud M_s,\hspace{.3cm}
\label{MVBSDE_instantaneous_triplet}
\tag{$\textrm{MV}_{\textrm{Inst}}^{\widetilde{\mathbf{S}}^{\textrm{mod}}}$}
\end{multlined}
\end{align}
where $\widetilde{\mathbf{S}}_\cdot:= \big(Y_\cdot,Z_\cdot ,U_\cdot, M_\cdot\big)$ 
denotes the solution of the BSDE \eqref{MVBSDE_instantaneous_triplet}
 and 
 \begin{gather*}
\widetilde{\mathbf{S}}^{\text{mod}}_\cdot:= \big(Y_\cdot,Z_\cdot  c^{(\mathbb{G},\mathbf{X})}_\cdot,\Gamma^{(\mathbb{G},\mathbf{X},\Theta)}(U)_\cdot\big).
\end{gather*}
Here, again, the notation will be clarified at the end of this section.
Let us just mention that for a random variable $R$, its law will be denoted by $\mathcal{L}(R)$.

We will first show existence and uniqueness results for mean-field systems of BSDEs of the form \eqref{mfBSDE_instantaneous_triplet} and for McKean--Vlasov BDSEs of the form \eqref{MVBSDE_instantaneous_triplet} in the general setting of \citet{papapantoleon2018existence,Papapantoleon_Possamai_Saplaouras_2023}.
In other words, the driving processes are general square-integrable martingales in general filtrations that are not necessarily stochastically continuous.
Hence, we can consider discrete- and continuous-time equations, as well as continuous and purely discontinuous processes as drivers in a unified setting.

 L\'evy processes and more general semimartingales are natural stochastic processes when modeling phenomena in mathematical finance, because they can capture the real-world behavior of financial markets.
These models acknowledge that financial asset prices do not always move smoothly; instead, they may present abrupt changes due to unexpected events.
Moreover, the distribution of log-returns exhibits fat-tails and skews, which are captured easily by the laws of such processes. 
We refer the reader to \citet{Eberlein_Kallsen_2019, Tankov2003} or \citet{Schoutens_2003} for more details and references on models with jumps in finance.
Analogous considerations also hold for models in physics, see, for example \citet{metzler2000random}.

Our general framework allows to model the driving dynamics with continuous-time processes such as L\'evy processes and other semimartingales, with discrete-time processes such as random walks, as well as with superpositions of discrete- and continuous-time models, such as the so-called extended Grigelionis processes; see \citet[Def. 2.15]{kallsen1998semimartingale}.
We also refer the interested reader to \citet[Section 4]{papapantoleon2018existence} and \citet[Section 4]{Papapantoleon_Possamai_Saplaouras_2023} for further examples and applications of this general setting.

Moreover, we assume that the generators $f^{i,N}$ and $f$ of the BSDEs can be stochastic Lipschitz, while they can also depend on the initial segments of the paths of the solution $\mathbf{Y}$, \textit{i.e.} on $\mathbf{Y}|_{[0,\cdot]}$ instead of $\mathbf{Y}_\cdot$.
In fact, the results of \cref{sec:well-posedness_mf_MV_BSDES} are equally valid in the case where the dependence in the generator comes from $\mathbf{Y}|_{[0,\cdot-]}$ instead of $\mathbf{Y}|_{[0,\cdot]}$, respectively when we replace $Y_\cdot$ with $Y_{\cdot-}$; see the notation introduced in \eqref{inpath}.

Before we proceed to a comparison with the known literature, we owe a clarification regarding the way the terms ``mean-field'' and ``McKean--Vlasov'' are used. 
Indeed, in the literature so far, the term mean-field has been used to describe a rather general class of equations, which do not necessarily share a common type of dependence, \textit{e.g.}, the dependence of the system does not solely rely on the empirical measure.
Moreover, the term McKean--Vlasov has been occasionally used instead of the term mean-field, thus increasing the confusion around the class under consideration. 
In the previous paragraphs, we have indicated the way we understand a mean-field system, \textit{i.e.} when the dependence involves the empirical measure $L^N$, as well as a McKean--Vlasov equation, \textit{i.e.} when the dependence involves the law $\mathcal{L}$. 
However, the reader should keep in mind that in the articles mentioned below, not all of them understand the respective equations in the same way as used here. 

The term mean-field BSDE appeared for the first time in the works of \citet{Buckdahn_Djehiche_Li_Peng_2009} and \citet{buckdahn2009mean}.
 The aim of the former work, which was generalized by \citet{lauriere2022backward}, 
 is to analyze the properties of a specific approximation of the solution of a mean-field BSDE,
 while the aim of the latter is to provide a probabilistic interpretation of semilinear McKean--Vlasov PDEs by means of a mean-field BSDE.
 Since then, the literature on both types of BSDEs has expanded along the classical directions, \textit{e.g.},
 for controlled forward-backward dynamics see \citet{CarmonaDelarue2015MVFBSDE};
 for a generator exhibiting quadratic behavior see  
 \citet{HaoWenXiong2022Quadratic}, while for one exhibiting local monotonicity see \citet{Boufoussi2023};
 when a type of reflection is introduced see \citet{BRIAND20207021,LI201447,LiLuo2012,briand2021particles,PengLuo2024MFBSDE};
 for solutions in $L^p-$spaces see \citet{ChenXingZhang2020};
 when the fractional Brownian motion acts as the driver of the BSDE see \citet{Shi2021}.
 Finally, it is worth mentioning that the literature related to graphon-type systems has also been developed, see \textit{e.g.} \citet{BayraktarWuZhang2023} and the references therein.
 For applications in risk measures see \citet{ChenDumitrMincaSulem2023MFBSDE}, while for numerical approximations based on neural networks see 
\citet{Germain2022_MVBSDE_NN}.
Let us also mention here the very recent work of \citet{Agarwal_DeMarco_Gobet_Lopez-Salas_Noubiagain_Zhou_2024}, where a class of ``anticipative'' McKean--Vlasov BSDEs are studied in relation to initial margin computations for CCPs; the anticipativity means that the BSDEs depend on $Y|_{[\cdot,T]}$ and not just on $Y_\cdot$. 
See also the volumes of \citet{Carmona_Delarue_2018_I,Carmona_Delarue_2018_II}.
All the above works are presented under the umbrella of Brownian drivers, wherein we abusively include the work dealing with fractional Brownian motion.
In contrast, the literature related to jump processes is rather limited. 
Indeed, for the case of L\'evy drivers there are, to the best of our knowledge, only the works of \citet{djehiche2021propagation}, which further deals with the reflected case, and \citet{AminiCaoSulem2021GraphonMFBSDE}. 
The results presented in the current work generalize \citet{lauriere2022backward}, as well as \citet{djehiche2021propagation} in the non-reflected case, to the case the driving martingales of the stochastic integrals are square-integrable martingales, possibly both purely discontinuous. 

As a next step, we provide backward propagation of chaos results, \textit{i.e.} we show that as the size $N$ of the mean-field system of BSDEs \eqref{mfBSDE_instantaneous_triplet} grows to infinity, then this system converges to $N$ non-interacting McKean--Vlasov BSDEs of the form \eqref{MVBSDE_instantaneous_triplet}.
More precisely, for the propagation of chaos property the current framework allows for square--integrable martingales with independent increments as integrators of stochastic integrals, c\`adl\`ag predictable increasing processes as integrators of  Lebesgue--Stieltjes integrals, and it also allows the dependence of the generator from the initial segments of the paths of the solution $\mathbf{Y}$.
In fact, the results of \cref{sec:propagation} are also valid without change in case the dependence in the generator comes from $\mathbf{Y}|_{[0,\cdot-]}$ instead of $\mathbf{Y}|_{[0,\cdot]}$; see again the notation introduced in \eqref{inpath}. 
Similarly, in the instantaneous dependence case, we can replace $Y_\cdot$ with $Y_{\cdot-}$.

Although the propagation of chaos property has been extensively studied for (forward) stochastic differential equations (SDEs), see \textit{e.g.} the review articles of \citet{chaintron2022propagation,Chaintron_Diez_2022}, only a handful of papers have been published so far on the backward propagation of chaos; see \citet{Buckdahn_Djehiche_Li_Peng_2009,Hu_Ren_Yang_2023,djehiche2021propagation,buckdahn2009mean, LiDu2024, BayraktarWuZhang2023} and \citet{lauriere2022backward}.
Most of these works employ the coupling techniques pioneered by \citet{sznitman1991topics} in the proof of the propagation of chaos property; some recent references can be found in \citet{cardaliaguet2019master} and \citet{delarue2020master}.

The present work generalizes the aforementioned articles on the backward propagation of chaos in several directions. On the one hand, we consider more general driving martingales and also combine the discrete- and continuous-time settings in a single one.
On the other hand,  we introduce a new coupling technique, which allows us to obtain results under weaker conditions.
A crucial step in the proofs of backward propagation of chaos is to control the quantities that involve the solutions of the mean-field systems using quantities that involve the solutions of the McKean--Vlasov BSDEs, and thus be able to use that the latter are independent and identically distributed. 
In that respect, our method combines the triangle inequality of the Wasserstein distance, the simple inequality \eqref{ineq:empirical_measures_bound} and the \textit{a priori} estimates of \cref{lem:a_priori_estimates}, while the consideration of these solutions under a common probability space is made possible by the ``conservation of solutions'' result of \cref{lem:conserv_laws}; see also \cref{remark 2.3}. 
Then, when we consider all the solutions of each mean-field system simultaneously, the application of the aforementioned combination of techniques enables us to finally deduce the backward propagation of chaos.

Typically in the study of backward propagation of chaos,  the terminal conditions $(\xi^{i,N})_{i \in \{1,\dots,N\}}$ for the $N$ mean-field system are either assumed to be the same with the terminal conditions $(\xi^i)_{i \in \{1,\dots,N\}}$ of the corresponding first $N$ McKean--Vlasov equations or, less commonly, exchangeable, such that the differences $\{\xi^{i,N} - \xi^i\}_{i \in \{1,\dots,N\}}$ are identically distributed.
In the exchangeable case, it is then assumed that under some appropriate norms we have at the limit that $\xi^{i,N}  \xrightarrow[]{N \rightarrow \infty} \xi^i$, for every $i \in \mathbb{N}$. 
In the present work, the method of the proofs does not require such assumptions. 
Instead, it is sufficient to assume that 
\begin{align*}
    \|\xi^{i,N} - \xi^i\|^2\xrightarrow[N \rightarrow \infty]{|\cdot|} 0 , \text{  for every }i \in \mathbb{N},
\end{align*} 
and that the convergence happens at a uniform rate, \textit{i.e.} 
\begin{align*}
    \frac{1}{N} \sum_{i = 1}^{N} \|\xi^{i,N} - \xi^i\|^2 \xrightarrow[N \rightarrow \infty]{|\cdot|} 0 ,
\end{align*}
under an appropriate norm, which will be defined in the next section.
Moreover, in the literature, only the case where the generators remain the same for each particle and each system is studied. 
However, here, we can allow the generators to vary between both particles and systems, under the condition that every generator converges to a common generator in the limit.
More specifically, a sufficient condition is the existence of suitable functions $\{(g^{i,N})_{i\in\{1,\dots,N\}}\}_{N\in\mathbb{N}}$ such that $\|f^{i,N}_{\cdot}-f_{\cdot}\|_{\infty}^2\leq \alpha_{\cdot}^4 g^{i,N}_{\cdot} \hspace{0.1cm}\mathbb{P}\otimes C-$a.e., where $\alpha, A$
depend on the Lipschitz condition of $(\{f^{i,N}\}_{i\in\mathscr{N}})_{N\in\mathbb{N}},f$, which satisfy
\[
\mathbb{E}\Big[\int_{0}^{T}g^{i,N}_s \ud\mathcal{E}\big(\hat{\beta}A\big)_s\Big]\xrightarrow[N\rightarrow\infty]{|\cdot|} 0 
    \quad \text{and} \quad 
\frac{1}{N}\sum_{i=1}^{N}\mathbb{E}\Big[\int_{0}^{T}g^{i,N}_s \ud\mathcal{E}\big(\hat{\beta}A\big)_s\Big]\xrightarrow[N\rightarrow\infty]{|\cdot|} 0.
\]
We remark that above although $f^{i,N}$ were assumed stochastic, $f$ was assumed deterministic. But that was only for simplicity reasons as it can also be stochastic, under the obvious modifications in the proofs, where as it is customary we can work with copies of a stochastic driver and data of a prototypic probability space.

Finally, we provide convergence rates for the propagation of chaos in case the BSDEs depend on the instantaneous value of the solutions, \textit{i.e.} on $Y_\cdot$, under advanced integrability conditions for the solutions of the BSDEs.
These results generalize those found in \citet{lauriere2022backward} from the Brownian framework to processes with independent increments. 
\citet{lauriere2022backward} prove also that the requirements of the theorems, \textit{i.e.} the advanced integrability conditions on the solutions, can be satisfied under an additional specific linear growth condition for the generator.
Their method of proof relies on a delicate application of Girsanov's theorem in combination with a version of Gronwall's inequality. 
In the present work, we do not study under what conditions for the generator these integrability requirements can be satisfied, but we can say that they are trivially satisfied when the generator is bounded; see  \cref{rem:on_integrability_of_rates}\,\ref{bounded condition}. 
The method of \citet{lauriere2022backward} is not directly applicable in our setting, mainly due to the extra dependence of the generator on the process $U(\cdot)$.

This paper is organized as follows:
In \cref{sec:prelims}, we provide notation and preliminary results that are required for the remainder of this work.
In \cref{sec:apriori}, we provide new \textit{a priori} estimates for BSDEs utilizing the stochastic exponential.
In \cref{sec:well-posedness_mf_MV_BSDES}, we provide existence and uniqueness results for McKean--Vlasov BSDEs and systems of mean-field BSDEs.
In \cref{sec:propagation}, we prove the propagation of chaos property for both cases, when the BSDEs depend on the initial segment of the solution, and when they only depend on the instantaneous value of the solution.
Furthermore, we provide convergence rates for the ``classical'' setting, when the BSDEs depend on the instantaneous value of the solution.
Finally, the appendices contain certain proofs and other auxiliary results.


\subsection{Notation}\label{sec:Notation} 

The nature of the problem and the generality of our setting require a rather cumbersome notation, whose rationale we are going to explain in this part in order to ease the effort of the reader.
Let us start with some basic information. 
The set of positive natural numbers will be denoted by $\mathbb{N}$.
The natural number $N$ will be used exclusively to denote the total number of particles or players.
Let $N\in \mathbb{N}$, then we will use abusively the notation $\mathscr{N} := \{1,\dots,N\}$.
The letters $p,q,m,n, d, i$ and $j$ will denote arbitrary positive natural numbers. 
More precisely, the letters $p$, $n$ and $d$, will denote the dimension of the Euclidean state space of the It\=o integrator, the Euclidean state space of the spatial variable of an arbitrary integer-valued random measure, and the Euclidean state space of the stochastic integrals, respectively. 
The space $\mathbb{R}^{m\times q}$ consists of $m\times q-$matrices with real entries and is endowed with its usual norm, \emph{i.e.}, for $A\in\mathbb{R}^{m\times q}$ holds $|A|:=\textrm{Tr}[AA^\top]$, where $A^{\top}$ is the transpose of $A$. If at any point we want to indicate that we are dealing with a matrix we may alternatively use the symbol $\|\cdot\|$.
The reader should observe that we have suppressed in the norm symbol the notational dependence in the dimension.
The elements of $\mathbb{R}^m$ will be understood as column-vectors.
The letters $i,j$ will denote the index associated with an arbitrary particle or player.
This notation will precede any other notation attributed to the same object in order for the reader to clearly understand the intended particle or player; specific examples will be given at the end of this part.
The letter $\rho$ will be reserved to denote metrics. 
More precisely, for a metric space $\mathscr{X}$ its associated metric will be denoted by $\rho_{\mathscr{X}}$.
For product metric spaces, \emph{e.g.}, $\mathscr{X}=\mathscr{X}_1\times \mathscr{X}_2$, the associated metric will be given by 
$\rho_{\mathscr{X}}^2 := \rho_{\mathscr{X}_1}^2+\rho_{\mathscr{X}_2}^2$.

The symbols $\mathbb{P}, \mathbb{E}$, \emph{etc.} will be used in their obvious probabilistic interpretation.
Let us now declare our notational conventions.
The martingales that are intended to serve as integrators of stochastic integrals will be denoted by $X$ and by $\mathbf{X}$ for pairs of integrators.
More specifically, the It\=o integrator will be denoted by the symbol $\circ$, while a martingale associated with an integer-valued measure will be denoted by the symbol $\natural$.
The pair $(X^{i,\circ},X^{i,\natural})$, for example, will denote the martingales that drive the stochastic integrals for the $i-$th particle or player.
The first element of the pair will be the It\=o integrator, while the jumps of the second element will be associated with the integer-valued measure $\mu^{X^{\natural}}$, for which more details will be presented in the next section.  
The letter $\xi$ will be used to denote terminal conditions.

Since we deal with two different types of BSDEs, \emph{i.e.}, mean-field and McKean--Vlasov, we will associate special symbols to their solutions, especially in the parts of the current work in which both types appear.
The reader already had a taste of our choices in the introduction. 
Let us start with the mean-field BSDE systems.
Their solution will be denoted by an upper bar, except when it is absolutely clear, in which case we will omit it, and associated with an upper index which denotes the number of particles or players. 
The solution of the mean-field system of BSDEs with $N$-players \eqref{mfBSDE_instantaneous_triplet}, for example, is denoted by $\overline{\mathbb{S}}^{N}$.
Let us also address the various parts of the solution.
On the one hand, when we want to focus on the elements of $\overline{\mathbb{S}}^{N}$ associated with the $i-$th particle or player, \emph{i.e.}, $(Y^{i,N},Z^{i,N},U^{i,N}, M^{i,N})$, we will write $\overline{\mathbf{S}}^{i,N}$. 
On the other hand, when we want to focus on a specific part of the solution, \emph{e.g.}, say the $(Y^{i,N})_{i\in\mathscr{N}}-$part of the solution $\overline{\mathbb{S}}^{N}$, we will write $\overline{\mathbb{Y}}^{N}$; the meaning of the symbols $\overline{\mathbb{Z}}^{N}$, $\overline{\mathbb{U}}^{N}$ and $\overline{\mathbb{M}}^{N}$ is analogous.
Moreover, because in the generator we need to use a modification (not to be understood in the probabilistic sense) of part of the solution, we use the symbols $\overline{\mathbf{S}}^{\text{mod},i,N}$ and $\overline{\mathbb{S}}^{\text{mod},N}$; \emph{i.e.}, we have annotated the upper index with the term ``$\text{mod}$''.
Combining all of the above, the meaning of the symbols $\overline{\mathbb{Z}}^{\text{mod,N}}$, $\overline{\mathbb{U}}^{\text{mod,N}}$ should also be clear.
In general, the double upper index ``$i,N$'' associates the object with the $i-$th player of the $N-$player game; for example, $\xi^{i,N}$ is the terminal condition of the $i-$th player in an $N-$player game.
At this point, let us mention that we might abuse notation and rarely use the identity $\overline{\mathbb{S}}^{N} = (\overline{\mathbb{Y}}^{N}, \overline{\mathbb{Z}}^{N}, \overline{\mathbb{U}}^{N}, \overline{\mathbb{M}}^{N})$, which is true up to a permutation of the elements.

Let us now turn our attention to the notation of McKean--Vlasov BSDEs.
The respective solution will be denoted by an upper wide tilde, except when it is absolutely clear, in which case we will omit it.
The solution of the McKean--Vlasov BSDE \eqref{MVBSDE_instantaneous_triplet}, for example, is denoted by $\widetilde{\mathbf{S}}$, while $\widetilde{\mathbf{S}}^{\text{mod}}$ denotes the modified part that appears within the generator.
Once a McKean--Vlasov BSDE is associated with the $i-$th particle or player, the respective solutions will be denoted by  $\widetilde{\mathbf{S}}^i$, and $\widetilde{\mathbf{S}}^{\text{mod},i}$.
In particular, the terminal condition of the $i-$th particle or player will be denoted by $\xi^i$.
Moreover, when we need to collect the solutions of the equations associated with the first $N-$particles or players (resp. first $\hat{k}-$particles or players) for $N\in\mathbb{N}$ diverging to infinity (resp. for $\hat{k}\in\mathbb{N}$ fixed), we will write $\widetilde{\mathbb{S}}^{N}$ and $\widetilde{\mathbb{S}}^{\text{mod},N}$,
\emph{i.e.} $\widetilde{\mathbb{S}}^{N}:=(\widetilde{\mathbf{S}}^1,\dots,\widetilde{\mathbf{S}}^{N})$ and 
$\widetilde{\mathbb{S}}^{\text{mod},N}:=(\widetilde{\mathbf{S}}^{\text{mod},1}, \dots,\widetilde{\mathbf{S}}^{\text{mod},N})$,
(resp. $\widetilde{\mathbb{S}}^{\hat{k}}$ and $\widetilde{\mathbb{S}}^{\text{mod},\hat{k}}$,
\emph{i.e.} $\widetilde{\mathbb{S}}^{\hat{k}}:=(\widetilde{\mathbf{S}}^1,\dots,\widetilde{\mathbf{S}}^{\hat{k}})$ and 
$\widetilde{\mathbb{S}}^{\text{mod},\hat{k}}:=(\widetilde{\mathbf{S}}^{\text{mod},1}, \dots,\widetilde{\mathbf{S}}^{\text{mod},\hat{k}})$).
Additionally, for the $Y-$part of $\widetilde{\mathbb{S}}^{N}$ we will write $\widetilde{\mathbb{Y}}^{N}$, resp. for the $Y-$part of $\widetilde{\mathbb{S}}^{\hat{k}}$ we will write $\widetilde{\mathbb{Y}}^{\hat{k}}$.
The notation for the remaining parts, \emph{i.e.}, 
$\widetilde{\mathbb{Z}}^{N}$, 
$\widetilde{\mathbb{U}}^{N}$, 
$\widetilde{\mathbb{M}}^{N}$,  as well as for  
$\widetilde{\mathbb{Z}}^{\text{mod},N}$, 
$\widetilde{\mathbb{U}}^{\text{mod},N}$, should be obvious in view of the rationale we have adopted.


\section{Preliminaries}
\label{sec:prelims}

In this section, we are going to give a short overview of known results that will be useful later. 
Let $(\Omega,\mathcal{G},\mathbb{G},\mathbb{P})$ denote a complete stochastic basis in the sense of \citet[I.1.3]{jacod2013limit}.\footnote{Given that the filtration satisfies the usual conditions, every $\mathbb{G}-$martingale admits a c\`adl\`ag version. Accordingly, without any further explicit mention, we shall identify each martingale with its c\`adl\`ag version.}
Once there is no ambiguity about the reference filtration, we are going to suppress  the dependence on $\mathbb{G}$.


\subsection{Martingales}
Let us denote by $\mathcal{H}^2(\mathbb{G};\mathbb{R}^q)$ the set of square integrable $\mathbb{G}-$martingal-\newline es, \emph{i.e.}, 
\begin{align*}
\mathcal{H}^2(\mathbb{G};\mathbb{R}^q)
    := \big\{ L:\Omega\times\mathbb{R}_+ {}\longrightarrow \mathbb{R}^q, L \text{ is a $\mathbb{G}-$martingale with }\sup_{t\in\mathbb{R}_+}\mathbb{E}[\vert L_t\vert^2]<\infty \big\}, 
\end{align*}
equipped with its usual norm, 
\begin{align*}
    \Vert L\Vert^2_{\mathcal{H}^2(\mathbb{G};\mathbb{R}^q)}:= \mathbb{E}[ \vert L_\infty \vert^2]= \mathbb{E}\left[{{\rm Tr}[\langle L\rangle_\infty^{\mathbb{G}}]}\right],
\end{align*}
where $\langle L \rangle^{\mathbb{G}}$ denotes the $\mathbb{G}-$predictable quadratic variation of $L$, which is the $\mathbb{G}-$predictable compensator of the $\mathbb{G}-$optional quadratic variation $[L]$.

Let us also define a notion of orthogonality between two square integrable martingales. 
More precisely, we say that $L\in \mathcal{H}^2(\mathbb{G};\mathbb{R}^q)$ and $Y \in \mathcal{H}^2(\mathbb{G};\mathbb{R}^m)$ are (mutually) orthogonal if and only if $\langle L,Y\rangle^{\mathbb{G}}=0$. Moreover, for a subset $\mathcal{M}$ of $\mathcal{H}^2(\mathbb{G};\mathbb{R}^q)$, we denote the space of martingales orthogonal to each component of every element of $\mathcal{M}$ by $\mathcal{M}^{\perp_{\mathbb{G}}}$, \emph{i.e.},
\begin{align*}
\mathcal{M}^{\perp_{\mathbb{G}}}            
  &:=\big\{Y\in\mathcal{H}^2(\mathbb{G};\mathbb{R}),\   \langle Y,L\rangle^{\mathbb{G}}=0\text{ for every } L\in\mathcal{M}\big\}.
\end{align*}
A martingale $L\in\mathcal{H}^2(\mathbb{G};\mathbb{R}^q)$ will be called a \emph{purely discontinuous} martingale if $L_0=0$ and if each of its components is orthogonal to all continuous martingales of $\mathcal{H}^2(\mathbb{G};\mathbb{R}).$ 
Using \cite[Corollary I.4.16]{jacod2013limit} we can decompose $\mathcal{H}^2(\mathbb{G};\mathbb{R}^q)$ as follows
\begin{align}\label{SqIntMartDecomp}
\mathcal{H}^{2}(\mathbb{G};\mathbb{R}^q)
=\mathcal{H}^{2,c}(\mathbb{G};\mathbb{R}^q)\oplus\mathcal{H}^{2,d}(\mathbb{G};\mathbb{R}^q),
\end{align}
where $\mathcal{H}^{2,c}(\mathbb{G};\mathbb{R}^q)$ denotes the subspace of $\mathcal{H}^2(\mathbb{G};\mathbb{R}^q)$ consisting of all continuous square--integrable martingales and $\mathcal{H}^{2,d}(\mathbb{G};\mathbb{R}^q)$ denotes the subspace of $\mathcal{H}^2(\mathbb{G};\mathbb{R}^q)$ consisting of all purely discontinuous square--integrable martingales. 

Let us also provide a classical example of the decomposition of the space of square--integrable martingales; we will later expand this result to a more general setting.
Using \cite[Theorem I.4.18]{jacod2013limit}, any square--integrable $\mathbb{G}-$martingale $L\in\mathcal{H}^{2}(\mathbb{G};\mathbb{R}^q)$ admits a unique (up to $\mathbb{P}-$indistinguishability) decomposition
\begin{align}\label{deco_2.1}
L_\cdot=L_0+L^c_\cdot+L^d_\cdot,
\end{align}
where $L^c_0=L^d_0=0$. 
The process $L^c\in\mathcal{H}^{2,c}(\mathbb{G};\mathbb{R}^q)$ will be called the \emph{continuous martingale part of $L$}, while the process $L^d\in\mathcal{H}^{2,d}(\mathbb{G};\mathbb{R}^q)$ will be called the \emph{purely discontinuous martingale part of $L$}. 
The pair $(L^c, L^d)$ is called the \emph{natural pair of $L$ (under $\mathbb{G}$).}


\subsection{It\=o stochastic integrals}

Using \cite[Section III.6.a]{jacod2013limit}, in order to define the stochastic integral with respect to a square--integrable martingale $X \in \mathcal{H}^2(\mathbb{G};\mathbb{R}^p)$, we need to select a $\mathbb{G}-$predictable, non-decreasing, and right continuous process $C^{\mathbb{G}}$ with the property that 
\begin{align}\label{comp_property}
\langle X\rangle^{\mathbb{G}}_\cdot 
    = \int_{(0,\cdot]} \frac{\ud \langle X\rangle_s^{\mathbb{G}}}{\ud C^{\mathbb{G}}_s}\ud C_s^{\mathbb{G}},
\end{align} 
where the equality is understood component-wise.
In other words, $\frac{\ud \langle X\rangle^{\mathbb{G}}}{\ud C^{\mathbb{G}}}$ is a predictable process with values in the set of  all symmetric, non-negative definite $(p\times p)-$matrices.
Then, we define the set of integrable processes as
\begin{multline*}
\mathbb{H}^2(\mathbb{G},X;\mathbb{R}^{d \times p}) 
    := \Bigg\{ Z:(\Omega\times\mathbb{R}_+,\mathcal{P}^{\mathbb{G}}) 
    {}\longrightarrow (\mathbb{R}^{d\times p},\mathcal{B}(\mathbb{R}^{d\times p})),\\
    \mathbb{E}\left[{\int_0^\infty {\rm Tr}\left[Z_t \frac{\ud \langle X\rangle_t^{\mathbb{G}}}{\ud C_t^{\mathbb{G}}}Z_t^{\top}\right] \ud C_t^{\mathbb{G}}} \right] <\infty \Bigg\},
\end{multline*}
where $\mathcal{P}^{\mathbb{G}}$ denotes the $\mathbb{G}-$predictable $\sigma-$field on $\Omega\times\mathbb{R}_+$; see \cite[Definition I.2.1]{jacod2013limit}. 
The associated stochastic integrals will be denoted either by $Z\cdot X$ or by $\int Z_s \textup{d} X_s$.
In case we need to underline the filtration under which the It\=o stochastic integral is defined, we will write either $(Z\cdot X)^{\mathbb{G}}$ or $(\int Z_s \ud X_s)^{\mathbb{G}}$.
The most important relation of the stochastic integral is the following formula for its predictable quadratic variation (see \cite[Theorem III.6.4.c)]{jacod2013limit}) 
\begin{align*}
\left(Z \frac{\ud \langle X\rangle^{\mathbb{G}}}{\ud C^{\mathbb{G}}} Z^\top\right)\cdot C^{\mathbb{G}} = \langle Z\cdot X\rangle^{\mathbb{G}}.
\end{align*}
Hence, we have the analogon of the usual It\=o isometry 
\begin{align*}
\|{Z}\|_{\mathbb{H}^2(\mathbb{G},X;\mathbb{R}^{d \times p})}^2
    := \mathbb{E}\left[{\int_0^\infty {\rm Tr}\left[Z_t \frac{\ud \langle X\rangle_t^{\mathbb{G}}}{\ud C_t^{\mathbb{G}}}Z_t^\top\right]\ud C_t^{\mathbb{G}}} \right]
= \mathbb{E}\left[{\rm Tr}[\langle Z\cdot X \rangle^{\mathbb{G}}_\infty]\right].
\end{align*}
We will denote the space of It\=o stochastic integrals of processes in 
$\mathbb{H}^2(\mathbb{G},X;\mathbb{R}^{d\times p})$ with respect to $X$ by $\mathcal{L}^2(\mathbb{G},X;\mathbb{R}^d)$.
In particular, for $X^c\in\mathcal{H}^{2,c}(\mathbb{G};\mathbb{R}^p)$ we remind the reader that, by \cite[Theorem III.4.5]{jacod2013limit}, $Z\cdot X^c\in\mathcal{H}^{2,c}(\mathbb{G};\mathbb{R}^d)$ for every $Z\in\mathbb{H}^2(\mathbb{G},X^c;\mathbb{R}^{d\times p})$, \emph{i.e.}, $\mathcal{L}^2(\mathbb{G},X^c;\mathbb{R}^d)\subset\mathcal{H}^{2,c}(\mathbb{G};\mathbb{R}^d)$.


\subsection{Integrals with respect to an integer-valued random measure}

Let us now expand the space, and accordingly the predictable $\sigma$--algebra, in order to construct measures that depend also on the height of the jumps of a stochastic process, that is
\begin{align*}
\big(\widetilde{\Omega},\widetilde{\mathcal{P}}^{\mathbb{G}}\big):=
    \big(\Omega\times\mathbb{R}_{+}\times \mathbb{R}^{n},\mathcal{P}^{\mathbb{G}}\otimes \mathcal{B}({\mathbb{R}^{n}})\big).
\end{align*} 
A measurable function 
$U:\big(\widetilde{\Omega},\widetilde{\mathcal{P}}^{\mathbb{G}}\big)\longmapsto \left(\mathbb{R}^{d},\mathcal{B}({\mathbb{R}^{d}}) \right)$ 
is called \emph{$\widetilde{\mathcal{P}}^{\mathbb{G}}-$measurable function} and, abusing notation, the space of these functions will also be denoted by $\widetilde{\mathcal{P}}^{\mathbb{G}}$.
In particular, we will denote by $\widetilde{\mathcal{P}}^{\mathbb{G}}_+$ the space of non-negative $\widetilde{\mathcal{P}}^{\mathbb{G}}-$measurable functions.\footnote{We will adopt analogous notation for any non-negative measurable function, for example, for a $\sigma-$algebra $\mathcal{A}$, the set $\mathcal{A}_+$ denotes the set of non-negative $\mathcal{A}-$measurable functions. }

We say that $\mu$ is a random measure if $\mu := \{\mu\left( \omega;\ud t,\ud x \right)\}_{\omega\in\Omega}$ is a family of non-negative measures defined on $\left(\mathbb{R}_{+}\times\mathbb{R}^{n},\mathcal{B}(\mathbb{R}_{+})\otimes\mathcal{B}(\mathbb{R}^{n})\right)$, satisfying identically
$\mu\left(\omega;\{0\}\times\mathbb{R}^{n}\right) = 0$. 
Consider a function $U \in \widetilde{\mathcal{P}}^{\mathbb{G}}$, then we define the process
\begin{align*}
U * \mu_\cdot(\omega) :=
            \begin{cases}
            \displaystyle \int_{(0,\cdot]\times\mathbb{R}^{n}} 
U\left(\omega,s,x\right) \mu(\omega; \ud s, \ud x),\textrm{ if } \int_{(0,\cdot]\times\mathbb{R}^{n}} 
|U\left(\omega,s,x\right)| \mu(\omega;\ud s, \ud x)<\infty,\\
        \displaystyle   \infty,\textrm{ otherwise}.
            \end{cases}    
\end{align*}
Let $X\in\mathcal{H}^{2}(\mathbb{G};\mathbb{R}^n)$, we associate to it the $\mathbb{G}-$optional integer-valued random measure $\mu^{X}$ on $\mathbb{R}_+\times\mathbb{R}^{n}$ defined by its jumps via the formula
\begin{align}\label{random measure}
\mu^X(\omega;\ud t,\ud x):= \sum_{s>0} \mathds{1}_{ \{ \Delta X_s(\omega) \neq 0\} } \delta_{(s,\Delta X_s(\omega))}(\ud t,\ud x),
\end{align}
where, for any $z\in\mathbb{R}_{+}\times\mathbb{R}^{n} $, $\delta_z$ denotes the Dirac measure at the point $z$; see also \cite[Proposition II.1.16]{jacod2013limit} which verifies that $\mu^X$ is $\mathbb{G}-$optional and $\mathcal{P}^{\mathbb{G}}$--$\sigma$--finite. 
Notice that $\mu^X(\omega;\mathbb{R}_+\times\{0\})=0.$
Moreover, for a $\mathbb{G}-$predictable stopping time $\sigma$ we define the random variable 
\begin{multline*} 
\int_{\mathbb{R}^{n}}U(\omega,\sigma,x)\mu^{X}(\omega;\{\sigma\}\times \ud x) :=\\
 U(\omega,\sigma(\omega),\Delta X_{\sigma(\omega)}(\omega))\mathds{1}_{\{\Delta 
X_\sigma\neq 0, \, |U(\omega,\sigma(\omega),\Delta 
X_{\sigma(\omega)}(\omega))|<\infty\}}.
\end{multline*}
Since $X\in\mathcal{H}^2(\mathbb{G};\mathbb{R}^n)$, the $\mathbb{G}-$compensator of $\mu^X$ under $\mathbb{P}$ exists, see 
\cite[Theorem II.1.8]{jacod2013limit}. 
This is the unique, up to a $\mathbb{P}-$null set, $\mathbb{G}-$predictable random measure $\nu^{(\mathbb{G},X)}$ on $\mathbb{R}_+\times\mathbb{R}^{n}$ such that it holds for every $U \in \widetilde{\mathcal{P}}^{\mathbb{G}}_+$
\begin{align*}
\mathbb{E}\left[U * \mu^X_\infty \right] 
= \mathbb{E}\left[U * \nu^{(\mathbb{G},X)}_\infty\right].
 \end{align*}

Let $U \in \widetilde{\mathcal{P}}^{\mathbb{G}}_+$ and consider a $\mathbb{G}-$predictable time $\sigma$, whose graph\footnote{See \cite[Notation I.1.22]{jacod2013limit} and the comments afterwards for the distinction between the graph of the stopping time $\sigma$ compared to the graph of the mapping $\sigma:\Omega\to \overline{\mathbb{R}}_+$. 
Namely, the former is the restriction on $\Omega\times \mathbb{R}_+$ of the latter.} is denoted by $\llbracket\sigma\rrbracket$.
We define the random variable 
\begin{align}\label{hat_U}
\widehat{U}_{\sigma}^{(\mathbb{G},X)}
:=
    (U\mathds{1}_{\llbracket \sigma \rrbracket})*\nu^{(\mathbb{G},X)}_\infty.
\end{align}
Using \cite[Property II.1.11]{jacod2013limit}, we have
\begin{align}\label{PropertyII-1-11}
\widehat{U}_{\sigma}^{(\mathbb{G},X)} 
= \mathbb{E} \big[ (U\mathds{1}_{\llbracket \sigma \rrbracket})*\mu^X_\infty \big|\mathcal{G}_{\sigma-}\big]
\end{align}
In particular, for $U = 1$ we define
\begin{align}\label{hat_zeta}
\zeta_{\sigma}^{(\mathbb{G},X)} (\omega)
:=\widehat{1}^{(\mathbb{G},X)}_{\sigma} = \mathds{1}_{\llbracket \sigma \rrbracket}*\nu^{(\mathbb{G},X)}_\infty.
\end{align}

In order to define the stochastic integral of a function $U \in \widetilde{\mathcal{P}}^{\mathbb{G}}$ with respect to the \emph{$\mathbb{G}-$compensated integer-valued random measure 
${\widetilde{\mu}}^{(\mathbb{G},X)}\!:=\mu^X-\nu^{(\mathbb{G},X)}$}, we will consider the following class
\begin{multline*}
G_2(\mathbb{G},\mu^X;\mathbb{R}^d)
:=\bigg\{
U:\big(\widetilde{\Omega},\widetilde{\mathcal{P}}^{\mathbb{G}}\big){}\longrightarrow
\big(\mathbb{R}^{d},\mathcal{B}(\mathbb{R}^{d})\big),\\ 
    \mathbb{E}\bigg[\sum_{t>0} \left|U(t,\Delta X_t)\mathds{1}_{\{\Delta 
X_t\neq0 \}}-\widehat{U}_t^{(\mathbb{G},X)}\right|^2\bigg]
    <\infty
\bigg\}.
\end{multline*}
Any element of $G_2(\mathbb{G},\mu^X;\mathbb{R}^d)$ can be associated to an element of $\mathcal{H}^{2,d}(\mathbb{G};\mathbb{R}^d)$, uniquely up to $\mathbb{P}-$indisting-\newline uishability, via 
\begin{align*}
G_2(\mathbb{G},\mu^X;\mathbb{R}^d)\ni U{}\longmapsto U\star{\widetilde{\mu}}^{(\mathbb{G},X)}\in \mathcal{H}^{2,d}(\mathbb{G};\mathbb{R}^d),
\end{align*}
see \cite[Definition II.1.27, Proposition II.1.33.a]{jacod2013limit} and \citet[Theorem XI.11.21]{he2019semimartingale}.
We call $U\star{\widetilde{\mu}}^{(\mathbb{G},X)}$ the \emph{stochastic integral of $U$ with respect to ${\widetilde{\mu}}^{(\mathbb{G},X)}$}. 
Let us point out that for an arbitrary function of $G_2(\mathbb{G},\mu^X;\mathbb{R}^d)$ the two processes $U*(\mu^X - \nu^{(\mathbb{G},X)})$ and $U\star \widetilde{\mu}^{(\mathbb{G},X)}$ are not equal.\footnote{In the finite-variation case the two processes are identical.} 
We will make use of the following notation for the space of stochastic integrals with respect to ${\widetilde{\mu}}^{X}$ which are square integrable martingales
\begin{align*}
\mathcal{K}^2(\mathbb{G},\mu^X;\mathbb{R}^d)
  := \left\{U\star{\widetilde{\mu}}^{(\mathbb{G},X)},\  U\in G_2(\mathbb{G},\mu^X;\mathbb{R}^d)\right\}. 
\end{align*}
Moreover, by \cite[Theorem II.1.33]{jacod2013limit} or \cite[Theorem 11.21]{he2019semimartingale}, we have
\begin{align*}
\mathbb{E}\left[\langle U\star{\widetilde{\mu}}^{(\mathbb{G},X)}\rangle^{\mathbb{G}}_{\infty}\right]<\infty \text{ if and only if }
U\in G_2(\mathbb{G},\mu^X;\mathbb{R}^d),
\end{align*}
which enables us to define the following more convenient space
\begin{align*}
\mathbb{H}^2(\mathbb{G},\mu^X;\mathbb{R}^d)
 :=\left\{ U:\big(\widetilde{\Omega},\widetilde{\mathcal{P}}^{\mathbb{G}}\big) 
  {}\longrightarrow \big(\mathbb{R}^{d},\mathcal{B}(\mathbb{R}^{d})\big),\ 
  \mathbb{E}\left[
  \text{Tr} 
  \Big[\langle U\star{\widetilde{\mu}}^{(\mathbb{G},X)}\rangle^{\mathbb{G}}_\infty\Big]\right]<\infty \right\},
\end{align*}
and we emphasize that we have the direct identification
$    \mathbb{H}^2(\mathbb{G},\mu^X;\mathbb{R}^d) = G_2(\mathbb{G},\mu^X;\mathbb{R}^d).$
Let us finish this subsection with the following useful formulas
\begin{align*}
\mathbb{E}\left[\text{Tr}\langle U\star{\widetilde{\mu}}^{(\mathbb{G},X)}\rangle^{\mathbb{G}}_{\infty}\right] 
    &= \mathbb{E}\left[\sum_{t>0} \left|U(t,\Delta X_t)\mathds{1}_{\{\Delta X_t\neq0 \}}-\widehat{U}_t^{(\mathbb{G},X)}\right|^2\right]\\
    &= \mathbb{E}\left[\sum_{t>0} \left|\int_{\mathbb{R}^n}U(t,x)\,\mu^{X}(t,\ud x) - \int_{\mathbb{R}^n}U(t,x)\,\nu^{(\mathbb{G},X)}(t,\ud x)\right|^2\right].
\end{align*}


\subsection{Dol\'eans-Dade measure and disintegration} 

Let $\mu$ be a random measure, then we define on $\big(\widetilde{\Omega},\mathcal{G}_{\infty} \otimes \mathcal{B}(\mathbb{R}_+) \otimes \mathcal{B}(\mathbb{R}^n)\big)$ its Dol\'{e}ans-Dade measure $M_{\mu}$ as follows
\begin{align*}
M_{\mu}(\mathds{1}_A) := \mathbb{E}\left[\mathds{1}_A * \mu_{\infty}\right],
\end{align*}
for every $A \in \mathcal{G}_{\infty} \otimes \mathcal{B}(\mathbb{R}_+) \otimes \mathcal{B}(\mathbb{R}^n)$.
In particular, for a given $X\in\mathcal{H}^{2}(\mathbb{G};\mathbb{R}^n)$ with non-trivial purely discontinuous part, $M_{\mu^{X}}$ is $\sigma$--integrable with respect to $\widetilde{\mathcal{P}}^{\mathbb{G}}$.
Therefore, we can define for every $W\in \big(\mathcal{G}_{\infty} \otimes \mathcal{B}(\mathbb{R}_+) \otimes \mathcal{B}(\mathbb{R}^n)\big)_+$ its $M_{\mu^{X}}-$conditional expectation with respect to $\widetilde{\mathcal{P}}^{\mathbb{G}}$, which we denote by $M_{\mu^{X}}[W |\widetilde{\mathcal{P}}]$. 
Furthermore, since $\nu^{(\mathbb{G},X)}$ is the $\mathbb{G}-$compensator of $\mu^X$ under $\mathbb{P}$, by definition, we have for every $W\in \widetilde{\mathcal{P}}^{\mathbb{G}}_+$
\begin{align}\label{DD_eq}
   M_{\mu^{X}}(W)  = M_{\nu^{(\mathbb{G},X)}}(W).
\end{align} 
Let us denote by $|I|$ the map $\mathbb{R}^n\ni x\mapsto |x| + \mathds{1}_{\{0\}}(x)\in (0,\infty)$. 
Then, using the facts that  $M_{\mu^{X}}(|I|^2) < \infty$, $M_{\mu^{X}}(\Omega \times \mathbb{R}_+ \times \{0\}) = 0 = M_{\nu^{X}}(\Omega \times \mathbb{R}_+ \times \{0\})$ and that $|I|^2$ belongs to $\widetilde{\mathcal{P}}^{\mathbb{G}}_+$, we can use in the sequel the finite measures $|I|^2*\mu^{X}$, resp. $|I|^2*\nu^{(\mathbb{G},X)}$, as well as their associated Dol\'eans--Dade measures, for which identically holds
\begin{gather*}
M_{|I|^2*\mu^{X}}(\mathds{1}_A) = M_{\mu^{X}} \left(|I|^2\mathds{1}_A\right),\\
\text{ resp. } M_{|I|^2*\nu^{(\mathbb{G},X)}}(\mathds{1}_A) = M_{\nu^{(\mathbb{G},X)}}\left(|I|^2\mathds{1}_A\right),
\end{gather*}
for every $A \in \mathcal{G}_{\infty} \otimes \mathcal{B}(\mathbb{R}_+) \otimes \mathcal{B}(\mathbb{R}^n)$. Then, for every predictable, increasing process $C$ such that $|I|^2 * \nu^{(\mathbb{G},X)} \ll C$ $\mathbb{P}$--a.s., we get  from \cite[Theorem 5.14]{he2019semimartingale} that $\mathbb{P} \otimes (|I|^2 * \nu^{(\mathbb{G},X)}) \ll \mathbb{P} \otimes C$.

Consider a pair of martingales $\mathbf{X} := (X^\circ,X^{\natural}) \in \mathcal{H}^2(\mathbb{G};\mathbb{R}^p) \times \mathcal{H}^{2}(\mathbb{G};\mathbb{R}^n)$, and define 
\begin{align}\label{def_C}
    C^{(\mathbb{G},\mathbf{X})} 
    := \text{Tr}\left[\langle X^{\circ} \rangle^{\mathbb{G}}\right] + |I|^2 * \nu^{(\mathbb{G},X^{\natural})}.
\end{align}
Using the Kunita--Watanabe inequality in conjunction with the Young inequality, following the arguments of \cite[Proposition II.2.9]{jacod2013limit} we can easily verify that $C^{(\mathbb{G},\mathbf{X})}$ possesses the property described in \eqref{comp_property} with respect to $X^\circ$. 
We underline that $\frac{\ud \langle X^{\circ}\rangle^{\mathbb{G}}}{\ud C^{(\mathbb{G},\mathbf{X})}}$ is a $\mathbb{G}-$predictable process with values in the set of  all symmetric, non-negative definite $(p\times p)-$matrices. 
Moreover, let us define the square-root process
\begin{equation}\label{def_c}
    c^{(\mathbb{G},\mathbf{X})} := \bigg(\frac{\ud \langle X^{\circ} \rangle^{\mathbb{G}}}{\ud C^{(\mathbb{G},\mathbf{X})}}\bigg)^{\frac{1}{2}},
\end{equation}
with values in the set of symmetric, non-negative definite $(p\times p)-$matrices. 
Using the diagonalization property of these matrices and the results of \citet{azoff1974borel}, one can easily show that $c^{(\mathbb{G},\mathbf{X})}$ will also be a $\mathbb{G}-$predictable process with values in the set of  all symmetric, non-negative definite $(p\times p)-$matrices.
Additionally, one notices that we can disintegrate  
$\nu^{(\mathbb{G},X^{\natural})}$, \emph{i.e.}, we can determine kernels 
$K^{(\mathbb{G},\mathbf{X})} :\left(\Omega \times \mathbb{R}_+, \mathcal{P}^{\mathbb{G}}\right) {}\longrightarrow \mathcal{R}\left(\mathbb{R}^n,\mathcal{B}(\mathbb{R}^n)\right)$, where $\mathcal{R}\left(\mathbb{R}^n,\mathcal{B}(\mathbb{R}^n)\right)$ are the Radon measures on $\mathbb{R}^n$, such that 
\begin{align}\label{def:Kernels}
  \nu^{(\mathbb{G},X^{\natural})}(\omega,\ud t,\ud x) = K^{(\mathbb{G},\mathbf{X})}(\omega,t,\ud x)\ud C^{(\mathbb{G},\mathbf{X})}_t(\omega).
\end{align}
The kernels $K^{(\mathbb{G},\mathbf{X})}$ are $\mathbb{P} \otimes C^{(\mathbb{G},\mathbf{X})}-$unique, as can be deduced by a straightforward Dynkin class argument.

At this point, let us make an immediate, yet crucial, remark for the current work.
\begin{remark}\label{rem:immersion_no_change_in_comp}
Assuming that $\mathbb{G}$ is immersed in $\mathbb{H}$, \emph{i.e.}, $\mathbb{H}$ is a filtration such that $\mathcal{G}_t\subseteq \mathcal{H}_t$ for every $t\in\mathbb{R}_+$ and, additionally, it possesses the property that every $\mathbb{G}-$martingale is an $\mathbb{H}-$martingale, then $C^{(\mathbb{G},\mathbf{X})}$ and $C^{(\mathbb{H},\mathbf{X})}$ are indistinguishable.
Indeed, one immediately has that 
\begin{align*}
    \langle X^{\circ}\rangle^{\mathbb{G}} - 
    \langle X^{\circ}\rangle^{\mathbb{H}}
    = (\langle X^{\circ}\rangle^{\mathbb{G}} - X^{\circ}\cdot\left(X^{\circ}\right)^{\top})
    +(X^{\circ}\cdot\left(X^{\circ}\right)^{\top} - \langle X^{\circ}\rangle^{\mathbb{H}})
\end{align*}
is an $\mathbb{H}-$martingale, $\mathbb{H}-$predictable and of finite variation.
Hence, since its initial value is $0$, it is the zero martingale, which proves that $\langle X^{\circ}\rangle^{\mathbb{G}}$ and $\langle X^{\circ}\rangle^{\mathbb{H}}$ are indistinguishable.
One can argue analogously for the processes $|I|^2*\nu^{(\mathbb{G},X^{\natural})}$ and $|I|^2*\nu^{(\mathbb{H},X^{\natural})}$.
In other words, we are allowed to interchange the filtration symbol in the notation of \eqref{def_C}, or even to omit it.
\end{remark}


\subsection{Orthogonal decompositions}
\label{subsec:Orthog_Decomp}
Let us now state the decomposition results that will be used to solve the BSDEs of interest; for more details, we refer to \citet[Section 2.2.1]{papapantoleon2018existence}.

Let $\mathbf{X}:=(X^\circ,X^\natural)\in\mathcal{H}^2(\mathbb{G};\mathbb{R}^p)\times\mathcal{H}^{2,d}(\mathbb{G};\mathbb{R}^n)$ with $M_{\mu^{X^\natural}}[\Delta X^\circ|\widetilde{\mathcal{P}}^{\mathbb{G}}]=0$. 
Using this assumption, we get that for $Y^1 \in \mathcal{L}^2(\mathbb{G},X^\circ;\mathbb{R}^p)$ and $Y^2 \in \mathcal{K}^2(\mathbb{G},\mu^{X^\natural};\mathbb{R})$ it holds that $\langle Y^1,Y^2 \rangle = 0$; see \emph{e.g.} \citet[Theorem 13.3.16]{cohen2015stochastic}.
Now, an arbitrary element of $\mathcal{H}^2(\mathbb{G};\mathbb{R})$, say $Y$, can be decomposed as follows 
\begin{align}\label{eq:orth-deco}
Y= Y_0 + Z\cdot {X}^\circ + U\star\widetilde{\mu}^{X^\natural} + {N},
\end{align}
for some $(Z,U)\in\mathbb{H}^2(\mathbb{G},{X}^\circ;\mathbb{R}^{p})\times \mathbb{H}^2(\mathbb{G},\mu^{X^\natural};\mathbb{R})$ and $N\in\mathcal H^2(\mathbb{G};\mathbb{R})$ such that
 $\langle {X}^{\circ},N\rangle = 0$ and $M_{\mu^{X^\natural}}\big[ \Delta {N} |\widetilde{\mathcal{P}}^{\mathbb{G}} \big]=0$. 
Moreover, this decomposition is unique, up to indistinguishability. 
The interested reader may find the proof at \cite[Proposition 2.5]{papapantoleon2018existence}.

Then we define
\begin{align*}
\mathcal{H}^2(\mathbf{X}^{\perp_{\mathbb{G}}};\mathbb{R}^d)  
    := \big(\mathcal{L}^2(\mathbb{G},X^\circ;\mathbb{R}^{d\times p})\oplus\mathcal{K}^2(\mathbb{G},\mu^{X^\natural};\mathbb{R}^d)\big)^{\perp_{\mathbb{G}}}.
\end{align*}  
Subsequently, we have the following description for $\mathcal{H}^2(\mathbf{X}^{\perp_{\mathbb{G}}};\mathbb{R}^d)$, which is \cite[Proposition 2.6]{papapantoleon2018existence}.

\begin{proposition}\label{prop:CharacterOrthogSpace}
Let $\mathbf{X}:=(X^\circ,X^\natural)\in\mathcal{H}^2(\mathbb{G};\mathbb{R}^p)\times\mathcal{H}^{2,d}(\mathbb{G};\mathbb{R}^n)$ with \newline $M_{\mu^{X^\natural}}[\Delta X^\circ|\widetilde{\mathcal{P}}^{\mathbb{G}}]=0$.
Then
\begin{align*}
\mathcal{H}^2(\mathbf{X}^{\perp_{\mathbb{G}}};\mathbb{R}^d) = \big\{ L\in\mathcal{H}^2(\mathbb{G};\mathbb{R}^d),\ \langle X^{\circ}, L\rangle^{\mathbb{G}}=0 \text{ and } M_{\mu^{X^\natural}}[\Delta L|\widetilde{\mathcal{P}}^{\mathbb{G}}]=0\big\}.
\end{align*}
Moreover, the space $\big(\mathcal{H}^2(\mathbf{X}^{\perp_{\mathbb{G}}};\mathbb{R}^d), \Vert \cdot\Vert_{\mathcal{H}^2(\mathbb{G};\mathbb{R}^d)}\big)$ is closed.
\end{proposition}

Summing up the previous results, we arrive at the decomposition result that is going to dictate the structure of the BSDEs in our setting
\begin{align*}
\mathcal{H}^2(\mathbb{G};\mathbb{R}^d) 
= \mathcal{L}^2(\mathbb{G},X^\circ;\mathbb{R}^d) \oplus \mathcal{K}^2(\mathbb{G},\mu^{X^\natural};\mathbb{R}^d)\oplus \mathcal{H}^2(\mathbf{X}^{\perp_{\mathbb{G}}};\mathbb{R}^d),
\end{align*}
where each of the spaces appearing in the equality above is closed.

\begin{remark}\label{remark 2.3}
 In \cref{lem:conserv_laws} we essentially show that for $\mathbf{X}:=(X^\circ,X^\natural)\in\mathcal{H}^2(\mathbb{G};\mathbb{R}^p)\times\mathcal{H}^{2,d}(\mathbb{G};\mathbb{R}^n)$ with $M_{\mu^{X^\natural}}[\Delta X^\circ|\widetilde{\mathcal{P}}^{\mathbb{G}}]=0$, and a filtration $\mathbb{H}$ such that $\mathbb{G}$ is immersed in $\mathbb{H}$, the above decomposition for elements of $\mathcal{H}^2(\mathbb{G};\mathbb{R}^d)$ is conserved with respect to the corresponding spaces of $\mathbb{H}$.
\end{remark}


\subsection{Stochastic exponential}

Let $A$ be a finite variation, c\`adl\`ag process, and define the process 
\begin{align}\label{stocheq}
    \mathcal{E}(A)_\cdot := \ue^{A_\cdot} \prod_{s \leq \cdot} \frac{1 + \Delta A_s}{\ue^{\Delta A_s}},
\end{align}
which is called the stochastic exponential of $A$. 
Using the trivial inequalities $ 0\leq 1 + x \leq \ue^x$, for all $x \geq -1$, and the usual properties of the jumps of finite variation processes, we can easily see that the above process is well defined, adapted, c\`adl\`ag and locally bounded. 
The main functionality of the above process is that it satisfies the SDE 
\begin{align}\label{equation:stochexp}
\mathcal{E}(A)_t = 1 + \int_{0}^{t} \mathcal{E}(A)_{s-} \ud A_s.
\end{align}
This fact is proved using It\=o's formula, and yields that $\mathcal{E}(A)$ also has finite variation; see \textit{e.g.} \cite[Section 15.1]{cohen2015stochastic}. 

In the sequel, we will need additional properties for a process expressed as a stochastic exponential; therefore, we collect them in the next lemma.
In order to simplify notation, we adopt the following convention: whenever we write $\Delta A \neq -1$, we mean that the set $\{\Delta A_t = -1 \text{ for some }t\in\mathbb{R}_+\}$ is evanescent. 
Analogous will be the understanding for $\Delta A \ge a$, for any $a\in\mathbb{R}_+$, as well as for $\mathcal{E}(A)\neq 0$, \emph{etc}.

\begin{lemma}\label{lem:StochExp}
Let $A$ be a c\`adl\`ag process of finite variation.
\begin{enumerate}[label=\textrm{(\roman*)}]
\item\label{item:StochExp1} If $\Delta A \neq -1$, then $\mathcal{E}(A) \neq 0$.
\item\label{item:StochExp2} 
    If $\Delta A \neq -1$, then 
    $\mathcal{E}(A)^{-1} = \mathcal{E}\left(- \overline{A}\right)$, where $\overline{A}_\cdot := A_\cdot - \sum_{s \leq \cdot} \frac{(\Delta A_s)^2}{(1 + \Delta A_s)}$.
\item\label{item:StochExp3} 
    If $\Delta A \geq - 1$, 
    then $0 \leq \mathcal{E}(A)_\cdot \leq \ue^{A_\cdot}$.
    Actually,  $\Delta A > - 1$ if and only if $\mathcal{E}(A)>0$.
\item\label{item:StochExp4} 
    If $A$ is non-decreasing, then $\mathcal{E}(A)$ is non-decreasing.
    If $A$ is non-increasing with $\Delta A>-1$, then $\mathcal{E}(A)$ is non-increasing.
\item\label{item:StochExp5} 
    If $B$ is another finite variation process, then we have the identity  $$\mathcal{E}(A) \mathcal{E}(B) = \mathcal{E}(A + B +[A,B]),$$ where $[A,B]_\cdot : = \sum _{s \leq \cdot} \Delta A_s \Delta B_s$.
\item\label{item:StochExp6} 
    Let $\overline{A}$ be as defined in \ref{item:StochExp2}, then we have the identities
    \begin{gather*}
    \Delta \overline A_\cdot
      = \frac{\Delta A_{\cdot}}{1 + \Delta A_{\cdot}}
        \ \text{ and }\
        \overline A_\cdot = A_0 + \int_0^\cdot \frac{1}{1+\Delta A_s} \ud A_s.
    \end{gather*}
    Moreover, if $A$ is non-decreasing, then $\mathcal{E}(-\overline{A})$ is non-increasing.
\item\label{item:StochExp7} 
    Let $\gamma,\delta \geq 0$ and $\overline{A}$ as defined in \ref{item:StochExp2}.
    Define 
    \begin{align*}
        \widetilde{A}^{\delta,\gamma}_\cdot:= 
        \delta A_\cdot - \overline{\gamma A}_\cdot - [\delta A, \overline{\gamma A}]_\cdot.
    \end{align*}
    Then, from \ref{item:StochExp2} and \ref{item:StochExp5} it trivially holds
    $\mathcal{E}(\widetilde{A}^{\delta,\gamma}) = \mathcal{E}(\delta A)\mathcal{E}(\gamma A)^{-1}$.
    Therefore,
    \begin{gather*}
    \Delta \widetilde{A}^{\delta,\gamma}_\cdot
     = \frac{\Delta((\delta - \gamma) A)_{\cdot}}{1 + \Delta (\gamma A)_{\cdot}}
    \ \text{ and }\
    \widetilde{A}^{\delta,\gamma}_\cdot 
        = A_0 + \int_0^\cdot \frac{1}{1+\Delta(\gamma A)_s} \ud ((\delta-\gamma)A)_s .
    \end{gather*}
    If $A$ is non-decreasing, then $\Delta (\widetilde{A}^{\delta,\gamma}) > -1$.
\item \label{item:StochExp8}
    Let $A$ be non-decreasing with $\Delta A \leq \Phi$, for some $\Phi \geq 0$.
    Then
    \begin{align*}
    \int_{t}^{T}\mathcal{E}(\gamma A)^{-1}_{s-} \,\ud A_s
    \leq \frac{1 + \gamma\Phi}{\gamma} \mathcal{E}(\gamma A)^{-1}_{t-}.
\end{align*}
\end{enumerate}
\end{lemma}

\begin{proof}
The above properties are fairly standard and follow from relatively simple calculations; one may consult \cite[Section 15.1]{cohen2015stochastic} for \ref{item:StochExp1}--\ref{item:StochExp5}. 
Regarding the other properties, let $\overline{A}$ be understood as in \ref{item:StochExp2}.

One immediately derives \ref{item:StochExp6} and \ref{item:StochExp7} once the respective continuous and discontinuous parts are compared.
Indeed, regarding \ref{item:StochExp6} we have, on the one hand, that the continuous part of $\overline{A}$ is $A^c$, while $\Delta \overline{A}_\cdot = \frac{(\Delta A_\cdot)^2}{1 + \Delta A_\cdot}$.
On the other hand, 
\begin{align*}
    \left( \int_0^t \frac{1}{1 + \Delta A_s}\ud A_s\right)^c =
     \int_0^t \frac{1}{1 + \Delta A_s}\ud A_s^c = A^c_{t},
\end{align*}
where the second equality is true because the process $\Delta A$ is non-zero only countably many times and the measure associated to $A^c$ is atomless.
Moreover, when one compares the respective jump processes for every $t\geq 0$
\begin{align*}
    \Delta \overline{A}_t
    =\Delta A_t - \frac{(\Delta A_t)^2}{1 +\Delta A_t}
    = \frac{\Delta A_t}{1+ \Delta A_t}
    =\Delta \left( \int_0^\cdot \frac{1}{1 + \Delta A_s}\ud A_s\right)_t.
\end{align*}
Let us now verify that $\mathcal{E}(-\overline{A})$ is non-increasing. 
Firstly, let us observe that $\mathcal{E}(-\overline{A})_{-}$ is positive. 
Indeed, from the previous arguments
\begin{align*}
    0\leq \Delta(\overline{A})<1\Leftrightarrow 
    -1< \Delta(-\overline{A})\leq 0,
\end{align*}
which from \ref{item:StochExp3} yields the desired property.
Then, $\overline{A}$ is non-decreasing, or equivalently $-\overline{A}$ is non-increasing, a property inherited from $A$.
Now we can conclude in view of \eqref{equation:stochexp}.

Let us focus now on \ref{item:StochExp7}.
On the one hand, the continuous parts of $\widetilde{A}^{\delta, \gamma}$ and $(\delta-\gamma) A$ are identical.
On the other hand, we have for their discontinuous parts
\begin{align*}
    \Delta( \widetilde{A}^{\delta,\gamma} )
    &= \Delta( \delta A) - \Delta (\overline{\gamma A}) - \Delta( \delta A)\Delta (\overline{\gamma A})\\
    &= \Delta( \delta A) - \frac{\Delta (\gamma A)}{1+\Delta(\gamma A)} - \Delta( \delta A)\frac{\Delta (\gamma A)}{1+\Delta(\gamma A)}\\
    &=\frac{\Delta ((\delta - \gamma) A)}{1+\Delta(\gamma A)}.
    \numberthis\label{jump_tildeA}
\end{align*}
Using the fact that $\Delta A$ is non-zero only countably many times, we can conclude as before that
\begin{align*}
    \widetilde{A}^{\delta,\gamma}_\cdot 
    = A_0 + \int_0^\cdot \frac{1}{1+\Delta(\gamma A)_s} \ud \big((\delta- \gamma) A\big)_s.
\end{align*}
Finally, we only need to verify that $\Delta (\widetilde{A}^{\delta,\gamma}) > -1$, which is trivial in view of the following equivalences
\begin{equation*}
    -1 < 
    \Delta(\widetilde{A}^{\delta,\gamma})
    \overset{\eqref{jump_tildeA}}{=} 
    \frac{\Delta \big( (\delta - \gamma) A\big)}{1 + \Delta (\gamma A)}
    \Longleftrightarrow -1 - \Delta ( \gamma A) < \Delta \big((\delta-\gamma) A\big)
    \Longleftrightarrow -1 < \Delta( \delta A). 
\end{equation*}

We conclude with the proof of \ref{item:StochExp8}.
Using \ref{item:StochExp6} for the jumps of $\overline{\gamma A}$, yields 
    $-1<\Delta(- \overline{\gamma A})\leq 0,$
which further implies that $\mathcal{E}( -\overline{\gamma A}) > 0$; see \ref{item:StochExp3}. 
Then, \ref{item:StochExp6},
\begin{align*}
    \int_{t}^{T}\mathcal{E}(\gamma A)^{-1}_{s-} \,\ud A_s  
    &\overset{\ref{item:StochExp2}}{=}\int_{t}^{T}\mathcal{E}( -\overline{\gamma A})_{s-} \,\ud A_s 
    \overset{\ref{item:StochExp6}}{=}  \frac{1}{\gamma} \int_{t}^{T}(1 + \Delta(\gamma A_s))\mathcal{E}( -\overline{\gamma A})_{s-} \,\ud(\overline{\gamma A})_s\\
    &\leq 
    - \frac{1 + \gamma\Phi}{\gamma} \int_{t}^{T}\mathcal{E}( -\overline{\gamma A})_{s-} \,\ud(-\overline{\gamma A})_s 
    \overset{\eqref{equation:stochexp}}{=} -\frac{1 + \gamma\Phi}{\gamma} \mathcal{E}(-\overline{\gamma A})\big|_t^T \\
    &
    = \frac{1 + \gamma\Phi}{\gamma} \mathcal{E}(\gamma A)^{-1}\big|_T^t 
    \leq \frac{1 + \gamma\Phi}{\gamma} \mathcal{E}(\gamma A)^{-1}_t\\
    &
    \leq \frac{1 + \gamma\Phi}{\gamma} \mathcal{E}(\gamma A)^{-1}_{t-},
\end{align*}
where we used the fact that $\mathcal{E}(\gamma A)^{-1}_\cdot$ is non-increasing. 
\end{proof}


\subsection{Norms and spaces}\label{Norms and spaces}

We will largely follow the notation for norms and spaces of stochastic processes of \citet[Section 2.3]{papapantoleon2018existence}.
However, we will need to additionally keep track of the filtration under which we are working, since later many filtrations will appear in our framework.
We inform the reader that we will, essentially, omit the notational dependence on the time horizon $T$, since this will be fixed at the beginning of each subsequent section.

Let $\mathbf{X}:=(X^\circ,X^\natural)\in \mathcal{H}^2(\mathbb{G};\mathbb{R}^p)\times\mathcal{H}^{2,d}(\mathbb{G};\mathbb{R}^n)$ with 
$M_{\mu^{X^\natural}}[\Delta X^\circ|\widetilde{\mathcal{P}}^{\mathbb{G}}]=0$, and consider $A,C:(\Omega \times \mathbb{R}_{+},\mathcal{P}^{\mathbb{G}}) {}\longrightarrow (\mathbb{R}_{+},\mathcal{B}(\mathbb{R}_+))$ c\`adl\`ag and increasing.
The following spaces will appear in the analysis throughout this article, for $\beta \geq 0$ and $T$ a $\mathbb{G}-$stopping time:

\begin{gather*}
\begin{multlined}[\textwidth]
\mathbb{L}^{2}_{\beta}(\mathcal{G}_{T}, A;\mathbb{R}^d)
:= \bigg\{\xi \text{ is an } \mathbb{R}^{d} -\text{valued, } \mathcal{G}_{T}-  \text{random variable s.t.},\\
\|\xi\|^{2}_{\mathbb{L}^{2}_{\beta}(\mathcal{G}_{T}, A;\mathbb{R}^d)}: = \mathbb{E}\left[\mathcal{E}(\beta A)_{T-} |\xi|^2 \right] < \infty\bigg\},
\end{multlined}\\
\begin{multlined}[\textwidth]
\mathcal{H}^{2}_{\beta}(\mathbb{G},A;\mathbb{R}^d) 
:= \bigg\{M \in \mathcal{H}^{2}(\mathbb{G},A;\mathbb{R}^d) , \\
\|M\|^{2}_{\mathcal{H}^{2}_{\beta}(\mathbb{G},A;\mathbb{R}^d)}:= \mathbb{E} \bigg[\int_{0}^{T}\mathcal{E}(\beta A)_{t-} \, \ud\text{Tr}\big[\langle M \rangle^{\mathbb{G}}\big]_t \bigg] < \infty \bigg\},
\end{multlined}\\
\begin{multlined}[\textwidth]
    \mathbb{H}^{2}_{\beta}(\mathbb{G},A,C;\mathbb{R}^d) 
:= 
 \bigg\{\phi \text{ is an } \mathbb{R}^{d} -\text{valued, } \mathbb{G}-\text{optional process s.t.}\\
 \|\phi\|^{2}_{\mathbb{H}^{2}_{\beta}(\mathbb{G},A,C;\mathbb{R}^d)}:= \mathbb{E} \bigg[\int_{0}^{T}\mathcal{E}(\beta A)_{t-}|\phi|_{t}^2 \, \ud C_t \bigg] < \infty  \bigg\},
\end{multlined}
\\
\begin{multlined}[\textwidth]
\mathcal{S}^{2}_{\beta}(\mathbb{G},A;\mathbb{R}^d) 
:=
\bigg\{\phi \text{ is an } \mathbb{R}^{d} -\text{valued,} \hspace{0.1 cm} \mathbb{G}-\text{optional process s.t.}\\
\|\phi\|^{2}_{\mathcal{S}^{2}_{\beta}(\mathbb{G},A;\mathbb{R}^d)}:= 
\mathbb{E} \Big[\sup_{t \in [0,T]}\big\{\mathcal{E}(\beta A)_{t-}|\phi|_{t}^2\big\} \Big] < \infty  \bigg\},
\end{multlined}\\
\begin{multlined}[\textwidth]
\mathbb{H}^{2}_{\beta}(\mathbb{G},A,X^\circ;\mathbb{R}^{d \times p}) 
:=
    \bigg\{Z \in \mathbb{H}^{2}(\mathbb{G},X^{\circ};\mathbb{R}^{d \times p}),  \\
    \|Z\|^{2}_{\mathbb{H}^{2}_{\beta}(\mathbb{G},A,X^\circ;\mathbb{R}^{d \times p})}:= 
    \|Z\cdot X^{\circ}\|^2_{\mathcal{H}^{2}_{\beta}(\mathbb{G},A;\mathbb{R}^d)}
    < \infty  \bigg\},
\end{multlined}    \\
\begin{multlined}[\textwidth]
\mathbb{H}^{2}_{\beta}(\mathbb{G},A,\mu^{X^\natural};\mathbb{R}^d)
:= 
    \bigg\{U \in \mathbb{H}^{2}(\mathbb{G},\mu^{X^\natural};\mathbb{R}^d),  \\
    \|U\|^{2}_{\mathbb{H}^{2}_{\beta}(\mathbb{G},A,\mu^{X^\natural};\mathbb{R}^d)}:= 
    \|U\star \widetilde{\mu}^{X^{\natural}}\|^2_{\mathcal{H}^{2}_{\beta}(\mathbb{G},A;\mathbb{R}^d)}
    < \infty  \bigg\},
\end{multlined}
\shortintertext{and}
\begin{multlined}[\textwidth]
\mathcal{H}^{2}_{\beta}(\mathbb{G},A,\mathbf{X}^{\perp_{\mathbb{G}}};\mathbb{R}^{d}) 
:= \mathcal{H}^{2}(\mathbf{X}^{\perp_{\mathbb{G}}};\mathbb{R}^d) \cap  \mathcal{H}^{2}_{\beta}(\mathbb{G},A;\mathbb{R}^{d})\\
\text{ with } \|M\|^{2}_{\mathcal{H}^{2}_{\beta}(\mathbb{G},A,\mathbf{X}^{\perp_{\mathbb{G}}};\mathbb{R}^{d})}:= \|M\|^2_{\mathcal{H}^{2}_{\beta}(\mathbb{G},A;\mathbb{R}^d)}.
\end{multlined}
\end{gather*}

\begin{remark}\label{remark terminal random variable}
Although the domain of the paths of stochastic processes we work with is the whole $\mathbb{R}_+$, the quantities we presented above,  which are tailor made for our problem, exist only up to time $T$. 
This creates the usual situation where the objects of our spaces are in fact equivalent classes; this will always be implicitly assumed. 
Hence, we are able to assume, whenever needed, that any object of, for example, $\mathcal{S}^{2}_{\beta}(\mathbb{G},A;\mathbb{R}^d)$ is constant on the interval $[T,\infty)$. 
Later on we will exploit this property in the proofs of the existence and uniqueness theorems.
\end{remark}

Let us now present the norms for the solutions of the BSDEs in our setting.
Let 
\begin{align*}
\begin{multlined}[\textwidth]
    (Y,Z,U,M)\\ \in  
\mathcal{S}^{2}_{\beta}(\mathbb{G},A;\mathbb{R}^d) 
    \times  \mathbb{H}^{2}_{\beta}(\mathbb{G},A,X^\circ;\mathbb{R}^{d\times p}) 
    \times \mathbb{H}^{2}_{\beta}(\mathbb{G},A,\mu^{X^\natural};\mathbb{R}^d) 
    \times  \mathcal{H}^{2}_{\beta}(\mathbb{G},A,\mathbf{X}^{\perp_{\mathbb{G}}};\mathbb{R}^d),
\end{multlined}
\end{align*}
then we define
\begin{align}\label{def:star_norm}
\begin{multlined}[0.9\textwidth]
    \|\left(Y,Z,U,M\right) \|^{2}_{\star,\beta,\mathbb{G},A,\mathbf{X}}:=  \\
    \|Y\|^{2}_{\mathcal{S}^{2}_{\beta}(\mathbb{G},A;\mathbb{R}^d)} 
    + \| Z\cdot X^{\circ} + U\star\widetilde{\mu}^{X^\natural} + M\|^2_{\mathcal{H}^2_\beta(\mathbb{G},A;\mathbb{R}^d)}.
\end{multlined}
\end{align}

In some results which we will present, the $Y-$part of the solution will be required to live in a space smaller than $\mathcal{S}^2_\beta(\mathbb{G},A;\mathbb{R}^d)$. 
That is, it should live in  
\begin{align*}
    \mathscr{S}^{2}_{\beta}(\mathbb{G},A;\mathbb{R}^d) :=\footnotemark
    \big\{Y\in \mathcal{S}^2_{\beta}(\mathbb{G},A;\mathbb{R}^d) : 
    \| \alpha Y\|_{\mathbb{H}_{\beta}(\mathbb{G},A,C;\mathbb{R}^d)}<\infty \big\}
\end{align*}%
\footnotetext{We have omitted the notational dependence on $\alpha$ and $C$ in the newly introduced space and its norm. 
Not only for the sake of simplicity, but also because we foresee the connection between $\alpha, A$ and $C$ in the next sections, which is going to be of the form $A_\cdot=\int_0^\cdot \alpha^2_s \textup{d} C_s$. 
Hence, in the end, the notation we introduce at this point will be consistent later in terms of dependencies.}
with 
\begin{align}
    \label{def:norm_scriptS}
    \| \cdot \|^2_{\mathscr{S}^{2}_{\beta}(\mathbb{G},A,C,\alpha;\mathbb{R}^d)}:=
    \|\cdot\|^2_{\mathcal{S}^2_{\beta}(\mathbb{G},A;\mathbb{R}^d)} + 
    \|\alpha \cdot\|^2_{\mathbb{H}_{\beta}(\mathbb{G},A,C;\mathbb{R}^d)}.
\end{align}
Moreover, for 
\begin{align*}
\begin{multlined}[\textwidth]
    (Y,Z,U,M)\\ \in  
\mathscr{S}^{2}_{\beta}(\mathbb{G},A;\mathbb{R}^d) 
    \times  \mathbb{H}^{2}_{\beta}(\mathbb{G},A,X^\circ;\mathbb{R}^{d\times p}) 
    \times \mathbb{H}^{2}_{\beta}(\mathbb{G},A,\mu^{X^\natural};\mathbb{R}^d) 
    \times  \mathcal{H}^{2}_{\beta}(\mathbb{G},A,\mathbf{X}^{\perp_{\mathbb{G}}};\mathbb{R}^d)
\end{multlined}
\end{align*}
we define the following norm for the solution 
\begin{align}\label{def:blacksquare_norm}
\begin{multlined}[0.9\textwidth]
    \|\left(Y,Z,U,M\right) \|^{2}_{\text{\scalebox{0.5}{$\blacksquare$}},
    \beta,\mathbb{G},A,\mathbf{X}}:=  \\
    \|Y\|^{2}_{\mathscr{S}^{2}_{\beta}(\mathbb{G},A;\mathbb{R}^d)} + \| Z\cdot X^{\circ} + U\star\widetilde{\mu}^{X^\natural} + M\|^2_{\mathcal{H}^2_\beta(\mathbb{G},A;\mathbb{R}^d)}
\end{multlined}
\end{align}
In particular for the martingale part, we underline the immediate equality
\begin{align}\label{identity:norm_ZU}
\begin{multlined}[0.9\textwidth]
    \| Z\cdot X^{\circ} + U\star\widetilde{\mu}^{X^\natural} + M\|^2_{\mathcal{H}^2_\beta(\mathbb{G},A;\mathbb{R}^d)}=\\
    \|Z\|^{2}_{\mathbb{H}^{2}_{\beta}(\mathbb{G},A,X^\circ;\mathbb{R}^{d\times p})} + 
    \|U\|^{2}_{\mathbb{H}^{2}_{\beta}(\mathbb{G},A,\mu^{X^\natural};\mathbb{R}^d)} +  \|M\|^{2}_{\mathcal{H}^{2}_{\beta}(\mathbb{G},A,\mathbf{X}^{\perp_{\mathbb{G}}};\mathbb{R}^d)},
\end{multlined}
\end{align}
which is true because of the assumption $M_{\mu^{X^\natural}}[\Delta X^\circ | \widetilde{\mathcal{P}}^{\mathbb{G}}]=0$ (see also \cite[Theorem 13.3.16]{cohen2015stochastic}), in conjunction with the fact that $M\in\mathcal{H}^{2}_{\beta}(\mathbb{G},A,\mathbf{X}^{\perp_{\mathbb{G}}};\mathbb{R}^d)$.

Later, in the course of computations, we will need to rewrite the norms associated with the spaces $\mathbb{H}^{2}_{\beta}(\mathbb{G},A,$ $X^\circ;\mathbb{R}^{d \times p})$ and $\mathbb{H}^{2}_{\beta}(\mathbb{G},A,\mu^{X^\natural};\mathbb{R}^d) $ in terms of Lebesgue--Stieltjes integrals with respect to $C^{(\mathbb{G},\mathbf{X})}$, for $C^{(\mathbb{G},\mathbf{X})}$ as defined in \eqref{def_C}; one may consult \cite[Lemma 2.13]{papapantoleon2018existence} for the details.
Hence, for $(Z,U) \in 
\mathbb{H}^2_{\beta}(\mathbb{G}, A, X^\circ;\mathbb{R}^{d\times p}) $ $\times \mathbb{H}^2_{\beta}(\mathbb{G}, A,\mu^{X\natural};\mathbb{R}^d)$ we have
 \begin{align}
     \|Z\|^2_{\mathbb{H}^2_{\beta}(\mathbb{G}, A, X^\circ;\mathbb{R}^{d\times p})} = \mathbb{E}\bigg[\int_{0}^{T}\mathcal{E}({\beta} A)_{s-}
     |Z_s c^{(\mathbb{G},\mathbf{X})}_s|^2\ud C^{(\mathbb{G},\mathbf{X})}_s\bigg],
     \label{comp_norm_altern_circ}
 \end{align}
 where
 $c^{(\mathbb{G},\mathbf{X})}$ was defined in \eqref{def_c}, and
 \begin{align}
     \|U\|^2_{\mathbb{H}^2_{\beta}(\mathbb{G}, A,\mu^{X\natural};\mathbb{R}^d)} = \mathbb{E}\bigg[\int_{0}^{T}\mathcal{E}({\beta} A)_{s-}\left(\tnorm{U_s(\cdot)}^{(\mathbb{G},\mathbf{X})}_s\right)^2\ud C^{(\mathbb{G},\mathbf{X})}_s\bigg],\label{comp_norm_altern_natural}
 \end{align}
 where
\begin{align}
\begin{multlined}[0.9\textwidth]
\Big(\tnorm{U_s(\omega;\cdot)}^{(\mathbb{G},\mathbf{X})}_s(\omega)\Big)^2 
 :=
 \int_{\mathbb{R}^n}\big|U(\omega,s,x) - \widehat{U}_s^{(\mathbb{G},\mathbf{X})}(\omega)\big|^2 K^{(\mathbb{G},\mathbf{X})}_s(\omega,\ud x) \label{def:tnorm}\\
 + \big(1 - \zeta^{(\mathbb{G},X^{\natural})}_s(\omega)\big) \Delta C^{(\mathbb{G},\mathbf{X})}_s(\omega) \left|\int_{\mathbb{R}^n}U(\omega,s,x)\,K^{(\mathbb{G},\mathbf{X})}_s(\omega,\ud x)\right|^2,
\end{multlined}
\end{align}
with $K^{(\mathbb{G},\mathbf{X})}$ satisfying \eqref{def:Kernels}, \emph{i.e.},
 \begin{align*}
     \nu^{(\mathbb{G}, X^{\natural})}(\omega,\ud t,\ud x) 
     = K^{(\mathbb{G}, \mathbf{X})}(\omega,t,\ud x) \ud C^{(\mathbb{G},\mathbf{X})}_t(\omega),
 \end{align*} 
and $\zeta^{(\mathbb{G},X^\natural)}_{\cdot}=1_{\llbracket \sigma \rrbracket}*\nu^{(\mathbb{G},X^\natural)}_{\cdot}$, which was defined in \eqref{hat_zeta}.

\begin{remark}\label{rem:notation_beta_zero}
In order to simplify the notation whenever possible, if we consider one of the aforementioned spaces for $\beta=0$, then we will omit $0$. 
As a result, the dependence on the process $A$ is redundant, hence we will also omit the process $A$ from the notation of the respective space.
As an example, $\mathbb{L}^2(\mathcal{G}_T;\mathbb{R}^d)$ denotes the space $\mathbb{L}^2_0(\mathcal{G}_T,A;\mathbb{R}^d)$, which is the classical Lebesgue space.
\end{remark}

\subsection{The \texorpdfstring{\(\Gamma\)}{Gamma} function}\label{Gamma}

The integrand in the stochastic integral with respect to the purely-discontinuous martingale $\widetilde\mu^{X^\natural}$ is a suitable process $U$, and the driver of the BSDE $f$ will, obviously, depend on this process. 
However, in the present work, we cannot simply require that $f$ is Lipscitz in this argument with respect to the $\tnorm{\cdot}$-norm defined in \eqref{def:tnorm}, as done \textit{e.g.} in \citeauthor{papapantoleon2018existence} \cite{papapantoleon2018existence,Papapantoleon_Possamai_Saplaouras_2023}, because this norm depends on the filtration $\mathbb{G}$.
Hence, motivated by applications and connections to PDEs, see \cref{Remark 2.1.6} \ref{Remark 2.1.6. 3}, we define a composition of functions, called the $\Gamma$ function, with a free parameter $\Theta$.

\begin{definition}\label{def_Gamma_function}
Let $\mathbf{X}:=(X^\circ,X^\natural)\in \mathcal{H}^2(\mathbb{G};\mathbb{R}^p)\times\mathcal{H}^{2,d}(\mathbb{G};\mathbb{R}^n)$, 
let $C^{(\mathbb{G},\mathbf{X})}$ be as defined in \eqref{def_C} 
and $K^{(\mathbb{G},\mathbf{X})}$ that satisfies \eqref{def:Kernels}.
Additionally, let $\Theta $ be an $\mathbb{R}$--valued, $\widetilde{\mathcal{P}}^{\mathbb{G}}-$measurable function such that $|\Theta|\leq |I|$; the reader may recall that  $|I|(x) := |x| + \mathds{1}_{\{0\}}(x)$.
Define the process
$\Gamma^{(\mathbb{G}, \mathbf{X},\Theta)} : \mathbb{H}^{2}(\mathbb{G},\mu^{X^\natural};\mathbb{R}^d) {}\longrightarrow \mathcal{P}^{\mathbb{G}}(\mathbb{R}^d)$ such that, for every $s\in \mathbb{R}_+$, holds
\begin{multline*}
\Gamma^{(\mathbb{G},\mathbf{X},\Theta)}(U)_s(\omega) := \\
\int_{\mathbb{R}^n}\left(U(\omega,s,x) - \widehat{U}_s^{(\mathbb{G},X^{\natural})}(\omega)\right)\left(\Theta(\omega,s,x) - \widehat{\Theta}_s^{(\mathbb{G},X^{\natural})}(\omega)\right)\,K^{(\mathbb{G},\mathbf{X})}_s(\omega,\ud x)\\
\hspace{-8em}+ (1 - \zeta^{(\mathbb{G},X^{\natural})}_s(\omega))\Delta C^{(\mathbb{G},\mathbf{X})}_s(\omega)\\ \times\int_{\mathbb{R}^n}U(\omega,s,x)\,K^{(\mathbb{G},\mathbf{X})}_s(\omega,\ud x) \int_{\mathbb{R}^n}\Theta(\omega,s,x)\,K^{(\mathbb{G},\mathbf{X})}_s(\omega,\ud x).
\end{multline*}
\end{definition}

\begin{remark}\label{Remark 2.1.6}
\begin{enumerate}
    \item 
    Given that $X^\natural\in\mathcal{H}^2(\mathbb{G};\mathbb{R}^n)$, it is immediate that the process $\Theta\in\mathbb{H}^{2}(\mathbb{G},\mu^{X^\natural};\mathbb{R})$. 
    Therefore, for any $U\in\mathbb{H}^2(\mathbb{G},\mu^{X^\natural};\mathbb{R}^d)$, the process $\Gamma$ is well-defined $\mathbb{P}\otimes C^{(\mathbb{G},\mathbf{X})}-$a.e.
    \item 
    One would expect in $\Gamma$ a notational dependence on the kernel. 
    However, $K_1^{(\mathbb{G},\mathbf{X})},K_2^{(\mathbb{G},\mathbf{X})}$ satisfying \eqref{def:Kernels} and $U \in \mathbb{H}^{2}(\mathbb{G},\mu^{X^\natural};\mathbb{R}^d)$, so we have $\Gamma^{(\mathbb{G},\mathbf{X},\Theta)}(U)_s = \Gamma^{(\mathbb{G},\mathbf{X},\Theta)}(U)_s,$ $ \hspace{0.2cm} \mathbb{P} \otimes C^{(\mathbb{G},\mathbf{X})}-$a.e.; here, implicitly, the left hand side is defined with respect to $K_1$ and the right hand side with respect to $K_2$.
    In that way, we have uniqueness of the kernels that satisfy \eqref{def:Kernels}.
    Since in the respective computations all the equalities appearing will be taken under $\mathbb{P} \otimes C^{(\mathbb{G},\mathbf{X})}$, we have suppressed the notational dependence on the kernels.
    \item \label{Remark 2.1.6. 3}
    The  choice of $\Gamma$ was based on, and inspired by, applications.
    The reader may recall, for example, the connection between BSDEs and partial integro-differential equations, and the special structure that is required for the generator, see \emph{e.g.}  \citet{barles1997backward} or \citet[Section 4.2]{delong2013backward}. 
    Moreover, one can easily verify that $\Gamma$ is equal to Radon--Nikodym derivative 
    \begin{align*}
        \frac{\ud \langle U \star \widetilde{\mu}^{(\mathbb{G},X^\natural)},\Theta \star \widetilde{\mu}^{(\mathbb{G},X^\natural)} \rangle^{\mathbb{G}}}{\ud C^{(\mathbb{G},\mathbf{X})}}.
    \end{align*}
\end{enumerate}
\end{remark}

The next lemma can be viewed as a version of the Cauchy--Schwarz inequality and, essentially, justifies the definition of the process $\Gamma$.
In other words, the function $\Gamma$ is Lipschitz in the sense described below.
However, note that the inner characteristics of the function $\Gamma$ play no part in what follows in the other sections. If one can prove that results similar to \cref{lem:Gamma_is_Lipschitz} and \cref{lem:equalities_integrals} hold, then the rest will remain valid, under the appropriate modifications.

\begin{lemma}\label{lem:Gamma_is_Lipschitz}
Let $\mathbf{X} \in \mathcal{H}^2(\mathbb{G};\mathbb{R}^p)\times\mathcal{H}^{2,d}(\mathbb{G};\mathbb{R}^n)$ and 
$\Theta \in \widetilde{\mathcal{P}}^{\mathbb{G}}$ be an $\mathbb{R}$--valued function such that $|\Theta|\leq |I|$. 
Then, for every $U^1,U^2 \in \mathbb{H}^{2}(\mathbb{G},\mu^{X^\natural};\mathbb{R}^d)$, we have 
\begin{multline*}
\big| \Gamma^{(\mathbb{G},\mathbf{X},\Theta)}(U^1)_t(\omega) - \Gamma^{(\mathbb{G},\mathbf{X},\Theta)}(U^2)_t(\omega) \big|^2
\\ \leq 2 \left(\tnorm{U^1_t(\omega ; \cdot) - U^2_t(\omega; \cdot)}^{(\mathbb{G},\mathbf{X})}_t(\omega)\right)^2,\hspace{0.2cm}\mathbb{P} \otimes C^{(\mathbb{G},\mathbf{X})}-\text{a.e.}
\end{multline*}
\end{lemma}

\begin{proof}
Let $\delta U := U^1- U^2$ then, by the Cauchy--Schwarz inequality, we get 
\begin{align*}
&\big| \Gamma^{(\mathbb{G},\mathbf{X},\Theta)}(U^1)_t - \Gamma^{(\mathbb{G},\mathbf{X},\Theta)}(U^2)_t \big|^2\\
&
\begin{multlined}[t][\textwidth]
\leq 2
\left|\int_{\mathbb{R}^n}\Big[\delta U_t(x) 
    -  \widehat{\delta U}^{(\mathbb{G},\mathbf{X})}_t\Big]\hspace{0.1cm}
    \left[\Theta_t(x) - \widehat{\Theta}^{(\mathbb{G},\mathbf{X})}_t\right]\,K^{(\mathbb{G},\mathbf{X})}_t(\ud x)\right|^2\\
+ 2 (1 - \zeta^{(\mathbb{G},X^{\natural})}_t)^2
\big(\Delta C^{(\mathbb{G},\mathbf{X})}_t\big)^2 
\left|\int_{\mathbb{R}^n}\delta U_t(x)\,K^{(\mathbb{G},\mathbf{X})}_t(\ud x)\right|^2 \left|\int_{\mathbb{R}^n}\Theta_t(x)\,K^{(\mathbb{G},\mathbf{X})}_t(\ud x)\right|^2 
\end{multlined}\\
&
\begin{multlined}[t][\textwidth]
\leq  2 
\int_{\mathbb{R}^n}\left|\delta U_t(x) -  \widehat{\delta U}^{(\mathbb{G},\mathbf{X})}_t\right|^2\,K^{(\mathbb{G},\mathbf{X})}_t(\ud x) \hspace{0.1cm}\int_{\mathbb{R}^n}\left|\Theta_t(x) - \widehat{ \Theta}^{(\mathbb{G},\mathbf{X})}_t\right|^2\,K^{(\mathbb{G},\mathbf{X})}_t(\ud x)\\
+ 2 (1 - \zeta^{(\mathbb{G},X^{\natural})}_t)^2
\big(\Delta C^{(\mathbb{G},\mathbf{X})}_t\big)^2
\left|\int_{\mathbb{R}^n}\delta U_t(x)\,K^{(\mathbb{G},\mathbf{X})}_t(\ud x)\right|^2
\left|\int_{\mathbb{R}^n}\Theta_t(x)\,K^{(\mathbb{G},\mathbf{X})}_t(\ud x)\right|^2
\end{multlined}\\
&\leq  
2 \Big(\tnorm{U^2_t(\omega \hspace{0.1cm}; \cdot) - U^1_t(\omega \hspace{0.1cm}; \cdot)}^{(\mathbb{G},\mathbf{X})}_t(\omega)\Big)^2 \Big(\tnorm{\Theta_t(\cdot)}^{(\mathbb{G},\mathbf{X})}_t(\omega)\Big)^2,\hspace{0.2cm} \mathbb{P} \otimes C^{(\mathbb{G},\mathbf{X})}-\text{a.e}.
\end{align*}
Using \cite[11.21 Theorem, Part 3)]{he2019semimartingale}, $|\Theta| \leq |I|$ and $X^{\natural} \in \mathcal{H}^{2,d}(\mathbb{G}; \mathbb{R}^n)$, we have
\begin{align*}
    \langle {\Theta} \star \widetilde{\mu}^{(\mathbb{G},X^\natural)} \rangle^{\mathbb{G}}_{\cdot} 
    = |\Theta|^2 * \nu^{(\mathbb{G},X^\natural)}_{\cdot} - \sum_{s \leq \cdot} \big|\widehat\Theta^{(\mathbb{G},\mathbf{X})}_{s}\big|^2.
\end{align*}
Hence, for every $s,u \in \mathbb{Q}_+$ with $s < u$, we get
\begin{multline*}
    \langle \Theta \star \widetilde{\mu}^{(\mathbb{G},X^\natural)} \rangle^{\mathbb{G}}_{u} - \langle \Theta \star \widetilde{\mu}^{(\mathbb{G},X^\natural)} \rangle^{\mathbb{G}}_{s} 
    = |\Theta|^2 * \nu^{(\mathbb{G},X^\natural)}_{u} - |\Theta|^2 * \nu^{(\mathbb{G},X^\natural)}_{s} - \sum_{s<r\leq u} \big| \widehat{\Theta}^{(\mathbb{G},\mathbf{X})}_r\big|^2
    \\
    \leq |\Theta|^2 * \nu^{(\mathbb{G},X^\natural)}_{u} - |\Theta|^2 * \nu^{(\mathbb{G},X^\natural)}_{s} 
    = \left(|\Theta|^2 \mathds{1}_{\rrbracket s,u\rrbracket}\right) * \nu^{(\mathbb{G},X^\natural)}_{\infty}\\
    \leq \left(|I|^2 
    \mathds{1}_{\rrbracket s,u\rrbracket}\right) * \nu^{(\mathbb{G},X^\natural)}_{\infty} 
    = |I|^2 * \nu^{(\mathbb{G},X^\natural)}_{u} - |I|^2 * \nu^{(\mathbb{G},X^\natural)}_{s}, \hspace{0.2cm} \mathbb{P}-\text{a.e}.
\end{multline*}
By a straightforward monotone class argument, we have
\begin{align*}
    \ud \langle \Theta \star \widetilde{\mu}^{(\mathbb{G},X^\natural)} \rangle^{\mathbb{G}} \leq \ud \big(|I|^2 * \nu^{(\mathbb{G},X^\natural)}\big),\ \mathbb{P}-\text{a.e}.
\end{align*}
Therefore, using the above, we get
\begin{align*}
\left(\tnorm{\Theta_t(\cdot)}^{(\mathbb{G},\mathbf{X})}_t(\omega)\right)^2  = \frac{\ud \langle \Theta \star \widetilde{\mu}^{(\mathbb{G},X^\natural)} \rangle^{\mathbb{G}}}{\ud C^{(\mathbb{G},\mathbf{X})}} 
\leq \frac{\ud\big(|I|^2 * \nu^{(\mathbb{G},X^{\natural})}\big)}{\ud C^{(\mathbb{G},\mathbf{X})}} \leq 1, \hspace{0.2cm} \mathbb{P} \otimes C^{(\mathbb{G},\mathbf{X})}-\text{a.e.},
\end{align*}
and the proof is thus complete.
\end{proof}


\subsection{Wasserstein distance}\label{sec:Wasserstein_dist}

Let $\mathbb{X}$ be a Polish space endowed with a metric $\rho$, then we denote by $\mathscr{P}(\mathbb{X})$ the space of probability measures on $(\mathbb{X},\rho)$. 
Moreover, for every real $r\in [1,\infty)$, we define the probability measures on $\mathbb{X}$ with finite $r$ moment to be 
\begin{align*}
    \mathscr{P}_r(\mathbb{X}) :=\left\{\mu : \int_{\mathbb{X}}\rho(x_0,x)^r \,\mu(\ud x) < \infty \right\},\text{ for some } x_0 \in \mathbb{X}.
\end{align*}
By the triangle inequality and the fact that we consider probability (\textit{i.e.} finite) measures, it is immediate that the space $\mathscr{P}_r(\mathbb{X})$ is independent of the choice of $x_0$.
The space $\mathscr{P}(\mathbb{X})$ is equipped with the usual weak topology,\footnote{We use the term as probabilists do, \emph{i.e.}, the topological dual space of the set of probability measures is the set of continuous and bounded functions defined on $\mathbb{X}$. In other words, from the point of view of functional analysis, this is the weak${}^*-$ topology.} which we denote by $\mathcal{T}$. 
Let us recall the form of the elements for the usual basis of $\mathcal{T}$.

\begin{definition}
Let $f : \mathbb{X} {}\longrightarrow \mathbb{R}$ be any continuous and bounded function (i.e. $f \in C_b(\mathbb{X})$), then we denote by $I^{f} :\mathscr{P}(\mathbb{X}) {}\rightarrow \mathbb{R}$ the function where
\begin{align*}
\mathscr{P}(\mathbb{X}) \ni \mu \longmapsto  \int_{\mathbb{X}}f(x)\mu(\ud x) \in \mathbb{R}.
\end{align*}
\end{definition}

\noindent We can now give a description of the usual basis of $\mathcal{T}$
\begin{align}\label{2.15}
    \textbf{B}(\mathcal{T}) := \left\{ \bigcap_{i = 1}^{m}\left(I^{f_i}\right)^{-1}(A_i): m \in \mathbb{N}, f_i \in C_b(\mathbb{X}), A_i \text{ open sets of } \mathbb{R}
    \right\}.
\end{align} 
The weak topology is metrizable and, in fact, $(\mathscr{P}(\mathbb{X}),\mathcal{T})$ is a Polish space; see \citet[15.15 Theorem]{AliprantisBorder}. 
Moreover, it is well-known, see \emph{e.g.} \cite[15.3 Theorem]{AliprantisBorder}, that 
a sequence of probability measures $\{\mu_m\}_{m \in \mathbb{N}}$  converges weakly to a probability measure $\mu$
if and only if for every bounded and continuous function $f$ we have
\begin{align*}
\int_{\mathbb{X}}f(x)\,\mu_m(\ud x) \xrightarrow[m\to\infty]{} \int_{\mathbb{X}}f(x)\,\mu(\ud x).
\end{align*}

On $\mathscr{P}_r(\mathbb{X})$ we can define an even stronger mode of convergence, that allows for more functions to be tested. 
We will simply call it weak convergence in $\mathscr{P}_r(\mathbb{X})$.
This mode says that a sequence of probability measures $\{\mu_m\}_{m \in \mathbb{N}}$ converges weakly in $\mathscr{P}_r(\mathbb{X})$ to a probability measure $ \mu$
if and only if for every continuous function $f$ such that $|f(x)| \leq C (1 + \rho(x_0,x)^r)$,  where $C := C(f) \in \mathbb{R}_+$, we have 
\begin{align}\label{fW}
\int_{\mathbb{X}}f(x)\,\mu_m(\ud x) \xrightarrow[m\to\infty]{} \int_{\mathbb{X}}f(x)\,\mu(\ud x).
\end{align}
Of course, as before, it is immediate that in the above definition it does not matter which $x_0$ we choose. 

Now, one can easily see that the topology produced by this mode of convergence has a natural base, similar to the one of the weak topology. For the purpose of writing it down let us give the following definition.

\begin{definition}
Let $r\in [1,\infty)$ and $f : \mathbb{X} {}\longrightarrow \mathbb{R}$ any continuous function such that $|f(x)| \leq C (1 + \rho(x_0,x)^r)$,  where $C := C(f) \in \mathbb{R}_+$. Then we denote by $I^{f}_r :\mathscr{P}_r(\mathbb{X}) {}\rightarrow \mathbb{R}$ the function where
\begin{align*}
\mathscr{P}_r(\mathbb{X}) \ni \mu \longmapsto  \int_{\mathbb{X}}f(x)\mu(\ud x) \in \mathbb{R}.
\end{align*}
\end{definition}
Hence, similar to before, the natural base for the topology $\mathcal{T}_r$ of the weak convergence in $\mathscr{P}(\mathbb{X})_r$ is
\begin{equation}\label{2.25}
\begin{aligned}
    \textbf{B}_r(\mathcal{T}_r) := \Bigg\{ \bigcap_{i = 1}^{m}\left(I^{f_i}_r\right)^{-1}(A_i): m \in \mathbb{N}, f_i \in C(\mathbb{X})\hspace{.1cm}\text{such that}\\|f_i(x)| \leq C(f_i) (1 + \rho(x_0,x)^r), A_i \text{ open sets of } \mathbb{R}
    \Bigg\}.
    \end{aligned}
\end{equation}
Additionally, 
$\mathcal{T}_r$
is metrizable from a metric with nice properties.
This metric is called the Wasserstein distance of order $r$.
More precisely, given two probability measures $\mu, \nu \in \mathscr{P}_r(\mathbb{X})$ we define the Wasserstein distance of order $r$ between them to be 
\begin{equation}\label{def:wasserstein}
    \mathcal{W}_{r,\rho}^r(\mu,\nu) := \inf_{\pi \in \Pi(\mu,\nu)}\left\{\int_{\mathbb{X} \times \mathbb{X}}\rho(x,y)^r\,\pi(\ud x,\ud y)\right\},
\end{equation}
where $\Pi(\mu,\nu)$ are the probability measures on $\mathbb{X} \times \mathbb{X}$ with marginals $\mu$ and $\nu$.
The interested reader may consult \citet[Theorem 6.9]{villani2009optimal} for the fact that $\mathcal{W}_{r,\rho}$ metrizes  $\mathscr{P}_r(\mathbb{X})$.
Moreover, $\mathscr{P}_r(\mathbb{X})$ with this mode of convergence is a Polish space, as can be seen from \cite[Theorem 6.18]{villani2009optimal}. So, we arrive to our next result.
\begin{lemma}\label{thm:weak_conv_r}
Let $r\in[1,\infty)$ and $\mu \in {\mathscr{P}_r}(\mathbb{X})$, then there exists a sequence $\{f^{\mu}_k\}_{k \in \mathbb{N}} \subseteq C(\mathbb{X})$ such that $|f^{\mu}_k(x)| \leq C(f^{\mu}_k) (1 + \rho(x_0,x)^r)$ for every $k\in\mathbb{N}$, such that a sequence of probability measures $\{\mu_m\}_{m \in \mathbb{N}}\subseteq\mathscr{P}_r(\mathbb{X})$ converges weakly to $\mu$ if and only if 
\begin{align*}
    \int_{\mathbb{X}}f^{\mu}_k(x)\,\mu_m(\ud x) \xrightarrow[m\to\infty]{}  \int_{\mathbb{X}}f^{\mu}_k(x)\,\mu(\ud x), \text{ for all } k \in \mathbb{N}.
\end{align*}
\end{lemma}
\begin{proof}
For every open ball with center $\mu$ and radius $\frac{1}{m}$ for $m \in \mathbb{N}$, denoted by $B_{W_{r,\rho}}(\mu,\frac{1}{m})$, there exists $D_m \in \textbf{B}_r(\mathcal{T}_r)$ where $\mu \in D_m \subseteq B_{W_{r,\rho}}(\mu,\frac{1}{m})$. 
Using \eqref{2.25}, we can conclude the proof.
\end{proof}
\begin{remark}\label{rem:Wasserstein_measurable}
One can view the Wasserstein distance of order $r$ as a non-negative function on the Polish space $(\mathscr{P}(\mathbb{X}),\mathcal{T}) \times (\mathscr{P}(\mathbb{X}),\mathcal{T})$, with the possibility of taking infinite values. 
Then, from \cite[Remark 6.12]{villani2009optimal} we have that $\mathcal{W}_{r,\rho}$ is lower semi-continuous, hence measurable.
\end{remark}
A useful inequality for the Wasserstein distance of order $r$, which is going to be used multiple times hereinafter, concerns the distance between two empirical measures on $\mathbb{X}$. 
Given $N \in \mathbb{N}$ and $\mathbf{x}^{N}:=(x^1,\dots,x^N),$ $\mathbf{y}^{N}:=(y^1,\dots,y^N) \in \mathbb{X}^{N}$, we have the empirical measures 
\begin{align*}
    L^N(\mathbf{x}^{N}) := \frac{1}{N}\sum_{i = 1}^{N} \delta_{x^i} \text{ and } L^N(\mathbf{y}^{N}) := \frac{1}{N}\sum_{i = 1}^{N} \delta_{y^i},
\end{align*}
where $\delta_{x}$ is the Dirac measure on $x\in\mathbb{X}$. 
Then, we have 
\begin{align}\label{ineq:empirical_measures_bound}
    \mathcal{W}_{r,\rho}^r\big(L^N(\mathbf{x}^{N}),L^N(\mathbf{y}^{N})\big) 
    \leq \frac{1}{N} \sum_{i = 1}^{N}\rho(x^i,y^i)^r.
\end{align}
The above inequality is immediate if in the definition of the Wasserstein distance \eqref{def:wasserstein} we choose the probability measure 
\begin{align*}
    \pi := \frac{1}{N}\sum_{i =1}^{N}\delta_{(x^i,y^i)},
\end{align*}
where $\delta_{(x,y)}$ is the Dirac measure on $(x,y)\in \mathbb{X} \times \mathbb{X}$. 
The marginals are immediately seen to be $L^N(\mathbf{x}^{N})$ and $L^N(\mathbf{y}^{N})$.

\begin{remark}\label{Re_2.9}
Finally, note that if $\rho$ is a bounded metric, then for every $r\in [1,\infty)$ we have $\mathscr{P}(\mathbb{X}) = \mathscr{P}_r(\mathbb{X})$, and every function $f$ as in \eqref{fW} belongs to $C_b(\mathbb{X})$. 
Hence, we see that $\mathcal{W}_{r,\rho}$ metrizes the weak convergence on $\mathscr{P}(\mathbb{X})$, see \cite[Corollary 6.13]{villani2009optimal}.
\end{remark}


\subsection{Skorokhod space}

Let us denote by $\mathbb{D}^d := \{f:[0,\infty) {}\longrightarrow \mathbb{R}^d: f\text{ is c\`adl\`ag}\}$ the space of c\`adl\`ag paths for every $d \in \mathbb{N}$. 
We supply $\mathbb{D}^d$ with its usual $J_1$-metric, which we denote by $\rho_{J_1}$. 
Endowed with this metric, $\mathbb{D}^d$ becomes a Polish space. 
We are not going to get into the specifics of $\rho_{J_1}$, as we will only need a couple of its basic properties. 
Firstly, for every $x,y \in \mathbb{D}^d$, we have 
\begin{align}\label{skorokineq}
\rho_{J_1}(x,y) \leq \sup_{s \in [0,\infty)}\{|x_s - y_s|\} \wedge 1.
\end{align}
Secondly, the Borel $\sigma-$algebra that $\rho_{J_1}$ generates coincides with the usual product $\sigma-$algebra on ${(\mathbb{R}^d)}^{[0,\infty)}$ that the projections generate. 
To be more precise, we have 
\begin{align}\label{skoalgebra}
    \mathcal{B}_{\rho_{J_1}}(\mathbb{D}^d) = \sigma\Big(\textrm{Proj}_s^{-1}(A): A \in \mathcal{B}(\mathbb{R}^d), s \in [0,\infty)\Big) \cap \mathbb{D}^d,
\end{align}
where ${(\mathbb{R}^d)}^{[0,\infty)} \ni x \overset{\textrm{Proj}_s}{\longmapsto} x(s)\in \mathbb{R}^d$, for every $s \in [0,\infty)$. 
Additional results on the Skorokhod space are available from \cite[Chapter 15]{he2019semimartingale} or \cite[Chapter VI]{jacod2013limit}. 

Using the fact that the $\sigma-$algebra on $\Omega$ is $\mathcal{G}$, it is obvious from \eqref{skoalgebra} that every $\mathcal{G} \otimes \mathcal{B}([0,\infty))-$jointly measurable c\`adl\`ag process $X$ can be seen as a function with domain $\Omega$ and taking values in $\mathbb{D}^d$ such that $X$ is $\big(\mathcal{G}/\mathcal{B}_{\rho_{J_1}}(\mathbb{D}^d)\big)-$measurable and \emph{vice versa}: every $\big(\mathcal{G}/\mathcal{B}_{\rho_{J_1}}(\mathbb{D}^d)\big)-$measurable random variable $X$ can be seen as a $\mathcal{G} \otimes \mathcal{B}([0,\infty))-$jointly measurable c\`adl\`ag process.

\begin{remark}\label{rem:Sko}
Later, when we say that a collection of c\`adl\`ag processes is independent or is identically distributed or is exchangeable, they will be understood as $\big(\mathcal{G}/\mathcal{B}_{\rho_{J_1}}(\mathbb{D}^d)\big)-$measurable random variables.
\end{remark}

Finally, for every $a \in (0,\infty)$ and $x \in \mathbb{D}^d$, the initial segment $x|_{[0,a]}$, resp. $x|_{[0,a-]}$, will be understood as an element of $\mathbb{D}^d$, using the convention 
\begin{gather}\label{inpath}
x|_{[0,a]}(s) := x(s)\mathds{1}_{[0,a)}(s) + x(a)\mathds{1}_{[a,\infty)}(s),
\\ 
\text{ resp. }
x|_{[0,a-]}(s) := x(s)\mathds{1}_{[0,a)}(s) + x(a-)\mathds{1}_{[a,\infty)}(s).\nonumber
\end{gather}
For $a\in[0,\infty)$, we define $\mathbb{D}^d_{a} := \{x_{[0,a]}: x \in \mathbb{D}^d\}$, resp. $\mathbb{D}^d_{a-} := \{x_{[0,a-]}: x \in \mathbb{D}^d\}$, and we naturally have
\begin{equation*}
    \mathcal{B}_{\rho_{J_1}}(\mathbb{D}^d_{a}) = \mathcal{B}_{\rho_{J_1}}(\mathbb{D}^d) \cap \mathbb{D}^d_{a}, 
        \quad \text{ resp. } \quad 
    \mathcal{B}_{\rho_{J_1}}(\mathbb{D}^d_{a-}) = \mathcal{B}_{\rho_{J_1}}(\mathbb{D}^d) \cap \mathbb{D}^d_{a-}.
\end{equation*}

\section{A priori estimates}
\label{sec:apriori}

In this section, we provide \textit{a priori} estimates for BSDEs whose generator does not depend on the solution of the equation.
In other words, we deal with semimartingale representations, for which their norms, as well as the norm of their martingale parts, can be controlled in terms of the norms of the terminal value and of the integrated generator of the equation. 
These estimates will be our primary tool in order to prove the existence and uniqueness theorems for the mean-field as well as the McKean--Vlasov BSDEs in the subsequent sections, but also in order to prove the propagation of chaos statements. 
They expand the results of \citet[Section 3.4]{papapantoleon2018existence} by replacing the deterministic exponential with the stochastic one.
Two remarks are in order at this point:
Firstly, the use of the left-continuous stochastic exponential as the norm-weight allows to obtain a statement analogous to \cite[Lemma 3.8]{papapantoleon2018existence}.
Secondly, similar \textit{a priori} estimates for BSDEs have recently been obtained independently by \citet{possamai2023reflections}.

Intuitively speaking, the last term in \eqref{apriorieq} below is zero on average, as martingales are the stochastic analogs of deterministic constant functions. 
Because the norms are defined using an expectation, it should be possible to bound the norm of the solution $y$ using the norms of the generator $f$ and the terminal condition $\xi$. 
Then, it follows directly that one can bound the norm of $\eta$ using the norms of $f$ and $\xi$, by switching places in \eqref{apriorieq}.  

Recall that $(\Omega,\mathcal{G},\mathbb{G},\mathbb{P})$ denotes a complete stochastic basis, which we assume fixed within this section.
 In addition, fix a stopping time $T$, a predictable, c\`adl\`ag, non-decreasing process $C$, a predictable, non-vanishing, real-valued process $\alpha$ and define $A_{\cdot}:= \int_{0}^{\cdot}\alpha_s^2\,\ud C_s$.
In an effort to ease the notation within the proof of \cref{lem:a_priori_estimates}, we will omit the dependence of the norms on $\mathbb{G}$, $\mathcal{G}_T$, $C$, $A$ and the state space $\mathbb{R}^d$.
More precisely, $\mathbb{L}^2_\beta(\mathcal{G}_T,A;\mathbb{R}^d)$,
$\mathbb{H}^2_\beta(\mathbb{G},A,C;\mathbb{R}^d)$,
$\mathcal{S}^2_\beta(\mathbb{G},A;\mathbb{R}^d)$ and
$\mathcal{H}^2_\beta(\mathbb{G},A;\mathbb{R}^d)$, 
will be simply denoted by 
$\mathbb{L}^2_\beta$,
$\mathbb{H}^2_\beta$, 
$\mathcal{S}^2_\beta$ and
$\mathcal{H}^2_\beta$ respectively.

\begin{lemma}
\label{lem:a_priori_estimates}
Assume that we are given a $d$-dimensional semimartingale $y$ of the form
\begin{equation}\label{apriorieq}
 y_t = \xi + \int_{t}^{T}f_s \,\ud C_s - \int_{t}^{T}\,\ud\eta_s, 
\end{equation}
where $\xi \in \mathbb{L}^2(\mathcal{G}_T, A; \mathbb{R}^d)$, 
$f$ is a $d$-dimensional optional process, and $\eta \in \mathcal{H}^2(\mathbb{G};\mathbb{R}^d)$. 
In addition, assume that there exists some $\Phi \ge 0$ such that $\Delta A \leq \Phi, \mathbb{P} \otimes C-$almost everywhere. 
Moreover, suppose that there exists a $\beta \in (0,\infty)$ such that
\begin{align}\label{conditions:a_priori_estimates}
    \| \xi\|_{\mathbb{L}^2_\beta(\mathcal{G}_T,A;\mathbb{R}^d)} + 
    \left\| \frac{f}{\alpha}\right\|_{\mathbb{H}^2_\beta(\mathbb{G},A,C;\mathbb{R}^d)} < \infty.
\end{align}
Then, for any $(\gamma,\delta) \in (0,\beta]^2$ with $\gamma \neq \delta$, we have
\begin{gather*}
\|\alpha y\|^{2}_{\mathbb{H}^2_{\delta}(\mathbb{G},A,C;\mathbb{R}^d)} 
    \leq \frac{2(1 + \delta \Phi)}{\delta} \|\xi\|^{2}_{\mathbb{L}^{2}_{\delta}(\mathcal{G}_T,A;\mathbb{R}^d)} 
    + 2 \frac{(1 + \gamma \Phi)^2}{\gamma |\delta - \gamma|} \left\|\frac{f}{\alpha}\right\|^{2}_{\mathbb{H}^{2}_{\gamma \vee \delta}(\mathbb{G},A,C;\mathbb{R}^d)}, \\ 
\|y\|^{2}_{\mathcal{S}^{2}_{\delta}(\mathbb{G},A;\mathbb{R}^d)} 
    \leq 8 \|\xi\|^2_{\mathbb{L}^{2}_{\delta}(\mathcal{G}_T,A;\mathbb{R}^d)} 
    + \frac{8(1 + \gamma \Phi)}{\gamma} \left\|\frac{f}{\alpha}\right\|^{2}_{\mathbb{H}^{2}_{\gamma \vee \delta}(\mathbb{G},A,C;\mathbb{R}^d)}
\shortintertext{and}
\begin{multlined}[0.9\textwidth]
    \|\eta\|^{2}_{\mathcal{H}^{2}_{\delta}(\mathbb{G},A;\mathbb{R}^d)} \\
    \leq (18 + 9\delta\Phi) \|\xi\|^2_{\mathbb{L}^{2}_{\delta}(\mathcal{G}_T,A;\mathbb{R}^d)} 
    + \Big(\frac{9{\big(1+(\gamma\vee\delta)\Phi\big)}}{\gamma \vee \delta} +  \frac{9\delta(1 + \gamma \Phi)^2}{\gamma |\delta - \gamma|}\Big) \left\|\frac{f}{\alpha}\right\|^{2}_{\mathbb{H}^{2}_{\gamma \vee \delta}(\mathbb{G},A,C;\mathbb{R}^d)}.
\end{multlined}
\end{gather*}
Therefore, putting the pieces together, we have
\begin{align*}
\begin{multlined}[\textwidth]
    \|y\|^{2}_{\mathcal{S}^{2}_{\delta}(\mathbb{G},A;\mathbb{R}^d)} + \|\eta\|^{2}_{\mathcal{H}^{2}_{\delta}(\mathbb{G},A,C;\mathbb{R}^d)} \\ 
    \leq (26 + 9 \delta \Phi)  \|\xi\|^2_{\mathbb{L}^{2}_{\delta}(\mathcal{G}_T,A;\mathbb{R}^d)} \\
    + \Big( \frac{8}{\gamma} + 17 \Phi + \frac{9}{\gamma \vee \delta} + 9\delta \frac{(1 + \gamma \Phi)^2}{\gamma |\delta - \gamma|} \Big)\left\|\frac{f}{\alpha}\right\|^{2}_{\mathbb{H}^{2}_{\gamma \vee \delta}(\mathbb{G},A,C;\mathbb{R}^d)}
\end{multlined}
\end{align*}
and
\begin{align*}
\begin{multlined}[0.99\textwidth]
        \|\alpha y\|^{2}_{\mathbb{H}^2_{\delta}(\mathbb{G},A,C;\mathbb{R}^d)} 
        +  \|y\|^{2}_{\mathcal{S}^{2}_{\delta}(\mathbb{G},A;\mathbb{R}^d)} 
        + \|\eta\|^{2}_{\mathcal{H}^{2}_{\delta}(\mathbb{G},A;\mathbb{R}^d)}\\ 
    \leq  
    \Big(26 + \frac{2}{\delta}+ (9\delta + 2)\Phi \Big) \|\xi\|^2_{\mathbb{L}^{2}_{\delta}(\mathcal{G}_T,A;\mathbb{R}^d)} 
    \\+ \Big( \frac{8}{\gamma} + 17 \Phi + \frac{9}{\gamma \vee \delta}
    + (9\delta + 2) \frac{(1 + \gamma \Phi)^2}{\gamma |\delta - \gamma|} \Big)\left\|\frac{f}{\alpha}\right\|^{2}_{\mathbb{H}^{2}_{\gamma \vee \delta}(\mathbb{G},A,C;\mathbb{R}^d)}
\end{multlined}.
\end{align*}
\end{lemma}

\begin{proof} 
We will initially deal with the upper bounds of the norms of the semimartingale $y$. 
By definition, we have that $\int_{t}^{T} \,\ud \eta_s = \eta_T - \eta_{T\wedge t}$. 
Because $y$ is adapted and $\eta\in\mathcal{H}^2(\mathbb{G};\mathbb{R}^d)$, we have, for every $t\ge 0$, that
\begin{equation}\label{apriorieq2}
 y_t  
 = \mathbb{E}\bigg[\xi +  \int_{t}^{T}f_s \,\ud C_s  \Big| \mathcal{G}_t \bigg]. 
\end{equation}
Hence, from the above identity it is evident that we need to study the following process
\begin{align}\label{def:F_Integral_process}
 F(t) :=  \int_{t}^{T}f_s \ud C_s. 
\end{align}
Let $\gamma \in \mathbb{R}_{+}$, then we have from the Cauchy--Schwarz inequality and \cref{lem:StochExp}.\ref{item:StochExp8} that 
\begin{align*}
    |F(t)|^2 
        &\leq  \int_{t}^{T}\mathcal{E}(\gamma A)^{-1}_{s-} \,\ud A_s  \int_{t}^{T} \mathcal{E}(\gamma A)_{s-} \frac{|f_s|^2}{\alpha^2_s} \,\ud C_s
        \\
        &
        \leq \frac{1 + \gamma \Phi}{\gamma}\mathcal{E}(\gamma A)^{-1}_{t-} \int_{t}^{T} \mathcal{E}(\gamma A)_{s-} \frac{|f_s|^2}{\alpha^2_s} \,\ud C_s. \numberthis \label{aprioriineq1}
\end{align*}
Using the assumption in \eqref{conditions:a_priori_estimates}, for $\gamma \in (0,\beta]$, we have
\begin{align*}
    \mathbb{E}\left[|F(0)|^2\right] < \infty.
\end{align*}
Next, for $\delta \in (0,\beta]$, we will integrate $|F(t)|^2$ with respect to $\mathcal{E}(\delta A)_{-} \ud A$, \emph{i.e.}, with respect to $\delta^{-1}\textup{d}\mathcal{E}(\delta A)$.
Before we proceed, we will use the notation of \cref{lem:StochExp}.\ref{item:StochExp7}, \emph{i.e.}
\begin{align*}
    \widetilde{A}^{\delta,\gamma}_\cdot:= \delta A_\cdot - \overline{\gamma A}_\cdot - [\delta A, \overline{\gamma A}]_\cdot,
\end{align*}
for $\overline{\gamma A}$ defined as in \cref{lem:StochExp}.\ref{item:StochExp2}.
We underline that $\mathcal{E}(\widetilde{A}^{\delta,\gamma})$ is (strictly) positive.
Indeed, this is straightforward from \ref{item:StochExp3} and \ref{item:StochExp7} of  \cref{lem:StochExp}. 
Now, we return to our aim and, with the help of \cref{lem:StochExp}.\ref{item:StochExp7}, and Tonelli's theorem, we get
\begin{align*}
\int_{0}^{T}\mathcal{E}(\delta A)_{t-} |F(t)|^2 \,\ud A_t 
&\overset{\eqref{aprioriineq1}}{\leq}  
\frac{1 + \gamma \Phi}{\gamma} 
    \int_{0}^{T}\int_{0}^{T}\mathcal{E}(\widetilde{A}^{\delta,\gamma})_{t-} \mathds{1}_{\rrbracket t,T\rrbracket}(s) 
    \mathcal{E}(\gamma A)_{s-} \frac{|f_s|^2}{\alpha^2_s} \,\ud C_s \,\ud A_t \\
&\overset{\phantom{\eqref{aprioriineq1}}}{=}   \frac{1 + \gamma \Phi}{\gamma} 
    \int_{0}^{T} \mathcal{E}(\gamma A)_{s-} \frac{|f_s|^2}{\alpha^2_s} \int_{0}^{T}\mathcal{E}(\widetilde{A}^{\delta,\gamma})_{t-}
    \mathds{1}_{\rrbracket t,T\rrbracket}(s) \,\ud A_t  \,\ud C_s\\
&\overset{\phantom{\eqref{aprioriineq1}}}{=}   \frac{1 + \gamma \Phi}{\gamma} 
    \int_{0}^{T} \mathcal{E}(\gamma A)_{s-} \frac{|f_s|^2}{\alpha^2_s} \int_{0}^{s-}\mathcal{E}(\widetilde{A}^{\delta,\gamma})_{t-} \,\ud A_t  \,\ud C_s.
    \numberthis\label{ineq:after_integration}
\end{align*}
Let us concentrate for a moment on the term 
$\int_{0}^{s-}\mathcal{E}(\widetilde{A}^{\delta,\gamma})_{t-} \ud A_t $, by considering the two cases $\delta > \gamma$ and $\delta < \gamma$:

$\bullet\, \delta > \gamma$: 
Using \cref{lem:StochExp}.\ref{item:StochExp7} we derive the inequality
\begin{align*}
\int_{0}^{s-}\mathcal{E}(\widetilde{A}^{\delta,\gamma})_{t-} \,\ud A_t 
&
= \frac{1}{\delta - \gamma} \int_{0}^{s-}(1 + \Delta(\gamma A)_t)\mathcal{E}(\widetilde{A}^{\delta,\gamma})_{t-} \,\ud\widetilde{A}^{\delta,\gamma}_t
\\
&
\leq \frac{(1 + \gamma \Phi)}{\delta - \gamma}\mathcal{E}(\widetilde{A}^{\delta,\gamma})_{s-}.
\numberthis\label{ineq:case_delta_greater_gamma_1}
\end{align*}
Thus, returning to \eqref{ineq:after_integration} and in conjunction with \eqref{ineq:case_delta_greater_gamma_1} and \cref{lem:StochExp}.\ref{item:StochExp7}, we have
\begin{align*}
\int_{0}^{T}\mathcal{E}(\delta A)_{t-} |F(t)|^2 \,\ud A_t   
=
    \frac{(1 + \gamma \Phi)^2}{\gamma (\delta - \gamma)} \int_{0}^{T} \mathcal{E}(\delta A)_{s-}  \frac{|f_s|^2}{\alpha^2_s} \,\ud C_s,
\end{align*}
which is integrable for $\delta \leq \beta$.

$\bullet \hspace{0.1cm} \delta < \gamma$: 
Using \cref{lem:StochExp}.\ref{item:StochExp7} again, we deduce 
\begin{align*}
\int_{0}^{s-}\mathcal{E}(\widetilde{A}^{\delta,\gamma})_{t-} \,\ud A_t 
&\overset{\phantom{\eqref{equation:stochexp}}}{=} \frac{1}{\delta - \gamma} \int_{0}^{s-}(1 + \Delta(\gamma A)_t)\mathcal{E}(\widetilde{A}^{\delta,\gamma})_{t-} \,\ud\widetilde{A}^{\delta,\gamma}_t\\
&\overset{\eqref{equation:stochexp}}{=}\frac{1}{|\delta - \gamma|} \int_{0}^{s-}(1 + \Delta(\gamma A)_t)\,\ud\big(-\mathcal{E}(\widetilde{A}^{\delta,\gamma})\big)_{t} 
\leq \frac{1 + \gamma \Phi}{|\delta - \gamma|},
\numberthis
\label{ineq:case_gamma_greater_delta_1}
\end{align*}
where, for the inequality, we use that $\mathcal{E}(\widetilde{A}^{\delta,\gamma})$ is positive and non-decreasing; see   \ref{item:StochExp7} and \ref{item:StochExp3} of \cref{lem:StochExp}. Thus, returning to \eqref{ineq:after_integration} and in conjunction with \eqref{ineq:case_gamma_greater_delta_1}, we have
\begin{align*}
\int_{0}^{T}\mathcal{E}(\delta A)_{t-} |F(t)|^2 \,\ud A_t 
&\overset{\eqref{ineq:case_gamma_greater_delta_1}}{\leq} 
    \frac{(1 + \gamma \Phi)^2}{\gamma |\delta - \gamma|} \int_{0}^{T} \mathcal{E}(\gamma A)_{s-} 
    \frac{|f_s|^2}{\alpha^2_s} \,\ud C_s.
\end{align*}

In total, summing up the conclusions of the two cases, we have that for every $(\gamma,\delta) \in (0,\beta]^2$ with $\gamma \neq \delta$ we can rewrite \eqref{ineq:after_integration} as
\begin{align*}
    \int_{0}^{T}\mathcal{E}(\delta A)_{t-} |F(t)|^2 \,\ud A_t 
        \leq \frac{(1 + \gamma \Phi)^2}{\gamma |\delta - \gamma|} \int_{0}^{T} \mathcal{E}\big((\gamma \vee \delta) A\big)_{s-} \frac{|f_s|^2}{\alpha^2_s} \,\ud C_s
\end{align*}
or, equivalently -- in terms of the introduced notation -- as
\begin{equation}\label{aprioriineq3}
\mathbb{E}\bigg[\int_{0}^{T}\mathcal{E}(\delta A)_{t-} |F(t)|^2 \,\ud A_t\bigg] 
\leq \frac{(1 + \gamma \Phi)^2}{\gamma |\delta - \gamma|} \left\|\frac{f}{\alpha}\right\|^{2}_{\mathbb{H}^{2}_{\gamma \vee \delta}}. 
\end{equation}

We are now ready to estimate $\|\alpha y\|_{\mathbb{H}^{2}_{\delta}}$. Using \eqref{apriorieq2}, \eqref{def:F_Integral_process}, Jensen's inequality, the inequality $(a + b)^2 \leq 2(a^2 + b^2)$, in conjunction with the fact that $A$ is predictable, see \cite[Theorem 5.13]{he2019semimartingale}, and \eqref{aprioriineq3}, we have that
\begin{align*}
\|\alpha y\|^{2}_{\mathbb{H}^{2}_{\delta}} 
    = \mathbb{E}\bigg[\int_{0}^{T}\mathcal{E}(\delta A)_{t-} |y_t|^2 \,\ud A_t\bigg]
\leq \frac{2(1 + \delta \Phi)}{\delta} \|\xi\|^{2}_{\mathbb{L}^{2}_{\delta}} 
+ 2 \frac{(1 + \gamma \Phi)^2}{\gamma |\delta - \gamma|} \left\|\frac{f}{\alpha}\right\|^{2}_{\mathbb{H}^{2}_{\gamma \vee \delta}}. 
\end{align*}

We move on to the estimate of $\|y\|_{\mathcal{S}^{2}_{\delta}}$.
Once again, we will use \eqref{apriorieq2}, \eqref{def:F_Integral_process}, \eqref{aprioriineq1}, the vector analogue of the triangle inequality for conditional expectations, the inequality $(a + b)^2 \leq 2(a^2 + b^2)$ and the subadditivity of the square root function, \emph{i.e.},
\begin{align*}
\|y\|^{2}_{\mathcal{S}^{2}_{\delta}} 
&= \mathbb{E}\Big[ \sup_{0 \leq t \leq T}\left(\mathcal{E}(\delta A)_{t-} |y_t|^2\right)\Big]\\
&\leq  \mathbb{E}\Big[ \sup_{0 \leq t \leq T}
    \Big(\mathcal{E}(\delta A)^{\frac{1}{2}}_{t-} \mathbb{E}\Big[\big|\xi +  F(t)  \big|  \Big| \mathcal{G}_t \Big] \Big)^2\Big] \\
&\leq 2 \mathbb{E}\Big[ \sup_{0 \leq t \leq T}
    \Big( \mathbb{E}\Big[\sqrt{ \mathcal{E}(\delta A)_{t-} |\xi|^2 + \mathcal{E}(\delta A)_{t-}  |F(t)|^2 } \Big| \mathcal{G}_t \Big] \Big)^2\Big]. 
\end{align*}
At this point, we will split our analysis again in two cases:

$\bullet \hspace{0.1cm} \delta < \gamma$: By definition of the stochastic exponential, for $A$ increasing we have $0 < \mathcal{E}(\delta A)_{t-} \leq \mathcal{E}(\gamma A)_{t-}$ or equivalently $0 < \mathcal{E}(\delta A)_{t-} \mathcal{E}(\gamma A)^{-1}_{t-} \leq 1$.
Hence, by means of \eqref{aprioriineq1} and Doob's inequality we get
\begin{align*}
\|y\|^{2}_{\mathcal{S}^{2}_{\delta}} 
&\leq  8 \|\xi\|^2_{\mathbb{L}^{2}_{\delta}} 
        + \frac{8(1 + \gamma \Phi)}{\gamma}  \left\|\frac{f}{\alpha}\right\|^{2}_{\mathbb{H}^{2}_{\gamma}}.
\end{align*}

$ \bullet \, \delta > \gamma$: 
We will use the fact that $ \mathcal{E}(\delta A)_{\cdot} \mathcal{E}(\gamma A)^{-1}_{\cdot} = \mathcal{E}(\widetilde{A}^{\delta,\gamma})_{\cdot}$ is non-decreasing and positive; see in \cref{lem:StochExp} \ref{item:StochExp4} and \ref{item:StochExp7}, in conjunction with \ref{item:StochExp3}. Now, we can proceed exactly as in the previous case except that in the current case we transfer the term 
$\mathcal{E}(\delta A)_{t-} \mathcal{E}(\gamma A)^{-1}_{t-}$ inside the integral and we bound it from above by $\mathcal{E}(\delta A)_{s-} \mathcal{E}(\gamma A)^{-1}_{s-}$, for $s\ge t$. 
After cancellation, we get the same formulas with the difference that  inside the Lebesgue--Stieltjes integral we have $\mathcal{E}(\delta A)_{s-}$ in place of $\mathcal{E}(\gamma A)_{s-}$. 

Combining the two cases, we get 
\begin{align}\label{bound:y_sup_norm_apriori}
 \|y\|^{2}_{\mathcal{S}^{2}_{\delta}} 
    \leq 8 \|\xi\|^2_{\mathbb{L}^{2}_{\delta}} 
 + \frac{8(1 + \gamma \Phi)}{\gamma}  \left\|\frac{f}{\alpha}\right\|^{2}_{\mathbb{H}^{2}_{\gamma \vee \delta}}. 
\end{align}

What remains to derive the bound for $\|\eta\|_{\mathcal{H}^{2}_{\delta}}$. Let $\eta = (\eta^1,\dots,\eta^d)$. We have per coordinate $i \in \{1,\dots,d\}$ that
\begin{align*}
(\eta^i_T - \eta^i_{T\wedge t})^2 
    &= (\eta^i_{T})^2 -  (\eta^i_{T\wedge t})^2 - 2\eta^i_{T\wedge t} (\eta^i_T - \eta^i_{T\wedge t})\\
    &\begin{multlined}[0.8\textwidth]
    =  \big((\eta^i_{T})^2 - \langle\eta^i\rangle_T\big) - \big((\eta^i_{T\wedge t})^2 - \langle\eta^i\rangle_{T\wedge t}\big) \\
    - 2\eta^i_{T\wedge t} (\eta^i_T - \eta^i_{T\wedge t}) + \langle\eta^i\rangle_T - \langle\eta^i\rangle_{T\wedge t}.
    \end{multlined}
\end{align*}
Hence, because for every $\mathbb{G}-$martingale $M$ and $t \in \mathbb{R}_+$ holds that $\mathbb{E}\left[M_T\big|\mathcal{G}_t\right] = M_{T\wedge t}$ and $\eta^i, (\eta^i)^2 - \langle \eta^i \rangle$ are $\mathbb{G}-$martingales, we have from the linearity of the conditional expectation that
\begin{equation}\label{3.12}
 \mathbb{E}[|\xi - y_t + F(t)|^2 | \mathcal{G}_t] 
 \overset{\eqref{apriorieq}}{=} \mathbb{E}[|\eta_T - \eta_{T\wedge t}|^2   | \mathcal{G}_t] 
 = \mathbb{E}\bigg[\int_{t}^{T}\,\ud\text{Tr}[\langle\eta\rangle_s] \Big|\mathcal{G}_t\bigg].
\end{equation}
Direct computations via Tonelli's theorem yield
\begin{align*}
\|\eta\|^{2}_{\mathcal{H}^{2}_{\delta}} 
&\overset{\eqref{equation:stochexp}}{=} \mathbb{E}\bigg[\delta \int_{0}^{T}\int_{0}^{s-}\mathcal{E}(\delta A)_{t-}\,\ud A_t\,\ud\text{Tr}[\langle\eta\rangle_s] 
    + \int_0^T \ud \text{Tr}[\langle\eta\rangle_s]  \bigg]\\
&\overset{\phantom{\eqref{equation:stochexp}}}{=} \mathbb{E}\bigg[\delta \int_{0}^{T}\mathcal{E}(\delta A)_{t-}\int_{t}^{T}\,\ud\text{Tr}[\langle\eta\rangle_s]\,\ud A_t 
    + \int_0^T \ud \text{Tr}[\langle\eta\rangle_s]  \bigg].
\end{align*}
Regarding the first term on the right side of the above inequality, using the fact that $A$ is predictable,  we get
\begin{align*}
&\mathbb{E}\bigg[\int_{0}^{T}\mathcal{E}(\delta A)_{t-}\int_{t}^{T}\,\ud\text{Tr}[\langle\eta\rangle_s]\,\ud A_t\bigg] 
= \mathbb{E}\bigg[\int_{0}^{T}\mathcal{E}(\delta A)_{t-} \mathbb{E}\bigg[\int_{t}^{T}\,\ud\text{Tr}[\langle\eta\rangle_s] \Big|\mathcal{G}_t\bigg]\,\ud A_t\bigg] \\
&\hspace{2em}\overset{\eqref{3.12}}{=} \mathbb{E}\bigg[\int_{0}^{T}\mathcal{E}(\delta A)_{t-}  \mathbb{E}\big[|\xi - y_t + F(t)|^2 \big| \mathcal{G}_t\big]\,\ud A_t\bigg] \\
&\hspace{2em}\overset{\eqref{apriorieq2}}{\underset{\phantom{\eqref{aprioriineq3}}}{\leq}} 
    9  \mathbb{E}\bigg[\int_{0}^{T}\mathcal{E}(\delta A)_{t-}  |\xi|^2\,\ud A_t\bigg] 
    + 9\mathbb{E}\bigg[\int_{0}^{T}\mathcal{E}(\delta A)_{t-}  \mathbb{E}\big[|F(t)|^2 \big| \mathcal{G}_t\big]\,\ud A_t\bigg]\\
&\hspace{2em}\overset{\eqref{aprioriineq3}}{\leq} \frac{9(1 + \delta \Phi)}{\delta}   \|\xi\|^2_{\mathbb{L}^{2}_{\delta}} 
+\frac{9(1 + \gamma \Phi)^2}{\gamma |\delta - \gamma|} \left\|\frac{f}{\alpha}\right\|^{2}_{\mathbb{H}^{2}_{\gamma \vee \delta}}.
\end{align*}
Regarding the second term, we have for every $\epsilon,\zeta\geq 0$
\begin{align*}
\mathbb{E}\bigg[\int_0^T \ud \text{Tr}[\langle\eta\rangle_s]\bigg] 
&\overset{\eqref{3.12}}{\underset{\eqref{apriorieq2}}{\leq}} 
9 \mathbb{E}\big[ |\xi|^2\big] 
    + 9 \mathbb{E}\big[|F(0)|^2\big]
\overset{\eqref{aprioriineq1}}{\leq} 9\|\xi\|^2_{\mathbb{L}^{2}_{\epsilon}} 
+ \frac{9{\big(1+\zeta\Phi\big)}}{\zeta} \left\|\frac{f}{\alpha}\right\|^{2}_{\mathbb{H}^{2}_{\zeta}}.
\end{align*}
Combining all of the above, we finally arrive at the following
\begin{equation*}
\|\eta\|^{2}_{\mathcal{H}^{2}_{\delta}} 
\leq (18 + 9\delta\Phi) \hspace{0.1cm} \|\xi\|^2_{\mathbb{L}^{2}_{\delta}} 
+ \Big(\frac{9{\big(1+(\gamma\vee \delta)\Phi\big)}}{\gamma \vee \delta} + \frac{9\delta (1 + \gamma \Phi)^2}{\gamma |\delta - \gamma|} \Big)\left\|\frac{f}{\alpha}\right\|^{2}_{\mathbb{H}^{2}_{\gamma \vee \delta}}. \qedhere
\end{equation*}
\end{proof}
\begin{corollary}\label{cor:xi_vanishes_norms}
    Assume the framework of \cref{lem:a_priori_estimates} with $\xi\equiv 0$, $\beta\in(0,\infty)$ and $0\leq \gamma < \delta \leq \beta$.
    Then, 
    \begin{gather*}
    \|\alpha y\|^{2}_{\mathbb{H}^2_{\delta}(\mathbb{G},A,C;\mathbb{R}^d)} 
        \leq \frac{(1 + \gamma \Phi)^2}{\gamma (\delta - \gamma)} \left\|\frac{f}{\alpha}\right\|^{2}_{\mathbb{H}^{2}_{\delta}(\mathbb{G},A,C;\mathbb{R}^d)}, \\ 
    \|y\|^{2}_{\mathcal{S}^{2}_{\delta}(\mathbb{G},A;\mathbb{R}^d)} 
        \leq  \frac{4(1 + \gamma \Phi)}{\gamma} \left\|\frac{f}{\alpha}\right\|^{2}_{\mathbb{H}^{2}_{\delta}(\mathbb{G},A,C;\mathbb{R}^d)}
    \shortintertext{and}
        \|\eta\|^{2}_{\mathcal{H}^{2}_{\delta}(\mathbb{G},A;\mathbb{R}^d)}
        \leq  \Big(\frac{4{\big(1+\delta\Phi\big)}}{\delta} +  \frac{4\delta(1 + \gamma \Phi)^2}{\gamma (\delta - \gamma)}\Big) \left\|\frac{f}{\alpha}\right\|^{2}_{\mathbb{H}^{2}_{\delta}(\mathbb{G},A,C;\mathbb{R}^d)}.
\end{gather*}
\end{corollary}

The previous lemma, as well as its \cref{cor:xi_vanishes_norms}, will be crucial in the construction of the contraction operators of the main theorems dealing with the existence and uniqueness of solutions of mean-field and McKean--Vlasov BSDEs. 
Foreseeing the coefficients that appear in the aforementioned procedure, we define them here and provide the respective optimality results in the following Proposition. 

Let $\mathcal{C}_\beta := \{(\gamma,\delta) \in (0,\beta]^2 : \gamma < \delta\}$ and define 
\begin{gather*}
M_\star^\Phi(\beta): = \inf_{(\gamma,\delta) \in \mathcal{C}_\beta} 
    \Big\{  \frac{9(1+\delta \Phi)}{\delta} +\frac{8 (1+\gamma\Phi)}{ \gamma} + \frac{9\delta}{\delta - \gamma}\frac{(1 + \gamma \Phi)^2}{\gamma} \Big\}\\
\widehat{M}_\star^\Phi(\beta): = \inf_{(\gamma,\delta) \in \mathcal{C}_\beta} 
    \Big\{\frac{4(1+\delta \Phi)}{\delta} +\frac{4 (1+\gamma\Phi)}{ \gamma} +  \frac{4\delta}{\delta - \gamma}\frac{(1 + \gamma \Phi)^2}{\gamma} \Big\}\\
M_{\scalebox{0.4}{$\blacksquare$}}^{\Phi}(\beta): = \inf_{(\gamma,\delta) \in \mathcal{C}_\beta} 
    \Big\{ \frac{9(1+\delta \Phi)}{\delta} +\frac{8 (1+\gamma\Phi)}{ \gamma} +  \frac{9\delta+2}{\delta - \gamma}\frac{(1 + \gamma \Phi)^2}{\gamma}\Big\}
\shortintertext{and}
\widehat{M}_{\scalebox{0.4}{$\blacksquare$}}^{\Phi}(\beta): = \inf_{(\gamma,\delta) \in \mathcal{C}_\beta} 
    \Big\{\frac{4(1+\delta \Phi)}{\delta} +\frac{4 (1+\gamma\Phi)}{ \gamma} +  \frac{4\delta+1}{\delta - \gamma}\frac{(1 + \gamma \Phi)^2}{\gamma} \Big\}.
\end{gather*}
We have the following obvious comparisons
\begin{align}\label{comparison:M_Phi}
\widehat{M}_\star^\Phi(\beta) < M_\star^\Phi(\beta) \text{ and }    
\widehat{M}_{\scalebox{0.4}{$\blacksquare$}}^\Phi(\beta) < M_{\scalebox{0.4}{$\blacksquare$}}^\Phi(\beta).    
\end{align}
In order to complete our analysis, we provide below asymptotic bounds for $M_\star^\Phi(\beta)$ and $M_{\scalebox{0.4}{$\blacksquare$}}^{\Phi}(\beta)$ with respect to $\Phi$.

A few remarks are appropriate at this point, in order to explain their roles and clarify the choices we made.
Let us start with an explanation of the notation.
Obviously, the lower index indicates the norm with which the coefficient are associated.
The presence of the wide hat indicates that the coefficients were obtained through \cref{cor:xi_vanishes_norms}, while the absence of the wide hat indicates that the coefficients were obtained through \cref{lem:a_priori_estimates}.

Let us focus now on their application. 
In the literature on BSDEs, the application of \cref{cor:xi_vanishes_norms} is the appropriate way to construct the optimal contraction constant for the existence and uniqueness results.
This is indeed the case because the processes mapped via the contraction have -- by construction -- the same terminal random variable.
Hence, the respective difference vanishes yielding \cref{cor:xi_vanishes_norms} the optimal option at this point.
However, in order to apply the \textit{a priori} estimates in the proof of the backward propagation of chaos, we need to control differences between the mean-field system and the McKean--Vlasov BSDEs, where the terminal conditions will not vanish. 
This generality dictates the use of \cref{lem:a_priori_estimates}.

\begin{remark}
Let us point out that the discrepancy of the coefficients between \cref{lem:a_priori_estimates} and \cref{prop:infima_for_M} in this work, and \citet[Proposition 5.4.]{possamai2023reflections} is due to the difference in the definition of the norms, where we use the left limit of the stochastic exponential instead of the right one as in \cite{possamai2023reflections}, while they also focus on the special case where $\xi = 0$ (when adapted to our notation).
\end{remark}

\begin{proposition}\label{prop:infima_for_M}
Let $\Phi \geq 0$ and $\beta \in (0,\infty)$, then we have
\begin{align*}
    M_{\star}^{\Phi}(\beta) 
    &= \min_{\gamma \in (0,\beta)} \left\{\frac{9}{\beta}+9\Phi + 8 \frac{(1 + \gamma\Phi)}{ \gamma} + 9 \hspace{0.1cm} \frac{\beta}{\beta - \gamma}\frac{(1 + \gamma \Phi)^2}{\gamma} \right\} \numberthis
    \label{eq_3.12}\\
    &= \left(6\sqrt{17} + 35\right)\left(\frac{1}{\beta}+ \Phi\right)
    \shortintertext{and}
    M_{\scalebox{0.4}{$\blacksquare$}}^{\Phi}(\beta)
    &= \min_{\gamma \in (0,\beta)} \left\{ \frac{9}{\beta}+9\Phi +  8 \frac{(1 + \gamma\Phi)}{ \gamma} + \frac{2 + 9\beta}{\beta - \gamma}\hspace{0.1cm} \frac{(1 + \gamma \Phi)^2}{\gamma}\right\} \numberthis
    \label{eq_3.14}\\
    &=  \left( 2 \sqrt{\frac{2}{\beta} + 9}\sqrt{\frac{2}{\beta} + 17}  + \frac{4}{\beta} + 35\right) \left(\frac{1}{\beta}+ \Phi\right).
\end{align*}
Hence, we get that 
\begin{align}\label{M_ineq}
\lim_{\beta \rightarrow \infty} M_\star^\Phi(\beta) = \lim_{\beta \rightarrow \infty} M_{\scalebox{0.4}{$\blacksquare$}}^{\Phi}(\beta) = 
 \left(6\sqrt{17} + 35\right)\Phi.
\end{align}
\end{proposition}

\begin{proof}
We will present here only the part of the results that is needed in the main text, \textit{i.e.} for \eqref{eq_3.12} and \eqref{eq_3.14}. 
The remainder of this proof is omitted as it also involves only elementary calculus, but is available upon request.  

Using the definition of $M_\star^\Phi(\beta),M_{\scalebox{0.4}{$\blacksquare$}}^{\Phi}(\beta)$ we can make a couple of observations. 
The first one is that we should only examine the case $\delta = \beta$, because for every pair $(\gamma,\delta) \in \mathcal{C}_{\beta}$ we have
    \begin{align*}
        \frac{9}{\delta} + 8 \frac{(1 + \gamma\Phi)}{ \gamma} + 9 \hspace{0.1cm} \frac{\delta}{\delta - \gamma}\frac{(1 + \gamma \Phi)^2}{\gamma} &\geq \frac{9}{\beta} + 8 \frac{(1 + \gamma\Phi)}{ \gamma} + 9 \hspace{0.1cm} \frac{\beta}{\beta - \gamma}\frac{(1 + \gamma \Phi)^2}{\gamma},\\
        \frac{9}{\delta} +  8 \frac{(1 + \gamma\Phi)}{ \gamma} + \frac{2 + 9\delta}{\delta - \gamma}\hspace{0.1cm} \frac{(1 + \gamma \Phi)^2}{\gamma} &\geq \frac{9}{\beta} +  8 \frac{(1 + \gamma\Phi)}{ \gamma} + \frac{2 + 9\beta}{\beta - \gamma}\hspace{0.1cm} \frac{(1 + \gamma \Phi)^2}{\gamma}.
    \end{align*}
    The second one is that
    \begin{align*}
        \lim_{\gamma \rightarrow 0^+} \frac{9}{\beta} + 8 \frac{(1 + \gamma\Phi)}{ \gamma} + 9 \hspace{0.1cm} \frac{\beta}{\beta - \gamma}\frac{(1 + \gamma \Phi)^2}{\gamma} 
        &=  \lim_{\gamma \rightarrow \beta^-} \frac{9}{\beta} + 8 \frac{(1 + \gamma\Phi)}{ \gamma} + 9 \hspace{0.1cm} \frac{\beta}{\beta - \gamma}\frac{(1 + \gamma \Phi)^2}{\gamma} 
        \\&= + \infty
        \shortintertext{and}
         \lim_{\gamma \rightarrow 0^+} \frac{9}{\beta} +  8 \frac{(1 + \gamma\Phi)}{ \gamma} + \frac{2 + 9\beta}{\beta - \gamma}\hspace{0.1cm} \frac{(1 + \gamma \Phi)^2}{\gamma} 
         &=  \lim_{\gamma \rightarrow \beta^-} \frac{9}{\beta} +  8 \frac{(1 + \gamma\Phi)}{ \gamma} +\frac{2 + 9\beta}{\beta - \gamma}\hspace{0.1cm} \frac{(1 + \gamma \Phi)^2}{\gamma} 
         \\&= + \infty.
    \end{align*}
    Therefore, we have that
\begin{align*}
    M_{\star}^{\Phi}(\beta) &= \min_{\gamma \in (0,\beta)} \left\{\frac{9}{\beta}+9\Phi + 8 \frac{(1 + \gamma\Phi)}{ \gamma} + 9 \hspace{0.1cm} \frac{\beta}{\beta - \gamma}\frac{(1 + \gamma \Phi)^2}{\gamma} \right\} \\
    \shortintertext{and}
    M_{\scalebox{0.4}{$\blacksquare$}}^{\Phi}(\beta) &= \min_{\gamma \in (0,\beta)} \left\{ \frac{9}{\beta}+9\Phi +  8 \frac{(1 + \gamma\Phi)}{ \gamma} + \frac{2 + 9\beta}{\beta - \gamma}\hspace{0.1cm} \frac{(1 + \gamma \Phi)^2}{\gamma}\right\}. 
\end{align*}
\end{proof}


\section{Well-posedness of mean-field and McKean--Vlasov BSDEs}
\label{sec:well-posedness_mf_MV_BSDES}

In this section, we provide general existence and uniqueness results for mean-field systems of BSDEs and for McKean--Vlasov BSDEs, in a setting where the filtrations are stochastically discontinuous and the stochastic integrals are defined with respect to general $\mathbb{L}^2$--martingales.
In other words, we consider discrete-time and continuous-time  mean-field systems of BSDEs and McKean--Vlasov BSDEs in a unified setting, where the driving processes are general and include both diffusions and jumps.
We consider first a ``path-dependent'' version of  mean-field systems of BSDEs and McKean--Vlasov BSDEs, where the generator depends on the initial segment of the solution $Y$, see 
\eqref{mfBSDE_initial_path_triplet} and
\eqref{MVBSDE_initial_path_onlyY}. 
Then, we also provide existence and uniqueness results for ``classic'' mean-field systems of BSDEs and McKean--Vlasov BSDEs that depend only on the instantaneous value of $Y$, see \eqref{mfBSDE_instantaneous_triplet} and \eqref{MVBSDE_instantaneous_triplet}, under weaker assumptions.


\subsection{Mean-field system of BSDEs}
\label{subsec:MFBSDE}

\subsubsection{Framework and structure of mean-field systems of BSDEs}
\label{subsubsec:Framework_structure_mfBSDE}
We will first introduce the setting for the initial existence and uniqueness theorem, which concerns mean-field systems of BSDEs, \textit{i.e.}, systems of BSDEs where the empirical measure of the processes affects the generator of the equation.
Let us fix a stochastic basis $(\Omega,\mathcal{G},\mathbb{G},\mathbb{P})$, a $\mathbb{G}-$stopping time $T$, $N \in \mathbb{N}$ and a $\hat{\beta}>0$; recall that $\mathscr{N}=\{1,\dots,N\}$. 
The stochastic basis supports the following:
\begin{enumerate} [label=\textrm{\textbf{(MF\arabic*)}}]
\item\label{mf_martingale_integrators} \label{G1}
    The family of martingales $(\mathbf{X}^i)_{i\in\mathscr{N}}$, where for each $i\in\mathscr{N}$
    \begin{align*}
       \mathbf{X}^i=(X^{i,\circ},X^{i,\natural}) \in  \mathcal{H}^2(\mathbb{G};\mathbb{R}^p) \times \mathcal{H}^{2,d}(\mathbb{G};\mathbb{R}^n)
    \end{align*}%
    such that $M_{\mu^{X^{i,\natural}}}[\Delta X^{i,\circ}|\widetilde{\mathcal{P}}^\mathbb{G}]=0.$\footnote{Since the filtration $\mathbb{G}$ is given as well as the pairs $\mathbf{X}^i$, for $i\in\mathscr{N}$, we will make use of $C^{(\mathbb{G},\mathbf{X}^i)}$, resp. 
    $c^{(\mathbb{G},\mathbf{X}^i)}$, as defined in \eqref{def_C}, resp. \eqref{def_c}. 
    Moreover, we will use the kernels $K^{(\mathbb{G},\mathbf{X}^i)}$ as determined by \eqref{def:Kernels}.}
\item \label{mf_generator}
    The generator $f^{N}:=(f^{i,N})_{i\in\mathscr{N}}$, where for $i\in\mathscr{N}$
    \begin{align*}
    f^{i,N}: \Omega \times \mathbb{R}_+ \times \mathbb{X} \times \mathbb{R}^{d \times p} \times \mathbb{R}^d \times 
    \mathscr{P}(\mathscr{X})\longrightarrow \mathbb{R}^{d},\footnotemark
    \end{align*}%
\footnotetext{In other words, the state space of the generator $f^{N}$ is $\mathbb{R}^{dN}$.}%
    where $(\mathbb{X},\rho_{\mathbb{X}})$ and $(\mathscr{X}, \rho_{\mathscr{X}})$ are metric spaces.\footnote{In the current work, $(\mathbb{X}, \rho_{\mathbb{X}})$ will be either $( \mathbb{D}^d,\rho_{J_1})$ or $(\mathbb{R}^d, \rho_{|\cdot|})$, while $\mathscr{X}$ will be either $\mathbb{X}\times \mathbb{R}^{d\times p}\times \mathbb{R}^d$ or $\mathbb{X}$.
    For the sake of simplicity, we introduce the condition \ref{mf_generator} using the metric spaces $\mathbb{X}$ and $\mathscr{X}$, which we will later specify in the statement of the theorems, based on the cases we consider.}
    The generator is, additionally, such that for any $(y,z,u,\mu)$ $\in \mathbb{X} \times \mathbb{R}^{d \times p} \times \mathbb{R}^d \times \mathscr{P}(\mathscr{X})$ the map 
    \begin{align*}
        (\omega,t) \longmapsto f^N(\omega,t,y,z,u,\mu) \text{ is }\mathbb{G}-\text{progressively measurable}
    \end{align*}
    and satisfies the following (stochastic) Lipschitz condition: for any $(y,z,u,\mu)$, $(y',z',u',\mu')  \in \mathbb{X} \times \mathbb{R}^{d \times p} \times \mathbb{R}^d \times \mathscr{P}(\mathscr{X})$ it holds that
    \begin{align*}
    \begin{multlined}[0.9\textwidth]
    |f^{i,N}(\omega,t,y,z,u,\mu) - f^{i,N}(\omega,t,y',z',u',\mu')|^2\\
    \leq  r^{i,N}(\omega,t) \rho_{\mathbb{X}}(y,y')^2
        + \vartheta^{i,N,\circ}(\omega,t)  |z - z'|^2 \\ 
        +  \vartheta^{i,N,\natural}(\omega,t) |u - u'|^2 
        + \vartheta^{i,N,*}(\omega,t)  \mathcal{W}^2_{2,\rho_{\mathscr{X}}}\left(\mu,\mu'\right),
    \end{multlined}
    \end{align*}
     where  $(r,\vartheta^{\circ},\vartheta^{\natural},\vartheta^{i,N,*}): \left(\Omega \times \mathbb{R}_+, \mathcal{P}^\mathbb{G}\right)
         \longrightarrow \left(\mathbb{R}^4_+,\mathcal{B}\left(\mathbb{R}^4_+\right)\right)$.
\item \label{mf_alpha_A} 
    The $\mathbb{G}$--predictable and c\`adl\`ag processes
    \begin{align*}
      A^{(\mathbb{G},\mathbf{X}^i,f^{N})}_\cdot := \int_{0}^{\cdot}\alpha^2_s\,\ud C^{(\mathbb{G},\mathbf{X}^i)}_s,
    \end{align*}
    are such that 
    there exists $\Phi > 0$ with
    \begin{align*}
    \Delta A^{(\mathbb{G},\mathbf{X}^i,f^{N})}_t(\omega) \leq \Phi,\ \mathbb{P} \otimes C^{(\mathbb{G},\mathbf{X}^i)}-\text{a.e.},
    \end{align*} for all $i \in \mathscr{N}$,
    where $\alpha^2_{\cdot} := \max\{(\sqrt{r^{i}_{\cdot}})_{i\in\mathscr{N}},(\vartheta^{i,N,\circ}_{\cdot})_{i\in\mathscr{N}}, (\vartheta^{i,N,\natural}_{\cdot})_{i\in\mathscr{N}}, \big(h(\vartheta^{i,N,*}_{\cdot})\big)_{i\in\mathscr{N}}\}$, with $\alpha^2>0$, for a function $h$ to be specified.\footnote{We suppress the dependence on the choice of the function $h$ as in the statement of each theorem this function will be fixed.}

\item\label{mf_A_cases}
\begin{enumerate}[label=\textrm{\textbf{(\roman*)}}]
\item\label{mf_StochExp_Bound}  
    There exists $ \Lambda_{\hat{\beta}} > 0$ such that 
    \begin{align*}
    \max_{i \in \mathscr{N}}\Big\{\mathcal{E}\big(\hat{\beta} A^{(\mathbb{G},\mathbf{X}^i,f^{N})} \big)_{T}\Big\} \leq \Lambda_{\hat{\beta}}. 
\end{align*}
\item \label{mf_A_identical} For $i,j \in \mathscr{N}$ we have $ A^{(\mathbb{G},\mathbf{X}^i,f^{N})}= A^{(\mathbb{G},\mathbf{X}^j,f^{N})}$.\footnote{The equality is understood up to evanescence. Moreover, in view of the definition of $A^{(\mathbb{G},\mathbf{X}^i,f^{N})}$, for $i\in\mathscr{N}$, this condition is equivalent to $C^{(\mathbb{G},\mathbf{X}^i)}= C^{(\mathbb{G},\mathbf{X}^j)}$ for $i,j\in\{1,\ldots,N\}$. We prefer to present it in the way we did, because it will be more convenient for the justification of the computations.}
\end{enumerate} 

\item\label{mf_Theta_Gamma} 
    There exist functions $(\Theta^i)_{i \in \mathscr{N}}$ and $\Gamma$ as in \cref{def_Gamma_function}, where for each $i\in\mathscr{N}$ the data for the definition are 
    $(\mathbb{G}, \mathbf{X}^i, C^{(\mathbb{G},\mathbf{X}^i)},K^{(\mathbb{G},\mathbf{X}^i)})$.   
\item \label{mf_terminal_rv} 
    The terminal conditions $(\xi^{i,N})_{i \in \mathscr{N}}$ satisfy, for $i\in\mathscr{N}$, that 
    \begin{align*}
        \xi^{i,N} \in \mathbb{L}^2_{\hat{\beta}}\big(\mathcal{G}_T,A^{(\mathbb{G},\mathbf{X}^i,f^{N})};\mathbb{R}^d\big).
    \end{align*}  
\item\label{mf_Integrability_Generator} 
    For every $i\in\mathscr{N}$,  $\big\| \frac{1}{\alpha} f^i(\cdot,0,0,0,\delta_0) \big\|_{\mathbb{H}^{2}_{\hat{\beta}}(\mathbb{G},A^{\mathbb{G}, \mathbf{X}^i,f^{N})},C^{(\mathbb{G}, \mathbf{X}^i)};\mathbb{R}^d)}<\infty$, where $\delta_0$ is the Dirac measure concentrated at $0$, the neutral element of the addition for the space $\mathscr{X}.$
\end{enumerate}

\begin{remark}
\label{rem:About_MFBSDE_framework}
\begin{enumerate}
\item \label{rem:about_a2_h_function}
In the description of \ref{mf_alpha_A}, more precisely in the definition of $\alpha^2$, the introduction of the function $h$ appears to be unavoidable in order to maintain a simple and unified presentation of the framework.
Let us initially focus on the other terms that appear in the definition of $\alpha^2$. 
That is, the bounds $r^{i,N}, \vartheta^{i,N,\circ}$ and $\vartheta^{i,N,\natural}$. 
In view of the definition of $\alpha^2$, see \ref{mf_alpha_A}, we have for $i\in\mathscr{N}$ and $(\omega,t)\in\llbracket 0,T\rrbracket$:
\begin{align}
    \label{properties_stoch_bounds}
    \frac{r^{i,N}_t}{\alpha^2_t}\leq \alpha^2_t, \quad
    \frac{\vartheta^{i,N,\circ}_t}{\alpha^2_t}\leq 1 \quad \text{ and } \quad
    \frac{\vartheta^{i,N,\natural}_t}{\alpha^2_t}\leq 1. 
\end{align}
In the course of the computations, a similar bound will be needed for $\vartheta^{i,N,*}$.
However, the structure of $\mathscr{X}$ dictates different approaches. 
On the one hand, when $\mathscr{X}$ is the product space $\mathbb{X}\times \mathbb{R}^{d\times p} \times \mathbb{R}^d$ the function $h$ is specified as the identity.
This leads to the bound
\begin{align}
    \label{properties_stoch_bounds_onlyWasserstein_identity}
    \frac{\vartheta^{i,N,*}_t}{\alpha^2_t}\leq 1,
\end{align}
 for $i\in\mathscr{N}$ and $(\omega,t)\in\llbracket 0,T\rrbracket$.
In other words, the bounds $(\vartheta^{i,N,*})_{i\in\mathscr{N}}$ should behave as the bounds $(\vartheta^{i,N,\circ})_{i\in\mathscr{N}}$ and $(\vartheta^{i,N,\natural})_{i\in\mathscr{N}}$. 
Without delving into further details at this point, this bound should be complemented with the additional condition $\alpha^2 \geq 1$ in order for the desired norms to be constructed; the details are presented in the proofs of \cref{thm:mfBSDE_initial_path_triplet,thm:mfBSDE_instantaneous_triplet}.
On the other hand, when $\mathscr{X}$ is the space $\mathbb{X}$ the function $h$ is specified as the square root.
This leads to the bound
\begin{align}
    \label{properties_stoch_bounds_onlyWasserstein_squareroot}
    \frac{\vartheta^{i,N,*}_t}{\alpha^2_t}\leq \alpha^2_t,
\end{align}
for $i\in\mathscr{N}$ and $(\omega,t)\in\llbracket 0,T\rrbracket$.
In other words, the bounds $(\vartheta^{i,N,*})_{i\in\mathscr{N}}$ should behave as the bounds $(r^{i})_{i\in\mathscr{N}}$; the details are presented in the proofs of \cref{thm:mfBSDE_initial_path_onlyY,thm:mfBSDE_instantaneous_onlyY_star}.

\item \label{rem:about_G6}
    Assumption \ref{mf_A_cases}.\ref{mf_A_identical} is not restrictive.
    Indeed, this is the case when 
    $(\mathbf{X}^i)_{i \in \mathscr{N}}$ is a family whose elements are identically distributed and define deterministic processes $C^{(\mathbb{G},\mathbf{X}^i)}$, \textit{i.e.}, from \cref{prop:independent_increments} have independent increments. 
    Then, 
    from \eqref{def_C} and \citet[6.23 Theorem]{he2019semimartingale} we would have for $i,j \in \mathscr{N}$ and $t \in \mathbb{R}_+$ that
        \begin{align*}
        C^{(\mathbb{G},\mathbf{X}^i)}_t 
        &=\mathbb{E}\Big[C^{(\mathbb{G},\mathbf{X}^i)}_t\Big] 
        = \mathbb{E}\Big[\textup{Tr}\big[\langle X^{i,\circ} \rangle^{\mathbb{G}}_t\big] \Big] + \mathbb{E}\left[|I|^2 * \nu^{(\mathbb{G},X^{i,\natural})}_t\right]\\
        &= \mathbb{E}\Big[\big| X^{i,\circ}_t\big|^2\Big] 
        + \mathbb{E}\left[|I|^2 * \mu^{X^{i,\natural}}_t\right] 
        = \mathbb{E}\Big[\big| X^{i,\circ}_t\big|^2\Big] 
        +\mathbb{E}\Big[\big|X^{i,\natural}_t\big|^2\Big]\\
        &= \mathbb{E}\Big[\big| X^{j,\circ}_t\big|^2\Big] 
        + \mathbb{E}\Big[\big|X^{j,\natural}_t\big|^2\Big] = \mathbb{E}\Big[\big| X^{j,\circ}_t\big|^2\Big] + \mathbb{E}\left[|I|^2 * \mu^{X^{j,\natural}}_t\right]\\
        &= \mathbb{E}\Big[\textup{Tr}\big[\langle X^{j,\circ} \rangle^{\mathbb{G}}_t\big] \Big] + \mathbb{E}\left[|I|^2 * \nu^{(\mathbb{G},X^{j,\natural})}_t\right] = \mathbb{E}\Big[C^{(\mathbb{G},\mathbf{X}^j)}_t\Big] = C^{(\mathbb{G},\mathbf{X}^j)}_t.
        \end{align*}
        In fact, the above argument includes that 
\begin{equation*}
    \textup{Tr}\big[\langle X^{i,\circ} \rangle^{\mathbb{G}}_t]=\textup{Tr}\big[\langle X^{j,\circ} \rangle^{\mathbb{G}}_t]\hspace{.2cm}\text{and}\hspace{.2cm}|I|^2 * \nu^{(\mathbb{G},X^{i,\natural})}_t=|I|^2 * \nu^{(\mathbb{G},X^{j,\natural})}_t.
\end{equation*}
More generally, in this situation, we have
\begin{equation*}
\begin{aligned}
    \langle X^{i,\circ} \rangle^{\mathbb{G}}_t = \mathbb{E}\Big[\langle X^{i,\circ} \rangle^{\mathbb{G}}_t\Big]& = \mathbb{E}\Big[X^{i,\circ}_t\cdot\left(X^{i,\circ}\right)^{\top}_t\Big]\\& = \mathbb{E}\Big[X^{j,\circ}_t\cdot\left(X^{j,\circ}\right)^{\top}_t\Big]= \mathbb{E}\Big[\langle X^{j,\circ} \rangle^{\mathbb{G}}_t\Big]= \langle X^{j,\circ} \rangle^{\mathbb{G}}_t.
    \end{aligned}
\end{equation*}

\item Let us repeat that Assumption \ref{mf_A_cases}.\ref{mf_StochExp_Bound} (and \ref{MV_A_cases}.\ref{MV_StochExp_Bound} below) leads to equivalence between the $\hat{\beta}$--norms and their $0$--counterparts.
However, we will need the $\hat{\beta}-$norms in order to construct the contractions, which will ultimately provide the fixed-points we are seeking.
\end{enumerate}

\end{remark}
%
%
%
%
%
%
%
%
%
%
%
%
%
Now, let us describe the structure of the mean-field we will consider. 
As we have already indicated in \ref{mf_generator}, the metric space $(\mathbb{X}, \rho_\mathbb{X})$ will be either $(\mathbb{D}^d,\rho_{J_1})$ or $(\mathbb{R}^d,\rho_{|\cdot|})$, while $(\mathscr{X},\rho_{\mathscr{X}})$ will be either $(\mathbb{X}\times \mathbb{R}^{d\times p}\times \mathbb{R}^d,\rho_{\mathbb{X}\times \mathbb{R}^{d\times p}\times \mathbb{R}^d})$
    or $(\mathbb{X},\rho_{\mathbb{X}})$ . 
This generality allows us to describe the problem of solving a mean-field BSDE system as seeking for a process $\overline{\mathbb{S}}^{N}:=(\overline{\mathbf{S}}^{i,N})_{i\in\mathscr{N}}$, where $\overline{\mathbf{S}}^{i,N}:=(Y^{i,N}, Z^{i,N},U^{i,N},M^{i,N})$ with%
\footnote{Recall \cref{rem:notation_beta_zero} for the spaces associated with $\beta=0$.}
\begin{align*}
\begin{multlined}[\textwidth]
\overline{\mathbf{S}}^{i,N} \in
    \mathcal{S}^{2}(\mathbb{G},A^{(\mathbb{G},\mathbf{X}^i,f^{N})};\mathbb{R}^d) 
    \times  \mathbb{H}^{2}(\mathbb{G},A^{(\mathbb{G},\mathbf{X}^i,f^{N})},X^{i,\circ};\mathbb{R}^{d\times p}) \\
    \times \mathbb{H}^{2}(\mathbb{G},A^{(\mathbb{G},\mathbf{X}^i,f^{N})},\mu^{X^{i,\natural}};\mathbb{R}^d) 
    \times  \mathcal{H}^{2}(\mathbb{G},A^{(\mathbb{G},\mathbf{X}^i,f^{N})},\mathbf{X}^{i,\perp_{\mathbb{G}}};\mathbb{R}^d),
\end{multlined}
\end{align*}
for every $i\in\mathscr{N}$ and such that 
\begin{align}
\begin{multlined}[0.9\textwidth]
\label{mfBSDE_general}
Y^{i,N}_t 
    = \xi^{i,N} 
        + \int^{T}_{t}f^{i,N} \Big(s,\overline{\mathbf{S}}^{\textrm{mod},i,N},L^N\big(\overline{\mathbb{S}}^{\textrm{mod},N}\big) \Big) \,\ud C^{(\mathbb{G},\mathbf{X}^i)}_s\\ 
        - \int^{T}_{t}Z^{i,N}_s \,  \ud X^{i,\circ}_s - \int^{T}_{t}\int_{\mathbb{R}^n}U^{i,N}_s(x) \, \widetilde{\mu}^{(\mathbb{G},X^{i,\natural})}(\ud s,\ud x) - \int^{T}_{t} \,\ud M^{i,N}_s,  
\end{multlined}
\tag{\textrm{MF}}
\end{align}
for $i\in\mathscr{N}$, where $\overline{\mathbf{S}}^{\textrm{mod},i,N}$ is a suitable ``modification'' (not to be understood in the probabilistic sense) of $\overline{\mathbf{S}}^{i,N}$ and $\overline{\mathbb{S}}^{\textrm{mod},N}:=(\overline{\mathbf{S}}^{\textrm{mod},i,N})_{i\in\mathscr{N}}$.\footnote{The modification is not needed when $\mathscr{X}\equiv \mathbb{X}$; see \cref{thm:mfBSDE_initial_path_onlyY}}
Before we proceed, let us emphasize once again that the aforementioned BSDE is indeed a system due to the dependence of the generator $f^{i,N}$ on the empirical measure of $\overline{\mathbb{S}}^{\textrm{mod},N}$. 
Had the empirical measure not been there, we would simply have $N$ uncoupled BSDEs.

Let us return to the two main cases, depending on the choice of the metric space $(\mathbb{X},\rho_{\mathbb{X}})$, that we will consider in this work and provide the remaining details. 

On the one hand, when $(\mathbb{X}, \rho_\mathbb{X}) = (\mathbb{D}^d,\rho_{J_1})$, we consider two subcases:$(\mathscr{X},\rho_{\mathscr{X}}) = (\mathbb{D}^d,\rho_{J_1})$ and 
$(\mathscr{X},\rho_{\mathscr{X}}) = (\mathbb{D}^d \times \mathbb{R}^{d\times p}\times \mathbb{R}^p,\rho_{\mathbb{D}^d \times \mathbb{R}^{d\times p}\times \mathbb{R}^p})$.
In the latter case, the mean-field BSDE system \eqref{mfBSDE_general} takes the form 
\begin{align}
\begin{multlined}[0.86\textwidth]
\label{mfBSDE_initial_path_triplet}
Y^{i,N}_t 
    = \xi^{i,N} 
        + \int^{T}_{t}f^{i,N}\Big(s,\overline{\mathbf{S}}^{\textrm{mod},i,N}_s,L^N\big(\overline{\mathbb{S}}^{\textrm{mod},N}_s\big) \Big) \,\ud C^{(\mathbb{G},\mathbf{X}^i)}_s\\ 
        - \int^{T}_{t}Z^{i,N}_s \,  \ud X^{i,\circ}_s - \int^{T}_{t}\int_{\mathbb{R}^n}U^{i,N}_s(x) \, \widetilde{\mu}^{(\mathbb{G},X^{i,\natural})}(\ud s,\ud x) - \int^{T}_{t} \,\ud M^{i,N}_s
\end{multlined}
\tag{$\textrm{MF}_{\textrm{Seg}}^{\overline{\mathbb{S}}^{\textrm{mod}}}$}
\end{align}
for $i\in\mathscr{N}$, with
\begin{align*}
    \overline{\mathbf{S}}^{\textrm{mod},i,N}_{\cdot}:=\big(Y^{i,N}|_{[0,\cdot]},Z^{i,N}_\cdot  c^{(\mathbb{G},\mathbf{X}^i)}_\cdot,\Gamma^{(\mathbb{G},\mathbf{X}^i,\Theta^i)}(U^{i,N})_\cdot\big) 
\end{align*}
for $i\in\mathscr{N}$, and
\begin{align*}
    \overline{\mathbb{S}}^{\textrm{mod},N} := (\overline{\mathbf{S}}^{\textrm{mod,1,N}},\dots, \overline{\mathbf{S}}^{\textrm{mod,N,N}});
\end{align*}
in other words, the system depends on the whole segment $[0,\cdot]$ of the solution process and the empirical measure on all the components of the solution process.
In the former case, the mean-field BSDE system \eqref{mfBSDE_general} takes the form, for $i\in\mathscr{N}$,
\begin{align}
\begin{multlined}[0.86\textwidth]
\label{mfBSDE_initial_path_onlyY}
Y^{i,N}_t 
    = \xi^{i,N} 
        + \int^{T}_{t}f^{i,N}\Big(s,\overline{\mathbf{S}}^{\textrm{mod},i,N}_s,L^N\big(\overline{\mathbb{Y}}^{N}|_{[0,s]}\big) \Big) \,\ud C^{(\mathbb{G},\mathbf{X}^i)}_s\\ 
        - \int^{T}_{t}Z^{i,N}_s \,  \ud X^{i,\circ}_s - \int^{T}_{t}\int_{\mathbb{R}^n}U^{i,N}_s(x) \, \widetilde{\mu}^{(\mathbb{G},X^{i,\natural})}(\ud s,\ud x) - \int^{T}_{t} \,\ud M^{i,N}_s.
\end{multlined}
\tag{$\textrm{MF}_{\textrm{Seg}}^{\overline{\mathbb{Y}}}$}
\end{align}
with 
\begin{gather*}
    \overline{\mathbb{Y}}^{N}|_{[0,\cdot]} :=(Y^{1,N}|_{[0,\cdot]}, \dots,Y^{N,N}|_{[0,\cdot]});
\end{gather*}
in other words, the system depends on the whole segment $[0,\cdot]$ of the solution process but the empirical measure only on the $Y$-component of the solution process.

On the other hand, when $(\mathbb{X}, \rho_\mathbb{X}) = (\mathbb{R}^d,\rho_{|\cdot|})$, we consider the two subcases: 
$(\mathscr{X},\rho_{\mathscr{X}}) = (\mathbb{D}^d \times \mathbb{R}^{d\times p}\times \mathbb{R}^p,\rho_{\mathbb{D}^d \times \mathbb{R}^{d\times p}\times \mathbb{R}^p})$
and 
$(\mathscr{X},\rho_{\mathscr{X}}) = (\mathbb{D}^d,\rho_{J_1})$.
In the former case, the mean-field BSDE system \eqref{mfBSDE_general} takes the form 
\begin{align}
\begin{multlined}[0.84\textwidth]
\label{mfBSDE_instantaneous_b_triplet}
Y^{i,N}_t 
    = \xi^{i,N} 
        + \int^{T}_{t}f^{i,N} \Big(s,\overline{\mathbf{S}}^{\textrm{mod},i,N}_s,L^N\big(\overline{\mathbb{S}}^{\textrm{mod},N}_s\big) \Big) \,\ud C^{(\mathbb{G},\mathbf{X}^i)}_s\\ 
        - \int^{T}_{t}Z^{i,N}_s \,  \ud X^{i,\circ}_s - \int^{T}_{t}\int_{\mathbb{R}^n}U^{i,N}_s(x) \, \widetilde{\mu}^{(\mathbb{G},X^{i,\natural})}(\ud s,\ud x) - \int^{T}_{t} \,\ud M^{i,N}_s
\end{multlined}
\tag{\ref{mfBSDE_instantaneous_triplet}}
\end{align}
for $i\in\mathscr{N}$, with
\begin{gather*}
    \overline{\mathbf{S}}^{\textrm{mod},i,N}_{\cdot}:=\big(Y^{i,N}_\cdot,Z^{i,N}_\cdot  c^{(\mathbb{G},\mathbf{X}^i)}_\cdot,\Gamma^{(\mathbb{G},\mathbf{X}^i,\Theta^i)}(U^{i,N})_\cdot\big);
\end{gather*}
in other words, the system depends on the instantaneous value of the solution process and the empirical measure on all the components of the solution process.
In the latter case, the mean-field BSDE system \eqref{mfBSDE_general} takes the form, for $i\in\mathscr{N}$,
\begin{align}
\begin{multlined}[0.86\textwidth]
\label{mfBSDE_instantaneous_b_onlyY}
Y^{i,N}_t 
    = \xi^{i,N} 
        + \int^{T}_{t}f^{i,N} \Big(s,\overline{\mathbf{S}}^{\textrm{mod},i,N}_s,L^N\big(\overline{\mathbb{Y}}^{N}_s\big) \Big) \,\ud C^{(\mathbb{G},\mathbf{X}^i)}_s\\ 
        - \int^{T}_{t}Z^{i,N}_s \,  \ud X^{i,\circ}_s - \int^{T}_{t}\int_{\mathbb{R}^n}U^{i,N}_s(x) \, \widetilde{\mu}^{(\mathbb{G},X^{i,\natural})}(\ud s,\ud x) - \int^{T}_{t} \,\ud M^{i,N}_s
\end{multlined}
\tag{$\textrm{MF}_{\textrm{Inst}}^{\overline{\mathbb{Y}}}$}
\end{align}
for $i\in\mathscr{N}$, with
\begin{gather*}
    \overline{\mathbb{Y}}^{N}_{\cdot} :=(Y^{1,N}_{\cdot}, \dots,Y^{N,N}_{\cdot}).  
\end{gather*}
in other words, the system depends on the instantaneous value of the solution process but the empirical measure only on the $Y$-component of the solution process.

%
%
%
%
%
%
%
%
%
%
%
%
%
%
%
%


\subsubsection{Statements of the mean-field BSDE theorems}
\label{subsubsec:mf_theorems_statements}

Let us now state the main theorems on the existence and uniqueness of the solution for the aforementioned mean-field BSDE systems.
\begin{theorem}\label{thm:mfBSDE_initial_path_triplet}
Under \ref{mf_martingale_integrators}, \ref{mf_generator} for $(\mathbb{X},\rho_{\mathbb{X}}) = (\mathbb{D}^d, \rho_{J_1})$ and $(\mathscr{X},\rho_{\mathscr{X}}) = (\mathbb{D}^d \times \mathbb{R}^{d\times p}\times \mathbb{R}^p,\rho_{\mathbb{D}^d \times \mathbb{R}^{d\times p}\times \mathbb{R}^p})$, \ref{mf_alpha_A} for $[0,\infty)\ni x\overset{h}{\mapsto} x $ and $\alpha^2_\cdot\geq 1$, \ref{mf_A_cases}, \ref{mf_Theta_Gamma}, \ref{mf_terminal_rv} and  \ref{mf_Integrability_Generator}, if
\begin{align*}
    \max\left\{4,\frac{2\Lambda_{\hat{\beta}}}{\hat{\beta}}\right\} 
    \widehat{M}^{\Phi}_{\star}(\hat{\beta}) < 1,
\end{align*} 
then the mean-field BSDE \eqref{mfBSDE_initial_path_triplet} admits a unique solution $\overline{\mathbb{S}}^{N}:=(\overline{\mathbf{S}}^{i,N})_{i\in\mathscr{N}}$ such that, for every $i\in\mathscr{N}$, $\|\overline{\mathbf{S}}^{i,N}\|_{\star,\hat{\beta},\mathbb{G}, A^{(\mathbb{G},\mathbf{X}^i,f^{N})},\mathbf{X}^i}<\infty$\footnote{The norm was defined in \eqref{def:star_norm}}.
\end{theorem}

\begin{theorem}\label{thm:mfBSDE_initial_path_onlyY}
Under \ref{mf_martingale_integrators}, \ref{mf_generator} for $(\mathbb{X},\rho_{\mathbb{X}}) = (\mathscr{X},\rho_{\mathscr{X}}) = (\mathbb{D}^d, \rho_{J_1})$, \ref{mf_alpha_A} for $[0,\infty)\ni x\overset{h}{\mapsto} \sqrt{x}$, \ref{mf_A_cases}.\ref{mf_StochExp_Bound}, \ref{mf_Theta_Gamma}, \ref{mf_terminal_rv} and  \ref{mf_Integrability_Generator}, if
\begin{align*}
    \max\left\{2,\frac{2\Lambda_{\hat{\beta}}}{\hat{\beta}}\right\} 
    \widehat{M}^{\Phi}_{\star}(\hat{\beta}) < 1,
\end{align*} 
then the mean-field BSDE \eqref{mfBSDE_initial_path_onlyY} admits a unique solution $\overline{\mathbb{S}}^{N}:=(\overline{\mathbf{S}}^{i,N})_{i\in\mathscr{N}}$ such that, for every $i\in\mathscr{N}$, 
$\|\overline{\mathbf{S}}^{i,N}\|_{\star,\hat{\beta},\mathbb{G}, A^{(\mathbb{G},\mathbf{X}^i,f^{N})},\mathbf{X}^i}<\infty$.
\end{theorem}

\begin{theorem}\label{thm:mfBSDE_instantaneous_triplet}
Under \ref{mf_martingale_integrators}, \ref{mf_generator} for $(\mathbb{X},\rho_{\mathbb{X}}) = (\mathbb{R}^d,\rho_{|\cdot|})$ and $(\mathscr{X},\rho_{\mathscr{X}}) = (\mathbb{R}^d \times \mathbb{R}^{d\times p}\times \mathbb{R}^p,\rho_{\mathbb{R}^d \times \mathbb{R}^{d\times p}\times \mathbb{R}^p})$, \ref{mf_alpha_A} for $[0,\infty)\ni x\overset{h}{\mapsto} x $ and $\alpha^2_\cdot\geq 1$, \ref{mf_A_cases}.\ref{mf_A_identical}, \ref{mf_Theta_Gamma}, \ref{mf_terminal_rv} and  \ref{mf_Integrability_Generator}, if
\begin{align*}
    4 \widehat{M}_{\scalebox{0.4}{$\blacksquare$}}^{\Phi}(\hat{\beta}) < 1,
\end{align*} 
then the mean-field BSDE \eqref{mfBSDE_instantaneous_b_triplet} admits a unique solution $\overline{\mathbb{S}}^{N}:=(\overline{\mathbf{S}}^{i,N})_{i\in\mathscr{N}}$ such that, for every $i\in\mathscr{N}$, 
$\|\overline{\mathbf{S}}^{i,N}\|_{\text{\scalebox{0.5}{$\blacksquare$}},\hat{\beta},\mathbb{G}, A^{(\mathbb{G},\mathbf{X}^i,f^{N})},\mathbf{X}^i}<\infty$\footnote{The norm was defined in \eqref{def:blacksquare_norm}}.
\end{theorem}

\begin{theorem}\label{thm:mfBSDE_instantaneous_onlyY_star}
Under \ref{mf_martingale_integrators}, \ref{mf_generator} for $(\mathbb{X},\rho_{\mathbb{X}}) = (\mathscr{X},\rho_{\mathscr{X}}) = (\mathbb{R}^d,\rho_{|\cdot|})$, \ref{mf_alpha_A} for $[0,\infty) \ni x\overset{h}{\mapsto} \sqrt{x}$, \ref{mf_A_cases}.\ref{mf_StochExp_Bound}, \ref{mf_Theta_Gamma}, \ref{mf_terminal_rv} and  \ref{mf_Integrability_Generator}, if
\begin{align*}
    \max \Big\{ 2,\frac{2\Lambda_{\hat{\beta}}}{\hat{\beta}} \Big\} \widehat{M}_{\star}^{\Phi}(\hat{\beta}) < 1,
\end{align*} 
then the mean-field BSDE \eqref{mfBSDE_instantaneous_b_onlyY} admits a unique solution $\overline{\mathbb{S}}^{N}:=(\overline{\mathbf{S}}^{i,N})_{i\in\mathscr{N}}$ such that, for every $i\in\mathscr{N}$, 
$\|\overline{\mathbf{S}}^{i,N}\|_{\star,\hat{\beta},\mathbb{G}, A^{(\mathbb{G},\mathbf{X}^i,f^{N})},\mathbf{X}^i}<\infty$.
\end{theorem}

\begin{theorem}\label{thm:mfBSDE_instantaneous_onlyY_blacksquare}
Under \ref{mf_martingale_integrators}, \ref{mf_generator} for $(\mathbb{X},\rho_{\mathbb{X}}) = (\mathscr{X},\rho_{\mathscr{X}}) = (\mathbb{R}^d,\rho_{|\cdot|})$, \ref{mf_alpha_A} for $[0,\infty) \ni x\overset{h}{\mapsto} \sqrt{x}$,  \ref{mf_A_cases}.\ref{mf_A_identical}, \ref{mf_Theta_Gamma}, \ref{mf_terminal_rv} and  \ref{mf_Integrability_Generator}, if
\begin{align*}
    2 \widehat{M}_{\scalebox{0.4}{$\blacksquare$}}^{\Phi}(\hat{\beta}) < 1,
\end{align*} 
then the mean-field BSDE \eqref{mfBSDE_instantaneous_b_onlyY} admits a unique solution $\overline{\mathbb{S}}^{N}:=(\overline{\mathbf{S}}^{i,N})_{i\in\mathscr{N}}$ such that, for every $i\in\mathscr{N}$, 
$\|\overline{\mathbf{S}}^{i,N}\|_{\text{\scalebox{0.5}{$\blacksquare$}},\hat{\beta},\mathbb{G}, A^{(\mathbb{G},\mathbf{X}^i,f^{N})},\mathbf{X}^i}<\infty$.
\end{theorem}

\begin{remark}
The last two statements differ only in terms of the conditions and norms of the solution.
\end{remark}

\begin{remark}\label{rem:notational_simplification}
We have avoided using product spaces in the statement of the aforementioned theorems for the sake of notational simplicity. 
However, the use of product spaces will be handy within the proofs. 
To this end, let us mention once again that we adopt the following convention hereinafter: whenever we consider a (finite) Cartesian product of normed spaces, the squared norm on this product will be the sum of the squared norms of the normed spaces used to construct the Cartesian product, and it will be simply denoted by $\|\cdot\|$, except if otherwise stated.
\end{remark}


\subsubsection{Proofs of the mean-field BSDE theorems}

Given that we have fixed our framework and the associated data, and for the sake of notational simplicity (again) within the proofs, we will drop the (seemingly) ``redundant'' notation as follows: for $i\in\mathscr{N}$
the objects $A^{(\mathbb{G},\mathbf{X}^i,f^{N})}$, $C^{(\mathbb{G},\mathbf{X}^i)}$, $c{(\mathbb{G},\mathbf{X}^i)}$,  
$K^{(\mathbb{G},\mathbf{X}^i)}$, 
$\Gamma^{(\mathbb{G},\mathbf{X}^i,\Theta^i)}$, 
and $\tnorm{\cdot}^{(\mathbb{G},\mathbf{X}^i)}$ 
will be denoted by $ A^{i}$, $C^{i}$, $c^i$, $K^{i}$, 
$\Gamma^{i}$, and $\tnorm{\cdot}^{i}$ respectively.
Moreover, within the proofs, we will omit the dependence of the norms on $\mathbb{G}$, $\mathcal{G}_T$, $C^{i}$, $A^{i}$ and the Euclidean state spaces.
More precisely, for $\beta\geq 0$, the spaces
$\mathbb{L}^2_\beta(\mathcal{G}_T,A^{i};\mathbb{R}^d)$,
$\mathcal{S}^2_\beta(\mathbb{G},A^{i};\mathbb{R}^d)$,
$\mathscr{S}^{2}_{\beta}(\mathbb{G},A^{i};\mathbb{R}^d)$,
$\mathbb{H}^2_\beta(\mathbb{G},A^{i},C^{i};\mathbb{R}^d)$
$\mathbb{H}^2_\beta(\mathbb{G},A^{i},X^{i,\circ};\mathbb{R}^{d\times p})$,
$\mathbb{H}^2_\beta(\mathbb{G},A^{i},\mu^{X^{i,\natural}};\mathbb{R}^d)$,
$\mathcal{H}^2_\beta(\mathbb{G},A^{i}, (\mathbf{X}^{i})^{\perp_\mathbb{G}};\mathbb{R}^d)$, 
and
$\mathcal{H}^2_\beta(\mathbb{G},A^{i};\mathbb{R}^d)$, 
will be simply denoted by 
$\mathbb{L}^{2,i}_{\beta}$,
$\mathcal{S}^{2,i}_{\beta}$,
$\mathscr{S}^{2,i}_{\beta}$,
$\mathbb{H}^{2,i}_{\beta}$,
$\mathbb{H}^{2,i,\circ}_{\beta}$,
$\mathbb{H}^{2,i,\natural}_{\beta}$,
$\mathcal{H}^{2,i,\perp}_{\beta}$
and
$\mathcal{H}^{2,i}_{\beta}$
 respectively.
 Additionally, for $\beta \geq 0$, we introduce the following notation for product spaces:
\begin{gather*}
    \mathcal{S}^2_{\beta,N}:=\Prod_{i = 1}^{N} \mathcal{S}^{2,i}_{\beta},\hspace{0.5em}
    \mathscr{S}^{2}_{\beta,N}:= \Prod_{i=1}^{N} \mathscr{S}^{2,i}_{\beta},\hspace{0.5em}
    \mathbb{H}^{2}_{\beta,N} :=\Prod_{i=1}^{N} \mathbb{H}^{2,i}_{\beta}, \\
    \mathbb{H}^{2,\circ}_{\beta,N} :=\Prod_{i=1}^{N} \mathbb{H}^{2,i,\circ}_{\beta}, \hspace{0.5em}
    \mathbb{H}^{2,\natural}_{\beta,N}:= \Prod_{i=1}^{N} \mathbb{H}^{2,i,\natural}_{\beta}, \hspace{0.5em}
    \mathcal{H}^{2,\perp}_{\beta,N}:=\Prod_{i=1}^{N} \mathcal{H}^{2,i,\perp}_{\beta}
    \shortintertext{ and }
    \mathscr{H}^{2}_{\beta,N}:=\mathbb{H}^{2,\circ}_{\beta,N} \times \mathbb{H}^{2,\natural}_{\beta,N}\times\mathcal{H}^{2,\perp}_{\beta,N}.
\end{gather*}

We will present in sufficient detail the proof of \cref{thm:mfBSDE_initial_path_triplet}, as this will serve as the reference proof for the other existence and uniqueness results.
\begin{proof}[Proof of \cref{thm:mfBSDE_initial_path_triplet}]
We will follow the classical approach for BSDEs defined at that generality, see \emph{e.g.} \citet{elkaroui1997general}, \citet[Theorem 3.5]{papapantoleon2018existence} or \citet{possamai2023reflections}.
Let 
\begin{align*}
(\mathbf{y}^{N},\mathbf{z}^{N},\mathbf{u}^{N},\mathbf{m}^{N}) \in 
\mathcal{S}^2_{\hat{\beta},N} \times 
\mathscr{H}^{2}_{\hat{\beta},N}
\end{align*}
with 
\begin{gather*}
\mathbf{y}^{N}:=(y^{1,N},\ldots,y^{N,N}),
\mathbf{z}^{N}:=(z^{1,N},\ldots,z^{N,N}),\\
\mathbf{u}^{N}:=(u^{1,N},\ldots,u^{N,N}) \text{ and }
\mathbf{m}^{N}:=(m^{1,N},\ldots,m^{N,N}).
\end{gather*}
Let us also define 
\begin{gather*}
    \overline{\mathbf{s}}^{i,N}_{\cdot}:=(y^{i,N}|_{[0,\cdot]},z^{i,N}_{\cdot},u^{i,N}_{\cdot}, m^{i,N}_{\cdot})
    \shortintertext{and} 
    \overline{\mathbf{s}}^{\textrm{mod},i,N}_{\cdot} :=\big(y^{i,N}|_{[0,\cdot]},z^{i,N}_{\cdot} c^{i}_{\cdot},\Gamma^{i}(u^{i,N})_{\cdot}\big)
\end{gather*}
for $i\in\mathscr{N}$.
Then, from \cref{subsec:Orthog_Decomp}, in conjunction with \cref{lem:a_priori_estimates} (for $0<\gamma<\delta\equiv\hat{\beta}$), we have the existence of 
\begin{align*}
    ({\mathbf{Y}}^{N}, {\mathbf{Z}}^{N}, {\mathbf{U}}^{N}, {\mathbf{M}}^{N})\in 
    \mathcal{S}^2_{\hat{\beta},N} \times
    \mathscr{H}^{2}_{\hat{\beta},N}
\end{align*}%
with 
\begin{gather*}
\mathbf{Y}^{N}:=(Y^{1,N},\ldots,Y^{N,N}),
\mathbf{Z}^{N}:=(Z^{1,N},\ldots,Z^{N,N}),\\
\mathbf{U}^{N}:=(U^{1,N},\ldots,U^{N,N}) \text{ and }
\mathbf{M}^{N}:=(M^{1,N},\ldots,M^{N,N}).
\end{gather*}
such that for every $i\in\mathscr{N}$ the following martingale
\begin{align*}
    L^{i,N}_{\cdot}:=\mathbb{E}\bigg[ \xi^{i,N} + \int_0^T f^{i,N}\big(s, \overline{\mathbf{s}}^{\textrm{mod},i,N}_s, L^N(\overline{\mathbf{s}}^{\textrm{mod},1,N}_s,\dots,\overline{\mathbf{s}}^{\textrm{mod},N,N}_s)  \big) \textup{d}C^i_s \Big| \mathcal{G}_{\cdot}\bigg]
\end{align*}
is represented as 
\begin{align*}
    L^{i,N}_{\cdot} = \mathbb{E}\big[L_T^{i}|\mathcal{G}_0\big] + Z^{i}\cdot X^{i,\circ}_{\cdot} + U^{i}\star \widetilde{\mu}^{i,\natural}_{\cdot} + M^{i}_{\cdot} 
\end{align*}
and 
\begin{multline*}
Y^{i,N}_t := \xi^{i,N} + \int^{T}_{t}f^{i,N} \left(s,\overline{\mathbf{s}}^{\text{mod},i,N}_s,L^N(\overline{\mathbf{s}}^{\text{mod},1,N}_s,\dots,\overline{\mathbf{s}}^{\text{mod},N,N}_s) \right) \,\ud C^i_s\\ 
- \int^{T}_{t}Z^{i,N}_s \,  \ud X^{i,\circ}_s - \int^{T}_{t}\int_{\mathbb{R}^n}U^{i,N}_s(x) \, \widetilde{\mu}^{i,\natural}(\ud s,\ud x) - \int^{T}_{t} \,\ud M^{i,N}_s.
\end{multline*}
Indeed, in view of \ref{mf_terminal_rv}, the above claim is legitimate as long as the term corresponding to the second summand of Condition \eqref{conditions:a_priori_estimates} is suitably integrable.
To this end, in view of \ref{mf_generator}, \cref{lem:Gamma_is_Lipschitz}, \eqref{ineq:empirical_measures_bound} and \eqref{skorokineq}   we have the following for every $i\in\mathscr{N}$:
\begin{align*}  
&\left\|\frac{1}{\alpha}f^{i,N}\left(\cdot,\overline{\mathbf{s}}^{\text{mod},i,N}_\cdot,L^N(\overline{\mathbf{s}}^{\text{mod},1,N}_\cdot,\dots,\overline{\mathbf{s}}^{\text{mod},N,N}_\cdot) \right)\right\|^{2}_{\mathbb{H}^{2,i}_{\hat{\beta}}}\\
&\begin{multlined}[0.9\textwidth]
\leq
2\left\|\frac{1}{\alpha}f^{i,N}\left(\cdot,0,0,0,\delta_0 \right)\right\|^{2}_{\mathbb{H}^{2,i}_{\hat{\beta}}}\\
\hspace{-7em}+2\big\|\frac{r^{i,N}}{\alpha}\sup_{(0,\cdot]}|y^{i,N}_\cdot|\big\|^{2}_{\mathbb{H}^{2,i}_{\hat{\beta}}} 
+ 2\big\|\frac{\vartheta^{i,N,\circ}}{\alpha} z^{i,N}\big\|^{2}_{\mathbb{H}^{2,i,\circ}_{\hat{\beta}}} 
    + 4 \big\|\frac{\vartheta^{i,N,\natural}}{\alpha} u^{i,N}\big\|^{2}_{\mathbb{H}^{2,i,\natural}_{\hat{\beta}}} 
    \\
    +  \frac{2}{N} \sum_{m = 1}^{N} \bigg\{ 
    \big\|\frac{\vartheta^{i,N,*}}{\alpha}\sup_{(0,\cdot]}|y^{m,N}_\cdot|\big\|^{2}_{\mathbb{H}^{2,i}_{\hat{\beta}}}
    +
    \big\|\frac{\vartheta^{i,N,*}}{\alpha}z^{m,N}\big\|^{2}_{ \mathbb{H}^{2,i,\circ}_{\hat{\beta}}} 
    +
    2\big\|\frac{\vartheta^{i,N,*}}{\alpha}u^{m,N}\big\|^{2}_{\mathbb{H}^{2,i,\natural}_{\hat{\beta}}}
    \bigg \}
\end{multlined}\\
&\begin{multlined}[0.9\textwidth]
\overset{\eqref{properties_stoch_bounds}}{\underset{\eqref{properties_stoch_bounds_onlyWasserstein_identity}}{\leq}}
2\left\|\frac{1}{\alpha}f^{i,N}\left(\cdot,0,0,0,\delta_0 \right)\right\|^{2}_{\mathbb{H}^{2,i}_{\hat{\beta}}}
+
2\frac{\Lambda_{\hat{\beta}}}{\hat{\beta}} \|y^{i,N}\|^{2}_{\mathcal{S}^{2}} 
    + 2\|z^{i,N}\|^{2}_{\mathbb{H}^{2,i,\circ}_{\hat{\beta}}} 
    + 4 \| u^{i,N}\|^{2}_{\mathbb{H}^{2,i,\natural}_{\hat{\beta}}} 
    \\
    +  \frac{2}{N} \sum_{m = 1}^{N} \bigg\{ 
    \frac{\Lambda_{\hat{\beta}}}{\hat{\beta}} \|y^{m,N}\|^{2}_{\mathcal{S}^{2}} 
    +
    \|z^{m,N}\|^{2}_{ \mathbb{H}^{2,m,\circ}_{\hat{\beta}}} 
    +
    2\|u^{m,N}\|^{2}_{\mathbb{H}^{2,m,\natural}_{\hat{\beta}}}
    \bigg \}
\numberthis\label{ineq:mf_condition_apriori_verification}
\end{multlined}\\
 &<\infty,
\end{align*}
by \ref{mf_Integrability_Generator} and our initial assumption on $(\mathbf{y}^N,\mathbf{z}^N,\mathbf{u}^N,\mathbf{n}^N)$.
We note that in the derivation of \eqref{ineq:mf_condition_apriori_verification}, more precisely in the construction of the stochastic exponentials whose bound $\Lambda_{\hat{\beta}}$ finally appears, we have used the assumption that $\alpha\geq 1$ together with \ref{mf_A_cases} (both \ref{mf_StochExp_Bound} and \ref{mf_A_identical}).

Hence, we can indeed define the function 
\begin{align*}
    \mathfrak{S}^{N} : 
    \mathcal{S}^2_{\hat{\beta},N} \times 
    \mathscr{H}^{2}_{\hat{\beta},N}
    \longrightarrow
    \mathcal{S}^2_{\hat{\beta},N} \times
    \mathscr{H}^{2}_{\hat{\beta},N}
\end{align*}
with 
\begin{align*}
\mathfrak{S}^{N}(\mathbf{y}^{N},\mathbf{z}^{N},\mathbf{u}^{N},\mathbf{m}^{N}) := ({\mathbf{Y}}^{N}, {\mathbf{Z}}^{N}, {\mathbf{U}}^{N}, {\mathbf{M}}^{N}).
\end{align*}
If $\mathfrak{S}^{N}$ is a contraction then, using Banach's fixed point theorem, we deduce that the solution is unique. 
To this end, for $j=1,2$,  let us arbitrarily choose 
\begin{align*}
    (\mathbf{y}^{N,(j)},\mathbf{z}^{N,(j)},\mathbf{u}^{N,(j)},\mathbf{m}^{N,(j)}) \in 
\mathcal{S}^2_{\hat{\beta},N} \times
\mathscr{H}^{2}_{\hat{\beta},N}
\end{align*}
with 
\begin{align*}
\mathfrak{S}^{N}(\mathbf{y}^{N,(j)},\mathbf{z}^{N,(j)},\mathbf{u}^{N,(j)},\mathbf{m}^{N,(j)}) := ({\mathbf{Y}}^{N,(j)}, {\mathbf{Z}}^{N,(j)}, {\mathbf{U}}^{N,(j)}, {\mathbf{M}}^{N,(j)}),
\end{align*}
with 
$\mathbf{y}^{N,(j)}:=(y^{1,N,(j)},\ldots,y^{N,N,(j)})$, \emph{etc.}, 
$\mathbf{Y}^{N,(j)}:=(Y^{1,N,(j)},\ldots,Y^{N,N,(j)})$, \emph{etc.}
and $\overline{\mathbf{s}}^{i,N,(j)}, \overline{\mathbf{s}}^{\textrm{mod},i,N,(j)}$, for $i\in\mathscr{N}$, following the same rationale as above.
Additionally, for two similar objects indexed on $j$, say $\big(w^{(j)}\big)_{j=1,2}$, we will use the notation $\delta w^{(*)} :=w^{(2)}-w^{(1)} $; \emph{e.g.}, $\delta \mathbf{y}^{N,(*)}=\mathbf{y}^{N,(2)}-\mathbf{y}^{N,(1)}$.
Following this rationale loosely, we define for every $i \in \mathscr{N}$
\begin{multline}
    \delta f^{i,N,(*)}
    := f^{i,N} \left(s,\overline{\mathbf{s}}^{\text{mod},i,N,(2)}_s,L^N(\overline{\mathbf{s}}^{\text{mod},1,N,(2)}_s,\dots,\overline{\mathbf{s}}^{\text{mod},N,N,(2)}_s) \right)\\
    - f^{i,N} \left(s,\overline{\mathbf{s}}_s^{\text{mod},i,N,(1)},L^N(\overline{\mathbf{s}}_s^{\text{mod},1,N,(1)},\dots,\overline{\mathbf{s}}_s^{\text{mod},N,N,(1)}) \right).
\end{multline}
 
Following completely analogous arguments as in \eqref{ineq:mf_condition_apriori_verification}, we can extract the following inequality
\begin{align*}  
&\hspace{-1em}\sum_{i=1}^N\Big\|\frac{1}{\alpha}\delta f^{i,N,(*)}\Big\|^{2}_{\mathbb{H}^{2,i}_{\hat{\beta}}}\\
&\begin{multlined}[0.9\textwidth]
\leq 
\sum_{i=1}^N\Big\{\frac{\Lambda_{\hat{\beta}}}{\hat{\beta}} \|\delta y^{i,N,(*)}\|^{2}_{\mathcal{S}^{2,i}_{\hat{\beta}}}
    + \|\delta z^{i,N,(*)}\|^{2}_{ \mathbb{H}^{2,i,\circ}_{\hat{\beta}}} 
    + 2 \|\delta u^{i,N,(*)}\|^{2}_{\mathbb{H}^{2,i,\natural}_{\hat{\beta}}}  \Big\}
    \\
    + \sum_{m = 1}^{N} \bigg\{ 
    \frac{\Lambda_{\hat{\beta}}}{\hat{\beta}} \|\delta y^{m,N,(*)}\|^{2}_{\mathcal{S}^{2,m}_{\hat{\beta}}} 
    +
    \|\delta z^{m,N,(*)}\|^{2}_{ \mathbb{H}^{2,m,\circ}_{\hat{\beta}}} 
    +
    2\|\delta u^{m,N,(*)}\|^{2}_{\mathbb{H}^{2,m,\natural}_{\hat{\beta}}}
    \bigg \}
\end{multlined}\\
&\begin{multlined}[0.9\textwidth]
\leq 
\max\Big\{4,\frac{2}{\hat{\beta}}\Lambda_{\hat{\beta}}\Big\}
\big\| \big(\delta \mathbf{y}^{N,(*)}, \delta \mathbf{z}^{N,(*)}, \delta \mathbf{u}^{N,(*)}, \delta \mathbf{m}^{N,(*)}\big)\big\|^{2}_{\mathcal{S}^{2}_{\hat{\beta},N}\times \mathscr{H}^{2}_{\hat{\beta}},N},
\end{multlined}
\numberthis\label{ineq:mf_sum_norm_delta_f_initial_path}
\end{align*}
for $i\in\mathscr{N}$.
Then, employing \cref{cor:xi_vanishes_norms} (\emph{i.e.}, we use the fact that $\delta \mathbf{Y}^{N,(*)}_T=0$) per coordinate, \cref{prop:infima_for_M} and summing over $i\in\mathscr{N}$, we have (recall the notation for the norm on the Cartesian product)
\begin{align*}
&\big\|\delta \mathfrak{S}^{N}(\mathbf{y}^{N,(*)},\mathbf{z}^{N,(*)},\mathbf{u}^{N,(*)},\mathbf{m}^{N,(*)})\big\|^2_{\mathcal{S}^{2}_{\hat{\beta},N}\times \mathscr{H}^{2}_{\hat{\beta},N} }\\
&
\overset{\phantom{(\delta \mathbf{Y}^{N,(*)}_T=0)}}{=}
    \| (\delta Y^{i,N,(*)},\delta Z^{i,N,(*)},\delta U^{i,N,(*)},\delta M^{i,N,(*)})\|_{\mathcal{S}^{2}_{\hat{\beta}}\times \mathscr{H}^{2}_{\hat{\beta},N}} \\
&\overset{(\delta \mathbf{Y}^{N,(*)}_T=0)}{\leq}
\widehat{M}^{\Phi}_{\star}(\hat{\beta}) \sum_{i = 1}^{N}
    \Big\|\frac{1}{\alpha}\delta f^{i,N,(*)} \Big\|^{2}_{\mathbb{H}^{2,i}_{\hat{\beta}}}\\
&\overset{\eqref{ineq:mf_sum_norm_delta_f_initial_path}}{\underset{\phantom{(\delta \mathbf{Y}^{N,(*)}_T=0}}{\leq}}
\max\Big\{4,\frac{2}{\hat{\beta}}\Lambda_{\hat{\beta}}\Big\}
\widehat{M}^{\Phi}_{\star}(\hat{\beta}) 
\big\| \big(\delta \mathbf{y}^{N,(*)}, \delta \mathbf{z}^{N,(*)}, \delta \mathbf{u}^{N,(*)}, \delta \mathbf{m}^{N,(*)}\big)\big\|^{2}_{\mathcal{S}^{2}_{\hat{\beta},N}\times _{\mathscr{H}^{2}_{\hat{\beta}},N} },
\end{align*}
which provides the desired contraction.
\end{proof}
\begin{remark}\label{rem:about_proof_thm_mfBSDE_initial_path}
Let us provide at this point some comments related to the proof of \cref{thm:mfBSDE_initial_path_triplet}.
\begin{enumerate}
    \item 
        In order to control the Wasserstein distance between empirical measures, we derived
        \begin{multline*}
        \mathcal{W}_{2,\rho_{\mathscr{X}}}^2 \big(L^N (\overline{\mathbf{s}}^{\text{mod},1,N,(2)}_s,\dots,\overline{\mathbf{s}}^{\text{mod},N,N,(2)}_s), L^{N}(\overline{\mathbf{s}}^{\text{mod},1,N,(1)}_s,\dots,\overline{\mathbf{s}}^{\text{mod},N,N,(1)}_s)\big)\\
        \leq 
            \frac{1}{N}\sum_{m = 1}^{N}\Big[\sup_{r \in [0,s]}|\delta y^{m,N,(*)}_r|^2
            +|(\delta z_s^{m,N,(*)})c^{m}_t|^2
            +2\Big(\tnorm{\delta u_s^{m,N,(*)}}^{m}_s\Big)^2\Big]. 
        \end{multline*}
        The reader may observe that no expectation appears on the right-hand side of the inequality, which is something to be expected because of the nature of the left-hand side.
        Indeed, there we have a Wasserstein distance of empirical measures, which depends on $\omega$.
        Hence, in general, there cannot exist a deterministic upper bound for this random variable. 
    \item 
    It is easy to show that
    $L^N(\overline{\mathbb{S}}|_{[0,s]}^{\text{mod},N})$ is a $\mathbb{G}-$optional process; see also the proof of \cref{Lemma_5.4} for similar arguments. 
    \item\label{rem:about_proof_thm_mfBSDE_initial_path_2}
    In the derivation of \eqref{ineq:mf_condition_apriori_verification} (and of \eqref{ineq:mf_sum_norm_delta_f_initial_path}) one faces the problem of multiplying the running maximum of processes lying within $\mathcal{S}^2_{\hat{\beta}}(\mathbb{G},A^{(\mathbb{G},\mathbf{X}^i,f^{N})};\mathbb{R}^d)$ with the stochastic exponential associated to $A^{(\mathbb{G},\mathbf{X}^m,f^{N})}$, for $m\neq i$. 
    In general, one cannot derive such estimates, except for special cases; this is the purpose of introducing Condition \textup{\ref{mf_A_cases}.\ref{mf_StochExp_Bound}}-\ref{mf_A_identical}.

 \item 
    In view of \cref{rem:about_proof_thm_MVBSDE_initial_path}.\ref{rem:about_proof_thm_MVBSDE_initial_path_3}, we may also consider to prove the existence and uniqueness of the solution of \eqref{mfBSDE_instantaneous_triplet} and \eqref{mfBSDE_instantaneous_b_onlyY} under a condition that involves the mean of the stochastic exponentials, however that is not possible.
    This can be easily explained.
    To do this, we must recall and combine two facts. 
    The first is the previous remark \ref{rem:about_proof_thm_mfBSDE_initial_path_2}.
    The second one is that due to \eqref{ineq:MV_Wasserstein_bound} the Wasserstein distance provides an expectation, which allows us to integrate it with respect to a stochastic exponential, thus factorizing the respective mean values.
    However, this is not the case for the inequality before \eqref{ineq:mf_condition_apriori_verification} (and of \eqref{ineq:mf_sum_norm_delta_f_initial_path}).
\end{enumerate}
\end{remark}

\begin{proof}[Proof of \cref{thm:mfBSDE_initial_path_onlyY}]
    Essentially, the Proof of \cref{thm:mfBSDE_initial_path_triplet} can be followed verbatim with the following change:
In the derivation of \eqref{ineq:mf_condition_apriori_verification}, \eqref{properties_stoch_bounds_onlyWasserstein_identity} is replaced by \eqref{properties_stoch_bounds_onlyWasserstein_squareroot} in view of the choice $x\overset{h}{\mapsto}\sqrt{x}$; recall \cref{rem:About_MFBSDE_framework}.\ref{rem:about_a2_h_function}.
Hence, the assumption $\alpha\geq 1$ imposed in \cref{thm:mfBSDE_initial_path_triplet} becomes redundant.
Moreover, the presence in the empirical measure solely of the $\mathbb{Y}-$part yields \ref{mf_A_cases}.\ref{mf_A_identical} also redundant. 
\end{proof}

\begin{proof}[Proof of \cref{thm:mfBSDE_instantaneous_triplet}]
Again, the Proof of \cref{thm:mfBSDE_initial_path_triplet} can be followed verbatim with the following changes:

Obviously the initial, up to time $\cdot$ segments of the paths of the $Y-$parts in the proof of \cref{thm:mfBSDE_initial_path_triplet} should be substituted by the instantaneous values at $\cdot$, \emph{e.g.}, $y^{i,N}|_{[0,\cdot]}$ should be substituted by $y^{i,N}_{\cdot}$, \emph{etc.}  
Under the assumption that $y^{N},y^{N,(1)},y^{N,(2)} \in \mathscr{S}^2_{\hat\beta,N}$, the essential change, now, is that in the derivation of \eqref{ineq:mf_condition_apriori_verification} Assumption \ref{mf_A_cases}.\ref{mf_StochExp_Bound} is not used, and we, therefore, need to deal with the $\mathbb{H}^2-$norm of the $\mathbb{Y}-$parts of the solution.
Given the fact that this norm cannot, in general, uniquely identify the paths of the integrand process, we choose to extend the $\mathbb{H}^2-$norms, which leads to the use of the $\|\cdot\|_{\mathscr{S}^2}-$norms and consequently of the $\|\cdot\|_{\scalebox{0.4}{$\blacksquare$}}-$norms.
\end{proof}
%
%
\begin{proof}[Proof of \cref{thm:mfBSDE_instantaneous_onlyY_star}]
The Proof of \cref{thm:mfBSDE_initial_path_onlyY} can be followed verbatim with the following obvious change: the initial, up to time $\cdot$ segments of the paths of the $Y-$parts in the proof of \cref{thm:mfBSDE_initial_path_triplet} should be substituted by the instantaneous values at $\cdot$, \emph{e.g.}, $y^{i,N}|_{[0,\cdot]}$ should be substituted by $y^{i,N}_{\cdot}$, \emph{etc.}  
\end{proof}

\begin{proof}[Proof of \cref{thm:mfBSDE_instantaneous_onlyY_blacksquare}]
The proof of \cref{thm:mfBSDE_instantaneous_triplet} can be followed verbatim with the following obvious change:
In the derivation of \eqref{ineq:mf_condition_apriori_verification}, \eqref{properties_stoch_bounds_onlyWasserstein_identity} is replaced by \eqref{properties_stoch_bounds_onlyWasserstein_squareroot} in view of the choice $x\overset{h}{\mapsto}\sqrt{x}$; recall \cref{rem:About_MFBSDE_framework}.\ref{rem:about_a2_h_function}.
\end{proof}


\subsection{McKean--Vlasov BSDEs}
\label{MV-BSDE}

\subsubsection{Framework and structure of McKean--Vlasov BSDEs}
\label{subsubsec:Framework_structure_MVBSDE}
Let us now introduce the setting for the existence and uniqueness theorems related to McKean--Vlasov BSDEs, \textit{i.e.} BSDEs where the law of (part of the) solution affects the generator of the equation. 
Let us fix a stochastic basis $(\Omega,\mathcal{G},\mathbb{G},\mathbb{P})$, a $\mathbb{G}-$stopping time $T$ and $\hat{\beta}>0$.
The stochastic basis supports the following:
\begin{enumerate} [label=\textrm{\textbf{(MV\arabic*)}}]
    \item\label{MV_martingale_integrators} 
    The pair of martingales $\mathbf{X} := (X^\circ,X^{\natural}) \in \mathcal{H}^2(\mathbb{G};\mathbb{R}^p) \times \mathcal{H}^{2,d}(\mathbb{G};\mathbb{R}^n)$ such that $M_{\mu^{X^{\natural}}}[\Delta X^\circ|\widetilde{\mathcal{P}}^{\mathbb{G}}]=0$.\footnote{Since the filtration $\mathbb{G}$ is given as well as the pair $\mathbf{X}$, we will make use of $C^{(\mathbb{G},\mathbf{X})}$, resp. 
    $c^{(\mathbb{G},\mathbf{X})}$, as defined in \eqref{def_C}, resp. \eqref{def_c}. 
    Moreover, we will use the kernels $K^{(\mathbb{G},\mathbf{X})}$ as determined by \eqref{def:Kernels}.}

    \item\label{MV_generator} 
    The generator $f: \Omega \times \mathbb{R}_+ \times \mathbb{X} \times \mathbb{R}^{d \times p} \times \mathbb{R}^d \times \mathscr{P}(\mathscr{X}) \longrightarrow \mathbb{R}^d$
    where $(\mathbb{X},\rho_{\mathbb{X}})$ and $(\mathscr{X}, \rho_{\mathscr{X}})$ are metric spaces.\footnote{As in \cref{subsec:MFBSDE}, $(\mathbb{X}, \rho_{\mathbb{X}})$ will be either $( \mathbb{D}^d,\rho_{J_1})$ or $(\mathbb{R}^d, \rho_{|\cdot|})$, while $\mathscr{X}$ will be either $\mathbb{X}\times \mathbb{R}^{d\times p}\times \mathbb{R}^d$ or $\mathbb{X}$.}
    The generator is, additionally, such that for any $(y,z,u,\mu) \in \mathbb{X} \times \mathbb{R}^{d \times p} \times \mathbb{R}^d \times \mathscr{P}(\mathscr{X})$ the map 
    \begin{align*}
        (\omega,t) \longmapsto f(\omega,t,y,z,u,\mu) \text{ is }\mathbb{G}-\text{progressively measurable}
    \end{align*}
    and satisfies the following (stochastic) Lipschitz condition: for any $(y,z,u,\mu)$, $(y',z',u',\mu')  \in \mathbb{X} \times \mathbb{R}^{d \times p} \times \mathbb{R}^d \times \mathscr{P}(\mathscr{X})$ it holds that
    \begin{align*}
    \begin{multlined}[0.9\textwidth]
    |f(\omega,t,y,z,u,\mu) - f(\omega,t,y',z',u',\mu')|^2\\
    \leq  r(\omega,t) \rho_{\mathbb{X}}(y,y')^2
        + \vartheta^{\circ}(\omega,t)  |z - z'|^2 \\ 
        +  \vartheta^{\natural}(\omega,t) |u - u'|^2 
        + \vartheta^{*}(\omega,t)  \mathcal{W}^2_{2,\rho_{\mathscr{X}}}\left(\mu,\mu'\right),
    \end{multlined}
    \end{align*}
     where  $(r,\vartheta^{\circ},\vartheta^{\natural},\vartheta^{*}): \left(\Omega \times \mathbb{R}_+, \mathcal{P}^\mathbb{G}\right)
         \longrightarrow \left(\mathbb{R}^4_+,\mathcal{B}\left(\mathbb{R}^4_+\right)\right)$.

\item \label{MV_alpha_A} 
    The $\mathbb{G}$--predictable and c\`adl\`ag process
    \begin{align*}
      A^{(\mathbb{G},\mathbf{X},f)}_\cdot := \int_{0}^{\cdot}\alpha^2_s\,\ud C^{(\mathbb{G},\mathbf{X})}_s,
    \end{align*}
    is such that there exists $\Phi > 0$ with
    
    \begin{align*}
    \Delta A^{(\mathbb{G},\mathbf{X},f)}_t(\omega) \leq \Phi,\ \mathbb{P} \otimes C^{(\mathbb{G},\mathbf{X})}-\text{a.e.},
    \end{align*}
    where $\alpha^2_{\cdot} := \max\{\sqrt{r_{\cdot}},\vartheta^{\circ}_{\cdot}, \vartheta^{\natural}_{\cdot}, h(\vartheta^{*}_{\cdot})\}$, with $\alpha^2>0$, for a function $h$ to be specified.\footnote{For notational simplicity we have suppressed the dependence on $(\mathbb{G},f)$ in the notation of $\alpha$, but we have carried it in the notation of $A^{(\mathbb{G},\mathbf{X},f)}$. Moreover, we suppress the dependence on the choice of the function $h$ as in the statement of each theorem this function will be fixed.} 
\item \label{MV_A_cases} 
    \begin{enumerate}[label=\textrm{\textbf{(\roman*)}}]
        \item\label{MV_StochExp_Bound} There exists $\Lambda_{\hat{\beta}} > 0$ such that $\mathcal{E}\big(\hat{\beta} A^{(\mathbb{G},\mathbf{X}, f)}\big)_T \leq \Lambda_{\hat{\beta}} $ $\mathbb{P}-$a.s.
        \item\label{MV_StochExp_Integrable} $\mathbb{E}\Big[ \mathcal{E}\big(\hat{\beta} A^{(\mathbb{G},\mathbf{X},f)}\big)_T\Big]<\infty$.
        \item\label{MV_A_deterministic} The processes $\alpha$ and $C^{(\mathbb{G},\mathbf{X})}$ are deterministic.\footnote{
       This immediately give us that the process $A^{(\mathbb{G},\mathbf{X},f)}$ is also deterministic.}
    \end{enumerate}
\item\label{MV_Theta_Gamma} There exist functions $\Theta$ and $\Gamma$ as in \cref{def_Gamma_function}, where the data for the definition are the pair $(\mathbb{G}, \mathbf{X})$, the process $C^{(\mathbb{G},\mathbf{X})}$ and the kernels $K^{(\mathbb{G},\mathbf{X})}$.
    
    \item\label{MV_terminal_rv} The terminal condition satisfies $\xi \in \mathbb{L}^2_{\hat{\beta}}(\mathcal{G}_T,A^{(\mathbb{G},\mathbf{X}, f)};\mathbb{R}^d)$.
      
    \item\label{MV_Integrability_Generator} 
    $\big\| \frac{1}{\alpha} f(\cdot,0,0,0,\delta_0) \big\|_{\mathbb{H}^{2}_{\hat{\beta}}(\mathbb{G},A^{\mathbb{G}, \mathbf{X},f)},C^{(\mathbb{G}, \mathbf{X})};\mathbb{R}^d)}<\infty$, where $\delta_0$ is the Dirac measure concentrated at $0$, the neutral element of the addition for the space $\mathscr{X}.$
\end{enumerate}

Let us now present the McKean--Vlasov BSDEs we will consider; the cases for $(\mathbb{X},\rho_{\mathbb{X}})$ and $(\mathscr{X},\rho_{\mathscr{X}})$ will be the same as in \cref{subsec:MFBSDE}, and we will follow analogous rationale in tagging the equations.
This generality allows us to describe the problem of solving a McKean--Vlasov BSDE as seeking a process $\widetilde{\mathbf{S}}:=(Y, Z,U,M)$ with
\begin{align*}
\begin{multlined}[\textwidth]
\widetilde{\mathbf{S}} \in
    \mathcal{S}^{2}(\mathbb{G},A^{(\mathbb{G},\mathbf{X},f)};\mathbb{R}^d) 
    \times  \mathbb{H}^{2}(\mathbb{G},A^{(\mathbb{G},\mathbf{X},f)},X^{\circ};\mathbb{R}^{d\times p}) \\
    \times \mathbb{H}^{2}(\mathbb{G},A^{(\mathbb{G},\mathbf{X},f)},\mu^{X^{\natural}};\mathbb{R}^d) 
    \times  \mathcal{H}^{2}(\mathbb{G},A^{(\mathbb{G},\mathbf{X},f)},\mathbf{X}^{\perp_{\mathbb{G}}};\mathbb{R}^d),
\end{multlined}
\end{align*}%
and such that 
\begin{align}
\begin{multlined}[0.9\textwidth]
\label{MVBSDE_general}
Y_t = \xi 
        + \int^{T}_{t}f \big(s,\widetilde{\mathbf{S}}^{\textrm{mod}},\mathcal{L}(\widetilde{\mathbf{S}}^{\textrm{mod}}) \big) \,\ud C^{(\mathbb{G},\mathbf{X}^i)}_s\\ 
        - \int^{T}_{t}Z_s \,  \ud X^{\circ}_s - \int^{T}_{t}\int_{\mathbb{R}^n}U_s(x) \, \widetilde{\mu}^{(\mathbb{G},X^{\natural})}(\ud s,\ud x) - \int^{T}_{t} \,\ud M_s,  
\end{multlined}
\tag{\textrm{MV}}
\end{align}
where $\widetilde{\mathbf{S}}^{\textrm{mod}}$ is a suitable ``modification'' (not to be understood in the probabilistic sense) of $\widetilde{\mathbf{S}}$.\footnote{The modification is not needed when $\mathscr{X}\equiv \mathbb{X}$; see \cref{thm:MVBSDE_initial_path_onlyY}}
Let us return to the two main cases, depending on the choice of the metric space $(\mathbb{X},\rho_{\mathbb{X}})$, that we will consider in this work and provide the remaining details. 

On the one hand, when $(\mathbb{X}, \rho_\mathbb{X}) = (\mathbb{D}^d,\rho_{J_1})$, we consider two subcases: $(\mathscr{X},\rho_{\mathscr{X}}) = (\mathbb{D}^d,\rho_{J_1})$ and $(\mathscr{X},\rho_{\mathscr{X}}) = (\mathbb{D}^d \times \mathbb{R}^{d\times p}\times \mathbb{R}^p,\rho_{\mathbb{D}^d \times \mathbb{R}^{d\times p}\times \mathbb{R}^p})$.
In the latter case, the McKean--Vlasov BSDE \eqref{MVBSDE_general} takes the form 
\begin{align}
\begin{multlined}[0.86\textwidth]
\label{MVBSDE_initial_path_triplet}
Y_t 
    = \xi 
        + \int^{T}_{t}f\big(s,\widetilde{\mathbf{S}}^{\textrm{mod}},\mathcal{L}(\widetilde{\mathbf{S}}^{\textrm{mod}}_s) \big) \,\ud C^{(\mathbb{G},\mathbf{X}^i)}_s\\ 
        - \int^{T}_{t}Z_s \,  \ud X^{\circ}_s - \int^{T}_{t}\int_{\mathbb{R}^n}U_s(x) \, \widetilde{\mu}^{(\mathbb{G},X^{\natural})}(\ud s,\ud x) - \int^{T}_{t} \,\ud M_s
\end{multlined}
\tag{$\textrm{MV}_{\textrm{Seg}}^{\widetilde{\mathbf{S}}^{\textrm{mod}}}$}
\end{align}
 with
\begin{align*}
    \widetilde{\mathbf{S}}^{\textrm{mod}}_{\cdot}:=\big(Y|_{[0,\cdot]},Z_\cdot  c^{(\mathbb{G},\mathbf{X})}_\cdot,\Gamma^{(\mathbb{G},\mathbf{X},\Theta)}(U)_\cdot\big).
\end{align*}
In the former case, the McKean--Vlasov BSDE \eqref{MVBSDE_general} takes the form, for $i\in\mathscr{N}$,
\begin{align}
\begin{multlined}[0.86\textwidth]
\label{MVBSDE_initial_path_onlyY}
Y_t     = \xi 
        + \int^{T}_{t}f\big(s,\widetilde{\mathbf{S}}^{\textrm{mod}},\mathcal{L}(Y|_{[0,s]}) \big) \,\ud C^{(\mathbb{G},\mathbf{X})}_s\\ 
        - \int^{T}_{t}Z_s \,  \ud X^{\circ}_s - \int^{T}_{t}\int_{\mathbb{R}^n}U_s(x) \, \widetilde{\mu}^{(\mathbb{G},X^{\natural})}(\ud s,\ud x) - \int^{T}_{t} \,\ud M_s
\end{multlined}
\tag{$\textrm{MV}_{\textrm{Seg}}^{Y}$}
\end{align}
with 
\begin{align*}
    \widetilde{\mathbf{S}}^{\textrm{mod}}_{\cdot}:=\big(Y|_{[0,\cdot]},Z_\cdot  c^{(\mathbb{G},\mathbf{X})}_\cdot,\Gamma^{(\mathbb{G},\mathbf{X},\Theta)}(U)_\cdot\big).
\end{align*}
In other words, in the first case, the BSDE depends on the law of all the components of the solution process, while in the second one only on the law of $Y$; in both cases, it also depends on the whole segment $[0,\cdot]$ of the processes.

On the other hand, when $(\mathbb{X}, \rho_\mathbb{X}) = (\mathbb{R}^d,\rho_{|\cdot|})$, we consider the two subcases: 
$(\mathscr{X},\rho_{\mathscr{X}}) = (\mathbb{D}^d \times \mathbb{R}^{d\times p}\times \mathbb{R}^p,\rho_{\mathbb{D}^d \times \mathbb{R}^{d\times p}\times \mathbb{R}^p})$
and 
$(\mathscr{X},\rho_{\mathscr{X}}) = (\mathbb{D}^d,\rho_{J_1})$.
In the former case, the McKean--Vlasov BSDE \eqref{MVBSDE_general} takes the form 
\begin{align}
\begin{multlined}[0.84\textwidth]
\label{MVBSDE_instantaneous_b_triplet}
Y_t     = \xi 
        + \int^{T}_{t}f \big(s,\widetilde{\mathbf{S}}^{\textrm{mod}},\mathcal{L}(\widetilde{\mathbf{S}}^{\textrm{mod}}_s) \big)  \Big) \,\ud C^{(\mathbb{G},\mathbf{X})}_s\\ 
        - \int^{T}_{t}Z_s \,  \ud X^{\circ}_s - \int^{T}_{t}\int_{\mathbb{R}^n}U_s(x) \, \widetilde{\mu}^{(\mathbb{G},X^{\natural})}(\ud s,\ud x) - \int^{T}_{t} \,\ud M_s
\end{multlined}
\tag{\ref{MVBSDE_instantaneous_triplet}}
\end{align}
with
\begin{gather*}
    \widetilde{\mathbf{S}}^{\textrm{mod}}_{\cdot}:=\big(Y_\cdot,Z_\cdot  c^{(\mathbb{G},\mathbf{X})}_\cdot,\Gamma^{(\mathbb{G},\mathbf{X},\Theta)}(U)_\cdot\big).
\end{gather*}
In the latter case, the McKean--Vlasov BSDE \eqref{MVBSDE_general} takes the form, for $i\in\mathscr{N}$,
\begin{align}
\begin{multlined}[0.86\textwidth]
\label{MVBSDE_instantaneous_b_onlyY}
Y_t     = \xi 
        + \int^{T}_{t}f \big(s,\widetilde{\mathbf{S}}^{\textrm{mod}},\mathcal{L}(Y_s) \big)  \,\ud C^{(\mathbb{G},\mathbf{X})}_s\\ 
        - \int^{T}_{t}Z_s \,  \ud X^{\circ}_s - \int^{T}_{t}\int_{\mathbb{R}^n}U_s(x) \, \widetilde{\mu}^{(\mathbb{G},X^{\natural})}(\ud s,\ud x) - \int^{T}_{t} \,\ud M_s
\end{multlined}
\tag{$\textrm{MV}_{\textrm{Inst}}^{Y}$}
\end{align}
with
\begin{gather*}
    \widetilde{\mathbf{S}}^{\textrm{mod}}_{\cdot}:=\big(Y_\cdot,Z_\cdot  c^{(\mathbb{G},\mathbf{X})}_\cdot,\Gamma^{(\mathbb{G},\mathbf{X},\Theta)}(U)_\cdot\big).
\end{gather*}
Now, in the first case, the BSDE depends on the law of all the components of the solution process, while in the second one only on the law of $Y$; in both cases, it also depends on the instantaneous value of the processes.


\subsubsection{Statements of the McKean--Vlasov BSDE theorems}
\label{subsubsec:MV_theorems_statements}

We are now ready to state the main theorems regarding the existence and uniqueness of the solution of the aforementioned McKean--Vlasov BSDEs.

\begin{theorem}\label{thm:MVBSDE_initial_path_triplet}
Under \ref{MV_martingale_integrators}, \ref{MV_generator} for $(\mathbb{X},\rho_{\mathbb{X}}) = (\mathbb{D}^d, \rho_{J_1})$ and $(\mathscr{X},\rho_{\mathscr{X}}) = (\mathbb{D}^d \times \mathbb{R}^{d\times p}\times \mathbb{R}^p,\rho_{\mathbb{D}^d \times \mathbb{R}^{d\times p}\times \mathbb{R}^p})$, \ref{MV_alpha_A} for $[0,\infty)\ni x\overset{h}{\mapsto} x $ and $\alpha^2_\cdot\geq 1$, \ref{MV_A_cases}.\ref{MV_StochExp_Bound}, \ref{MV_Theta_Gamma}, \ref{MV_terminal_rv} and  \ref{MV_Integrability_Generator}, if
\begin{align*}
    \max\left\{4,\frac{2\Lambda_{\hat{\beta}}}{\hat{\beta}}\right\} 
    \widehat{M}^{\Phi}_{\star}(\hat{\beta}) < 1,
\end{align*} 
then the McKean--Vlasov BSDE \eqref{MVBSDE_initial_path_triplet} admits a unique solution $\widetilde{\mathbf{S}}$ such that
$\|\widetilde{\mathbf{S}}\|_{\star,\hat{\beta},\mathbb{G}, A^{(\mathbb{G},\mathbf{X},f)},\mathbf{X}}<\infty$\footnote{The norm was defined in \eqref{def:star_norm}}.
\end{theorem}

\begin{theorem}\label{thm:MVBSDE_initial_path_onlyY}
Under \ref{MV_martingale_integrators}, \ref{MV_generator} for $(\mathbb{X},\rho_{\mathbb{X}}) = (\mathscr{X},\rho_{\mathscr{X}}) = (\mathbb{D}^d, \rho_{J_1})$, \ref{MV_alpha_A} for $[0,\infty)\ni x\overset{h}{\mapsto} \sqrt{x}$ , \ref{MV_A_cases}.\ref{MV_StochExp_Bound}, \ref{MV_Theta_Gamma}, \ref{MV_terminal_rv} and  \ref{MV_Integrability_Generator}, if
\begin{align*}
    \max\left\{2,\frac{2\Lambda_{\hat{\beta}}}{\hat{\beta}}\right\} 
    \widehat{M}^{\Phi}_{\star}(\hat{\beta}) < 1,
\end{align*} 
then the McKean--Vlasov BSDE \eqref{MVBSDE_initial_path_onlyY} admits a unique solution $\widetilde{\mathbf{S}}$ such that 
$\|\widetilde{\mathbf{S}}\|_{\star,\hat{\beta},\mathbb{G}, A^{(\mathbb{G},\mathbf{X},f)},\mathbf{X}}<\infty$.
\end{theorem}

\begin{theorem}\label{thm:MVBSDE_instantaneous_triplet}
Under \ref{MV_martingale_integrators}, \ref{MV_generator} for $(\mathbb{X},\rho_{\mathbb{X}}) = (\mathbb{R}^d,\rho_{|\cdot|})$ and $(\mathscr{X},\rho_{\mathscr{X}}) = (\mathbb{R}^d \times \mathbb{R}^{d\times p}\times \mathbb{R}^p,\rho_{\mathbb{R}^d \times \mathbb{R}^{d\times p}\times \mathbb{R}^p})$, \ref{MV_alpha_A} for $[0,\infty)\ni x\overset{h}{\mapsto} x $ and $\alpha^2_\cdot\geq 1$, \ref{MV_A_cases}.\ref{MV_StochExp_Integrable}, \ref{MV_Theta_Gamma}, \ref{MV_terminal_rv} and  \ref{MV_Integrability_Generator}, if
\begin{align*}
    \max\Big\{4,\frac{\mathbb{E}\big[\mathcal{E}\big(\hat{\beta}A^{(\mathbb{G},\mathbf{X},f)}\big)_T \big]}{\hat{\beta}}\Big\} \widehat{M}^{\Phi}_{\scalebox{0.4}{$\blacksquare$}}(\hat{\beta})<1,
\end{align*} 
then the McKean--Vlasov BSDE \eqref{MVBSDE_instantaneous_b_triplet} admits a unique solution $\widetilde{\mathbf{S}}$ such that 
$\|\widetilde{\mathbf{S}}\|_{\text{\scalebox{0.5}{$\blacksquare$}},\hat{\beta},\mathbb{G}, A^{(\mathbb{G},\mathbf{X},f)},\mathbf{X}}<\infty$\footnote{The norm was defined in \eqref{def:blacksquare_norm}}.
\end{theorem}

\begin{theorem}\label{thm:MVBSDE_instantaneous_triplet_b}
Under \ref{MV_martingale_integrators}, \ref{MV_generator} for $(\mathbb{X},\rho_{\mathbb{X}}) = (\mathbb{R}^d,\rho_{|\cdot|})$ and $(\mathscr{X},\rho_{\mathscr{X}}) = (\mathbb{R}^d \times \mathbb{R}^{d\times p}\times \mathbb{R}^p,\rho_{\mathbb{R}^d \times \mathbb{R}^{d\times p}\times \mathbb{R}^p})$, \ref{MV_alpha_A} for $[0,\infty)\ni x\overset{h}{\mapsto} x $ and $\alpha^2_\cdot\geq 1$, \ref{MV_A_cases}.\ref{MV_A_deterministic}, \ref{MV_Theta_Gamma}, \ref{MV_terminal_rv} and  \ref{MV_Integrability_Generator}, if
\begin{align*}
    4\widehat{M}^{\Phi}_{\scalebox{0.4}{$\blacksquare$}}(\hat{\beta})<1,
\end{align*} 
then the McKean--Vlasov BSDE \eqref{MVBSDE_instantaneous_b_triplet} admits a unique solution $\widetilde{\mathbf{S}}$ such that $\|\widetilde{\mathbf{S}}\|_{\text{\scalebox{0.5}{$\blacksquare$}},\hat{\beta},\mathbb{G}, A^{(\mathbb{G},\mathbf{X},f)},\mathbf{X}}<\infty$.
\end{theorem}

\begin{theorem}\label{thm:MVBSDE_instantaneous_onlyY_blacksquare}
Under \ref{MV_martingale_integrators}, \ref{MV_generator} for $(\mathbb{X},\rho_{\mathbb{X}}) = (\mathscr{X},\rho_{\mathscr{X}}) = (\mathbb{R}^d,\rho_{|\cdot|})$, \ref{MV_alpha_A} for $[0,\infty) \ni x\overset{h}{\mapsto} \sqrt{x}$,  \ref{MV_A_cases}.\ref{MV_StochExp_Integrable}, \ref{MV_Theta_Gamma}, \ref{MV_terminal_rv} and  \ref{MV_Integrability_Generator}, if
\begin{align*}
    \max\Big\{2,\frac{\mathbb{E}\big[\mathcal{E}(\hat{\beta}A^{(\mathbb{G},\mathbf{X},f)})_T \big]}{\hat{\beta}}\Big\} \widehat{M}^{\Phi}_{\scalebox{0.4}{$\blacksquare$}}(\hat{\beta})< 1,
\end{align*} 
then the McKean--Vlasov BSDE \eqref{MVBSDE_instantaneous_b_onlyY} admits a unique solution $\widetilde{\mathbf{S}}$ such that $\|\widetilde{\mathbf{S}}\|_{\text{\scalebox{0.5}{$\blacksquare$}},\hat{\beta},\mathbb{G}, A^{(\mathbb{G},\mathbf{X},f)},\mathbf{X}}<\infty$.
\end{theorem}

\begin{theorem}\label{thm:MVBSDE_instantaneous_onlyY_star}
Under \ref{MV_martingale_integrators}, \ref{MV_generator} for $(\mathbb{X},\rho_{\mathbb{X}}) = (\mathscr{X},\rho_{\mathscr{X}}) = (\mathbb{R}^d,\rho_{|\cdot|})$, \ref{MV_alpha_A} for $[0,\infty) \ni x\overset{h}{\mapsto} \sqrt{x}$, \ref{MV_A_cases}.\ref{MV_StochExp_Bound}, \ref{MV_Theta_Gamma}, \ref{MV_terminal_rv} and  \ref{MV_Integrability_Generator}, if
\begin{align*}
    \Big\{ 2,\frac{2\Lambda_{\hat{\beta}}}{\hat{\beta}} \Big\} \widehat{M}_{\star}^{\Phi}(\hat{\beta}) < 1,
\end{align*} 
then the McKean--Vlasov BSDE \eqref{MVBSDE_instantaneous_b_onlyY} admits a unique solution $\widetilde{\mathbf{S}}$ such that $\|\widetilde{\mathbf{S}}\|_{\star,\hat{\beta},\mathbb{G}, A^{(\mathbb{G},\mathbf{X},f)},\mathbf{X}}<\infty$.
\end{theorem}

\begin{theorem}\label{thm:MVBSDE_instantaneous_onlyY_star_deterministic}
Under \ref{MV_martingale_integrators}, \ref{MV_generator} for $(\mathbb{X},\rho_{\mathbb{X}}) = (\mathscr{X},\rho_{\mathscr{X}}) = (\mathbb{R}^d,\rho_{|\cdot|})$, \ref{MV_alpha_A} for $[0,\infty) \ni x\overset{h}{\mapsto} \sqrt{x}$, \ref{MV_A_cases}.\ref{MV_A_deterministic}, \ref{MV_Theta_Gamma}, \ref{MV_terminal_rv} and  \ref{MV_Integrability_Generator}, if
\begin{align*}
     2 \widehat{M}^{\Phi}_{\scalebox{0.4}{$\blacksquare$}}(\hat{\beta}) < 1,
\end{align*} 
then the McKean--Vlasov BSDE \eqref{MVBSDE_instantaneous_b_onlyY} admits a unique solution $\widetilde{\mathbf{S}}$ such that $\|\widetilde{\mathbf{S}}\|_{\scalebox{0.4}{$\blacksquare$},\hat{\beta},\mathbb{G}, A^{(\mathbb{G},\mathbf{X},f)},\mathbf{X}}<\infty$.
\end{theorem}

\subsubsection{Proofs of the McKean--Vlasov BSDE theorems}
In order to proceed with the proofs of the stated theorems, we will initially use the proof of \cref{thm:mfBSDE_initial_path_triplet} as a reference in order to minimize replication of the required arguments.
To this end, the reader should consider the following ``simplifications'' as a starting point: $i=N=1$, and it is the law of the respective part of the solution in place of the empirical measure.
We underline that even in this simple case where $i=N=1$, the empirical measure is a random measure, while the law that appears in the McKean--Vlasov BSDEs generator is deterministic. 
Thus, the following change applies to all of the following cases: in the derivation of \eqref{ineq:mf_condition_apriori_verification}, and consequently of \eqref{ineq:mf_sum_norm_delta_f_initial_path} as well, the estimation \eqref{ineq:empirical_measures_bound} should be replaced by the following, 
\begin{align}
    \label{ineq:MV_Wasserstein_bound}
    \mathcal{W}_{2,\rho_{\mathscr{X}}}^2 \big(\mathcal{L}(\widetilde{\mathbf{T}}^1),\mathcal{L}(\widetilde{\mathbf{T}}^2)\big)
    \leq \mathbb{E} \big[ \rho_{\mathscr{X}}^2(\widetilde{\mathbf{T}}^1,\widetilde{\mathbf{T}}^2)\big],
\end{align}
for $\mathscr{X}-$valued random variables $\widetilde{\mathbf{T}}^1, \widetilde{\mathbf{T}}^2$.
Notationally, we drop all indices $i$ and $N$ and replace $\overline{\mathbf{s}}$, $\overline{\mathbf{S}}$, \emph{etc.}, by $\tilde{\mathbf{s}}$, $\widetilde{\mathbf{S}}$, and so forth.
Additionally, for notational simplicity, we also write $A$ instead of $A^{(\mathbb{G},\mathbf{X},f)}$, \emph{etc}.

\begin{proof}[Proof of Theorem \ref{thm:MVBSDE_initial_path_triplet}]
In view of the previous remarks, on the one hand, \eqref{ineq:mf_condition_apriori_verification} now reads 
\begin{align*}  
&\Big\|\frac{1}{\alpha}f\left(\cdot,\widetilde{\mathbf{s}}^{\text{mod}}_\cdot,\mathcal{L}(\widetilde{\mathbf{s}}^{\text{mod}}_\cdot) \right)\Big\|^{2}_{\mathbb{H}^{2}_{\hat{\beta}}}\\
&\begin{multlined}[0.85\textwidth]
\underset{\phantom{\ref{MV_A_cases}.\ref{MV_StochExp_Bound}}}{\overset{\eqref{properties_stoch_bounds}}{\underset{\eqref{properties_stoch_bounds_onlyWasserstein_identity}}{\leq}}}
2\Big\|\frac{1}{\alpha}f\big(\cdot,0,0,0,\delta_0 \big)\Big\|^{2}_{\mathbb{H}^{2}_{\hat{\beta}}}
+2\mathbb{E} \Big[ \int_0^T\mathcal{E}(\hat{\beta} A^{(\mathbb{G},\mathbf{X},f)}_{s-}) \sup_{u\in[0,s]}|y_u|^2 \textup{d} A^{(\mathbb{G},\mathbf{X},f)}_s\Big]\\  
+2\mathbb{E} \Big[ \int_0^T\mathcal{E}(\hat{\beta} A^{(\mathbb{G},\mathbf{X},f)}_{s-}) \mathbb{E}[\sup_{u\in[0,s]}|y_u|^2] \textup{d} C^{(\mathbb{G},\mathbf{X})}_s\Big]
    +4\|z\|^{2}_{ \mathbb{H}^{2,\circ}_{\hat{\beta}}} 
    +8\|u\|^{2}_{\mathbb{H}^{2,\natural}_{\hat{\beta}}}
\end{multlined}\\
&\begin{multlined}[0.85\textwidth]
\overset{\ref{MV_A_cases}.\ref{MV_StochExp_Bound}}{\leq}
2\Big\|\frac{1}{\alpha}f\big(\cdot,0,0,0,\delta_0 \big)\Big\|^{2}_{\mathbb{H}^{2}_{\hat{\beta}}}
    + \frac{4\Lambda_{\hat{\beta}}}{\hat{\beta}} \| y\|^{2}_{ \mathcal{S}^{2}_{\hat{\beta}}} 
    +4\|z\|^{2}_{ \mathbb{H}^{2,\circ}_{\hat{\beta}}} 
    +8\|u\|^{2}_{\mathbb{H}^{2,\natural}_{\hat{\beta}}}
\end{multlined}
\numberthis\label{ineq:MV_condition_apriori_verification}\\
&\overset{\ref{MV_Integrability_Generator}}{\underset{\phantom{\ref{MV_A_cases}.\ref{MV_StochExp_Bound}}}{<}}
\infty,
\end{align*}
where in the second inequality we used the fact that $\alpha^2\geq 1$.
On the other hand, \eqref{ineq:mf_sum_norm_delta_f_initial_path} now reads
\begin{align*}  
&\hspace{-1em}\Big\|\frac{1}{\alpha}\delta f^{(*)}\Big\|^{2}_{\mathbb{H}^{2}_{\hat{\beta}}}\\
&\begin{multlined}[0.85\textwidth]
\leq 
\mathbb{E} \Big[ \int_0^T\mathcal{E}(\hat{\beta} A^{(\mathbb{G},\mathbf{X},f)}_{s-})  \sup_{u\in[0,s]}|\delta y^{(*)}_u\big|^2 \textup{d} A^{(\mathbb{G},\mathbf{X},f)}_s\Big] 
    + 2\|\delta z^{(*)}\|^{2}_{ \mathbb{H}^{2,\circ}_{\hat{\beta}}} 
    + 4\|\delta u^{(*)}\|^{2}_{\mathbb{H}^{2,\natural}_{\hat{\beta}}}
    \\
    +\mathbb{E} \Big[ \int_0^T\mathcal{E}(\hat{\beta} A^{(\mathbb{G},\mathbf{X},f)}_{s-}) \mathbb{E}[\sup_{u\in[0,s]}|\delta y^{(*)}_u|^2] \textup{d} C^{(\mathbb{G},\mathbf{X})}_s\Big]
\end{multlined}\\
&\begin{multlined}[0.85\textwidth]
\leq  \frac{2\Lambda_{\hat{\beta}}}{\hat{\beta}} \| \delta y^{(*)}\|^{2}_{ \mathcal{S}^{2}_{\hat{\beta}}} 
    + 2\|\delta z^{(*)}\|^{2}_{ \mathbb{H}^{2,\circ}_{\hat{\beta}}} 
    + 4\|\delta u^{(*)}\|^{2}_{\mathbb{H}^{2,\natural}_{\hat{\beta}}}
\end{multlined}\\
&\begin{multlined}[0.9\textwidth]
\leq 
\max\Big\{4,\frac{2}{\hat{\beta}}\Lambda_{\hat{\beta}}\Big\}
\big\| \big(\delta \mathbf{y}^{(*)}, \delta \mathbf{z}^{(*)}, \delta \mathbf{u}^{(*)}, \delta \mathbf{m}^{(*)}\big)\big\|^{2}_{\mathcal{S}^{2}_{\hat{\beta}}\times \mathscr{H}^{2}_{\hat{\beta}}},
\numberthis\label{ineq:MV_norm_delta_f_initial_path}
\end{multlined}
\end{align*}
where in the second inequality we used the fact that $\alpha^2\geq 1$.
Thus, the contraction constant derived from above is 
$\max\big\{4,\frac{2}{\hat{\beta}}\Lambda_{\hat{\beta}}\big\} \widehat{M}^{\Phi}_{\star}$ which is assumed to be smaller than 1.
\end{proof}

\begin{proof}[Proof of Theorem \ref{thm:MVBSDE_initial_path_onlyY}]
The use of the function $x\overset{h}{\mapsto} \sqrt{x}$ allows the use of \eqref{properties_stoch_bounds_onlyWasserstein_squareroot} instead of \eqref{properties_stoch_bounds_onlyWasserstein_identity}, which does not require the assumption $\alpha^2\geq 1$ used in the previous theorem.
Then, \eqref{ineq:MV_condition_apriori_verification} as well as \eqref{ineq:MV_norm_delta_f_initial_path} improve their coefficients.
As a result, the contraction constant is $\max\big\{2,\frac{2}{\hat{\beta}}\Lambda_{\hat{\beta}}\big\} \widehat{M}^{\Phi}_{\star}$ which is assumed to be smaller than 1.
\end{proof}
\begin{proof}[Proof of Theorem \ref{thm:MVBSDE_instantaneous_triplet}]
We will adapt the computations of \cref{thm:MVBSDE_initial_path_triplet} under the assumption that $y,y^{(1)},y^{(2)}$ $\in \mathscr{S}^2_{\hat\beta}$.
Inequality \eqref{ineq:MV_condition_apriori_verification} now reads 
\begin{align*}  
&\Big\|\frac{1}{\alpha}f\left(\cdot,\widetilde{\mathbf{s}}^{\text{mod}}_\cdot,\mathcal{L}(\widetilde{\mathbf{s}}^{\text{mod}}_\cdot) \right)\Big\|^{2}_{\mathbb{H}^{2}_{\hat{\beta}}}\\
&\begin{multlined}[0.15\textwidth]
\underset{\phantom{\ref{MV_A_cases}.\ref{MV_StochExp_Integrable}}}{\overset{\eqref{properties_stoch_bounds}}{\underset{\eqref{properties_stoch_bounds_onlyWasserstein_identity}}{\leq}}}
2\Big\|\frac{1}{\alpha}f\big(\cdot,0,0,0,\delta_0 \big)\Big\|^{2}_{\mathbb{H}^{2}_{\hat{\beta}}}
+2\mathbb{E} \Big[ \int_0^T\mathcal{E}(\hat{\beta} A^{(\mathbb{G},\mathbf{X},f)}_{s-}) \alpha_s^2 |y_s|^2 \textup{d} C^{(\mathbb{G},\mathbf{X})}_s\Big]\\  
+2\mathbb{E} \Big[ \int_0^T\mathcal{E}(\hat{\beta} A^{(\mathbb{G},\mathbf{X},f)}_{s-}) \mathbb{E}[|y_s|^2] \textup{d} C^{(\mathbb{G},\mathbf{X})}_s\Big]
    +4\|z\|^{2}_{ \mathbb{H}^{2,\circ}_{\hat{\beta}}} 
    +8\|u\|^{2}_{\mathbb{H}^{2,\natural}_{\hat{\beta}}}
\end{multlined}\\
&\begin{multlined}[0.15\textwidth]
\overset{\phantom{\ref{MV_A_cases}.\ref{MV_StochExp_Integrable}}}{\leq}
2\Big\|\frac{1}{\alpha}f\big(\cdot,0,0,0,\delta_0 \big)\Big\|^{2}_{\mathbb{H}^{2}_{\hat{\beta}}}
    + \max\Big\{2,\frac{2}{\hat{\beta}}\mathbb{E}\big[\mathcal{E}(\hat{\beta}A^{(\mathbb{G},\mathbf{X},f)})_T \big]\Big\} \| y\|^{2}_{ \mathscr{S}^{2}_{\hat{\beta}}} 
    +4\|z\|^{2}_{ \mathbb{H}^{2,\circ}_{\hat{\beta}}} 
\end{multlined}
\numberthis\label{ineq:MV_condition_apriori_verification_instantaneous}\\
&\hspace{9cm}+8\|u\|^{2}_{\mathbb{H}^{2,\natural}_{\hat{\beta}}} \overset{\ref{MV_Integrability_Generator}}{\underset{{\ref{MV_A_cases}.\ref{MV_StochExp_Integrable}}}{<}}
\infty,
\end{align*}
where in the second inequality we used the fact that $\alpha^2\geq 1$.
Moreover, \eqref{ineq:mf_sum_norm_delta_f_initial_path} now reads
\begin{align*}  
&\hspace{-1em}\Big\|\frac{1}{\alpha}\delta f^{(*)}\Big\|^{2}_{\mathbb{H}^{2}_{\hat{\beta}}}\\
&\begin{multlined}[0.85\textwidth]
\leq 
\mathbb{E} \Big[ \int_0^T\mathcal{E}(\hat{\beta} A^{(\mathbb{G},\mathbf{X},f)}_{s-}) \alpha^2_s |\delta y^{(*)}_s|^2 \textup{d} C^{(\mathbb{G},\mathbf{X})}_s\Big] 
    + 2\|\delta z^{(*)}\|^{2}_{ \mathbb{H}^{2,\circ}_{\hat{\beta}}} 
    + 4\|\delta u^{(*)}\|^{2}_{\mathbb{H}^{2,\natural}_{\hat{\beta}}}
    \\
    +\mathbb{E} \Big[ \int_0^T\mathcal{E}(\hat{\beta} A^{(\mathbb{G},\mathbf{X},f)}_{s-}) \mathbb{E}[|\delta y^{(*)}_s|^2] \textup{d} C^{(\mathbb{G},\mathbf{X})}_s\Big]
\end{multlined}\\
&\begin{multlined}[0.85\textwidth]
\leq  \max\Big\{1, \frac{1}{\hat{\beta}}\mathbb{E}\big[\mathcal{E}(\hat{\beta}A^{(\mathbb{G},\mathbf{X},f)})_T \big] \Big\} \| \delta y^{(*)}\|^{2}_{ \mathscr{S}^{2}_{\hat{\beta}}} 
    + 2\|\delta z^{(*)}\|^{2}_{ \mathbb{H}^{2,\circ}_{\hat{\beta}}} 
    + 4\|\delta u^{(*)}\|^{2}_{\mathbb{H}^{2,\natural}_{\hat{\beta}}}
\end{multlined}\\
&\begin{multlined}[0.9\textwidth]
\leq 
\max\Big\{4,\frac{1}{\hat{\beta}}\mathbb{E}\big[\mathcal{E}(\hat{\beta}A^{(\mathbb{G},\mathbf{X},f)})_T \big]\Big\}
\big\| \big(\delta \mathbf{y}^{(*)}, \delta \mathbf{z}^{(*)}, \delta \mathbf{u}^{(*)}, \delta \mathbf{m}^{(*)}\big)\big\|^{2}_{\mathscr{S}^{2}_{\hat{\beta}}\times \mathscr{H}^{2}_{\hat{\beta}}},
\numberthis\label{ineq:MV_norm_delta_f_instantaneous}
\end{multlined}
\end{align*}
where in the second inequality we used the fact that $\alpha^2\geq 1$.
Thus, the contraction constant derived from above is 
$\max\Big\{4,\frac{1}{\hat{\beta}}\mathbb{E}\big[\mathcal{E}(\hat{\beta}A^{(\mathbb{G},\mathbf{X},f)})_T \big]\Big\} \widehat{M}^{\Phi}_{\scalebox{0.4}{$\blacksquare$}}$ which is assumed to be smaller than 1.
\end{proof}
\begin{proof}[Proof of Theorem \ref{thm:MVBSDE_instantaneous_triplet_b}]
The proof of \cref{thm:MVBSDE_instantaneous_triplet} can be followed verbatim with the change that now we have the equality
$$
 \mathbb{E} \Big[ \int_0^T\mathcal{E}(\hat{\beta} A^{(\mathbb{G},\mathbf{X},f)}_{s-}) \mathbb{E}[|\delta y^{(*)}_s|^2] \textup{d} C^{(\mathbb{G},\mathbf{X})}_s\Big] = \mathbb{E} \Big[ \int_0^T\mathcal{E}(\hat{\beta} A^{(\mathbb{G},\mathbf{X},f)}_{s-}) |\delta y^{(*)}_s|^2 \textup{d} C^{(\mathbb{G},\mathbf{X})}_s\Big].
$$
Hence, the desired versions of  \eqref{ineq:MV_condition_apriori_verification_instantaneous} and \eqref{ineq:MV_norm_delta_f_instantaneous} follow.
\end{proof}
\begin{proof}[Proof of Theorem \ref{thm:MVBSDE_instantaneous_onlyY_blacksquare}]
The proof of \cref{thm:MVBSDE_instantaneous_triplet} can be followed verbatim with the following obvious change:
the choice $x\overset{h}{\mapsto}\sqrt{x}$ does not require the assumption $\alpha^2\geq 1$; recall \cref{rem:About_MFBSDE_framework}.\ref{rem:about_a2_h_function}.
Then \eqref{ineq:MV_condition_apriori_verification_instantaneous} and \eqref{ineq:MV_norm_delta_f_instantaneous} improve their coefficients, which lead to the contraction constant 
$\max\Big\{2,\frac{1}{\hat{\beta}}\mathbb{E}\big[\mathcal{E}(\hat{\beta}A^{(\mathbb{G},\mathbf{X},f)})_T \big]\Big\} \widehat{M}^{\Phi}_{\scalebox{0.4}{$\blacksquare$}}$.
\end{proof}
\begin{proof}[Proof of Theorem \ref{thm:MVBSDE_instantaneous_onlyY_star}]
Under the assumption that $y,y^{(1)},y^{(2)} \in \mathcal{S}^2_{\hat\beta}$ and the substitution of \ref{MV_A_cases}.\ref{MV_StochExp_Integrable} by \ref{MV_A_cases}.\ref{MV_StochExp_Bound}, the adaptation of the computations in the proof of \cref{thm:MVBSDE_instantaneous_onlyY_blacksquare} are immediate and lead to the contraction constant 
$\max\Big\{2,\frac{1}{\hat{\beta}}\Lambda_{\hat{\beta}}\Big\} \widehat{M}^{\Phi}_{\star}$.
\end{proof}

\begin{proof}[Proof of Theorem \ref{thm:MVBSDE_instantaneous_onlyY_star_deterministic}]
The proof of \cref{thm:MVBSDE_instantaneous_triplet_b} can be followed verbatim with the following obvious change:
the choice $x\overset{h}{\mapsto}\sqrt{x}$ does not require the assumption $\alpha^2\geq 1$; recall \cref{rem:About_MFBSDE_framework}.\ref{rem:about_a2_h_function}.
Then \eqref{ineq:MV_condition_apriori_verification_instantaneous} and \eqref{ineq:MV_norm_delta_f_instantaneous} improve their coefficients, which lead to the contraction constant 
$2\widehat{M}^{\Phi}_{\scalebox{0.4}{$\blacksquare$}}$.
\end{proof}
\begin{remark}\label{rem:about_proof_thm_MVBSDE_initial_path}
Let us now provide some remarks related to specific points of the previous proofs:
\begin{enumerate}
    \item If $T$ 
is deterministic and finite, then Conditions \ref{MV_A_cases}.\ref{MV_StochExp_Bound}, \ref{MV_StochExp_Integrable} are automatically satisfied. However, when $T = \infty$ these does not necessarily hold in general. 
    
    \item\label{remark 4.5(i)}
    Let $\widetilde{\mathbf{s}}^2 \in \mathcal{S}^{2}_{\hat{\beta}}(\mathbb{G},A^{(\mathbb{G},\mathbf{X}, f)};\mathbb{R}^d)\times \mathscr{H}_{\hat{\beta}}^2$, then from \eqref{2.15} we have that $\mathcal{L}({\mathbf{s}}|_{[0,\cdot]}^{\text{mod}})$ is a (deterministic) measurable process, indeed this holds because ($\mathbb{X}:= \mathbb{D}^d\times \mathbb{R}^{d \times p}\times \mathbb{R}^d$) for every continuous and bounded function $h:\mathbb{X} \rightarrow \mathbb{R}$, the deterministic process $\int_{\mathbb{X}}h(x)\mathcal{L}({\mathbf{s}}|_{[0,t]}^{\text{mod}})(\ud x) = \int_{\Omega}h({\mathbf{s}}|_{[0,t]}^{\text{mod}})\mathbb{P}(\ud\omega)$ is measurable by Fubini's theorem. 
    \item 
    The imposition of the bound $\Lambda_{\hat{\beta}}$ seems inevitable if we want to consider path-dependent BSDEs, \textit{i.e.} those whose generator $f$ depends on the initial segment of the paths of the solution $Y$. Indeed, in the proofs of \cref{thm:MVBSDE_initial_path_triplet,thm:MVBSDE_initial_path_onlyY} we need to multiply the stochastic exponential with the (square of the) running maximum of $y^2-y^1$.
    One could possibl, be able to proceed without an assumption on boundedness of the stochastic exponential, if it were possible to extract \emph{a priori} estimates for the running maximum analogous to \cref{lem:a_priori_estimates}.
    In this case, one expects in \cref{lem:a_priori_estimates} an integrability condition of the form
    $\mathbb{E}[\mathcal{E}(\hat{\beta}A^{(\mathbb{G},\mathbf{X}, f)})_{T-} \sup_{s\in [ 0,T]}|y_s|^2]<\infty$ to appear, which is clearly stronger than the $\mathcal{S}^2-$norm we are using. 
    However, in this work, we have not pursued such \emph{a priori} estimates.

     \item\label{rem:about_proof_thm_MVBSDE_initial_path_3}
    For the BSDEs \eqref{MVBSDE_instantaneous_b_triplet} and \eqref{MVBSDE_instantaneous_b_onlyY} whose generator $f$ depends at time $s$ on the instantaneous value of the solution $Y$, \emph{e.g.}, on $Y_s$ or $Y_{s-}$, and not on the initial - up to time $s$ - segment of its paths, \emph{i.e.}, on $Y|_{[0,s]}$ or $Y|_{[0,s-]}$, we could proceed under the assumption that $\mathcal{E}(\hat{\beta}A^{(\mathbb{G},\mathbf{X},f)})_T$ is integrable, for the $\hat{\beta}$ determining standard data.
    However, in this case we needed to seek a solution related to the $\|\cdot\|_{\scalebox{0.4}{$\blacksquare$}}$-norm instead of the $\|\cdot\|_{\star}$-norm
    such that 
    \begin{align*}
    Y \in \mathcal{S}^2_{\hat{\beta}}(\mathbb{G},A^{(\mathbb{G},\mathbf{X},f)};\mathbb{R}^d)
    \quad \text{ and } \quad 
    \alpha Y\in \mathbb{H}^2_{\hat{\beta}}(\mathbb{G},A^{(\mathbb{G},\mathbf{X},f)},C^{(\mathbb{G},\mathbf{X})};\mathbb{R}^d),
    \end{align*}
    instead of simply $Y \in \mathcal{S}^2_{\hat{\beta}}(\mathbb{G},A^{(\mathbb{G},\mathbf{X},f)};\mathbb{R}^d)$.
\end{enumerate}
\end{remark}

\section{Backward propagation of chaos in a general setting: the stochastic Lipschitz case}
\label{sec:propagation}

In this section, we will present general results on the propagation of chaos for particles that satisfy a backward SDE, as announced in the introduction of this work.
The technique for our proof is novel and differs from the classical results in the literature, in that we make repeated use of the \textit{a priori} estimates and the Strong Law of Large Numbers, instead of using a delicate change of measure and an application of Girsanov's theorem or some form of Gronwall's inequality or It\=o's lemma, as in \citet{lauriere2022backward} and the literature on the propagation of chaos for forward SDEs.
This method allows us to consider a more general setting, and we work with asymmetric mean-field systems and general square integrable drivers with jumps.

Despite the fact that we have already provided existence and uniqueness results for the respective mean-field BSDE systems and McKean--Vlasov BSDEs, there are still quite a few preparatory and auxiliary results that will be required before we present the proof of the respective propagation of chaos statements. 
These auxiliary results are presented here, while their proofs are deferred to \cref{subsec_app:tehnical_lemmata,subsec:Conservation_of_solutions}. 

Let us recall that stochastic processes are typically associated with a reference filtration. 
More specifically, `compensator-type' processes are defined with respect to a predictable $\sigma-$algebra, which depends on the reference filtration. 
In the propagation of chaos statements, on the one hand we want the solutions of the McKean--Vlasov BSDEs to be independent; hence, we solve each one in its own filtration. 
On the other hand, when we solve the mean-field systems of BSDEs we work in filtrations larger than those of the corresponding McKean--Vlasov BSDEs. 
Therefore, the question naturally arises whether the solutions of the McKean--Vlasov BSDEs with respect to the smaller filtrations remain solutions when we work in the larger filtrations. 
The appendices  \ref{subsec_app:tehnical_lemmata} and \ref{subsec:Conservation_of_solutions} provide the tools necessary to prove such `conservation of solutions' type results under the condition of immersion of filtrations.

Moreover, as expected, the framework we are going to use for the propagation of chaos will be based on the common ground of the frameworks we used in the previous sections, suitably enriched and reinforced wherever required. 


\subsection{Setting} 
Let $T\in(0,\infty]$ be a deterministic time, $\hat{\beta}>0$ and $(\Omega,\mathcal{G},\mathbb{P})$ be a complete probability space which supports the following:
\begin{enumerate} [label=\textup{\textbf{(PC\arabic*)}}]
    \item\label{PC_martingale_integrators} A sequence of independent and identically distributed processes
    $(\mathbf{X}^i)_{i \in \mathbb{N}}$ with independent increments, where for every $i\in \mathbb{N}$  
    \begin{align*}
        \mathbf{X}^i=(X^{i,\circ},X^{i,\natural})\in \mathcal{H}^2(\mathbb{F}^i;\mathbb{R}^p) \times \mathcal{H}^{2,d}(\mathbb{F}^i;\mathbb{R}^n)
    \end{align*} such that
    $M_{\mu^{X^{i,\natural}}}[\Delta X^{i,\circ}|\widetilde{\mathcal{P}}^{\mathbb{F}^i}] = 0$ and $\mathbb{F}^i:=(\mathcal{F}^i_t)_{t\geq 0}$ is the usual augmentation of the natural filtration of $\mathbf{X}^i$.\footnote{Since for every $i\in\mathbb{N}$ the filtration $\mathbb{F}^i$ is associated to $\mathbf{X}^i$,  we will make use of $C^{(\mathbb{F}^i,\mathbf{X}^i)}$, resp. 
    $c^{(\mathbb{F}^i,\mathbf{X}^i)}$, as defined in \eqref{def_C}, resp. \eqref{def_c}. 
    Moreover, we will use the kernels $K^{(\mathbb{F}^i,\mathbf{X}^i)}$ as determined by \eqref{def:Kernels}. 
    However, in view of the independent increment assumption, all the aforementioned objects are deterministic; see \cref{prop:independent_increments}.
    Hence, in view of the identical distribution assumption and Remarks \ref{rem:About_MFBSDE_framework} \ref{rem:about_G6} and \ref{rem:immersion_no_change_in_comp}, we can omit the upper index $(\mathbb{F}^i,\mathbf{X}^i)$ and denote them simply by $C,c$ and $K$.\label{footnote:deterministic_cCK}} 
    \item \label{PC_generator}
    A sequence of generators $(f^{N})_{N\in\mathbb{N}}$ and a generator $f$.
    The aforementioned generators possess the following properties:\footnote{Recall that $\mathbb{X}$ will be either $\mathbb{D}^d$ or $\mathbb{R}^d$, while $\mathscr{X}$ will be either $\mathbb{X}\times \mathbb{R}^{d\times p}\times \mathbb{R}^d$ or $\mathbb{X}$.}
    \begin{enumerate}[label=\textup{\textbf{(\roman*)}}]
        \item \label{PC_generator_fNsequence_structure}
                For every $N\in\mathbb{N}$, $f^{N}$ satisfies \ref{mf_generator} for deterministic Lipschitz bounds $(r^{i,N})_{i\in\mathscr{N}}$  $(\vartheta^{i,N,\circ})_{i\in\mathscr{N}}$, $(\vartheta^{i,N,\natural})_{i\in\mathscr{N}}$, $(\vartheta^{i,N,*})_{i\in\mathscr{N}}$.\footnote{The reader might observe that in the mean-field framework we have suppressed the notational dependence of the conditions on $N$. }

        \item \label{PC_generator_fdeterministic_structure}
                The generator $f$ is deterministic and satisfies \ref{MV_generator} with Lipschitz bounds $(r, \vartheta^{\circ},\vartheta^{\natural},\vartheta^{*})$.

        \item \label{PC_generator_convergence condition}
            The following boundedness property\footnote{See \ref{PC_alpha_A} for the definitions of $\alpha, A$.} holds
            \begin{align*}
            \|f^{i,N}_{\cdot}-f_{\cdot}\|_{\infty}^2\leq \alpha_{\cdot}^4 g^{i,N}_{\cdot}, \hspace{0.5cm} \mathbb{P}\otimes C-a.e.,
            \end{align*}
            for a sequence of sets $\big((g^{i,N})_{i\in\mathscr{N}}\big)_{N\in\mathbb{N}}$ of functions which are $\mathbb{P}\otimes\mathcal{E}\big(\hat{\beta}A\big)-$integrable, jointly measurable with respect to $\mathcal{G}\otimes\mathcal{B}([0,T])$ and real-valued, such that 
            \begin{align*}
            \mathbb{E}\Big[\int_{0}^{T}g^{i,N}_s \ud\mathcal{E}\big(\hat{\beta}A\big)_s\Big]\xrightarrow[N\rightarrow\infty]{|\cdot|} 0,
            \end{align*}
            for every $i\in\mathbb{N}$, and
            \begin{align*}
            \frac{1}{N}\sum_{i=1}^{N}\mathbb{E}\Big[\int_{0}^{T}g^{i,N}_s \ud\mathcal{E}\big(\hat{\beta}A\big)_s\Big]\xrightarrow[N\rightarrow\infty]{|\cdot|} 0.
            \end{align*}
    \end{enumerate}
    \item\label{PC_alpha_A}
    The c\`adl\`ag function
    \begin{align*}
      A_\cdot := \int_{0}^{\cdot}\alpha^2_s\,\ud C_s, 
    \end{align*}
    is such that there exists $\Phi > 0$ with
    \begin{align*}
    \Delta A_t(\omega) \leq \Phi,\ \mathbb{P} \otimes C-\text{a.e.},
    \end{align*}
    where 
    \begin{align*}
        \begin{multlined}[0.9\textwidth]
            \alpha^2_{\cdot} := \sup\Big\{
        \sqrt{r_{\cdot}},\Big(\big(\sqrt{r_{\cdot}^{i,N}}\big)_{i\in\mathscr{N}}\Big)_{N\in\mathbb{N}},
        \vartheta^{\circ}_{\cdot}, \big((\vartheta_{\cdot}^{i,N,\circ})_{i\in\mathscr{N}}\big)_{N\in\mathbb{N}},\\
        \vartheta^{\natural}_{\cdot}, \big((\vartheta_{\cdot}^{i,N,\natural})_{i\in\mathscr{N}}\big)_{N\in\mathbb{N}}, 
        h(\vartheta^{*}_{\cdot}), \big((h(\vartheta_{\cdot}^{i,N,*}))_{i\in\mathscr{N}}\big)_{N\in\mathbb{N}}
        \Big\},
        \end{multlined}
    \end{align*} with $\alpha^2>0$, for a function $h$ to be specified.
    \item\label{PC_A_cases} There exists $\Lambda_{\hat{\beta}} > 0$ such that
    $\mathcal{E}(\hat{\beta} A_T)<\Lambda_{\hat{\beta}}$.
    \item\label{PC_Theta_Gamma} 
    The deterministic, Borel measurable function $\Theta$ such that $|\Theta|\leq |I|$ and the functions
    $\Gamma$ as in \cref{def_Gamma_function}, where for each $N\in\mathbb{N}$ and for each $i\in\mathscr{N}$ the data for the definition are $(\mathbb{F}^i, \mathbf{X}^i,C,K)$ for the $i-$th McKean--Vlasov BSDE and $(\mathbb{F}^\mathscr{N}, \mathbf{X}^i,C,K)$ for the $N-$th mean-field BSDE system.\footnote{We underline that $\Theta \in \widetilde{\mathcal{P}}^{\mathbb{F}^i}, \widetilde{\mathcal{P}}^{\mathbb{F}^\mathscr{N}}$, for each $N\in\mathbb{N}$ and for each $i\in\mathscr{N}$. Furthermore, for the functions $\Gamma$ we note \cref{lem:equalities_integrals}.}
\item\label{PC_terminal_rv} The sequence of independent and identically distributed terminal random variables $(\xi^i)_{i \in \mathbb{N}}$ and the sequence of terminal conditions $\big((\xi^{i,N})_{i \in \mathscr{N}}\big)_{N\in \mathbb{N}}$ satisfy
\begin{gather*}
    \xi^i\in \mathbb{L}^2_{\hat{\beta}}(\mathcal{F}^i_T, A;\mathbb{R}^d) \text{ for every }i\in\mathbb{N}
    \shortintertext{ and }
    \xi^{i,N}\in \mathbb{L}^2_{\hat{\beta}}(\mathcal{F}^{\mathscr{N}}_T, A;\mathbb{R}^d) \text{ for every }i\in\mathscr{N}, N\in\mathbb{N},
\end{gather*}
where $\mathcal{F}^{\mathscr{N}}:=\bigvee_{i\in\mathscr{N}} \mathcal{F}^i_T$. Moreover, 
    \[
        \|\xi^{i,N} - \xi^i\|^2_{\mathbb{L}^2_{\hat{\beta}}(\mathcal{G}_T, A;\mathbb{R}^d)} \xrightarrow[N \rightarrow \infty]{} 0 \text{ for every }i\in\mathbb{N},
    \] 
    and 
    \[
        \frac{1}{N} \sum_{i = 1}^{N} \|\xi^{i,N} - \xi^i\|^2_{\mathbb{L}^2_{\hat{\beta}}(\mathcal{G}_T, A;\mathbb{R}^d)} \xrightarrow[N \rightarrow \infty]{}0.
    \]
\item \label{PC_Integrability_Generator}  For every $N\in\mathbb{N}$ and $i\in\mathscr{N}$, and for $\mathbb{F}^{\mathscr{N}}:=\bigvee_{i\in\mathscr{N}} \mathbb{F}^i$, holds
   \begin{gather*}
   \big\| \frac{1}{\alpha} f^{i,N}(\cdot,0,0,0,\delta_0)\big\|_{\mathbb{H}^2_{\hat{\beta}}(\mathbb{F}^{\mathscr{N}},A,C;\mathbb{R}^d)}<\infty,
   \end{gather*}
   where $\delta_0$ is the Dirac measure on the domain of the last argument concentrated at $0$, the neutral element of the addition.
\end{enumerate}
\begin{remark}\label{rem:comments_H_Cond}
\begin{enumerate}
    \item\label{rem:comments_H_Cond1}
    By construction, $\{\mathbb{F}^i\}_{i \in \mathbb{N}}$ is a sequence of independent filtrations on $(\Omega,\mathcal{G},\mathbb{P})$.
    Moreover, for every $N \in \mathbb{N}$, we define the filtration $\mathbb{F}^{\mathscr{N}} := \bigvee_{m = 1}^{N}\mathbb{F}^m$.
    Using \citet[Theorem 1]{Wu1982} we have that $\mathbb{F}^{\mathscr{N}}$ satisfies the usual conditions.
    Let $i \in \mathscr{N}$ and $N\in\mathbb{N}$ then, 
    a direct consequence of the independence of filtrations is that every $\mathbb{F}^i-$martingale, remains martingale under $\mathbb{F}^{\mathscr{N}}$, \emph{i.e.}, the filtration $\mathbb{F}^i$ is immersed in the filtration $\mathbb{F}^{\mathscr{N}}$.
    In particular, $X^{i,\natural} \in \mathcal{H}^{2,d}(\mathbb{F}^{\mathscr{N}}; \mathbb{R}^n)$; see \cref{cor:same_stoch_integ_wrtIVRM}.
     Additionally, from the assumption $M_{\mu^{X^{i,\natural}}}[\Delta X^{i,\circ}|\widetilde{\mathcal{P}}^{\mathbb{F}^i}]
    = 0$ one can deduce that $M_{\mu^{X^{i,\natural}}}[\Delta X^{i,\circ}|\widetilde{\mathcal{P}}^{\mathbb{F}^{\mathscr{N}}}]= 0$ also holds; one can follow the exact same arguments as in \cref{lem:conserv_laws}. 

    \item\label{rem:comments_H_Cond2} 
    Let $i\in\mathscr{N}$ and $N\in\mathbb{N}$. Then due to the independence of $\{\mathbb{F}^i\}_{i\in\mathbb{N}}$, we have that $\mathbb{F}^i$ is immersed in $\mathbb{F}^{\mathscr{N}}$, and $\mathbb{F}^{\mathscr{N}}$ is immersed in $\bigvee_{m=1}^\infty \mathbb{F}^m$. 

    \item\label{rem:comments_Y_ident_distr}
     The sequence of driving martingales associated to the McKean--Vlasov equations are copies of one another (I.I.D.), and the same holds for the terminal random variables. 
     Hence, from \cref{lem:conserv_laws} (see also \cref{subsec_app:construction}), we can view the solutions also as copies of one another, in other words, as strong solutions under the larger filtration and conclude their uniqueness in law. 
     In other words, the solutions of the McKean--Vlasov BSDEs are identically distributed. 
     We note that the same holds true for the processes $\{\Gamma^{(\mathbb{F}^{i},\mathbf{X}^i,\Theta)}(U^{i})\}_{i\in\mathbb{N}}$; recall the representation of $\Gamma$ provided in \cref{Remark 2.1.6} \ref{Remark 2.1.6. 3}.

     \item Combining \ref{PC_generator} and \ref{PC_Integrability_Generator} we get 
     \begin{equation*}
         \int_0^T \mathcal{E}\big(\hat{\beta}A\big)_{s-}\frac{1}{\alpha_s^2}|f(s,0,0,0,\delta_0)|^2\textup{d} C_s <\infty.
     \end{equation*}
    \end{enumerate}
    \end{remark}

\begin{proposition}\label{prop:independent_increments}
Let $\mathbb{G}$ be a filtration on $(\Omega,\mathcal{G},\mathbb{P})$ that satisfies the usual conditions. 
Consider a pair $\mathbf{X} := (X^{\circ},X^\natural) \in \mathcal{H}^2(\mathbb{G};\mathbb{R}^p) \times \mathcal{H}^2(\mathbb{G};\mathbb{R}^n)$ with independent increments.
Then, the processes $C^{(\mathbb{G},\mathbf{X})}$ and $ c^{(\mathbb{G},\mathbf{X})}$, as defined via \eqref{def_C} and \eqref{def_c}, are deterministic.
\end{proposition}
\begin{proof}
Let $j\in\{1,\dots,p\}$, then $X^{\circ,j}$ denotes the $j-$element of the $p-$dimensional process $X^{\circ}$.
Analogously, $x^j$ denotes the $j-$element of the $p-$dimensional vector $x$.
Using \citet[Definition 6.27]{he2019semimartingale}, we have that the dual predictable projection of the process $\sum_{s \leq t}|\Delta X^{\circ,j}_s|^2$ is equal to $(x^j)^2 * \nu^{(\mathbb{G},X^{\circ})}$, for every $j \in \{1,\dots,p\}$.
Then, using \citet[Corollaries 7.87, 7.88]{medvegyev2007stochastic}, we get that $\langle X^{\circ} \rangle^{\mathbb{G}}$ and $\nu^{(\mathbb{G},X^{\natural})}$ are deterministic.
\end{proof}
\begin{remark}
The independence of the increments of the martingale $\mathbf{X}$ is equivalent to its associated triplet being deterministic; see \citet[Corollary 7.87]{medvegyev2007stochastic} or \citet[Theorem II.4.15]{jacod2013limit}.
\end{remark}
%
%
%
%
%
%
%
%
%
%
%
%
%
%
%
%
%
%
%
%
\subsection{Statements of the backward propagation of chaos theorems}
In this subsection, we will present propagation of chaos theorems between systems of mean-field BSDEs \eqref{mfBSDE_general} and McKean--Vlasov BSDEs \eqref{MVBSDE_general}, for different cases, depending on whether the driver of the BSDE depends on all the components of the BSDE solution or just $Y$, and whether this dependence includes only the instantaneous value of the processes or the whole initial segment $[0,\cdot]$.
Let us fix $\hat{k}\in\mathbb{N}$, then we will focus on the first $\hat{k}$ players of the mean-field system and prove that their asymptotic behavior is the one dictated by the respective McKean--Vlasov dynamics.

\begin{theorem}[Propagation of chaos]\label{thm:prop_chaos_seg}
    Let $\hat{k}\in\mathbb{N}$. Under \ref{PC_martingale_integrators}, \ref{PC_generator} for $(\mathbb{X},\rho_{\mathbb{X}}) = (\mathbb{D}^d, \rho_{J_1})$ and $(\mathscr{X},\rho_{\mathscr{X}}) = (\mathbb{D}^d \times \mathbb{R}^{d\times p}\times \mathbb{R}^p,\rho_{\mathbb{D}^d \times \mathbb{R}^{d\times p}\times \mathbb{R}^p})$, \ref{PC_alpha_A} for $[0,\infty)\ni x\overset{h}{\mapsto} x $ and $\alpha^2_\cdot\geq 1$, \ref{PC_A_cases}, \ref{PC_Theta_Gamma}, \ref{PC_terminal_rv} and  \ref{PC_Integrability_Generator}, if
\begin{align*}
    \max\left\{12,\frac{6\Lambda_{\hat{\beta}}}{\hat{\beta}}\right\}
    M^{\Phi}_{\star}(\hat{\beta}) < 1,
\end{align*} 
then 
\begin{enumerate}[label=(\roman*)]
    \item the  mean-field $\mathbb{F}^{\mathscr{N}}-$BSDE \eqref{mfBSDE_initial_path_triplet} admits a unique solution $\overline{\mathbb{S}}^{N}:=(\overline{\mathbf{S}}^{i,N})_{i\in\mathscr{N}}$, for every $N\in\mathbb{N}$,
    such that, for every $i\in\mathscr{N}$, 
    $\|\overline{\mathbf{S}}^{i,N}\|_{\star,\hat{\beta},\mathbb{F}^{\mathscr{N}}, A,\mathbf{X}^i}<\infty$\footnote{The norm was defined in \eqref{def:star_norm}.},

    \item the McKean--Vlasov $\mathbb{F}^{i}-$BSDE \eqref{MVBSDE_initial_path_triplet} admits a unique solution $\widetilde{\mathbf{S}}^{i}$, such that 
$\|\widetilde{\mathbf{S}}^{i}\|_{\star,\hat{\beta},\mathbb{F}^i, A,\mathbf{X}^i}<\infty$, for every $i\in \mathbb{N}$
and
    \item  the following propagation of chaos statement holds
  \begin{align*}
   \lim_{N \rightarrow \infty} \sum_{i = 1}^{\hat{k}} \big\| \overline{\mathbf{S}}^{i,N} - \widetilde{\mathbf{S}}^i \big\|^{2}_{\star,\hat{\beta},\mathbb{F}^{\mathscr{N}},A,\mathbf{X}^i} = 0.
\end{align*}
\end{enumerate}
\end{theorem}

\begin{theorem}[Propagation of chaos]\label{thm:prop_chaos_seg_onlyY}
    Let $\hat{k}\in\mathbb{N}$. Under \ref{PC_martingale_integrators}, \ref{PC_generator} for $(\mathbb{X},\rho_{\mathbb{X}}) = (\mathscr{X},\rho_{\mathscr{X}}) = (\mathbb{D}^d, \rho_{J_1})$, \ref{PC_alpha_A} for $[0,\infty)\ni x\overset{h}{\mapsto} \sqrt{x}$, \ref{PC_A_cases}, \ref{PC_Theta_Gamma}, \ref{PC_terminal_rv} and  \ref{PC_Integrability_Generator}, if
\begin{align*}
    \max\left\{12,\frac{6\Lambda_{\hat{\beta}}}{\hat{\beta}}\right\} 
    M^{\Phi}_{\star}(\hat{\beta}) < 1,
\end{align*} 
then 
\begin{enumerate}[label=(\roman*)]
    \item the  mean-field $\mathbb{F}^{\mathscr{N}}-$BSDE \eqref{mfBSDE_initial_path_onlyY} admits a unique solution $\overline{\mathbb{S}}^{N}:=(\overline{\mathbf{S}}^{i,N})_{i\in\mathscr{N}}$, for every $N\in\mathbb{N}$,
    such that, for every $i\in\mathscr{N}$, 
    $\|\overline{\mathbf{S}}^{i,N}\|_{\star,\hat{\beta},\mathbb{F}^{\mathscr{N}}, A,\mathbf{X}^i}<\infty$,

    \item the McKean--Vlasov $\mathbb{F}^{i}-$BSDE \eqref{MVBSDE_initial_path_onlyY} admits a unique solution $\widetilde{\mathbf{S}}^{i}$, such that 
$\|\widetilde{\mathbf{S}}^{i}\|_{\star,\hat{\beta},\mathbb{F}^i, A,\mathbf{X}^i}<\infty$, for every $i\in \mathbb{N}$
and
    \item  the following propagation of chaos statement holds
  \begin{align*}
   \lim_{N \rightarrow \infty} \sum_{i = 1}^{\hat{k}} \big\| \overline{\mathbf{S}}^{i,N} - \widetilde{\mathbf{S}}^i \big\|^{2}_{\star,\hat{\beta},\mathbb{F}^{\mathscr{N}},A,\mathbf{X}^i} = 0.
\end{align*}
\end{enumerate}
\end{theorem}

\begin{theorem}[Propagation of chaos]\label{thm:prop_chaos_inst}
    Let $\hat{k}\in\mathbb{N}$. Under \ref{PC_martingale_integrators}, \ref{PC_generator} for $(\mathbb{X},\rho_{\mathbb{X}}) = (\mathbb{R}^d, \rho_{|\cdot|})$ and $(\mathscr{X},\rho_{\mathscr{X}}) = (\mathbb{R}^d \times \mathbb{R}^{d\times p}\times \mathbb{R}^p,\rho_{\mathbb{R}^d \times \mathbb{R}^{d\times p}\times \mathbb{R}^p})$, \ref{PC_alpha_A} for $[0,\infty)\ni x\overset{h}{\mapsto} x $ and $\alpha^2_\cdot\geq 1$,  \ref{PC_Theta_Gamma}, \ref{PC_terminal_rv} and  \ref{PC_Integrability_Generator}, if
\begin{align*}
    12 M^{\Phi}_{\scalebox{0.4}{$\blacksquare$}}(\hat{\beta}) < 1,
\end{align*} 
then 
\begin{enumerate}[label=(\roman*)]
    \item the  mean-field $\mathbb{F}^{\mathscr{N}}-$BSDE \eqref{mfBSDE_instantaneous_b_triplet} admits a unique solution $\overline{\mathbb{S}}^{N}:=(\overline{\mathbf{S}}^{i,N})_{i\in\mathscr{N}}$, for every $N\in\mathbb{N}$,
    such that, for every $i\in\mathscr{N}$, 
    $\|\overline{\mathbf{S}}^{i,N}\|_{\scalebox{0.4}{$\blacksquare$},\hat{\beta},\mathbb{F}^{\mathscr{N}}, A,\mathbf{X}^i}<\infty$\footnote{The norm was defined in \eqref{def:blacksquare_norm}},

    \item the McKean--Vlasov $\mathbb{F}^{i}-$BSDE \eqref{MVBSDE_instantaneous_b_triplet} admits a unique solution $\widetilde{\mathbf{S}}^{i}$, such that 
$\|\widetilde{\mathbf{S}}^{i}\|_{\scalebox{0.4}{$\blacksquare$},\hat{\beta},\mathbb{F}^i, A,\mathbf{X}^i}<\infty$, for every $i\in \mathbb{N}$
and
    \item  the following propagation of chaos statement holds
  \begin{align*}
   \lim_{N \rightarrow \infty} \sum_{i = 1}^{\hat{k}} \big\| \overline{\mathbf{S}}^{i,N} - \widetilde{\mathbf{S}}^i \big\|^{2}_{\scalebox{0.4}{$\blacksquare$},\hat{\beta},\mathbb{F}^{\mathscr{N}},A,\mathbf{X}^i} = 0.
\end{align*}
\end{enumerate}
\end{theorem}

\begin{theorem}[Propagation of chaos]\label{thm:prop_chaos_inst_onlyY}
    Let $\hat{k}\in\mathbb{N}$. Under \ref{PC_martingale_integrators}, \ref{PC_generator} for $(\mathbb{X},\rho_{\mathbb{X}}) = (\mathscr{X},\rho_{\mathscr{X}}) = (\mathbb{R}^d,\rho_{|\cdot|})$, \ref{PC_alpha_A} for $[0,\infty) \ni x\overset{h}{\mapsto} \sqrt{x}$,  \ref{PC_Theta_Gamma}, \ref{PC_terminal_rv} and  \ref{PC_Integrability_Generator}, if
\begin{align*}
    12 M^{\Phi}_{\scalebox{0.4}{$\blacksquare$}}(\hat{\beta}) < 1,
\end{align*} 
then 
\begin{enumerate}[label=(\roman*)]
    \item the  mean-field $\mathbb{F}^{\mathscr{N}}-$BSDE \eqref{mfBSDE_instantaneous_b_onlyY} admits a unique solution $\overline{\mathbb{S}}^{N}:=(\overline{\mathbf{S}}^{i,N})_{i\in\mathscr{N}}$, for every $N\in\mathbb{N}$,
    such that, for every $i\in\mathscr{N}$, 
    $\|\overline{\mathbf{S}}^{i,N}\|_{\scalebox{0.4}{$\blacksquare$},\hat{\beta},\mathbb{F}^{\mathscr{N}}, A,\mathbf{X}^i}<\infty$,
    
    \item the McKean--Vlasov $\mathbb{F}^{i}-$BSDE \eqref{MVBSDE_instantaneous_b_onlyY} admits a unique solution $\widetilde{\mathbf{S}}^{i}$, such that 
$\|\widetilde{\mathbf{S}}^{i}\|_{\scalebox{0.4}{$\blacksquare$},\hat{\beta},\mathbb{F}^i, A,\mathbf{X}^i}<\infty$, for every $i\in \mathbb{N}$
and
    \item  the following propagation of chaos statement holds
  \begin{align*}
   \lim_{N \rightarrow \infty} \sum_{i = 1}^{\hat{k}} \big\| \overline{\mathbf{S}}^{i,N} - \widetilde{\mathbf{S}}^i \big\|^{2}_{\scalebox{0.4}{$\blacksquare$},\hat{\beta},\mathbb{F}^{\mathscr{N}},A,\mathbf{X}^i} = 0.
\end{align*}
\end{enumerate}
\end{theorem}

\subsection{Proofs of the backward propagation of chaos theorems}

Our approach is based on the strong comparison, \emph{i.e.}, under the norm, of the solutions of mean-field systems of BSDEs and McKean--Vlasov BSDEs.
This requires to embed these BSDEs under a common, large enough, filtration. 
Consequently, when we refer to a specific BSDE, we often need to distinguish under which filtration we consider this. 
Of course, a crucial question is whether the actual solution remains the same under a change of filtration or not. 
This has an affirmative answer under specific circumstances; see \cref{subsec:Conservation_of_solutions}.
We have already introduced the notation $\mathbb{F}^{i}$, for $i\in\mathbb{N}$, and $\mathbb{F}^{\mathscr{N}}$, for $N\in\mathbb{N}$ where $\mathscr{N}:=\{1,\dots,n\}$; see \ref{PC_martingale_integrators} and \ref{PC_Integrability_Generator}.
Let us also introduce the notation $\mathbb{F}^{\mathbb{N}}:=\bigvee_{i\in\mathbb{N}}\mathbb{F}^i$.
Thus, under this framework, the solution of a mean-field BSDE system remains the same under $\mathbb{F}^{\mathscr{N}}$ and under $\mathbb{F}^{\mathbb{N}}$, while the solution of the McKean--Vlasov BSDE for the $i-$th player remains the same under $\mathbb{F}^i$, under $\mathbb{F}^{\mathscr{N}}$, and under $\mathbb{F}^{\mathbb{N}}$. Moreover, for any object $Q$ indexed in both $\{i,N\}$ and $N$ we introduce the following notation
\begin{align}\label{notation:delta_object}
    \delta^N Q^i := Q^{i,N} - Q^{i}.
\end{align}

\begin{proof}[Proof of \cref{thm:prop_chaos_seg}]
The proof will be divided into multiple steps, so let us briefly describe them.
Initially, we will provide an estimate for the normed difference between the solutions of the mean-field and the McKean--Vlasov BSDEs; this estimate utilizes \cref{lem:a_priori_estimates}. 
Then, at each subsequent step, we will prove the convergence of each term appearing on the upper bound of this estimate.
Let us point out that the order of the subsequent steps should be chosen wisely, because the computations in each one depend on preceding results. 
This way, we avoid the repetition of certain computations.

Let $N\geq \hat{k}$. 
Let $\overline{\mathbb{S}}^{N}= (\overline{\mathbf{S}}^{i,N})_{i\in\mathscr{N}}$ be the solution of the mean-field BSDE under $\mathbb{F}^{\mathscr{N}}$ and
$\widetilde{\mathbf{S}}^{i}$ be the solution of McKean--Vlasov BSDE under $\mathbb{F}^{i}$.
We underline that, in view of \cref{lem:conserv_laws}, the aforementioned solutions remain solutions of the respective BSDEs under $\mathbb{F}^{\mathscr{N}}$ or even under $\mathbb{F}^{\mathbb{N}}$.
This property will be used to embed the processes into a common filtered space, enabling an immediate and strong comparison.

Using the notation introduced above, let us consider the difference between the mean-field and the McKean--Vlasov BSDE
\begin{align*}
&\begin{multlined}[0.9\textwidth]
\delta^N Y^{i}_t 
 = Y^{i,N}_t - Y^i_t 
= \delta^N \xi^{i} \\
+ \int^{T}_{t} \Big\{   f^{i,N}\big(s,\overline{\mathbf{S}}^{\textrm{mod},i,N}_s,L^N( \overline{\mathbb{S}}^{\textrm{mod},N}_s)\big)
-                       f\big(s,\widetilde{\mathbf{S}}^{\textrm{mod},i}_s,\mathcal{L}( \widetilde{\mathbf{S}}^{\textrm{mod},i}_s)\big) \Big\} \, \ud C_s
\numberthis\label{system:prop_syst}\\
- \int_{t}^{T}\,\ud\Big(\left(\delta^N Z^{i}\right) \cdot X^{i,\circ} + \left(\delta^N U^{i}\right) \star {\widetilde{\mu}}^{(\mathbb{F}^{\mathscr{N}},X^{i,\natural})} + \delta^N M^{i}\Big)_{s}, 
\end{multlined}
\end{align*}
which crucially yields another system of BSDEs; later we will apply the a priori estimates to this system as well.
Let us introduce the following auxiliary notation:
\begin{align*}
    \hspace{2.5cm}\psi^{i,N} &:=
         f^{i,N}\big(s,\overline{\mathbf{S}}^{\textrm{mod},i,N},L^N( \overline{\mathbb{S}}^{\textrm{mod},N})\big)
                        -f\big(s,\widetilde{\mathbf{S}}^{\textrm{mod},i},\mathcal{L}( \widetilde{\mathbf{S}}^{\textrm{mod},i})\big),\\
    \psi^{i,N,[1]} &:=   f^{i,N}\big(s,\overline{\mathbf{S}}^{\textrm{mod},i,N},L^N( \overline{\mathbb{S}}^{\textrm{mod},N})\big)
                        -f\big(s,\overline{\mathbf{S}}^{\textrm{mod},i,N},L^N( \overline{\mathbb{S}}^{\textrm{mod},N})\big),\\
    \psi^{i,N,[2]} &:=   f\big(s,\overline{\mathbf{S}}^{\textrm{mod},i,N},L^N( \overline{\mathbb{S}}^{\textrm{mod},N})\big)
                        -f\big(s,\widetilde{\mathbf{S}}^{\textrm{mod},i},L^N( \widetilde{\mathbb{S}}^{\textrm{mod},N})\big)
\shortintertext{and}
    \psi^{i,N,[3]} &:=   f\big(s,\widetilde{\mathbf{S}}^{\textrm{mod},i},L^N( \widetilde{\mathbb{S}}^{\textrm{mod},N})\big)
                        -f\big(s,\widetilde{\mathbf{S}}^{\textrm{mod},i},\mathcal{L}( \widetilde{\mathbf{S}}^{\textrm{mod},i})\big),
\end{align*}
that will be utilized in the proof later.
Applying \cref{lem:a_priori_estimates} to \eqref{system:prop_syst}, we obtain, for every $i\in\mathscr{N}\equiv\{1,\dots,N\}$, $N\in\mathbb{N}$, an estimate of the form\footnote{For ease of notation we avoid the detailed description of the norms as it is clear from the context which norm is used.}
\begin{align*}
    \| \overline{\mathbf{S}}^{i,N} - \widetilde{\mathbf{S}}^i\|^2
    &\leq C_{\xi} \| \delta^N \xi^i\|^2 +  M^{\Phi}_{\star}(\hat{\beta}) \Big\|\frac{1}{\alpha}\psi^{i,N}\Big\|^2\\
    &\leq C_{\xi} \| \delta^N \xi^i\|^2 + 3 M^{\Phi}_{\star}(\hat{\beta}) \Big(\big\|\frac{1}{\alpha}\psi^{i,N,[1]}\big\|^2 + \big\|\frac{1}{\alpha}\psi^{i,N,[2]}\big\|^2 + \big\|\frac{1}{\alpha}\psi^{i,N,[3]}\big\|^2\Big)\\
    &\begin{multlined}[0.85\textwidth]
        \leq C_{\xi} \| \delta^N \xi^i\|^2 
        + \frac{3}{\hat{\beta}} M^{\Phi}_{\star}(\hat{\beta}) \mathbb{E}\Big[\int_0^T g^{i,N}_s \textup{d}\mathcal{E}( \hat{\beta} A)_s\Big] \\
        +  \max\left\{6,\frac{3\Lambda_{\hat{\beta}}}{\hat{\beta}}\right\}  M^{\Phi}_{\star}(\hat{\beta}) \Big\{\|\overline{\mathbf{S}}^{i,N} - \widetilde{\mathbf{S}}^{i}\|^2 + \frac{1}{N} \sum_{m=1}^N\|\overline{\mathbf{S}}^{m,N} - \widetilde{\mathbf{S}}^{m}\|^2\Big\}\\
        + \frac{6}{\hat{\beta}} M^{\Phi}_{\star}(\hat{\beta})  \int_0^T \mathbb{E}\Big[\mathcal{W}^2_{2,\rho_\mathscr{X}}\big(L^N( \widetilde{\mathbb{S}}^{\textrm{mod},N}_s), \mathcal{L}( \widetilde{\mathbf{S}}^{\textrm{mod},i}_s)\big) \Big] \textup{d}\mathcal{E}(\hat{\beta} A)_{s},
    \numberthis\label{ineq:delta_BSDE_a_priori_estimates}
    \end{multlined}
\end{align*}

Let us point out that the right-hand side of \eqref{ineq:delta_BSDE_a_priori_estimates} also depends on $\overline{\mathbf{S}}^{i,N} - \widetilde{\mathbf{S}}^i$. 
Therefore, we will need to obtain some auxiliary results for the desired conclusion.
These will be presented as intermediate steps:
\begin{enumerate}[label=\textit{Step~\arabic*}:, ref=Step~\arabic*]
    \item\label{Step1} 
    We will first show that, for every $i\in\mathbb{N}$,
    \begin{align}\label{lim:LLN_Step1}
        \int_0^T \mathbb{E}\Big[\mathcal{W}^2_{2,\rho_\mathscr{X}}\big(L^N( \widetilde{\mathbb{S}}^{\textrm{mod},N}_s), \mathcal{L}( \widetilde{\mathbf{S}}^{\textrm{mod},i}_s)\big) \Big] \textup{d}\mathcal{E}(\hat{\beta} A)_{s}\xrightarrow[N\to\infty]{}0,
    \end{align}
     which settles the convergence of the last term on the right-hand-side of \eqref{ineq:delta_BSDE_a_priori_estimates}.    
    This proof will be decomposed into three subtasks.
    In the first one, using the law of large numbers we will prove, for every $i\in\mathbb{N}$, that 
    \begin{align}\label{lim:LLN_Step1_Task1}
        \mathcal{W}^2_{2,\rho_\mathscr{X}}\big(L^N( \widetilde{\mathbb{S}}^{\textrm{mod},N}_s), \mathcal{L}( \widetilde{\mathbf{S}}^{\textrm{mod},i}_s)\big) \xrightarrow[N\to\infty]{}0 \quad \mathbb{P}-a.s. \text{ for every }s\in[0,T].
    \end{align}
    In the second one, we will show, for every $i\in\mathbb{N}$ and $s\in[0,T]$, that 
    \begin{align}\label{lim:LLN_Step1_Task2}
        \mathbb{E}\Big[\mathcal{W}^2_{2,\rho_\mathscr{X}}\big(L^N( \widetilde{\mathbb{S}}^{\textrm{mod},N}_s), \mathcal{L}( \widetilde{\mathbf{S}}^{\textrm{mod},i}_s)\big)\Big] \xrightarrow[N\to\infty]{}0.
    \end{align}
    In the third one, we will deduce the existence of a measurable and integrable function such that for every $i,N\in\mathbb{N}$ it dominates the measurable function
    \begin{align*}
        [0,T]\ni s \mapsto \mathbb{E}\Big[\mathcal{W}^2_{2,\rho_\mathscr{X}}\big(L^N( \widetilde{\mathbb{S}}^{\textrm{mod},N}_s), \mathcal{L}( \widetilde{\mathbf{S}}^{\textrm{mod},i}_s)\big) \Big].
    \end{align*}
    Thus, by means of the dominated convergence theorem we will conclude the convergence \eqref{lim:LLN_Step1}.
 
Let us recall a simple, yet crucial remark:
the solution of the McKean--Vlasov $\mathbb{F}^i-$BSDE \eqref{MVBSDE_initial_path_triplet} is an independent copy of the McKean--Vlasov $\mathbb{F}^j-$BSDE \eqref{MVBSDE_initial_path_triplet}, for every $i\neq j$ with $i,j\in\mathbb{N}$.
This is the outcome of the independence of the filtrations $\mathbb{F}^i$ and $\mathbb{F}^j$ generated respectively by the (independent) integrators $\mathbf{X}^i$ and $\mathbf{X}^j$, for every $i\neq j$ with $i,j\in\mathbb{N}$, in conjunction with \cref{lem:conserv_laws}.
The latter allows to recognize the aforementioned solutions as strong solutions under the larger filtration $\mathbb{F}^{\mathscr{N}}$, or even $\mathbb{F}^{\mathbb{N}}$, and conclude that their laws are identical. 
Thus, in particular, we have for every $i,j \in \mathbb{N}$ and $s\in[0,T]$ that $\mathcal{L}(\widetilde{\mathbf{S}}^{\textrm{mod},i}_s) = \mathcal{L}(\widetilde{\mathbf{S}}^{\textrm{mod},j}_s)$\footnote{Being more precise, we remind the reader that the objects $c,C$ and $K$ are deterministic; see \cref{footnote:deterministic_cCK}.}.
Therefore, it is sufficient to prove all the above for $i=1$.

Let us present the arguments for the validity of the first task described above, \emph{i.e.}, the validity of \eqref{lim:LLN_Step1_Task1}.
Essentially, we apply \citet[Definition 6.8 \textit{(iv)} and Theorem 6.9]{villani2009optimal}, with $x_0 = 0$, \cref{thm:weak_conv_r} and the aforementioned remark to get from the Strong Law of Large Numbers the desired result.
More precisely, we might recall the comments provided in \cref{sec:Wasserstein_dist}, and, in particular, the fact that the Wasserstein distance metrizes the weak convergence of measures in $\mathscr{P}_r(\mathscr{X})$, which allows to translate it to the more standard convergence of integrals, known from the usual weak convergence of measures in $\mathscr{P}(\mathscr{X})$. Fix $s \in [0,T]$, from \cref{thm:weak_conv_r} we only need to show - we follow the notation of \cref{thm:weak_conv_r} - that for a suitable sequence $\big\{f^{\mathcal{L}(\widetilde{\mathbf{S}}^{\textrm{mod},1}_s)}_k\big\}_{k\in\mathbb{N}}\subseteq C(\mathscr{X})$ such that 
\begin{align*}
|f^{\mathcal{L}(\widetilde{\mathbf{S}}^{\textrm{mod},1}_s)}_k(x)| \leq C (1 + \rho_{\mathscr{X}}(0,x)^2),    
\end{align*}
for a constant $C>0$ which depends on $f^{\mathcal{L}(\widetilde{\mathbf{S}}^{\textrm{mod},1}_s)}_k$, it holds
\begin{align*}
    \lim_{N \rightarrow \infty} \int_{\mathscr{X}}f^{\mathcal{L}( \widetilde{\mathbf{S}}^{\textrm{mod},1}_s)}_k\,\textup{d} L^N( \widetilde{\mathbb{S}}^{\textrm{mod},N}_s) 
    = \int_{\mathscr{X}}f^{\mathcal{L}( \widetilde{\mathbf{S}}^{\textrm{mod},1}_s)}_k\,\textup{d}\mathcal{L}( \widetilde{\mathbf{S}}^{\textrm{mod},1}_s) \, \mathbb{P} -a.s.,
\end{align*}
for all $k \in \mathbb{N}$.
Fix $k \in \mathbb{N}$, the above equality is equivalently written as
\begin{align*}
   \lim_{N \rightarrow \infty} \frac{1}{N} \sum_{m = 1}^{N} f^{\mathcal{L}(\widetilde{\mathbf{S}}^{\textrm{mod},1}_s)}_k(\widetilde{\mathbf{S}}^{\textrm{mod},m}_s) 
    = \mathbb{E}\left[f^{\mathcal{L}(\widetilde{\mathbf{S}}^{\textrm{mod},1}_s)}_k(\widetilde{\mathbf{S}}^{\textrm{mod},1}_s)\right]
\mathbb{P} -a.s.,
\end{align*}
which, in fact, is the Strong Law of Large Numbers for the independent and identically distributed random variables 
$\Big(f^{\mathcal{L}(\widetilde{\mathbf{S}}^{\textrm{mod},1})}_k(\widetilde{\mathbf{S}}^{\textrm{mod},m})\Big)_{m \in \mathbb{N}}$; see \citet[Theorem 3]{varadarajan_conv_sample}.

Next, we proceed with the verification of \eqref{lim:LLN_Step1_Task2}. To this end, let us initially provide some preparatory estimates. 
Namely, by applying the triangle inequality for the Wasserstein distance as well as \eqref{ineq:empirical_measures_bound} we get that
\begin{align*}
&\mathcal{W}_{2,\rho_{\mathscr{X}}}^2\Big(L^N\big(\widetilde{\mathbb{S}}^{\textrm{mod},N}_s\big),
    \mathcal{L}(\widetilde{\mathbf{S}}^{\textrm{mod},1}_s)\Big) \\
&\hspace{2em}\leq 2 \mathcal{W}_{2,\rho_{\mathscr{X}}}^2\left(L^N\big(\widetilde{\mathbb{S}}^{\textrm{mod},N}_s\big),\delta_0\right)
    + 2 \mathcal{W}_{2,\rho_{\mathscr{X}}}^2\big(\delta_0,\mathcal{L}( \widetilde{\mathbf{S}}^{\textrm{mod},1}_s)\big)\\
&\hspace{2em}\begin{multlined}[0.8\textwidth]
\leq 2 \frac{1}{N} \sum_{m = 1}^{N} \bigg[\sup_{u\in[0,s]}|Y^m_u|^2 + \|Z^{m}_s  c_s\|^2 + |\Gamma^{(\mathbb{F}^{m},\mathbf{X}^m,\Theta)}(U^{m})_s|^2\bigg]\\
\hspace{.5cm} + 2 \mathbb{E}\left[\sup_{u\in[0,s]}|Y^1_u|^2+\|Z^{1}_s  c_s\|^2 + |\Gamma^{(\mathbb{F}^{1},\mathbf{X}^1,\Theta)}(U^{1})_s|^2\right].
\label{ineq:Wasserstein_Empirical_MVLaw_upper_bound}
\end{multlined}
 \numberthis
\end{align*}
For the validity of the desired claim, it is sufficient to prove that the right-hand side of \eqref{ineq:Wasserstein_Empirical_MVLaw_upper_bound} is uniformly integrable. 
This would imply the uniform integrability of the left-hand side, as dominated by a uniformly integrable one.
Using that $\big( \widetilde{\mathbf{S}}^{\textrm{mod},i}_\cdot\big)_{i\in\mathbb{N}}$ is identically distributed, then \emph{a fortiori} the sequences $(\sup_{u\in[0,s]} |Y^m_u|^2)_{m\in\mathbb{N}}$, $(\|Z^m_sc_s\|^2)_{m\in\mathbb{N}}$ and $(|\Gamma^{(\mathbb{F}^m,\mathbf{X}^m,\Theta)}(U^m)_s|^2)_{m\in\mathbb{N}}$ are identically distributed. 
Additionally, by the finiteness of the norms, the aforementioned sequences are integrable, for every $s\in[0,T]$.
Combining the above it is, in particular, implied the uniform integrability of the respective sequences.
We now conclude the desired uniform integrability of the right-hand side of \eqref{ineq:Wasserstein_Empirical_MVLaw_upper_bound}, using the fact that sums of uniformly integrable families remain uniformly integrable, together with the fact that the convex hull of a uniformly integrable family is also uniformly integrable.

Finally, we need to verify that the dominated convergence theorem can be indeed applied.
Using the information presented in the above lines for every $s \in [0,T]$ and by applying the expectation on \eqref{ineq:Wasserstein_Empirical_MVLaw_upper_bound} we have
\begin{align*}
\begin{multlined}[0.9\textwidth]
\mathbb{E}\Big[\mathcal{W}^2_{2,\rho_\mathscr{X}}\big(L^N( \widetilde{\mathbb{S}}^{\textrm{mod},N}_s), \mathcal{L}( \widetilde{\mathbf{S}}^{\textrm{mod},i}_s)\big) \Big]\\
    \leq 4\mathbb{E}\Big[\sup_{u\in[0,s]}|Y^1_u|^2+\|Z^{1}_s  c_s\|^2 + |\Gamma^{(\mathbb{F}^{1},\mathbf{X}^1,\Theta)}(U^{1})_s|^2\Big].
\end{multlined}
\end{align*}
Furthermore, we have for the right-hand side of the estimate above
\begin{align*}
\begin{multlined}[0.9\textwidth]
\int_{0}^{T}\mathbb{E}\Big[\sup_{u\in[0,s]}|Y^1_u|^2+\|Z^{1}_s  c_s\|^2 + |\Gamma^{(\mathbb{F}^{1},\mathbf{X}^1,\Theta)}(U^{1})_s|^2\Big]\,\ud \mathcal{E}\big(\hat{\beta} A\big)_{s} \\
\leq  \hat{\beta} \Lambda_{\hat{\beta}} \|\left(Y^{1},Z^{1},U^{1},M^1\right) \|^{2}_{\star,\hat{\beta},\mathbb{F}^{1},\alpha,C,\mathbf{X}^1}.
\end{multlined}
\end{align*}
Hence, the Borel--measurable function $$[0,T] \ni s \longmapsto 4\mathbb{E}\left[\sup_{u\in[0,s]}|Y^1_u|^2+\|Z^{1}_s  c_s\|^2 + |\Gamma^{(\mathbb{F}^{1},\mathbf{X}^1,\Theta)}(U^{1})_s|^2\right]$$ is $\mathcal{E}(\hat{\beta} A)-$integrable and dominates the sequence 
\begin{align*}
\Big([0,T]\ni s\mapsto \mathbb{E}\Big[ W_{2,\rho_{\mathscr{X}}}^2\big(L^N\big(\widetilde{\mathbb{S}}^{\textrm{mod},N}_s\big),\mathcal{L}\big( \widetilde{\mathbf{S}}^{\textrm{mod},1}_s\big)\big)\Big]\Big)_{N \in \mathbb{N}}.
\end{align*}

    \item\label{Step2} Under the condition $ \max\left\{12,\frac{6\Lambda_{\hat{\beta}}}{\hat{\beta}}\right\} M^{\Phi}_{\star}(\hat{\beta})<1$, we will prove that 
    \begin{align*}
        \frac{1}{N} \sum_{m=1}^N \| \overline{\mathbf{S}}^{m,N} - \widetilde{\mathbf{S}}^{m}\| \xrightarrow[N\to\infty]{}0.
    \end{align*}
    Indeed, summing \eqref{ineq:delta_BSDE_a_priori_estimates} over $i\in\mathscr{N}$ and dividing by $N$ we obtain
    \begin{align*}
    &\frac{1}{N} \sum_{i=1}^N \| \overline{\mathbf{S}}^{i,N} - \widetilde{\mathbf{S}}^{i}\| \\
    &\hspace{2em}\leq \frac{C_{\xi}}{1-\max\left\{12,\frac{6\Lambda_{\hat{\beta}}}{\hat{\beta}}\right\} M^{\Phi}_{\star}(\hat{\beta})} \frac{1}{N} \sum_{i=1}^N \| \delta^N \xi^i\|^2 \\
    &\hspace{5em}+ \frac{3 M^{\Phi}_{\star}(\hat{\beta})}{\hat{\beta}\Big(1-\max\left\{12,\frac{6\Lambda_{\hat{\beta}}}{\hat{\beta}}\right\} M^{\Phi}_{\star}(\hat{\beta})\Big)} \frac{1}{N} \sum_{i=1}^N \mathbb{E}\Big[\int_0^T g^{i,N}_s \textup{d}\mathcal{E}( \hat{\beta} A)_s\Big] \\
    &\hspace{5em}+ \frac{6 M^{\Phi}_{\star}(\hat{\beta})}{\hat{\beta}\Big(1-\max\left\{12,\frac{6\Lambda_{\hat{\beta}}}{\hat{\beta}}\right\} M^{\Phi}_{\star}(\hat{\beta})\Big)} 
        \\&\hspace{4cm}\mathbb{E}\Big[ \int_0^T \mathcal{W}^2_{2,\rho_\mathscr{X}}\big(L^N( \widetilde{\mathbb{S}}^{\textrm{mod},N}_s), \mathcal{L}( \widetilde{\mathbf{S}}^{\textrm{mod},1}_s)\big) \textup{d}\mathcal{E}(\hat{\beta} A)_{s} \Big],
    \end{align*}
    which provides the desired convergence in view of \ref{PC_terminal_rv}, \ref{PC_generator}.\ref{PC_generator_convergence condition} and \ref{Step1}.

    \item\label{Step3} 
    Returning to \eqref{ineq:delta_BSDE_a_priori_estimates} and summing over $\{1,\dots,\hat{k}\}$ we obtain 
    \begin{align*}
        \sum_{i=1}^{\hat{k}} \| \overline{\mathbf{S}}^{i,N} - \widetilde{\mathbf{S}}^{i}\|^2 
        & \leq \frac{C_{\xi}}{1-\max\left\{6,\frac{3\Lambda_{\hat{\beta}}}{\hat{\beta}}\right\} M^{\Phi}_{\star}(\hat{\beta})} \sum_{i=1}^{\hat{k}} \|\delta^N \xi^i\|^2\\
        &               + \frac{3 M^{\Phi}_{\star}(\hat{\beta})}{\hat{\beta}\Big(1-\max\left\{6,\frac{3\Lambda_{\hat{\beta}}}{\hat{\beta}}\right\} M^{\Phi}_{\star}(\hat{\beta})\Big)} \sum_{i=1}^{\hat{k}} \mathbb{E}\Big[\int_0^T g^{i,N}_s \textup{d}\mathcal{E}( \hat{\beta} A)_s\Big]\\
        & + \frac{\max\left\{6,\frac{3\Lambda_{\hat{\beta}}}{\hat{\beta}}\right\} M^{\Phi}_{\star}(\hat{\beta})\hat{k}}{1-\max\left\{6,\frac{3\Lambda_{\hat{\beta}}}{\hat{\beta}}\right\} M^{\Phi}_{\star}(\hat{\beta})} \frac{1}{N}\sum_{m=1}^{N} \| \overline{\mathbf{S}}^{m,N} - \widetilde{\mathbf{S}}^{m}\|^2\\
        & + \frac{6 M^{\Phi}_{\star}(\hat{\beta})\hat{k}}{\hat{\beta}\Big(1-\max\left\{6,\frac{3\Lambda_{\hat{\beta}}}{\hat{\beta}}\right\} M^{\Phi}_{\star}(\hat{\beta})\Big)}
        \\&\hspace{2cm}\mathbb{E}\Big[ \int_0^T \mathcal{W}^2_{2,\rho_\mathscr{X}}\big(L^N( \widetilde{\mathbb{S}}^{\textrm{mod},N}_s), \mathcal{L}( \widetilde{\mathbf{S}}^{\textrm{mod},1}_s)\big) \textup{d}\mathcal{E}(\hat{\beta} A)_{s} \Big],
    \end{align*}
    which provides the desired convergence in view of \ref{PC_terminal_rv}, \ref{PC_generator}.\ref{PC_generator_convergence condition}, \ref{Step2} and \ref{Step1}.
\end{enumerate}
\end{proof}

\begin{proof}[Proof of \cref{thm:prop_chaos_seg_onlyY}]
The Proof of \cref{thm:prop_chaos_seg} can be followed verbatim with the following change:
In the derivation of \eqref{ineq:delta_BSDE_a_priori_estimates}, \eqref{properties_stoch_bounds_onlyWasserstein_identity} is replaced by \eqref{properties_stoch_bounds_onlyWasserstein_squareroot} in view of the choice $x\overset{h}{\mapsto}\sqrt{x}$; recall \cref{rem:About_MFBSDE_framework}.\ref{rem:about_a2_h_function}.
Hence, the assumption $\alpha\geq 1$ imposed in \cref{thm:prop_chaos_seg} becomes redundant.
Moreover, the presence of the $\mathbb{Y}-$part in the empirical measure only does not alter the argumentation.
\end{proof}

\begin{proof}[Proof of \cref{thm:prop_chaos_inst}]
The Proof of \cref{thm:prop_chaos_seg} can be followed verbatim with the following change:
We use the $\mathbb{H}^2-$norm for the $\mathbb{Y}-$parts of the solution, which leads to the use of the $\|\cdot\|_{\mathscr{S}^2}-$norms and consequently of the $\|\cdot\|_{\scalebox{0.4}{$\blacksquare$}}-$norms.
\end{proof}

\begin{proof}[Proof of \cref{thm:prop_chaos_inst_onlyY}]
The Proof of \cref{thm:prop_chaos_inst} can be followed verbatim with the following change:
In the derivation of \eqref{ineq:delta_BSDE_a_priori_estimates}, \eqref{properties_stoch_bounds_onlyWasserstein_identity} is replaced by \eqref{properties_stoch_bounds_onlyWasserstein_squareroot} in view of the choice $x\overset{h}{\mapsto}\sqrt{x}$; recall \cref{rem:About_MFBSDE_framework}.\ref{rem:about_a2_h_function}.
Hence, the assumption $\alpha\geq 1$ imposed in \cref{thm:prop_chaos_seg} becomes redundant.
Moreover, the presence of the $\mathbb{Y}-$part in the empirical measure only does not alter the argumentation.
\end{proof}

\subsection{Immediate corollaries}
In this subsection, we will present some immediate corollaries of the backward propagation of chaos theorems. 
In order to ease the presentation, we will introduce an auxiliary notation. 
Namely, for $\hat{k}\in\mathbb{N}$ and $N\in\mathbb{N}$ such that $N\geq \hat{k}$ we will denote by $\overline{\mathbb{S}}^{N}_{\cdot}|_{\hat{k}}$ the part of the solution $\overline{\mathbb{S}}^{N}_{\cdot}$ that is associated with the first $\hat{k}$ players of the $N-$player mean field BSDE system, \emph{i.e.}, $\overline{\mathbb{S}}^{N}_{\cdot}|_{\hat{k}}:=(\overline{\mathbf{S}}^{1,N}_\cdot, \dots,\overline{\mathbf{S}}^{\hat{k},N}_\cdot)$. 
Analogously, 
$\overline{\mathbb{S}}^{\textrm{mod},N}_{\cdot}|_{\hat{k}}:=(\overline{\mathbf{S}}^{\textrm{mod},1,N}_\cdot, \dots,\overline{\mathbf{S}}^{\textrm{mod},\hat{k},N}_\cdot)$.
Let us recall that for the McKean--Vlasov solutions associated to the first $\hat{k}$ players, we have  
$\widetilde{\mathbb{S}}^{\hat{k}}_{\cdot}:=(\widetilde{\mathbf{S}}^{1}_\cdot, \dots,\widetilde{\mathbf{S}}^{\hat{k}}_\cdot)$ and 
$\widetilde{\mathbb{S}}^{\textrm{mod},\hat{k}}_{\cdot}:=(\widetilde{\mathbf{S}}^{\textrm{mod},1}_\cdot, \dots,\widetilde{\mathbf{S}}^{\textrm{mod},\hat{k}}_\cdot)$.

The following consist of the analogons of the usual convergence in law statements for the solutions of the mean-field systems to the solutions of the McKean--Vlasov BSDEs.
\begin{corollary}\label{cor:prop_chaos_first}
Let $\hat{k}\leq N\in\mathbb{N}$. 
The following are true for any set of assumptions of Theorems \ref{thm:prop_chaos_seg}-\ref{thm:prop_chaos_inst_onlyY}:
\begin{enumerate}
\item $\displaystyle     \lim_{N \rightarrow \infty} \Big\| \mathcal{W}^2_{2,\rho_{\mathscr{X},\hat{k}}}\Big(\mathcal{L}\big(\overline{\mathbb{S}}^{\textrm{mod},N}_\cdot|_{\hat{k}}\big),\mathcal{L}\big(\widetilde{\mathbb{S}}^{\textrm{mod},\hat{k}}_\cdot\big)\Big)\Big\|_{\mathbb{H}^{2}_{\hat{\beta}}(\mathbb{F}^{\mathscr{N}},A,C;\mathbb{R})}
    =0.$
\item  $\displaystyle \lim_{N \rightarrow \infty} \sup_{t \in [0,T]} \mathcal{W}^2_{2,\mathbb{X}}\big(\mathcal{L}(\overline{\mathbb{Y}}^{N}_t|_{\hat{k}}),\mathcal{L}(\widetilde{\mathbb{Y}}^{\hat{k}}_t)\big) = 0$, where, in case $\mathbb{X}=\mathbb{D}$, for each $t\in[0,T]$ the initial segments (up to time $t$) of the paths are considered.
\end{enumerate}
\end{corollary}

\begin{proof}
We will only prove the case where $(\mathbb{X},\rho_{\mathbb{X}}) = (\mathbb{D}^d, \rho_{J_1})$ and $(\mathscr{X},\rho_{\mathscr{X}}) = (\mathbb{D}^d \times \mathbb{R}^{d\times p}\times \mathbb{R}^p,\rho_{\mathbb{D}^d \times \mathbb{R}^{d\times p}\times \mathbb{R}^p})$; the others can be treated analogously.
Let $t \in [0,T]$.
Then, by definition of the Wasserstein distance of order two, we have\footnote{The reader might recall the notation introduced in \eqref{notation:delta_object}.}
\begin{align*}
\begin{multlined}[0.9\textwidth]
W^2_{2,\rho_{\mathscr{X},\hat{k}}}\left(\mathcal{L}\big(\overline{\mathbb{S}}^{\textrm{mod},N}_t|_{\hat{k}}\big),\mathcal{L}\big(\widetilde{\mathbb{S}}^{\textrm{mod},\hat{k}}_t\big)\right) 
\leq \int_{(\mathbb{D}^d \times \mathbb{R}^{d \times p} \times \mathbb{R}^d)^{2\hat{k}}}\rho_{\mathscr{X},\hat{k}}(x,z)^2 \pi(\ud x,\ud z)\\
\leq  \sum_{i=1}^{\hat{k}} \mathbb{E}\Big[\sup_{s\in[0,t]}\{|\delta^i Y^{N}_s|^2\}+ \|(\delta^i Z^{N}_t)  c_t\|^2 
    + |\delta^i \Gamma^{(\mathbb{F}^{\mathscr{N}},\mathbf{X}^i,\Theta)}(U^{N})_t|^2\Big].
\end{multlined}
\end{align*}
Hence, we have
\begin{align*}
    \Big\| W^2_{2,\rho_{\mathscr{X},\hat{k}}}\Big(\mathcal{L}\big(\overline{\mathbb{S}}^{\textrm{mod},N}_\cdot|_{\hat{k}}\big),
        \mathcal{L}\big(\widetilde{\mathbb{S}}^{\textrm{mod},\hat{k}}_\cdot\big)\Big)\Big\|_{\mathbb{H}^{2}_{\hat{\beta}}(\mathbb{F}^{\mathscr{N}},A,C;\mathbb{R})}
    \leq  \max\left\{2,\frac{\Lambda_{\hat{\beta}}}{\hat{\beta}}\right\}  \sum_{i=1}^{\hat{k}} \| \overline{\mathbf{S}}^{i,N} - \widetilde{\mathbf{S}}^{i}\|^2_\star.
\end{align*}
Therefore, by \cref{thm:prop_chaos_seg} the proof of the first statement is complete.
The proof of the second statement follows using similar arguments and is therefore omitted.
\end{proof}

The next result, \cref{cor:conv_MFempirical_to_MVlaw}, is the analogon of \citet[Lemma 3.3]{lauriere2022backward}. 
These results consist of preparatory steps in order to derive convergence rates for the backward propagation of chaos.
Apart from their own independent interest, we would like to draw the reader's attention to their proofs. 
Namely, in \cref{cor:conv_MFempirical_to_MVlaw} we will essentially repeat the series of simple, but efficient, arguments as presented in the proofs of Theorems  \ref{thm:prop_chaos_seg}-\ref{thm:prop_chaos_inst_onlyY}; let us point out that the dependence of the generator can go beyond the empirical law of the $\overline{\mathbb{Y}}-$part of the solution.
Thus, no symmetry is needed with respect to the players' position in the $N-$player game, as well as no classical techniques, such as Gr\"onwall's lemma and Girsanov's theorem. 
This means that the subadditive relationship as described in the proof of \cref{cor:conv_MFempirical_to_MVlaw} can be directly derived. 
On the other hand, in \citet[Lemma 3.3]{lauriere2022backward} the analogous subadditive relationship is modified in a multiplicative one, as a result of the application of Gr\"onwall's inequality. 
However, the application of Girsanov's theorem was required before the application of Gr\"onwall's inequality.
This point is quite delicate and requires extra caution as -- following the notation of the Proof of \cite[Lem.~3.3, p.~14]{lauriere2022backward} -- the measure $Q$ used for Girsanov's theorem is connected to the $i-$th player.
However, the measure $Q$ is finally used for every player, as the $E_Q-$conditional expectation is calculated for every player.
This point can be, possibly, remedied by means of the symmetry assumed in the mean-field system of \cite{lauriere2022backward}. In total, the series of arguments we propose has the advantage of simplicity, which allows for their application in more general settings where Gr\"onwall's inequality and/or Girsanov's theorem may not be applicable.

\begin{corollary}\label{cor:conv_MFempirical_to_MVlaw}
Let $t \in [0,T]$ and consider any set of assumptions of Theorems \ref{thm:prop_chaos_seg}-\ref{thm:prop_chaos_inst_onlyY}.
Then, we have
\begin{align*}
\lim_{N \rightarrow \infty}\mathbb{E}\Big[\mathcal{W}_{2,|\cdot|}^2\big(L^N(\overline{\mathbb{Y}}^{N}_t),\mathcal{L}({Y}^1_t)\big)\Big] = 0.
\end{align*}
\end{corollary}

\begin{proof}
Using the triangle inequality of the Wasserstein distance and \eqref{ineq:empirical_measures_bound}, we have 
\begin{align*}
\begin{multlined}[0.9\textwidth]
\mathcal{W}_{2,|\cdot|}^2\big(L^N(\overline{\mathbb{Y}}^{N}_t),\mathcal{L}({Y}^1_t)\big) \\
\leq 2 \mathcal{W}_{2,|\cdot|}^2\big(L^N(\overline{\mathbb{Y}}^{N}_t),L^N({\widetilde{\mathbb{Y}}}^{N}_t) \big) + 2 \mathcal{W}_{2,|\cdot|}^2\big(L^N({\widetilde{\mathbb{Y}}}^{N}_t),\mathcal{L}({Y}^{1}_t)\big)\\
\leq 2  \frac{1}{N}\sum_{m = 1}^{N}|Y^{m,N}_t - Y^{m}_t|^2 
    + 2\mathcal{W}_{2,|\cdot|}^2\big(L^N\big({\widetilde{\mathbb{Y}}}^{N}_t\big),\mathcal{L}\big( Y^{1}_t\big)\big).
\end{multlined}
\end{align*}
We can then conclude, using \ref{Step2}, resp. \ref{Step1}, of the proofs of Theorems \ref{thm:prop_chaos_seg}-\ref{thm:prop_chaos_inst_onlyY} for the first, resp. second, summand on the right-hand side of the above inequality.
\end{proof}

Now, we can also provide a version of the Strong Law of Large Numbers with respect to the $\mathbb{L}^2$--convergence for the solutions of the mean-field systems. 
However, notice that for every $N\in \mathbb{N}$ and $t \in [0,T]$ the vector $\overline{\mathbb{Y}}^N_t=(Y^{1,N}_t,...,Y^{N,N}_t)$ consists of random variables which are neither independent nor exchangeable.

\begin{corollary}\label{Corollary 6.5}
Let $t \in [0,T]$ and consider any set of assumptions of Theorems \ref{thm:prop_chaos_seg}-\ref{thm:prop_chaos_inst_onlyY}.
Then, we have
\begin{align*}
    \lim_{N \rightarrow \infty} \sup_{t \in [0,T]} \mathbb{E}\bigg[\Big|\frac{1}{N}\sum_{i = 1}^{N}Y^{i,N}_t - \mathbb{E}[Y^1_t]\Big|^2\bigg] = 0.
\end{align*}
\end{corollary}

\begin{proof}
Direct computations yield
    \begin{align*}
        \mathbb{E}\bigg[ \bigg|\frac{1}{N}\sum_{m = 1}^{N}Y^{i,N}_t - \mathbb{E}[Y^1_t]\bigg|^2\bigg] 
        & = \mathbb{E}\bigg[\bigg|\frac{1}{N}\sum_{i = 1}^{N}(Y^{i,N}_t - Y^i_t) + \frac{1}{N}\sum_{i = 1}^{N}(Y^{i}_t - \mathbb{E}[Y^1_t])\bigg|^2\bigg]\\
        &\leq \frac{2}{N}\mathbb{E}\bigg[\sum_{i = 1}^{N}\big|Y^{i,N}_t - Y^i_t\big|^2\bigg] + \frac{2}{N^2}\mathbb{E}\bigg[\sum_{i = 1}^{N}\big|Y^{i}_t - \mathbb{E}[Y^1_t]\big|^2\bigg]\\
        &= \frac{2}{N}\sum_{i = 1}^{N}\mathbb{E}\Big[\big|Y^{i,N}_t - Y^i_t\big|^2\Big] + \frac{2}{N}\mathbb{E}\Big[\big|Y^{1}_t - \mathbb{E}[Y^1_t]\big|^2\Big]\\
        &\leq \frac{2}{N}\sum_{i = 1}^{N}\mathbb{E}\Big[\big|Y^{i,N}_t - Y^i_t\big|^2\Big] + \frac{2}{N}\mathbb{E}\big[\sup_{t\in[0,T]}\big|Y^{1}_t\big|^2\big];
    \end{align*}
    in the first inequality we used the well-known inequality $\big(\sum_{i = 1}^{N}a_i\big)^2 \leq  N \sum_{i = 1}^{N}a^2_i$ in conjunction with  \cref{rem:comments_H_Cond}\ref{rem:comments_Y_ident_distr}.
    Applying \ref{Step2} of the proofs of Theorems \ref{thm:prop_chaos_seg}-\ref{thm:prop_chaos_inst_onlyY} allows us to conclude.
\end{proof}


\subsection{Rates of convergence in the classical setting}
\label{Rates of convergence}

Let us consider now the setting of the previous subsection, where we have established a propagation of chaos result for players (or particles) that satisfy BSDEs. 
We are now interested in deriving convergence rates for this result. 
These rates will be based on the celebrated work of \citet{fournier2015rate} and some additional results presented below.

Assume that the conditions of Theorem \ref{thm:prop_chaos_inst} hold; moreover, assume there exists a function $R:\mathbb{N} \longrightarrow \mathbb{R}_+$ with $\lim_{N \rightarrow \infty}R(N) = 0$ such that
\begin{align*}
   \sup_{i \in \mathscr{N}} \bigg\{ \|\xi^{i,N} - \xi^i\|^2_{\mathbb{L}^2_{\hat{\beta}}(\mathcal{F}^{\mathscr{N}}_T, A;\mathbb{R}^d)}+\mathbb{E}\Big[\int_{0}^{T}g^{i,N}_s \ud\mathcal{E}\big(\hat{\beta}A\big)_s\Big]\bigg\} \leq R(N).
\end{align*}
The proof of Theorem \ref{thm:prop_chaos_inst} reveals that the solutions $\overline{\mathbb{S}}^N$ of the $N-$player mean-field BSDE \eqref{mfBSDE_instantaneous_b_triplet} and the solutions $\widetilde{\mathbb{S}}^N$ of the first $N$ McKean--Vlasov BSDEs \eqref{MVBSDE_instantaneous_triplet} satisfy the following inequalities
\begin{align}
\begin{multlined}[0.87\textwidth]
    \frac{1}{N}\sum_{m = 1}^{N} \big\|\overline{\mathbf{S}}^{m,N} - \widetilde{\mathbf{S}}^m \big\|^{2}_{\scalebox{0.4}{$\blacksquare$},\hat{\beta},\mathbb{F}^{\mathscr{N}},A,\mathbf{X}^m}
    \leq 
     \frac{3M^{\Phi}_{\scalebox{0.4}{$\blacksquare$}}(\hat{\beta})+\hat{\beta}C_{\xi}}{\hat{\beta}(1-12M^{\Phi}_{\scalebox{0.4}{$\blacksquare$}}(\hat{\beta})\big)}R(N)\\
    + \frac{6 M^{\Phi}_{\scalebox{0.4}{$\blacksquare$}}(\hat{\beta})}{\hat{\beta}(1-12M^{\Phi}_{\scalebox{0.4}{$\blacksquare$}}(\hat{\beta})\big)}\int_{0}^{T}\mathbb{E}\left[W_{2,|\cdot|}^2\left(L^N\left(\widetilde{\mathbb{S}}^{\textrm{mod},N}_s\right),\mathcal{L}\left( \widetilde{\mathbf{S}}^{\textrm{mod},1}_s\right)\right)\right] \,\ud \mathcal{E}\big(\hat{\beta} A\big)_{s},
\end{multlined}\label{ineq:rate_meanSquareNormSolution}
\end{align}
and, for $i\in \mathscr{N}$,
\begin{align}\label{ineq:rate_i_SquareNormSolution}
\begin{multlined}[0.87\textwidth]
\big\|\overline{\mathbf{S}}^{i,N} - \widetilde{\mathbf{S}}^i \big\|^{2}_{\scalebox{0.4}{$\blacksquare$},\hat{\beta},\mathbb{F}^{\mathscr{N}},A,\mathbf{X}^i} 
    \leq 
    \frac{\big(3M^{\Phi}_{\scalebox{0.4}{$\blacksquare$}}(\hat{\beta})+\hat{\beta}C_{\xi}\big)}{\hat{\beta}\big(1 - 12M^{\Phi}_{\scalebox{0.4}{$\blacksquare$}}(\hat{\beta})\big)} R(N) \\
    +\frac{6M^{\Phi}_{\scalebox{0.4}{$\blacksquare$}}(\hat{\beta})}{\hat{\beta}\big(1 - 12M^{\Phi}_{\scalebox{0.4}{$\blacksquare$}}(\hat{\beta})\big)} 
     \int_{0}^{T}\mathbb{E}\left[W_{2,|\cdot|}^2\left(L^N\left(\widetilde{\mathbb{S}}^{\textrm{mod},N}_s\right),\mathcal{L}\left( \widetilde{\mathbf{S}}^{\textrm{mod},1}_s\right)\right)\right] \,\ud \mathcal{E}\big(\hat{\beta} A\big)_{s}.
\end{multlined}
\end{align}
Therefore, in order to apply \citet[Theorem 1]{fournier2015rate} and derive convergence rates for the propagation of chaos results in the present setting, we need to know the finiteness of the quantities $\Lambda^{\textbf{S}}_{q,T}$ and $\Lambda^{\textbf{Y}}_{q,T}$ defined below; see \cref{Corollary 6.9.} for the justification. 
Obviously, one may directly extract analogous estimates for the equations \eqref{mfBSDE_instantaneous_b_onlyY}.

\begin{definition}
Let $q > 2$ be a real number and $T \in \mathbb{R}_+ \cup \{\infty\}$ be deterministic.
Then, define 
\begin{align*}
\Lambda^{\textbf{S}}_{q,T} &:= \frac{1}{\hat{\beta}}\int_{0}^{T}\left(\mathbb{E}\left[|Y^1_s|^q+ \|Z^{1}_s  c_s\|^q + |\Gamma^{(\mathbb{F}^{1},\mathbf{X}^1,\Theta)}(U^{1})_s|^q\right]\right)^{\frac{2}{q}} \ud \mathcal{E}\big(\hat{\beta} A\big)_s
\shortintertext{and}
\Lambda^{\textbf{Y}}_{q,T} &:= \frac{1}{\hat{\beta}}\int_{0}^{T}\left(\mathbb{E}\left[|Y^1_s|^q\right]\right)^{\frac{2}{q}} \ud \mathcal{E}\big(\hat{\beta} A\big)_s.
\end{align*}
\end{definition}

\begin{remark}
Using \cref{rem:comments_H_Cond}.\ref{rem:comments_Y_ident_distr}, for all $s \in [0,T]$, we have that $\mathbb{E}\big[|Y^i_s|^q\big] = \mathbb{E}\big[|Y^j_s|^q\big]$, for all $i,j \in \mathbb{N}$. 
Moreover, we obviously have that
\[
    \|\alpha Y^{1}\|^2_{\mathbb{H}^{2}_{\hat{\beta}}(\mathbb{G},A,C;\mathbb{R}^d)} = \frac{1}{\hat{\beta}}\int_{0}^{T}\mathbb{E}\left[|Y^1_s|^2\right] \ud \mathcal{E}\big(\hat{\beta} A\big)_s \leq \Lambda^{\textbf{Y}}_{q,T}.
\]
\end{remark}

\begin{theorem}\label{thm:Rates_triplet}
Let $\Lambda^{\textbf{S}}_{q,T} < \infty$ for some $q > 2$ and deterministic $T$.
Then, under the assumptions of Theorem \ref{thm:prop_chaos_inst},  there exists a constant $C_{d,p,q,2} > 0$, depending on $d,p,q,2$, such that 
\begin{align*}
\begin{multlined}[\textwidth]
\big\|\overline{\mathbf{S}}^{i,N} - \widetilde{\mathbf{S}}^i \big\|^{2}_{\scalebox{0.4}{$\blacksquare$},\hat{\beta},\mathbb{F}^{\mathscr{N}},A,\mathbf{X}^i}\\
    \leq \frac{\big(3M^{\Phi}_{\scalebox{0.4}{$\blacksquare$}}(\hat{\beta})+\hat{\beta}C_{\xi}\big)}{\hat{\beta}\big(1 - 12M^{\Phi}_{\scalebox{0.4}{$\blacksquare$}}(\hat{\beta})\big)} 
    R(N) + 
    \frac{6M^{\Phi}_{\scalebox{0.4}{$\blacksquare$}}(\hat{\beta})}{\hat{\beta}\big(1 - 12M^{\Phi}_{\scalebox{0.4}{$\blacksquare$}}(\hat{\beta})\big)} 
    \Lambda^{\textbf{S}}_{q,T} \hspace{0.1cm}C_{d,p,q,2}    \times \\
    \times
    \begin{cases}
          N^{-\frac{1}{2}} + N^{-\frac{q - 2}{q}}, & \text{if}\hspace{0.2cm}  (2+p)d < 4\hspace{0.2cm}\text{and}\hspace{0.2cm} q \neq 4\\
        N^{-\frac{1}{2}}\log(1 + N) + N^{-\frac{q - 2}{q}}, & \text{if}\hspace{0.2cm} (2+p)d = 4 \hspace{0.2cm}\text{and}\hspace{0.2cm} q \neq 4 \\
         N^{-\frac{2}{d}} + N^{-\frac{q - 2}{q}}, & \text{if}\hspace{0.2cm} (2+p)d > 4. 
    \end{cases}
\end{multlined}
\end{align*}
\end{theorem}

\begin{proof}
    Immediate from \eqref{ineq:rate_i_SquareNormSolution} and \citet[Theorem 1]{fournier2015rate}. 
\end{proof}

Using an inequality analogous to \eqref{ineq:rate_i_SquareNormSolution} in case the dependence in the measures stems only from the $Y-$part of the solution, we also have the following result.

\begin{theorem}\label{thm_Rates_onlyY}
Let $\Lambda^{\textbf{Y}}_{q,T} < \infty$ for some $q > 2$ and deterministic $T$. 
Then, under the assumptions of Theorem \ref{thm:prop_chaos_inst_onlyY}, there exists a constant $C_{d,q,2} > 0$, depending on $d,q,2$, such that 
\begin{align*}
\begin{multlined}[\textwidth]
\big\|\overline{\mathbf{S}}^{i,N} - \widetilde{\mathbf{S}}^i \big\|^{2}_{\scalebox{0.4}{$\blacksquare$},\hat{\beta},\mathbb{F}^{\mathscr{N}},A,\mathbf{X}^i}\\ 
    \leq 
    \frac{\big(3M^{\Phi}_{\scalebox{0.4}{$\blacksquare$}}(\hat{\beta})+\hat{\beta}C_{\xi}\big)}{\hat{\beta}\big(1 - 12M^{\Phi}_{\scalebox{0.4}{$\blacksquare$}}(\hat{\beta})\big)} 
    R(N) + 
    \frac{6M^{\Phi}_{\scalebox{0.4}{$\blacksquare$}}(\hat{\beta})}{\hat{\beta}\big(1 - 12M^{\Phi}_{\scalebox{0.4}{$\blacksquare$}}(\hat{\beta})\big)} 
    \Lambda^{\textbf{Y}}_{q,T} \hspace{0.1cm}C_{d,q,2} 
    \times\\
    \times
    \begin{cases}
          N^{-\frac{1}{2}} + N^{-\frac{q - 2}{q}}, & \text{if}\hspace{0.2cm}  d < 4\hspace{0.2cm}\text{and}\hspace{0.2cm} q \neq 4\\
        N^{-\frac{1}{2}}\log(1 + N) + N^{-\frac{q - 2}{q}}, & \text{if}\hspace{0.2cm} d = 4 \hspace{0.2cm}\text{and}\hspace{0.2cm} q \neq 4 \\
         N^{-\frac{2}{d}} + N^{-\frac{q - 2}{q}}, & \text{if}\hspace{0.2cm} d > 4. 
    \end{cases}
\end{multlined}
\end{align*}
\end{theorem}

The next result provides sufficient conditions for controlling the quantities of interest in \eqref{ineq:rate_meanSquareNormSolution} and \eqref{ineq:rate_i_SquareNormSolution}, and thus to derive convergence rates for the propagation of chaos results.
These conditions can be easily checked, in contrast to the boundedness assumption on $\Lambda^{\textbf{Y}}_{q,T}$, however they are stronger and more restrictive; 
see also \cref{rem:on_integrability_of_rates}.\ref{bounded condition} for further discussion in that direction.

\begin{corollary}\label{Corollary 6.9.}
Let $\sup_{t \in[0,T]}\mathbb{E}\left[|Y^1_t|^q\right] < \infty$ for some $q > 2$ and deterministic $T < \infty$.
Then, under the assumptions of Theorem \ref{thm:prop_chaos_inst_onlyY}, there exists a constant $C_{d,q,2} > 0$, depending on $d,q,2$, such that for every $t \in [0,T]$ 
\begin{align*}
\begin{multlined}[\textwidth]
\mathbb{E}\left[\mathcal{W}_{2,|\cdot|}^2\big(L^N(\overline{\mathbb{Y}}^{N}_t),\mathcal{L}({Y}^1_t)\big)\right]
    \\
    \leq \frac{6M^{\Phi}_{\scalebox{0.4}{$\blacksquare$}}(\hat{\beta})+2\hat{\beta}C_{\xi}}{\hat{\beta}(1-12M^{\Phi}_{\scalebox{0.4}{$\blacksquare$}}(\hat{\beta})\big)}R(N)
    + \left(
    \frac{12 M^{\Phi}_{\scalebox{0.4}{$\blacksquare$}}(\hat{\beta})}{\hat{\beta}(1-12M^{\Phi}_{\scalebox{0.4}{$\blacksquare$}}(\hat{\beta})\big)}
    \Lambda^{\textbf{Y}}_{q,T} + 2\left(\mathbb{E}\left[|Y^1_t|^q\right]\right)^{\frac{2}{q}} \right)\hspace{0.1cm} C_{d,q,2} \times
    \\
    \times
    \begin{cases}
          N^{-\frac{1}{2}} + N^{-\frac{q - 2}{q}} &, \text{if}\hspace{0.2cm}  d < 4\hspace{0.2cm}\text{and}\hspace{0.2cm} q \neq 4\\
        N^{-\frac{1}{2}}\log(1 + N) + N^{-\frac{q - 2}{q}} &, \text{if}\hspace{0.2cm} d = 4 \hspace{0.2cm}\text{and}\hspace{0.2cm} q \neq 4 \\
         N^{-\frac{2}{d}} + N^{-\frac{q - 2}{q}} &, \text{if}\hspace{0.2cm} d > 4. 
    \end{cases}
    \end{multlined}
\end{align*}
\end{corollary}
\begin{proof}
By definition, $\sup_{t \in[0,T]}\mathbb{E}\left[|Y^1_t|^q\right] < \infty$ and $T < \infty$ implies $
\Lambda^{\textbf{Y}}_{q,T} < \infty$.
Using the triangle inequality for the Wasserstein distance and \eqref{ineq:empirical_measures_bound},
we have 
\begin{align*}
\begin{multlined}[0.9\textwidth]
\mathcal{W}_{2,|\cdot|}^2\big(L^N(\overline{\mathbb{Y}}^{N}_t),\mathcal{L}({Y}^1_t)\big) \\
\leq 2 \mathcal{W}_{2,|\cdot|}^2\big(L^N(\overline{\mathbb{Y}}^{N}_t),L^N({\widetilde{\mathbb{Y}}}^{N}_t) \big) + 2 \mathcal{W}_{2,|\cdot|}^2\big(L^N({\widetilde{\mathbb{Y}}}^{N}_t),\mathcal{L}({Y}^{1}_t)\big)\\
\leq 2  \frac{1}{N}\sum_{m = 1}^{N}|Y^{m,N}_t - Y^{m}_t|^2 
    + 2\mathcal{W}_{2,|\cdot|}^2\big(L^N\big({\widetilde{\mathbb{Y}}}^{N}_t\big),\mathcal{L}\big( Y^{1}_t\big)\big).
\end{multlined}
\end{align*}
Hence, from the analog of \eqref{ineq:rate_meanSquareNormSolution} in case the dependence in the measure stems only from the $Y-$part of the solution, and \citet[Theorem 1]{fournier2015rate} we can conclude.
\end{proof}

\begin{remark}\label{rem:on_integrability_of_rates}
\begin{enumerate}
    \item 
\label{bounded condition}
Let $q > 2$ and $T < \infty$ be deterministic, then we have from Jensen's inequality that
\begin{align*}
    |Y^1_t|^q \leq \left(4 \hspace{0.1cm}C_T + 4\right)^{\frac{q}{2}}  \mathbb{E}\left[ |\xi^1|^q + \left(\int^{T}_{0}\left|f\left(s,Y^1_s,Z^1_s  c_s,\Gamma^{(\mathbb{F}^1,\mathbf{X}^1,\Theta)}(U^1)_s,\mathcal{L}(Y^1_s)\right)\right|^2 \, \ud C_s\right)^{\frac{q}{2}} \bigg| \mathcal{F}^1_{t} \right ].
\end{align*}
Hence, we can satisfy the requirement $\Lambda^{\textbf{Y}}_{q,T} < \infty$ by an appropriate boundedness condition on $f$ and an advanced integrability condition on $\xi^1$. 
We leave the problem of finding optimal conditions for the requirements $\Lambda^{\textbf{S}}_{q,T},\Lambda^{\textbf{Y}}_{q,T} < \infty$ open for future research.
Let us also point out that \citet{lauriere2022backward} provide sufficient conditions for deriving rates of convergence in the case of BSDEs driven by Brownian motions.

\item \label{sec_5.3}
The proofs of \cref{thm:prop_chaos_seg} and \cref{thm:prop_chaos_inst} allow us to deduce that in order to 
derive convergence rates for the path-dependent BSDEs one would 
be required to analyze the convergence rates of the empirical measure for random variables on the path space with respect to the Wasserstein distance, in the spirit of \citet{fournier2015rate}.
\end{enumerate}
\end{remark}

\begin{appendix}

\section{Auxiliary results}\label{sec_app:aux_result}

\subsection{Construction of a space satisfying \ref{PC_martingale_integrators}}\label{subsec_app:construction}

In this subsection, we will show that there exists a space satisfying \ref{PC_martingale_integrators}. Since we discuss about a sequence of independent identically distributed processes, one naturally expects to construct a countable product space, denoted by  $(\Omega, \mathcal{G},\mathbb{P})$, based on a prototype probability space $(\Omega^1, \mathcal{G}^1,\mathbb{P}^1)$. 
On this prototype probability space we will construct the pair of martingales $\mathbf{X}^1$ which satisfies the desired properties. 
As one may expect, this is not a condition that is trivially satisfied. 
Hence, we are led to consider specific cases, which nevertheless demonstrate the generality of the framework we are using. 
Let us recall that a L\'evy process is square--integrable if and only if for the corresponding L\'evy measure $\nu$ we have that
\begin{align*}
    \int_{\mathbb{R}^d}\mathds{1}_{[1,\infty)}(|x|)|x|^2 \nu(\ud x) < \infty.
\end{align*}

\begin{example}
    Let $(\Omega^1, \mathcal{G}^1,\mathbb{P}^1)$ be the probability space that supports two independent square--integrable L\'evy processes, say $(X^{1,\circ},X^{1,\natural})$,which are martingales with respect to their natural filtrations.
    We further assume that  $X^{1,\natural}$ is purely discontinuous. 
    The independence of the L\'evy processes implies the property of having no common jumps; see
    \citet[Proposition 5.3]{Tankov2003}.
    Hence the desired condition $M_{\mu^{X^{1,\natural}}}[\Delta X^{1,\circ}| \widetilde{\mathcal{P}}^{\mathbb{F}^{1}}]=0$ is trivially satisfied. 
    
    Let us mention, for completeness, that \cite[Proposition 5.3]{Tankov2003} refers to L\'evy processes with no Gaussian part, but this property remains valid also in the case we describe. 
\end{example}

\begin{example}\label{example B.2}
Let $({\Omega}^1,{\mathcal{G}}^1,\mathbb{P}^1)$ be a probability space that supports a $p$-dimensional, purely discontinuous square--integrable {L\'{e}vy} process $X^{1,\natural}$; for the construction see \citet[Theorem 4.6.17]{bichteler2002stochastic}. By taking product if necessary, we assume that our probability space also supports a sequence of independent random variables, $(h_k)_{k \in \mathbb{N}} \subseteq \mathbb{L}^2(\mathcal{G}^1; \mathbb{R}^n)$, such that the $\sigma-$algebras $\bigvee_{t \in \mathbb{R}_+}\sigma\big(X^{1,\circ}_t\big)$ and $\bigvee_{k = 1}^{\infty}\sigma\left(h^k\right)$ are independent and
\begin{align}\label{B.2}
  \mathbb{E}[h^k] = 0,  
  \forall k \in \mathbb{N}, \hspace{0.4cm}  \sum_{k = 1}^{\infty}\mathbb{E}\left[|h^k|^2\right] < \infty.
\end{align}
Furthermore, let $(t_k)_{k \in \mathbb{N}}\subseteq \mathbb{R}_+$ be a family of deterministic times indexed in increasing order. 
We define
\begin{align*}
    X^{1,\circ}(\omega^1,t) := \sum_{k = 1}^{\infty}h^k(\omega^1)\mathds{1}_{[t_k,\infty)}(t).
\end{align*}
Then we have that $\mathbf{X}^1:=(X^{1,\circ},X^{1,\natural}) \in \mathcal{H}^2({\mathbb{F}}^1;\mathbb{R}^p) \times \mathcal{H}^{2,d}({\mathbb{F}}^1;\mathbb{R}^n)$.
We remind the reader that ${\mathbb{F}}^1$ is the usual augmentation of the natural filtration of the pair. The martingale property for $X^{1,\circ}$ comes from \eqref{B.2}. Finally, because {L\'{e}vy} processes are quasi-left-continuous, they jump only at totally inaccessible times. Hence $X^{1,\circ}$ and $X^{1,\natural}$ have no common jumps and again the condition $M_{\mu^{X^{1,\natural}}}[\Delta X^{1,\circ}| \widetilde{\mathcal{P}}^{\mathbb{F}^{1}}]=0$ is trivially satisfied.
\end{example}

\begin{example}
Consider two independent random walks defined on the same grid. 
Then, if we denote as in \cref{example B.2} by $(h_k)_{k \in \mathbb{N}}$ the jumps of $X^{1,\circ}$, $(\widetilde{h}^k)_{k \in \mathbb{N}}$ the jumps of $X^{1,\natural}$ and $(t_k)_{k \in \mathbb{N}}\subseteq \mathbb{R}_+$ the grid, by reducing the general case to the one described from a single deterministic time, we have the desired property if the jumps are 0 on average. 
That is because $\mathcal{F}^1_{t_k-} = \left(\bigvee_{\{m \in \mathbb{N}: t_m < t_k\} }\sigma\big(h^m,\widetilde{h}^m\big)\right)
\bigvee \mathcal{N}^1$.
Here $\mathcal{N}^1$ is the $\sigma-$algebra generated from the null sets under $\mathbb{P}^1$. Note that $X^{1,\natural} \in \mathcal{H}^{2,d}({\mathbb{F}}^1;\mathbb{R}^n)$ due to the fact that has finite variation and \citet[6.23 Theorem 3)]{he2019semimartingale}. 
\end{example}

In the above examples key feature was the concept of independence. We now provide an example which illustrates that independence is not necessary.

\begin{example}
Let $({\Omega}^1,{\mathcal{G}}^1,\mathbb{P}^1)$ be a probability space that supports $h^1 \in \mathbb{L}^2(\mathcal{G}^1; \mathbb{R}^p)$ and $h^2 \in \mathbb{L}^2(\mathcal{G}^1; \mathbb{R}^n)$ such that 
\begin{align}\label{B.1}
    \mathbb{E}\big[h^1\big|\sigma(h^2)\big] = 0 \hspace{0.5cm}\text{and} \hspace{0.5cm}\mathbb{E}\big[h^2\big|\sigma(h^1)\big] = 0.
\end{align}
The relation that is expressed through \eqref{B.1} is a generalization of independence, when the random variables have zero expectation. Let $t_1,t_2 \in \mathbb{R}_+$, we define
\begin{align*}
X^{1,\circ}(\omega^1,t) := \begin{cases}0, \hspace{1cm}\text{if}\hspace{0.1cm} t < t_1\\
                                  h^1(\omega^1), \hspace{0.1cm}\text{if} \hspace{0.1cm} t \geq t_1 \end{cases}
                                  \text{and} \hspace{0.4cm}
X^{1,\natural}(\omega^1,t) := \begin{cases}0, \hspace{1cm}\text{if}\hspace{0.1cm} t< t_2\\
                                  h^2(\omega^1), \hspace{0.1cm}\text{if} \hspace{0.1cm} t \geq t_2. \end{cases}                               
\end{align*}
Then we have that $\mathbf{X}^1:=(X^{1,\circ},X^{1,\natural}) \in \mathcal{H}^2({\mathbb{F}}^1;\mathbb{R}^p) \times \mathcal{H}^{2,d}({\mathbb{F}}^1;\mathbb{R}^n)$; recall that ${\mathbb{F}}^1$ is the usual augmentation of the natural filtration of the pair. The martingale property comes from \eqref{B.1}. To see that $X^{1,\natural} \in \mathcal{H}^{2,d}({\mathbb{F}}^1;\mathbb{R}^n)$ note that $X^{1,\natural}$ has finite variation and use \citet[6.23 Theorem 3)]{he2019semimartingale}. Finally, from \eqref{B.1} and \citet[Lemma 13.3.15 (\textit{ii})]{cohen2015stochastic} we have $M_{\mu^{X^{\natural,1}}}[\Delta X^{1,\circ}|\widetilde{\mathcal{P}}^{{\mathbb{F}}^1}] = 0$.
\end{example}

In view of the presented examples, we may assume a canonical space $\Omega^1$ such that $({\Omega}^1,{\mathcal{G}}^1,\mathbb{P}^1)$ is a probability space and ${\mathbb{F}}^1$ is the usual augmentation of the natural filtration of a pair $\mathbf{X}^1 := (X^{1,\circ},X^{1,\natural}) \in \mathcal{H}^2({\mathbb{F}}^1;\mathbb{R}^p) \times \mathcal{H}^{2,d}({\mathbb{F}}^1;\mathbb{R}^n)$ (defined on the canonical space $\Omega^1$)
, with $M_{\mu^{X^{1,\natural}}}[\Delta X^{1,\circ}|\widetilde{\mathcal{P}}^{{\mathbb{F}}^1}] = 0$, where $\mu^{X^{1,\natural}}$ is the random measure generated by the jumps of $X^{1,\natural}$. Additionally, let a random variable $\xi^1 \in \mathbb{L}^2_{\hat{\beta}}({\mathcal{F}}^1_T;\mathbb{R}^d)$ for a deterministic time $T$, which will be assumed fixed from now on.

Then, let $\left\{({\Omega}^i,{\mathcal{G}}^i,{\mathbb{F}}^i,\mathbb{P}^i)\right\}_{i \in \mathbb{N}}$ be copies of the stochastic basis $({\Omega}^1,{\mathcal{G}}^1,{\mathbb{F}}^1,\mathbb{P}^1)$, $\{\mathbf{X}^i := (X^{i,\circ},X^{i,\natural})\}_{i \in \mathbb{N}}$ the corresponding copies of $(X^{1,\circ},X^{1,\natural})$ and $\{\xi^i\}_{i \in \mathbb{N}}$ the corresponding copies of $\xi^1$.
We define the product probability space $\big(\prod_{i = 1}^{\infty} {\Omega}^i, 
\bigotimes_{i = 1}^{\infty}{\mathcal{G}}^i,
\bigotimes_{i = 1}^{\infty} \mathbb{P}^i\big)$. 
We denote by $\widehat{\mathbb{F}}^i$ the augmented natural filtration of the pair $\mathbf{X}^i$ in the product space $\prod_{i = 1}^{\infty} {\Omega}^i$ under the probability measure $\mathbb{P} := \bigotimes_{i = 1}^{\infty} \mathbb{P}^i$.  
Because the pair $\mathbf{X}^i$ depends only on ${\omega}^i$ we have
\begin{align}\label{A.3.}
 \widehat{\mathbb{F}}^i 
 = \Big(\mathbb{F}^i \times \prod_{m \in \mathbb{N}\setminus \{i\}}^{\infty} \Omega^m\Big) \bigvee \mathcal{N},
\end{align}
where $\mathcal{N}$ is the $\sigma$--algebra generated from the subsets of the null sets under the measure $\mathbb{P}$.
 
Using the methods of \cref{cor:same_stoch_integ_wrtIVRM} we get that $\mathbf{X}^i \in \mathcal{H}^2(\widehat{\mathbb{F}}^i;\mathbb{R}^p) \times \mathcal{H}^{2,d}(\widehat{\mathbb{F}}^i;\mathbb{R}^n)$. 
To prove the identity $M_{\mu^{X^{i,\natural}}}[\Delta X^{i,\circ}|\widetilde{\mathcal{P}}^{\widehat{\mathbb{F}}^i}]= 0$ we work as in the end of the proof of \cref{lem:conserv_laws}. 
Let $\{\tau_k\}_{k \in \mathbb{N}}$ be a sequence of disjoint $\mathbb{F}^i-$stopping times that exhausts the jumps of $X^{i,\natural}$ and also satisfies the assumptions of \cite[Lemma 13.3.15 (\textit{ii})]{cohen2015stochastic}; it is known that such a sequence always exists for every $\mathbb{F}^i-$adapted, c\`adl\`ag process, \emph{e.g.}, see \cite[Definition I.1.30, Proposition I.1.32]{jacod2013limit}.
Of course, the aforementioned stopping times when viewed in the product space depend only on $\omega^i$.
Moreover, $X^{i,\circ}$ is also an $\mathbb{F}^i-$martingale. 
Hence, $\Delta X^{i,\circ}_{\tau_k}$ will be measurable with respect to $\mathcal{F}^i_{\infty}$, for every $k \in \mathbb{N}$. If we denote by $\widehat{\mathcal{F}}^{i}_{\tau_k -}$ the $\sigma-$algebra of events occurring strictly before the stopping time $\tau_k$ produced under the filtration $\widehat{\mathbb{F}}^{i}$ and with $\mathcal{F}^{i}_{\tau_k -}$ the respective $\sigma-$algebra under the filtration $\mathbb{F}^i$,
then from \eqref{A.3.} we have 
\begin{gather*}
    \widehat{\mathcal{F}}^{i}_{\tau_k -} = \Big(\mathcal{F}^{i}_{\tau_k -}
    \times \prod_{m \in \mathbb{N}\setminus \{i\}}^{\infty}
    \Omega^m \Big)\bigvee \mathcal{N}
    \shortintertext{and}
    \widehat{\sigma\big(\Delta X^{i,\natural}_{\tau_k}\big)} = \sigma(\Delta X^{i,\natural}_{\tau_k}) \times \prod_{m \in \mathbb{N}\setminus \{i\}}^{\infty} \Omega^m .
\end{gather*}
Then, because $\mathcal{N}$ is independent from any other sub $\sigma$--algebra of $\overline{\bigotimes_{i = 1}^{\infty}{\mathcal{G}}^i}$, where we denoted by $\overline{\bigotimes_{i = 1}^{\infty}{\mathcal{G}}^i}$ the completion under the measure $
\mathbb{P}$, we get
\begin{align*}
    &\mathbb{E}^{\mathbb{P}}
    \Big[\Delta X^{i,\circ}_{\tau_k}\Big|
    \widehat{\mathcal{F}}^{i}_{\tau_k -} 
    \bigvee \widehat{\sigma\big(\Delta X^{i,\natural}_{\tau_k}\big)} \Big]\\
    &= \mathbb{E}^{\mathbb{P}}
    \Big[\Delta X^{i,\circ}_{\tau_k}\Big|\Big(\big(\mathcal{F}^{i}_{\tau_k -} \bigvee \sigma(\Delta X^{i,\natural}_{\tau_k})\big) \times \prod_{m \in \mathbb{N}\setminus \{i\}}^{\infty} \Omega^m \Big) \bigvee \mathcal{N}\Big]\\
    &= \mathbb{E}^{
    \mathbb{P}}\left[\Delta X^{i,\circ}_{\tau_k}\Bigg|\left(\mathcal{F}^{i}_{\tau_k -} \bigvee \sigma(\Delta X^{i,\natural}_{\tau_k})\right) \times \prod_{m \in \mathbb{N}\setminus \{i\}}^{\infty} \Omega^m \right]\\
    &= \mathbb{E}^{\mathbb{P}^i}\left[\Delta X^{i,\circ}_{\tau_k}\Bigg|\mathcal{F}^{i}_{\tau_k -} \bigvee \sigma(\Delta X^{i,\natural}_{\tau_k}) \right] (\omega^i)= 0,
\end{align*}
where we used \cite[Section 9.7, Property (k) on p. 88]{williams1991probability} in the second equality and
\cite[Lemma 13.3.15 (\textit{ii})]{cohen2015stochastic} in the last equality.
 
Moreover, note from \eqref{A.3.} that the sequence $\{\widehat{\mathbb{F}}^i\}_{i\in\mathbb{N}}$ consists of independent filtrations of the probability space $\big(\prod_{i = 1}^{\infty} {\Omega}^i,
\overline{\bigotimes_{i = 1}^{\infty}{\mathcal{G}}^i},
\mathbb{P}\big)$, where we abused notation and denoted the extended measure again with $\mathbb{P}$.
 
Next, for every $i \in \mathbb{N}$ define the bimeasurable bijections $g^i :\big(\prod_{i = 1}^{\infty} {\Omega}^i,\overline{\bigotimes_{i = 1}^{\infty}
\mathcal{G}^i}\big) {}\longrightarrow \big(\prod_{i = 1}^{\infty} {\Omega}^i,
\overline{\bigotimes_{i = 1}^{\infty}{\mathcal{G}}^i}\big)$ by 
\begin{align*}
    g^i(\omega^1,
    \omega^2, \dots,
    \omega^{i-1},
    \omega^i,
    \omega^{i+1},\dots)
    := (\omega^i,
    \omega^2, \dots,
    \omega^{i-1},
    \omega^1,
    \omega^{i+1},\dots),
\end{align*}
\emph{i.e.}, the function $g^i$ switches the places of ${\omega}^i$ and $\omega^1$. 
It is easy to check the following properties of the sequence $\{g^i\}_{i \in \mathbb{N}}$:
\begin{enumerate}
    \item For every $i \in \mathbb{N}$, we have $g^i \circ g^i = \textrm{Id}_{\prod_{i = 1}^{\infty} \Omega^i}$,
    where $\textrm{Id}_{\prod_{i = 1}^{\infty} \Omega^i}$ is the identity function.
    
    \item For every $i \in \mathbb{N}$ and for every $t \in \mathbb{R}_+,$ 
    we have $ g^i\big(\widehat{\mathcal{F}}^1_t\big)^{-1} = \widehat{\mathcal{F}}^i_t$,
  
    \item\label{iv} For every $i \in \mathbb{N}$ and for every $A \in 
    \overline{\bigotimes_{i = 1}^{\infty}{\mathcal{G}}^i}$, 
    we have $\mathbb{P}(A) =
    \mathbb{P}(g^i(A)^{-1})$.

    \item  For every $i \in \mathbb{N}$, we have $\mathbf{X}^i := \mathbf{X}^1\circ (g^i,\textrm{Id}_{\mathbb{R}_+})$ and $\xi^i := \xi^1 \circ g^i$.
\end{enumerate}
Now, for every $i \in \mathbb{N}$ we have that
\begin{align*}
    \mathcal{P}^{\widehat{\mathbb{F}}^i} = \sigma\left(\left\{A_t \times (t,\infty) : t \in \mathbb{R}_+, A_t \in \widehat{\mathcal{F}}^i_t\right\} \cup \left\{A_0 \times \{0\} : A_0 \in \widehat{\mathcal{F}}^i_0 \right\}\right).
\end{align*}
Thus
, from \eqref{A.3.} and \ref{iv} above, we have that 
\begin{align*}
   Z \in \mathcal{P}^{\widehat{\mathbb{F}}^i} &\Longleftrightarrow  Z \circ (g^i,\textrm{Id}_{\mathbb{R}_+}) \in \mathcal{P}^{\widehat{\mathbb{F}}^1}
   \shortintertext{and}
    M \in \mathcal{H}^2(\widehat{\mathbb{F}}^i;\mathbb{R}^d)
    &\Longleftrightarrow M \circ (g^i,\textrm{Id}_{\mathbb{R}_+}) \in \mathcal{H}^2(\widehat{\mathbb{F}}^1;\mathbb{R}^d).
\end{align*}
From the above properties one can show that
\begin{align*}
   \langle X^{i,\circ} \rangle ^{\widehat{\mathbb{F}}^i} = \langle X^{1,\circ} \rangle ^{\widehat{\mathbb{F}}^1} \circ (g^i,\textrm{Id}_{\mathbb{R}_+}) \hspace{0.4cm} \text{and} \hspace{0.4cm} |I|^2 * \nu^{(\widehat{\mathbb{F}}^i,X^{i,\natural})} = |I|^2 * \nu^{(\widehat{\mathbb{F}}^1,X^{1,\natural})} \circ (g^i,\textrm{Id}_{\mathbb{R}_+}).
\end{align*}
Hence, from \eqref{def_C} we have $C^{(\widehat{\mathbb{F}}^i,\mathbf{X}^i)} = C^{(\widehat{\mathbb{F}}^1,\mathbf{X}^1)}\circ (g^i,\textrm{Id}_{\mathbb{R}_+})$ and 
$c^{(\widehat{\mathbb{F}}^i,\mathbf{X}^i)} = c^{(\widehat{\mathbb{F}}^1,\mathbf{X}^1)}\circ (g^i,\textrm{Id}_{\mathbb{R}_+})$.
    
Finally, from the existence and uniqueness theorems of \cref{MV-BSDE}, see also \cref{remark 2.3}, and the fact that $\mathcal{E}\big(\hat{\beta} A^{i}\big)_T = \mathcal{E}\big(\hat{\beta} A^{1}\big)\circ (g^i,T)$, due to symmetry, we have that for all $i \in \mathbb{N}$
\begin{align*}
  (Y^i,Z^i,U^i(\cdot),M^i) = (Y^1,Z^1,U^1(\cdot),M^1) \circ (g^i,\textrm{Id}_{\mathbb{R}_+}).
\end{align*}

\subsection{Auxiliary results}\label{subsec_app:tehnical_lemmata}

In this subsection, we will present some useful technical lemmata and their proofs.
Let $\mathbb{G}$ and $\mathbb{H}$ be filtrations on the probability space $(\Omega,\mathcal{G}^{\circ},\mathbb{P})$ such that $\mathbb{G}$ is immersed in $\mathbb{H}$ and both satisfy the usual conditions. 
Special cases of the results presented below appear in \citet{di2022propagation}. 
Although these are sufficient for our purposes, we present here the more general results for completeness.

\begin{lemma}\label{lem:equalities_integrals}
Let $U \in \widetilde{\mathcal{P}}^{\mathbb{G}}_+$ and $g \in \left(\mathcal{G} \otimes \mathcal{B}(\mathbb{R}_+)\right)_+$.
Consider $C^{(\mathbb{H},\mathbf{X})}$ as defined in \eqref{def_C} and $K^{(\mathbb{G},\mathbf{X})}$, $K^{(\mathbb{H},\mathbf{X})}$ as defined in \eqref{def:Kernels}, 
for every pair $\mathbf{X}:=(X^\circ,X^\natural)\in \mathcal{H}^2(\mathbb{G};\mathbb{R}^p)\times\mathcal{H}^{2}(\mathbb{G};\mathbb{R}^n)$. 
Then, we have that 
\begin{enumerate}[label=(\roman*)]
    \item \label{eq_int_1} for $\mathbb{P} \otimes C^{(\mathbb{H},\mathbf{X})}-a.e.$ $(\omega,t)\in\Omega\times \mathbb{R}_+$
\begin{align*}
    \int_{\mathbb{R}^n} U(\omega,t,x)\hspace{0.1cm}K^{(\mathbb{G},\mathbf{X})}(\omega,t,\ud x) =  \int_{\mathbb{R}^n} U(\omega,t,x)\hspace{0.1cm}K^{(\mathbb{H},\mathbf{X})}(\omega,t,\ud x).
\end{align*}

\item \label{eq_int_2}  
For $\mathbb{P}-a.e.$ 
\begin{multline*}
    \int_{\mathbb{R}_+ \times \mathbb{R}^n} g(\omega,t)\hspace{0.1cm}U(\omega, t,x)\hspace{0.1cm}\nu^{(\mathbb{G},X^{\natural})}(\omega,\ud t,\ud x) 
   \\ =  \int_{\mathbb{R}_+ \times \mathbb{R}^n} g(\omega,t)\hspace{0.1cm}U(\omega,t,x)\hspace{0.1cm}\nu^{(\mathbb{H},X^{\natural})}(\omega,\ud t,\ud x).
\end{multline*}

\item\label{eq_int_3}
    $U * \nu^{(\mathbb{G},X^{\natural})} =  U * \nu^{(\mathbb{H},X^{\natural})}$, up to evanescence.

\item\label{eq_int_4} 
$    \widehat{U}^{(\mathbb{G},X^{\natural})} = \widehat{U}^{(\mathbb{H},X^{\natural})}$, up to evanescence; see \eqref{hat_U} for their definition.

\item\label{eq_int_5} 
$    \zeta^{(\mathbb{G},X^{\natural})} = \zeta^{(\mathbb{H},X^{\natural})}$,
up to evanescence;  see \eqref{hat_zeta} for their definition.

\end{enumerate}
\end{lemma}
\begin{proof} 
Let us fix $\mathbf{X}:=(X^\circ,X^\natural)\in \mathcal{H}^2(\mathbb{G};\mathbb{R}^p)\times\mathcal{H}^{2}(\mathbb{G};\mathbb{R}^n)$. 
We remind the reader that the immersion of filtrations implies that $C^{(\mathbb{G},\mathbf{X})} = C^{(\mathbb{H},\mathbf{X})}$; see \cref{rem:immersion_no_change_in_comp}.
Therefore, we may simplify the notation and simply write $C$ for 
$C^{(\mathbb{H},\mathbf{X})}$.
We proceed to prove our claims:
\begin{enumerate}
    \item[{\ref{eq_int_1}}] 
    We consider $\{B_m\}_{m \in \mathbb{N}}\subseteq \widetilde{\mathcal{P}}^{\mathbb{G}}$ to be a partition of $\Omega \times \mathbb{R}_+ \times \mathbb{R}^n$ that makes $M_{\mu^{X^{\natural}}}$ $\sigma-$integrable with respect to $\widetilde{\mathcal{P}}^{\mathbb{G}}$.
    Then, we define the sequence $\{A_m\}_{m \in \mathbb{N}}$ as $A_m := \left(\bigcup_{k = 1}^m B_k\right) \hspace{0.05cm} \cap \hspace{0.05cm}\{|U| \leq m\}$; 
    of course $\mathds{1}_{A_m} \nearrow 1$ for every $(\omega,t,x)$.
For every $m \in \mathbb{N}$, we have that $\left(U\mathds{1}_{A_m}\right) * \mu^{X^{\natural}} - \left(U\mathds{1}_{A_m}\right) * \nu^{(\mathbb{G},X^{\natural})}$ is a $\mathbb{G}-$martingale of  finite variation,
while $\left(U\mathds{1}_{A_m}\right) * \mu^{X^{\natural}} - \left(U\mathds{1}_{A_m}\right) * \nu^{(\mathbb{H},X^{\natural})}$ is an $\mathbb{H}-$martingale of  finite variation. 
    In view of the immersion property, \emph{i.e.}, every $\mathbb{G}-$martingale is also an $\mathbb{H}-$martingale, 
    $\left(U\mathds{1}_{A_m}\right) * \nu^{(\mathbb{G},X^{\natural})} - \left(U\mathds{1}_{A_m}\right) * \nu^{(\mathbb{H},X^{\natural})}$ is a predictable, $\mathbb{H}-$martingale of  finite variation starting at $0$.
In other words, it is 0 up to indistinguishability for every $m\in\mathbb{N}$, which equivalently reads
\begin{align}
    \left(U\mathds{1}_{A_m}\right) * \nu^{(\mathbb{G},X^{\natural})} = \left(U\mathds{1}_{A_m}\right) * \nu^{(\mathbb{H},X^{\natural})}
    \label{eq:pred_comp_for_U_and_m}
\end{align}
up to indistinguishability for every $m\in\mathbb{N}$.
By \eqref{def:Kernels} \begin{align*}
&\int_{\mathbb{R}_+}\mathds{1}_{[0,s]}(t)
\int_{\mathbb{R}^n} U(\omega,t,x)\mathds{1}_{A_m}(\omega,t,x)\hspace{0.1cm}K^{(\mathbb{G},\mathbf{X})}(\omega,t,\ud x)\ud C_t\\
&\hspace{1em}=
\int_{\mathbb{R}_+}\mathds{1}_{[0,s]}(t)
\int_{\mathbb{R}^n} U(\omega,t,x)\mathds{1}_{A_m}(\omega,t,x)\hspace{0.1cm}K^{(\mathbb{H},\mathbf{X})}(\omega,t,\ud x)\ud C_t,
\end{align*}
up to evanescence,
for every $s \in \mathbb{Q}_+$ and $m\in\mathbb{N}$.
Recalling that $\{[0,s]\}_{s \in \mathbb{Q}_+}$ is a $\pi$-system whose $\lambda-$system produces $\mathcal{B}(\mathbb{R}_+)$, by an application of Dynkin's lemma we can replace $[0,s]$ in the above equality with any set $D \in \mathcal{B}(\mathbb{R}_+)$. 
Using the monotone convergence theorem, with respect to the sequence $\{\mathds{1}_{A_m}\}_{m\in\mathbb{N}}$, we get the desired result.
\item[{\ref{eq_int_2}}]
    Immediate from {\ref{eq_int_1}} and the disintegration formula \eqref{def:Kernels}.
\item[{\ref{eq_int_3}}]
Immediate from \eqref{eq:pred_comp_for_U_and_m} by means of monotone convergence.
\item[{\ref{eq_int_4}}]
 Immediate from \eqref{eq:pred_comp_for_U_and_m} because from \citet[5.27 Theorem, 2) and 11.11 Theorem]{he2019semimartingale} we have by monotone convergence
\begin{align*}
     \widehat{U}^{(\mathbb{G},X^{\natural})} &= \lim_{m \rightarrow \infty} \Delta\left((U\mathds{1}_{A_m}) * \nu^{(\mathbb{G},X^{\natural})}\right)\\
     \shortintertext{and}
     \widehat{U}^{(\mathbb{H},X^{\natural})} &= \lim_{m \rightarrow \infty} \Delta\left((U\mathds{1}_{A_m}) * \nu^{(\mathbb{H},X^{\natural})}\right).
 \end{align*}
 
\item[{\ref{eq_int_5}}]
This is a direct consequence of \ref{eq_int_4} for $U = 1$.
\end{enumerate} 
\end{proof}

\begin{corollary}\label{cor:same_stoch_integ_wrtIVRM}
    Let
    $Z \in \mathbb{H}^{2}(\mathbb{G},X^{\circ};\mathbb{R}^{d \times p})$, then 
    $(Z\cdot X^{\circ})^{\mathbb{G}} =(Z\cdot X^{\circ})^{\mathbb{H}}$, up to evanescence.
    Moreover, let
    $U \in G_2(\mathbb{G},\mu^{X^{\natural}})$, 
    then $U\star\widetilde{\mu}^{(\mathbb{G},X^{\natural})} = U\star\widetilde{\mu}^{(\mathbb{H},X^{\natural})}$, up to evanescence.
    In particular, for $\mathbb{R}^n \ni x \overset{\textup{Id}}{\longmapsto} x\in \mathbb{R}^n$ we have that  $X^{\natural} = \textup{Id}\star  \widetilde{\mu}^{(\mathbb{G},X^{\natural})} 
= \textup{Id}\star \widetilde{\mu}^{(\mathbb{H},X^{\natural})}$.
\end{corollary}

\begin{proof}
The claim is immediate for the It\=o stochastic integrals from their definition, see \citet[Definition I.2.1]{jacod2013limit}, and the fact that $C^{(\mathbb{G},\mathbf{X})} = C^{(\mathbb{H},\mathbf{X})}$. 

As for the stochastic integrals with respect to the integer-valued measure $\mu^{X^{\natural}}$, let $U \in G_2(\mathbb{G},\mu^{X^{\natural}})$.
Then, we have
\begin{align*}
    \|U\star\widetilde{\mu}^{(\mathbb{G},X^{\natural})}\|^2_{\mathcal{H}^2(\mathbb{G};\mathbb{R}^d)} &= \mathbb{E}\Big[\Big|U\star\widetilde{\mu}^{(\mathbb{G},X^{\natural})}\Big|^2_{\infty}\Big] \\
    & = \|U\star\widetilde{\mu}^{(\mathbb{G},X^{\natural})}\|^2_{\mathcal{H}^2(\mathbb{H};\mathbb{R}^d)}\\
    &\overset{\eqref{SqIntMartDecomp}}= \left\|\left(U\star\widetilde{\mu}^{(\mathbb{G},X^{\natural})}\right)^{(\mathbb{H},c)}\right\|^2_{\mathcal{H}^2(\mathbb{H};\mathbb{R}^d)} + \left\|\left(U\star\widetilde{\mu}^{(\mathbb{G},X^{\natural})}\right)^{(\mathbb{H},d)}\right\|^2_{\mathcal{H}^2(\mathbb{H};\mathbb{R}^d)}.
\end{align*}
Note that we denoted with
\begin{align*}
\left(\left(U\star\widetilde{\mu}^{(\mathbb{G},X^{\natural})}\right)^{(\mathbb{H},c)},\left(U\star\widetilde{\mu}^{(\mathbb{G},X^{\natural})}\right)^{(\mathbb{H},d)}\right)
\end{align*}
the unique pair in $\mathcal{H}^{2,c}(\mathbb{H};\mathbb{R}^d) \times \mathcal{H}^{2,d}(\mathbb{H};\mathbb{R}^d)$ such that 
\begin{align*}
U\star\widetilde{\mu}^{(\mathbb{G},X^{\natural})} = \left(U\star\widetilde{\mu}^{(\mathbb{G},X^{\natural})}\right)^{(\mathbb{H},c)} + \left(U\star\widetilde{\mu}^{(\mathbb{G},X^{\natural})}\right)^{(\mathbb{H},d)}.
\end{align*} 
Using \citet[6.23 Theorem]{he2019semimartingale} we have that
\begin{align*}
\|U\star\widetilde{\mu}^{(\mathbb{G},X^{\natural})}\|^2_{\mathcal{H}^2(\mathbb{G};\mathbb{R}^d)} = \mathbb{E}\left[\sum_{t > 0}\left|\Delta \left(U\star\widetilde{\mu}^{(\mathbb{G},X^{\natural})}\right)_t\right|^2\right] = \left\|\left(U\star\widetilde{\mu}^{(\mathbb{G},X^{\natural})}\right)^{(\mathbb{H},d)}\right\|^2_{\mathcal{H}^2(\mathbb{H};\mathbb{R}^d)}.
\end{align*}
Hence, 
\begin{align*}
\left\|\left(U\star\widetilde{\mu}^{(\mathbb{G},X^{\natural})}\right)^{(\mathbb{H},c)}\right\|^2_{\mathcal{H}^2(\mathbb{H};\mathbb{R}^d)} = 0,
\end{align*}
and $U\star\widetilde{\mu}^{(\mathbb{G},X^{\natural})} \in \mathcal{H}^{2,d}(\mathbb{H};\mathbb{R}^d)$.

Next, from 
\cref{lem:equalities_integrals}.{\ref{eq_int_4}}
\begin{align*}
\Delta\big(U\star\widetilde{\mu}^{(\mathbb{G},X^{\natural})}\big)_t 
&= U(\omega,t,\Delta X^{\natural}_t)\mathds{1}_{\{\Delta X^{\natural}\neq 0\}} - \widehat{U}_t^{(\mathbb{G},X^{\natural})}\\
&= U(\omega,t,\Delta X^{\natural}_t)\mathds{1}_{\{\Delta X^{\natural}\neq 0\}} - \widehat{U}_t^{(\mathbb{H},X^{\natural})}\\
&=\Delta\big(U\star\widetilde{\mu}^{(\mathbb{H},X^{\natural})}\big)_t.
\end{align*} 
Hence, from the above equality, because $U\star\widetilde{\mu}^{(\mathbb{G},X^{\natural})} - U\star\widetilde{\mu}^{(\mathbb{H},X^{\natural})} \in \mathcal{H}^{2,d}(\mathbb{H};\mathbb{R}^d)$, using again \citet[6.23 Theorem]{he2019semimartingale} we conclude that
\begin{align*}
\|U\star\widetilde{\mu}^{(\mathbb{G},X^{\natural})} - U\star\widetilde{\mu}^{(\mathbb{H},X^{\natural})}\|^2_{\mathcal{H}^2(\mathbb{H};\mathbb{R}^d)} =  \mathbb{E}\left[\sum_{t > 0}\left|\Delta \left(U\star\widetilde{\mu}^{(\mathbb{G},X^{\natural})} - U\star\widetilde{\mu}^{(\mathbb{H},X^{\natural})}\right)_t\right|^2\right] = 0,
\end{align*}
thus $U\star\widetilde{\mu}^{(\mathbb{G},X^{\natural})} = U\star\widetilde{\mu}^{(\mathbb{H},X^{\natural})}$, up to indistinguishability.

Finally, for $U = \textup{Id}$ we get the last claim.
\end{proof}

\subsection{Conservation of solutions under immersion of filtrations}
\label{subsec:Conservation_of_solutions}

In this subsection we will identify the solutions of the McKean--Vlasov BSDEs \eqref{MVBSDE_initial_path_triplet}, \eqref{MVBSDE_initial_path_onlyY}, \eqref{MVBSDE_instantaneous_b_triplet}, \eqref{MVBSDE_instantaneous_b_onlyY} when we fix all the elements of the data except for the filtrations.
\cref{lem:conserv_laws} makes precise the previous sentence.

Let us remind the reader of the necessary notation and terminology.
Let us fix $N \in \mathbb{N}$ and assume \ref{PC_martingale_integrators}--\ref{PC_Integrability_Generator}, or,  \ref{PC_martingale_integrators}--\ref{PC_alpha_A} and \ref{PC_Theta_Gamma}--\ref{PC_Integrability_Generator}, respectively for each case as in \cref{sec:propagation}.
For each $i\in\mathscr{N}$, the McKean--Vlasov BSDEs \eqref{MVBSDE_initial_path_triplet}, \eqref{MVBSDE_initial_path_onlyY}, \eqref{MVBSDE_instantaneous_b_triplet}, \eqref{MVBSDE_instantaneous_b_onlyY} associated to the data
$\big(\mathbf{X}^i,\mathbb{F}^{i},\Theta,\Gamma,T,\xi^i,f\big)$ under $\hat{\beta}$ admit, by Theorems \ref{thm:prop_chaos_seg}--\ref{thm:prop_chaos_inst_onlyY}, unique solutions, which are denoted by $(Y^i,Z^i,U^i,M^i)$.

\begin{remark}\label{Remark_for_Gamma}
Let $i \in \mathscr{N}$. 
Under \ref{PC_martingale_integrators}, \ref{PC_Theta_Gamma} and
 for $U \in \mathbb{H}^{2}(\mathbb{F}^i,X^{i,\natural};\mathbb{R}^d)$ 
 we have from \cref{lem:equalities_integrals}  that
\begin{align*}
    \Gamma^{(\mathbb{F}^i,\mathbf{X}^i,\Theta)}(U) = \Gamma^{(\mathbb{F}^{\mathscr{N}},\mathbf{X}^i,\Theta^i)}(U), \hspace{0.5cm} \mathbb{P} \otimes C
    -\text{a.e.}
\end{align*}
\end{remark}

\begin{lemma}[\textbf{Conservation of solutions}]\label{lem:conserv_laws}
Assume for each case the neccesary assumptions from \cref{sec:propagation} and fix $N\in\mathbb{N}$ and $i \in \mathscr{N}$. 
The unique solutions of the McKean--Vlasov BSDEs \eqref{MVBSDE_initial_path_triplet}, \eqref{MVBSDE_initial_path_onlyY}, \eqref{MVBSDE_instantaneous_b_triplet}, \eqref{MVBSDE_instantaneous_b_onlyY} associated to the data $\big(\mathbf{X}^i,\mathbb{F}^i,\Theta,\Gamma,T,\xi^i,f \big)$ under $\hat{\beta}$, are also the unique solutions of the McKean--Vlasov BSDEs \eqref{MVBSDE_initial_path_triplet}, \eqref{MVBSDE_initial_path_onlyY}, \eqref{MVBSDE_instantaneous_b_triplet}, \eqref{MVBSDE_instantaneous_b_onlyY} associated to the data
$\big(\mathbf{X}^i,\mathbb{F}^{\mathscr{N}},\Theta,\Gamma,T,\xi^i,f\big)$ under $\hat{\beta}$.
\end{lemma}

\begin{proof}
Let us fix $N\in\mathbb{N}$ and $i\in\mathscr{N}$. 
We are going to prove only the case for \eqref{MVBSDE_initial_path_onlyY}; the others follow analogously.
We denote by $(Y^{i},Z^{i},U^{i},M^{i})$ the solution of the McKean--Vlasov BSDE \eqref{MVBSDE_initial_path_onlyY} associated to the standard data 
$\big(\mathbf{X}^i,\mathbb{F}^i,\Theta,\Gamma,T,\xi^i,f \big)$ under $\hat{\beta}$.
From \cref{lem:equalities_integrals},  \cref{cor:same_stoch_integ_wrtIVRM},
\cref{Remark_for_Gamma} and  \cref{thm:prop_chaos_seg_onlyY} we deduce that it will be enough to show that $M^i \in \mathcal{H}^2({\mathbf{X}^i}^{\perp_{\mathbb{F}^{\mathscr{N}}}})$. From \cref{prop:CharacterOrthogSpace} we will need to check that 
\begin{align}
\label{aux:requirements}
    \langle X^{i,\circ}, M^i\rangle^{\mathbb{F}^{\mathscr{N}}} =0 \quad \text{ and } \quad M_{\mu^{X^{i,\natural}}}[\Delta M^i|\widetilde{\mathcal{P}}^{\mathbb{F}^{\mathscr{N}}}]=0.
\end{align} 
We remind the reader that $M^i \in \mathcal{H}^2({\mathbf{X}^i}^{\perp_{\mathbb{F}^i}})$, \emph{i.e.}, $\langle X^{i,\circ}, M^i\rangle^{\mathbb{F}^i} =0$ and $M_{\mu^{X^{i,\natural}}}[\Delta M^i|\widetilde{\mathcal{P}}^{\mathbb{F}^i}]=0$.

For the first equation in \eqref{aux:requirements}, we have that $X^{i,\circ} M^i$ remains an $\mathbb{F}^{\mathscr{N}}-$martingale, since $\mathbb{F}^i$ is immersed in $\mathbb{F}^{\mathscr{N}}$. 
Hence, $\langle X^{i,\circ}, M^i\rangle^{\mathbb{F}^{\mathscr{N}}}=0.$

For the second equation in \eqref{aux:requirements}, we are going to use \citet[Lemma 13.3.15 (\textit{ii})]{cohen2015stochastic}. 
The martingale $X^{i,\natural}$ is adapted to the filtration $\mathbb{F}^i$. 
Let $\{\tau_k\}_{k \in \mathbb{N}}$ be a sequence of disjoint $\mathbb{F}^i-$stopping times that exhausts the jumps of $X^{i,\natural}$ and also satisfies the assumptions of \cite[Lemma 13.3.15 (\textit{ii})]{cohen2015stochastic}; it is known that such a sequence always exists for every $\mathbb{F}^i-$adapted, c\`adl\`ag process, see \emph{e.g.} \citet[Definition I.1.30, Proposition I.1.32]{jacod2013limit}.
Moreover, $M^i$ is also an $\mathbb{F}^i-$martingale. 
Hence, $\Delta M^{i}_{\tau_k}$ will be measurable with respect to $\mathcal{F}^i_{\infty}$, for every $k \in \mathbb{N}$. If we denote by $\mathcal{F}^{\mathscr{N}}_{\tau_k -}$ the $\sigma-$algebra of events occuring strictly before the stopping time $\tau_k$ produced under the filtration $\mathbb{F}^{\mathscr{N}}$ and with $\mathcal{F}^{i}_{\tau_k -}$ the respective $\sigma-$algebra under the filtration $\mathbb{F}^i$, then we have 
\begin{align*}
    \mathcal{F}^{\mathscr{N}}_{\tau_k -} \subseteq \mathcal{F}^{i}_{\tau_k -}\bigvee \Big(\bigvee_{m \in \mathscr{N}\setminus \{i\}}\mathcal{F}^m_{\infty}\Big) \quad \text{ and} \quad
    \sigma(\Delta X^{i,\natural}_{\tau_k}) \subseteq \mathcal{F}^i_{\infty}.
\end{align*}
Finally, we get
\begin{gather*}
    \mathbb{E}\big[\Delta M^i_{\tau_k}\big|\mathcal{F}^{\mathscr{N}}_{\tau_k -} \bigvee \sigma(\Delta X^{i,\natural}_{\tau_k})\big]\\
    = \mathbb{E}\Big[\mathbb{E}\big[\Delta M^i_{\tau_k}\big|\mathcal{F}^{i}_{\tau_k -}\bigvee \big(\bigvee_{m \in \mathscr{N}\setminus \{i\}}\mathcal{F}^m_{\infty}\big) \bigvee \sigma(\Delta X^{i,\natural}_{\tau_k})\big]\Big|\mathcal{F}^{\mathscr{N}}_{\tau_k -} \bigvee \sigma(\Delta X^{i,\natural}_{\tau_k})\Big]\\
    = \mathbb{E}\Big[ 
    \mathbb{E} \big[ \Delta M^i_{\tau_k} \big| \mathcal{F}^i_{\tau_{k-}}\bigvee \sigma(\Delta X^{i,\natural}_{\tau_{k}})\big]
    \Big|\mathcal{F}^{\mathscr{N}}_{\tau_k -} \bigvee \sigma(\Delta X^{i,\natural}_{\tau_k})\Big]\\
    = \mathbb{E}\Big[ 
    M_{\mu^{X^{i,\natural}}}\big[\Delta M^i \big|\widetilde{\mathcal{P}}^{\mathbb{F}^i}\big](\tau_k,\Delta X^{i,\natural}_{\tau_k}) \Big|\mathcal{F}^{\mathscr{N}}_{\tau_k -} \bigvee \sigma(\Delta X^{i,\natural}_{\tau_k})\Big]\\
    = 0,
\end{gather*}
where we used the tower property for the first equality, 
\citet[Section 9.7, Property (k) on p. 88]{williams1991probability} for the second one, \cite[Lemma 13.3.15 (\textit{ii})]{cohen2015stochastic} in the second to last equality, and concluded in view of the known information $M_{\mu^{X^{i,\natural}}}[\Delta M^i |\widetilde{\mathcal{P}}^{\mathbb{F}^i}]=0$.
\end{proof}
\begin{lemma}\label{Lemma_5.4}
Let $i\in\mathscr{N}$. The process 
$\mathcal{W}_{2,\rho_{\mathbb{D}^d \times \mathbb{R}^{d \times p} \times \mathbb{R}^d}}^2\big(
L^N\big({\widetilde{\mathbb{S}}|_{[0,\cdot]}^{\text{mod},N}}\big),
\mathcal{L}(\widetilde{\mathbf{S}}|_{[0,\cdot]}^{\text{mod},i})\big)$ is optional with respect to the filtration $\mathbb{F}^{\mathscr{N}}$.
\end{lemma}
\begin{proof}
Set $\mathbb{X}:= \mathbb{D}^d \times \mathbb{R}^{d \times p} \times \mathbb{R}^d$ , from \cref{rem:Wasserstein_measurable} the Wasserstein distance of order 2 as a function $\mathcal{W}_{2,\rho_{\mathbb{X}}} : \mathscr{P}(\mathbb{X}) \times \mathscr{P}(\mathbb{X}) {}\rightarrow \mathbb{R}_+$ is measurable with respect to $\mathcal{B}_{\mathcal{T}}(\mathscr{P}(\mathbb{X}))$. Next, from \ref{remark 4.5(i)} of \cref{rem:about_proof_thm_MVBSDE_initial_path} the process $\mathcal{L}(\widetilde{\mathbf{S}}|_{[0,\cdot]}^{\text{mod},i})$ is predictable. Finally, for the random measure $L^N\big({\widetilde{\mathbb{S}}|_{[0,\cdot]}^{\text{mod},N}}\big)$ we have that for every $f \in C_b(\mathbb{X})$ the process
 \begin{align*}
L^N\big({\widetilde{\mathbb{S}}|_{[0,\cdot]}^{\text{mod},N}}\big) = \frac{1}{N}\sum_{m = 1}^{N} f(\widetilde{\mathbf{S}}|_{[0,\cdot]}^{\text{mod},m})
\end{align*}
is optional with respect to the filtration $\mathbb{F}^{\mathscr{N}}$, hence from \eqref{2.15} we are done.
\end{proof}
\end{appendix}

\bibliographystyle{abbrvnat}
\bibliography{bibliography}

\end{document}